\documentclass[11pt,letterpaper]{amsart}

\usepackage[utf8]{inputenc}
\usepackage[T1]{fontenc}
\usepackage{lmodern}
\usepackage[english]{babel}
\usepackage{microtype}
\usepackage[toc]{appendix}

\usepackage{amsmath,amssymb,amsfonts,amsthm,esint}
\usepackage{mathtools,accents}
\usepackage{mathrsfs}
\usepackage{aliascnt}
\usepackage{braket}
\usepackage{bm}
\usepackage{ulem}

\usepackage[a4paper,margin=3cm]{geometry}
\usepackage[citecolor=blue,colorlinks]{hyperref}

\usepackage{enumerate}
\usepackage{xcolor}

\makeatletter
\g@addto@macro\@floatboxreset\centering
\makeatother

\makeatletter
\def\newaliasedtheorem#1[#2]#3{
	\newaliascnt{#1@alt}{#2}
	\newtheorem{#1}[#1@alt]{#3}
	\expandafter\newcommand\csname #1@altname\endcsname{#3}
}
\makeatother

\numberwithin{equation}{section}

\newtheoremstyle{slanted}{\topsep}{\topsep}{\slshape}{}{\bfseries}{.}{.5em}{}

\theoremstyle{plain}
\newtheorem{theorem}{Theorem}[section]
\newaliasedtheorem{proposition}[theorem]{Proposition}
\newaliasedtheorem{lemma}[theorem]{Lemma}
\newaliasedtheorem{corollary}[theorem]{Corollary}
\newaliasedtheorem{counterexample}[theorem]{Counterexample}

\theoremstyle{definition}
\newaliasedtheorem{definition}[theorem]{Definition}
\newaliasedtheorem{question}[theorem]{Question}
\newaliasedtheorem{openquestion}[theorem]{Open Question}
\newaliasedtheorem{conjecture}[theorem]{Conjecture}

\theoremstyle{remark}
\newaliasedtheorem{remark}[theorem]{Remark}
\newaliasedtheorem{example}[theorem]{Example}

\newcommand{\setN}{\mathbb{N}}
\newcommand{\setQ}{\mathbb{Q}}
\newcommand{\setR}{\mathbb{R}}

\newcommand{\eps}{\varepsilon}

\let\altphi\phi
\let\phi\varphi
\let\varphi\altphi
\let\altphi\undefined

\newcommand{\abs}[1]{\left\lvert#1\right\rvert}
\newcommand{\norm}[1]{\left\lVert#1\right\rVert}

\newcommand{\weakto}{\rightharpoonup}

\newcommand{\Id}{\mathrm{Id}}

\let\div\undefined
\DeclareMathOperator{\div}{div}
\DeclareMathOperator{\Hess}{Hess}
\DeclareMathOperator{\sign}{sign}

\newcommand{\di}{\mathop{}\!\mathrm{d}}

\newcommand{\bs}{{\rm bs}}
\newcommand{\loc}{{\rm loc}}

\newcommand{\res}{\mathop{\hbox{\vrule height 7pt width .5pt depth 0pt
			\vrule height .5pt width 6pt depth 0pt}}\nolimits}

\DeclareMathOperator{\supp}{supp}

\newcommand{\Ch}{{\sf Ch}}

\DeclareMathOperator{\LipConst}{Lip}
\DeclareMathOperator{\Lip}{LIP}
\DeclareMathOperator{\Lipb}{LIP_b}
\DeclareMathOperator{\Lipbs}{LIP_\bs}
\DeclareMathOperator{\Cb}{C_b}

\DeclareMathOperator{\Cbs}{C_\bs}

\DeclareMathOperator{\lip}{lip} 
\DeclareMathOperator{\BV}{BV}
\DeclareMathOperator{\Per}{Per}
\DeclareMathOperator{\Ric}{Ric}

\newcommand{\haus}{\mathscr{H}}
\newcommand{\Leb}{\mathscr{L}}

\newcommand{\cK}{\mathcal{K}}
\newcommand{\ft}{\mathrm{t}}

\newcommand{\XX}{{\boldsymbol{X}}}

\newcommand{\dist}{\mathsf{d}}

\newcommand{\meas}{\mathfrak{m}}

\newcommand{\Test}{{\rm Test}}

\DeclareMathOperator{\CD}{CD}
\DeclareMathOperator{\RCD}{RCD}

\newfont{\tmpf}{cmsy10 scaled 2500}

\DeclareMathOperator{\Tan}{Tan}

\def\XXint#1#2#3{{\setbox0=\hbox{$#1{#2#3}{\int}$ }
		\vcenter{\hbox{$#2#3$ }}\kern-.6\wd0}}

\begin{document}

\title[
Minimal boundaries in $\RCD(K,N)$ spaces
]
{
Weak Laplacian bounds and minimal boundaries in non-smooth spaces with Ricci curvature lower bounds
}

\author{Andrea Mondino}\thanks{Andrea Mondino: Mathematical Institute, University of Oxford, UK, \\email: Andrea.Mondino@maths.ox.ac.uk} 
\author{Daniele Semola}\thanks{Daniele Semola: Mathematical Institute, University of Oxford, UK, \\email: daniele.semola.math@gmail.com}

\maketitle

\begin{abstract}
The goal of the paper is four-fold. In the setting of spaces with synthetic Ricci curvature lower bounds  (more precisely  $\RCD(K,N)$ metric measure spaces):
\begin{itemize}
\item we develop an intrinsic theory of Laplacian bounds in viscosity sense and in a pointwise, heat flow related sense, showing their equivalence also with Laplacian bounds in distributional sense;
 \item relying on these tools, we establish a PDE principle relating lower Ricci curvature bounds to the preservation of Laplacian lower bounds under the evolution via the $p$-Hopf-Lax semigroup, for general exponents $p\in[1,\infty)$. The principle admits a broad range of applications, going much beyond the topic of the present paper;
\item we prove sharp Laplacian bounds on the distance function from a set (locally) minimizing the perimeter with a flexible technique, not involving any regularity theory; this corresponds to vanishing mean curvature in the smooth setting and encodes also information about the second variation of the area;
\item we initiate a regularity theory for boundaries of sets (locally) minimizing the perimeter, obtaining sharp dimension estimates for their singular sets, quantitative estimates of independent interest even in the smooth setting and topological regularity away from the singular set.
\end{itemize}
The class of $\RCD(K,N)$ metric measure spaces includes as remarkable sub-classes: measured Gromov-Hausdorff limits of smooth manifolds with lower Ricci curvature bounds and  finite dimensional Alexandrov spaces with lower sectional curvature bounds. Most of the results are new also in these frameworks. Moreover, the tools that we develop here have applications to classical questions in Geometric Analysis on smooth, non compact Riemannian manifolds with lower Ricci curvature bounds.
\end{abstract}

\tableofcontents

\section{Introduction}

Minimal surfaces constitute a fascinating research topic across Analysis and Geometry, with strong connections with Topology and Mathematical Physics. Even if the field is extremely rich in results and techniques, arguably two cornerstones in the theory of minimal surfaces in Riemannian manifolds are:
\begin{itemize}
\item the regularity theory, asserting that a minimal surface is smooth away from a small (in the sense of Hausdorff dimension) singular set;
\item the first and second variations formulae, encoding at a differential level the fact that a minimal surface is a stationary point (resp. a local minimum or a min-max critical point) of the area functional.
\end{itemize}
Classically, the regularity theory is established for minimal surfaces in Euclidean ambient spaces, and then transplanted to the smooth curved setting of Riemannian manifolds by using local coordinates or Nash embedding theorem. While on the one hand this procedure gives sharp \textit{qualitative} regularity results (such as the dimension of the singular set),  on the other hand it is not completely satisfactory in terms of \textit{effective estimates}, which usually depend on quantities like the injectivity radius or the full Riemann curvature tensor.

A natural question (raised for instance in Gromov's lectures \cite[pp. 334-335]{Gromov21}) is to which extent one can develop a theory for minimal surfaces if the ambient space is non-smooth.
In the case of 2-dimensional minimal surfaces in (suitable) metric spaces, there has been recent progress by Lytchak-Wenger \cite{LW17, LW18, LW20} who successfully studied the Plateau problem together with geometric applications.

The aim of the present paper is to investigate the higher dimensional case of minimal boundaries in possibly non-smooth finite dimensional ambient spaces, satisfying Ricci curvature lower bounds in a synthetic sense. More precisely, the framework for the ambient space throughout the paper is the one of  $\RCD(K,N)$ metric measure spaces, for finite $N\in [1,\infty)$ and $K\in \setR$ (see \autoref{subsec:generalRCD} for a quick introduction and relevant bibliography). Here $K\in \setR$ plays the role of (synthetic) lower bound on the Ricci curvature and $N\in [1,\infty)$ plays the role of (synthetic) upper bound on the dimension. This class includes measured Gromov-Hausdorff limits of smooth manifolds with  Ricci curvature lower bounds (see \cite{CheegerColding97,CheegerColding2000a,CheegerColding2000b,CheegerNaber13,CheegerJiangNaber21}) and finite dimensional Alexandrov spaces with sectional curvature lower bounds (see \cite{BuragoGromovPerelman92,Petrunin}). Most of our results are new also in these more classical settings.
\medskip

The goal of the paper is four-fold. In the aforementioned setting of (possibly non-smooth)  $\RCD(K,N)$ metric measure spaces:
\begin{itemize}
\item we develop an intrinsic theory of Laplacian bounds in viscosity sense and in a pointwise, heat flow related sense, showing their equivalence also with  Laplacian bounds in distributional and (various) comparison senses;
\item we establish a PDE principle relating lower Ricci curvature bounds to the preservation of Laplacian lower bounds under the evolution via the $p$-Hopf-Lax semigroup, for general exponents $p\in[1,\infty)$;
\item we prove sharp Laplacian bounds on the distance function from a set (locally) minimizing the perimeter; this corresponds to vanishing mean curvature in the smooth setting (i.e. the first variation formula), encoding at the same time information about the \textit{second variation of the area along equidistant sets}. This is achieved with a flexible technique, independent of any regularity theory and applicable to solutions of different variational problems;
\item we initiate a regularity theory for boundaries of sets (locally) minimizing the perimeter, obtaining sharp Hausdorff dimension bounds for the singular set, Minkowski bounds, and topological regularity in a neighbourhood of the regular set (i.e., where the tangent is flat half-space).
\end{itemize}

Besides the deep theoretical interest towards developing Geometric Measure Theory under curvature bounds in a non smooth setting, the tools that we develop here find applications in the study of classical questions in Geometric Analysis on smooth non compact Riemannian manifolds with lower Ricci bounds, see for instance \cite{AntonelliPasqualettoPozzettaSemola22,Ding21c}.
In particular, due to the compactness and stability of $\RCD(K,N)$ spaces and to the stability of minimal boundaries, the aforementioned fourth goal  is a step towards an effective theory of minimal boundaries under lower Ricci curvature bounds, not depending on additional assumptions such as lower bounds on the injectivity radius or full Riemann curvature bounds.

We next illustrate the main results of the paper.

\subsection{Mean curvature bounds and minimal boundaries in a non smooth setting}\label{subsec:overview}

The subject of our study will be sets of finite perimeter that locally minimize the perimeter, according to the following.

\begin{definition}\label{def:locperIntro}
Let $(X,\dist,\meas)$ be an $\RCD(K,N)$ metric measure space and let $\Omega\subset X$ be an open domain. Let $E\subset X$ be a set of locally finite perimeter. We say that $E$ is \textit{locally perimeter minimizing in $\Omega$} if for any $x\in \Omega$ there exists $r_x>0$ such that $E$ minimizes the perimeter among all the perturbations that are compactly supported in $B_{r_x}(x)$, i.e. for any Borel set $F$ such that $E\Delta F\subset B_{r_x}(x)$ it holds
\begin{equation*}
\Per(E,B_{r_x}(x))\le \Per(F,B_{r_x}(x))\, .
\end{equation*}
\end{definition}

The above is a very general condition. For instance, smooth minimal hypersurfaces in Riemannian manifolds are locally boundaries of locally perimeter minimizing sets according to \autoref{def:locperIntro}, even though, in general, they do not minimize the perimeter among arbitrarily compactly supported variations. A simple example in this regard is the equator inside the sphere.

Let us define the comparison function $\ft_{K,N}: I_{K,N}\to \setR$ as
\begin{equation}\label{eq:deftKN}
\begin{split}
\ft_{K,N}(x)&:=
\begin{cases}
- \sqrt{ K(N-1)} \tan\big(\sqrt{\frac{K} {N-1}} x \big)\,& \quad \text{if } K>0\\
\quad 0\, & \quad \text{if } K=0 \\
\sqrt{-K(N-1)} \tanh\big(\sqrt{\frac{-K} {N-1}} x \big)\,& \quad \text{if } K<0 \;, 
\end{cases} 
\\
I_{K,N}&:=
\begin{cases}
\big( - \frac{\pi}{2} \sqrt{\frac{N-1}{K}},  \frac{\pi}{2} \sqrt{\frac{N-1}{K}}  \big)\,& \quad \text{if } K>0\\
\quad \setR \,& \quad \text{if } K\leq 0 \; .
\end{cases} 
\end{split}
\end{equation}

A fundamental property of locally area minimizing hypersurfaces in a smooth Riemannian manifold is that their mean curvature vanishes. Our first main result is the following sharp Laplacian comparison for the distance from a locally perimeter minimizing boundary. It shall be thought as a \textit{global and non smooth} counterpart of the smooth fact that the mean curvature vanishes for sets locally minimizing the perimeter.  

\begin{theorem}[\autoref{thm:meancurvminimal1}]\label{thm:meancurvminimal1Intro}
Let $(X,\dist,\haus^N)$ be an $\RCD(K,N)$ metric measure space. Let $E\subset X$ be a set of locally finite perimeter and assume that it is a local perimeter minimizer. Let $\dist_{\overline{E}}:X\setminus \overline{E}\to[0,\infty)$ be the distance function from $\overline{E}$. Then
\begin{equation}\label{eq:meanglobIntro}
\Delta \dist_{\overline{E}}\le \ft_{K,N} \circ \dist_{\overline{E}}\, \quad\text{on $X\setminus\overline{E}$}\, ,
\end{equation}
where $ \ft_{K,N}$ is defined in \eqref{eq:deftKN}.
\end{theorem}

\begin{remark}[How to interpret the Laplacian bounds]\label{rem:LaplacianBouundsDist}
The Laplacian bounds \eqref{eq:meanglobIntro} have to be intended in any of the equivalent ways stated in  \autoref{thm:mainequivlapla}, i.e. either in the viscosity, distributional, heat flow, or comparison senses (see \autoref{sec:LapBoundIntro} later in the introduction for an outline of the various notions). 
\end{remark}

\begin{remark}\label{rem:SharpUBLapIntro}
The upper bound \eqref{eq:meanglobIntro} is sharp already in the class of smooth Riemannian manifolds with Ricci curvature bounded below by $K\in \setR$ and dimension equal to $N\in \setN, N\geq 2$. Indeed, it is easily seen that:
\begin{itemize}
\item Case $K>0$. The distance function from a equatorial hyper-sphere inside the $N$-dimensional sphere of constant sectional curvature $K/(N-1)$ achieves equality in \eqref{eq:meanglobIntro}.
\item Case $K=0$. The distance function from a hyperplane in $\setR^{N}$ is harmonic, and thus  achieves equality in \eqref{eq:meanglobIntro}.
\item Case $K<0$. The distance function from  a horosphere inside the $N$-dimensional hyperbolic space of constant sectional curvature $K/(N-1)$ achieves equality in \eqref{eq:meanglobIntro}.
\end{itemize}
\end{remark}

Encoding mean curvature bounds through the Laplacian of the distance function as in \eqref{eq:meanglobIntro} is equivalent to the classical vanishing mean curvature condition for smooth hypersurfaces on Riemannian manifolds. Moreover, according to \cite{Wu79,Gromov91}, this is the right way to look at mean curvature bounds, having in mind the perspective of global differential geometry. As we shall explain, \eqref{eq:meanglobIntro} also encodes the information about the second variation of the perimeter on equidistant sets from $\overline{E}$ usually obtained with the second variation formula for the perimeter.

Let us mention that some proposals  of weak  notions of mean curvature bounds in the non-smooth setting have been put forward in \cite[Section 5]{Ketterer20} and \cite[Section 5.1]{CavallettiMondinoLorentz} by using localisation (also called needle decomposition) techniques.  Compared to such proposals, the remarkable advantage of the approach via Laplacian comparison \eqref{thm:meancurvminimal1Intro}, and key new point of the present work, is that we establish mean curvature bounds for solutions of variational problems, such as local perimeter minimizers. This makes the new tools very powerful for geometric applications.
\medskip

\autoref{thm:meancurvminimal1Intro} is  new even for Alexandrov spaces with sectional curvature bounded from below and for Ricci limit spaces. The proof is independent of the regularity theory for minimal boundaries and it avoids the first variation formula for the perimeter. Hence it is different from those present in the literature also when read on smooth Riemannian manifolds. Moreover, the technique that we develop here is flexible and can be applied to solutions of more general variational problems as the isoperimetric one, see \cite{AntonelliPasqualettoPozzettaSemola22}.\\
We remark that it is much simpler to prove the sharp Laplacian comparison for minimal boundaries inside Ricci limit spaces that can be obtained as limits of minimizing boundaries in smooth Riemannian manifolds with Ricci curvature uniformly bounded from below, essentially by passing to the limit the analogous statements for smooth manifolds. This assumption, however, would largely restrict the set of applications with respect to \autoref{thm:meancurvminimal1Intro}.

Extensions of some classical theorems in Riemannian Geometry such as Frankel's theorem \cite{Frankel66} about intersecting minimal hypersurfaces on closed manifolds with positive Ricci curvature and Simons' theorem \cite{Simons} about the non-existence of two-sided area-minimizing hypersurfaces on closed manifolds with positive Ricci curvature will follow as corollaries (see \autoref{thm:generalizedFrankel} and  \autoref{cor:SimonsThmRCD}), thus confirming the strength of this approach.\\
Moreover, \autoref{thm:meancurvminimal1Intro} plays a key role in the regularity theory, for instance in establishing Minkowski-type bounds on the singular set (see \autoref{thm:contestsing}).

\subsection{A regularity theory for minimal boundaries on $\RCD$ spaces}\label{ss:RegIntro}

A second main goal of this paper is to initiate the regularity theory of minimal boundaries on $\RCD$ spaces.  
This can be seen as a step towards an effective regularity theory for minimal hypersurfaces under lower Ricci bounds, where by effective we mean only depending on the ambient  Ricci curvature and volume lower bounds  (and independent of extra assumptions such as injectivity radius, or bounds on the full Riemann curvature tensor).

\medskip
Our first result in this direction is an $\eps$-regularity theorem in the spirit of De Giorgi's regularity theory for Euclidean minimal boundaries \cite{DeGiorgi61} and of the volume $\eps$-regularity theorem for manifolds with lower Ricci bounds originally due to Cheeger-Colding \cite{Colding97,CheegerColding97} (see \autoref{thm:epsregcolding} below).

\begin{definition}[$\eps$-regular points]\label{def:epsregularintro}
Let $\eps>0$. If $(X,\dist,\haus^N)$ is an $\RCD(-\eps,N)$ metric measure space and $E\subset X$ is a set of finite perimeter, minimizing the perimeter in $B_2(x)\subset X$, such that:
\begin{itemize}
\item[(i)] the ball $B_2(x)\subset X$ is $\eps$-GH close to the ball $B_2(0)\subset\setR^N$;
\item[(ii)] $E$ is $\eps$-close on $B_2(x)$ in the $L^1$ topology to $\{t<0\}\subset \setR^N$ and $\partial E\cap B_2(x)$ is $\eps$-GH close to $\{t=0\}\cap B_2(0)\subset\setR^N$, where we denoted by $t$ one of the canonical coordinates on $\setR^N$;
\end{itemize}
then we shall say that $E$ is $\eps$-regular at $x$ in $B_2(x)$.

The notion of  $\eps$-regular at $x$  in $B_r(x)$ can be introduced analogously by scaling.
\end{definition}

Notice that, as we prove in \autoref{thm:closurecompactnesstheorem}, $L^1$-convergence of perimeter minimizing open sets automatically self-improves to Hausdorff convergence of their boundaries in this setting.

\begin{theorem}[$\eps$-regularity]\label{thm:epsregularityintro}
Let $N>1$ be fixed. For any $\eps>0$ there exists $\delta=\delta(\eps,N)>0$ such that the following holds.
If $(X,\dist,\haus^N)$ is an $\RCD(-\delta,N)$ metric measure space, $E\subset X$ is perimeter minimizing on $B_4(x)\subset X$ and $E$ is $\delta$-regular in $B_2(x)$ then, for any $y\in \partial E\cap B_1(x)$ and any $0<r<1$, $E$ is $\eps r$-regular in $B_r(y)$.

Moreover, for any $0<\alpha<1$, there exists $\delta=\delta(\alpha,N)>0$ such that, if $x\in \partial E$ and $E$ is $\delta$-regular at $x$ on $B_2(x)$, then $\partial E\cap B_1(x)$ is $C^{\alpha}$-homeomorphic to the ball $B_1(0)\subset \setR^{N-1}$.
\end{theorem}

The uniform Reifenberg flatness of minimal boundaries on sequences of smooth manifolds converging in the Lipschitz sense  had been previously considered by Gromov in \cite{Gromov96,Gromov14b}. Here we remove the smoothness assumption, we rely only on the synthetic Ricci curvature lower bounds, and we relax the notion of closeness for the ambient spaces to Gromov-Hausdorff. This has the effect of largely broadening the set of possible applications, thanks to the well known precompactness of spaces with lower Ricci and upper dimension bounds in Gromov-Hausdorff sense and to the well established regularity theory for ambient spaces.\\
The main new idea that we introduce for the proof of \autoref{thm:epsregularityintro} is very robust. The same technique applies to general variational problems in the setting of spaces with lower Ricci bounds, as soon as there are enough stability and an $\eps$-regularity theorem with gap for the analogous problem in the Euclidean setting, see \autoref{rm:densitygapEuclidean} for the precise statement.
\medskip

\autoref{thm:epsregularityintro} 
is the building block to prove that the boundary of a  locally-perimeter-minimizing set is a  topological manifold away from sets of ambient codimention three. A difficulty, which is absent in the Euclidean theory, is that we need to control simultaneously the flatness of the ambient and the flatness of the hypersurface inside it.

\begin{theorem}[\autoref{thm:SummaryOa}]\label{thm:topregawayintro}
Let $(X,\dist,\haus^N)$ be an $\RCD(K,N)$ metric measure space. Let $E\subset X$ be a set of finite perimeter. Assume that $E$ is perimeter minimizing in $B_2(x)\subset X$ and $B_2(x)\cap \partial X=\emptyset$. 
Then, letting $\mathcal{S}^E$ be the set of singular boundary points of $\partial E$, i.e. those points where there exists a blow-up which is not a flat Euclidean half-space, it holds
\begin{equation}\label{hdim}
\dim_H(\mathcal{S}^E\cap B_2(x))\le N-3\, .
\end{equation}
Moreover, for any $0<\alpha<1$ there exists a relatively open set
\begin{equation*}
O_{\alpha}\subset \partial E\cap B_1(x)
\end{equation*}
such that 
\begin{itemize}
\item $\big(\partial E\setminus \mathcal{S}^E\big) \cap B_1(x) \subset O_{\alpha}\,$;  hence, in particular,  $\dim_H\big( (\partial E\setminus O_{\alpha}) \cap B_1(x)\big)\le N-3$;
\item $O_{\alpha}$ is $C^{\alpha}$-biH\"older homeomorphic to an $(N-1)$-dimensional open manifold.
\end{itemize}
\end{theorem}
\medskip

\noindent
Additionally, in \autoref{thm:sharpdimreg} we will prove a sharp dimension estimate 
\begin{equation}
\dim_H\left(\mathcal{S}^E\cap\mathcal{R}(X)\right)\le N-8\,,
\end{equation}
for the intersection of the singular set of the minimal boundary with the regular set $\mathcal{R}(X)$ of the ambient space.

\begin{remark}
The Hausdorff dimension estimate \eqref{hdim} is sharp in this context, as elementary examples illustrate (see \autoref{rem:ExsharpCodim3} and \autoref{rem:sharpThmSum}). It will be obtained through the classical dimension reduction pattern, but several new difficulties arise, due to the non smoothness of the ambient space (for instance it is not clear whether the classical monotonicity formula for minimal surfaces holds in such a general framework).

The $C^{0,\alpha}$ regularity of the manifold $O_{\alpha}$ containing the regular set matches the (currently known) regularity of the regular part $\mathcal{R}(X)$ of the ambient space $X$ (after Cheeger-Colding's metric Reifenberg Theorem  \cite[Appendix 1]{CheegerColding97}  and \cite{KapovitchMondino21}). Higher regularity of  $\partial E\setminus \mathcal{S}^E$ (e.g. contained in a Lipschitz manifold), would require first improving the structure theory of the ambient space.
\end{remark}

In \autoref{thm:contestsing} we will also obtain a Minkowski estimate for the quantitative singular sets of minimal boundaries in this framework, in the spirit of \cite{CheegerNaber13,CheegerNaber13b,NaberValtorta20}. The estimate has independent interest and it is new also for smooth manifolds with lower Ricci curvature and volume bounds (see \autoref{rem:RegNewSmooth}).\footnote{In \cite{Ding21a}, which appeared on the arXiv the day before the appearance of the present paper, Q. Ding has independently proved the first part of \autoref{thm:epsregularityintro} and the Hausdorff dimension estimate \eqref{hdim} under the additional assumption that the minimal boundary is a limit of minimal boundaries along a sequence of smooth manifolds with Ricci curvature and volume of unit balls uniformly bounded from below. These results played a fundamental role in the subsequent proof of the Poincar\'e inequality for minimal graphs over smooth manifolds with nonnegative Ricci curvature and Euclidean volume growth and of generalized versions of Bernstein's theorem in \cite{Ding21c}.}

\subsection{Outline of the strategy to establish the Laplacian bounds of \autoref{thm:meancurvminimal1Intro}}\label{ss:OutlineMCIntro}

On smooth Riemannian manifolds, minimal surfaces are critical points of the area functional. A key technical tool for this definition is the first variation formula.\\ 
For the sake of this presentation, let us focus on sets of finite perimeter in Euclidean ambient spaces. Given any such set $E\subset\setR^n$ and any smooth vector field $\XX$ with compact support in $B_r(x)\subset \setR^n$, we can consider the induced flow of diffeomorphisms $(\Phi_t)_{t\in(-\eps,\eps)}$ such that $\Phi_{0}=\Id$. Then
\begin{equation}\label{eq:firstvar}
\left.\frac{\di}{\di t}\right|_{t=0}\Per(\Phi_t(E)\cap B_r(x))=\int_{\mathcal{F}E\cap B_r(x)}\div_E\XX\di\Per_E\, ,
\end{equation} 
where $\div_E$ denotes the tangential divergence, $\mathcal{F}E$ denotes the so-called reduced boundary of the set of finite perimeter $E$ and $\Per_E$ its perimeter measure. When $E$ is an open set with smooth boundary, $\mathcal{F}E$ coincides with the topological boundary and $\Per_E$ is the surface measure, see \cite{Giusti84,Maggi12}.\\
If $E$ is locally perimeter minimizing, then a deep regularity result originally due to De Giorgi \cite{DeGiorgi61}  and refined by Federer (after work of Simons)  is that $\mathcal{F}E$ is smooth  and $\partial E\setminus \mathcal{F}E $ has ambient codimension $8$; moreover, \eqref{eq:firstvar} implies that  the  classical mean curvature vanishes on $\mathcal{F}E$.
\medskip

It is often advocated that Ricci curvature governs the distortion of volumes on a smooth Riemannian manifold. Indeed, it enters into the variation formula for the area element of the equidistant sets from a given smooth hypersurface, see \cite{HeintzeKarcher78,Gromov91}. If we consider a smooth minimal hypersurface, the first derivative of the area element of equidistant surfaces vanishes at $t=0$, moreover  the Ricci curvature in normal direction and the second fundamental form enter into the expression for the second derivative.\\
There are two main drawbacks of this approach: it only looks at the infinitesimal geometry near to the hypersurface and it requires smoothness, while usually minimal hypersurfaces are built through variational methods and global regularity is not guaranteed.
\medskip

Focusing on the first issue, it is possible to switch from an infinitesimal to a global perspective. If $\Sigma\subset M$ is a smooth minimal hypersurface inside a smooth Riemannian manifold with non-negative Ricci curvature, then the distance function $\dist_{\Sigma}$ is superharmonic on $M\setminus\Sigma$, see \cite{Wu79} and \autoref{subsec:othernotions}. 
This is a remarkable observation for the sake of developing an analogous theory on metric measure spaces, since it avoids the necessity of giving a meaning to the mean curvature of a hypersurface.
\medskip

Let us recall a classical argument  \cite{Gromov07} to deal with the aforementioned regularity issue in the setting of smooth Riemannian manifolds that was key  in the proof of the L\'evy-Gromov isoperimetric inequality.
The fundamental observation is that in order to bound the Laplacian of the distance function, minimality (in the stricter sense of local area minimizing) was only needed at footpoints of minimizing geodesics on the hypersurface itself. In various situations, deep regularity theorems (\cite{DeGiorgi61,Almgren76}) guarantee that minimal hypersurfaces are smooth in a neighbourhood of these points and the classical arguments can then be applied.

Given our current knowledge of $\RCD$ spaces, there is little hope that such an approach could prove \autoref{thm:meancurvminimal1Intro}: there is no first variation formula as \eqref{thm:meancurvminimal1Intro} available at the moment and, even more dramatically, the classical regularity theorems do not make sense in this non-smooth setting. The L\'evy-Gromov isoperimetric inequality has been generalized to the present framework in \cite{CavallettiMondino17}, avoiding the analysis of the mean curvature of isoperimetric sets (see also \cite{Klartag17}, dealing with smooth Riemannian manifolds). However, a sharper understanding of mean curvature bounds for solutions of variational problems is definitely needed for more refined developments.
\medskip

In \cite{CaffarelliCordoba93}, a different proof of the vanishing of the mean curvature for local minimizers of the perimeter functional was obtained in the Euclidean setting. It does not rely on the regularity theory for area minimizers nor on the first variation formula, rather, it follows the pattern of viscosity theory in partial differential equations. The possibility of following a similar pattern to prove the L\'evy-Gromov isoperimetric inequality on Alexandrov spaces was pointed out later in the research announcement \cite{Petrunin03}, together with the key remark that the sup-convolution could act as a counterpart of the more classical slicing with quadratic polynomials of the viscosity theory.
\medskip

Below, we outline the strategy that we will follow, inspired by \cite{CaffarelliCordoba93} and \cite{Petrunin03}, neglecting some of the regularity issues. 
\medskip

Consider a locally area minimizing hypersurface $\Sigma\subset\setR^n$, and assume that it is the boundary of a smooth domain $D$, locally minimizing the surface measure among all compactly supported perturbations.

Let $\dist_{\Sigma}: X\to [0,\infty)$ be the distance function from $\Sigma$, defined by:
\begin{equation*}
\dist_{\Sigma}(x):=\inf\{ \dist(x,y) \, :\, y\in \Sigma\}.
\end{equation*} 
We wish to prove that $\Delta\dist_{\Sigma}\le 0$  in the viscous sense on $\setR^n\setminus\Sigma$.
Let us suppose that this is not the case. Then there exist $x\in\setR^n\setminus\Sigma$ and a smooth function $\phi:U_x\to\setR$ such that
\begin{equation}\label{eq:lap>eps}
\Delta \phi(x)\ge \eps>0\, ,\quad \phi(x)=\dist_{\Sigma}(x)\, ,\quad \phi\le \dist_{\Sigma}\, .
\end{equation} 
Let us extend $\phi$ to a globally defined function $\hat{\phi}:\setR^n\to\setR$ such that $\hat{\phi}\le\dist_{\Sigma}$. Then we introduce $\tilde{\phi}:\setR^n\to\setR$ by
\begin{equation*}
\tilde{\phi}(y):=\max_{z\in\setR^n}\{\hat{\phi}(z)-\dist(z,y)\}\, .
\end{equation*}
The properties of $\tilde{\phi}$ that will be relevant for our purposes are the following:
\begin{itemize}
\item[(i)] $\tilde{\phi}$ is a $1$-Lipschitz map;
\item[(ii)] $\tilde{\phi}\le\dist_{\Sigma}$; 
\item[(iii)] let us denote by $x_{\Sigma}$ one of the footpoints of $x$ on $\Sigma$. Then $\tilde{\phi}=\dist_{\Sigma}$ along the minimal geodesic connecting $x$ to $x_{\Sigma}$;
\item[(iv)] suppose for simplicity that $x_{\Sigma}$ is the unique footpoint of $x$ on $\Sigma$. Then $\tilde{\phi}<\dist_{\Sigma}$ outside from the minimal geodesic connecting $x$ to $x_{\Sigma}$. Moreover, there is a neighbourhood $U_{x_{\Sigma}}$ of $x_{\Sigma}$ such that the maximum defining $\tilde{\phi}$ is achieved at points in a neighbourhood $U_x$ of $x$ for any $y\in U_{x_{\Sigma}}$;
\item[(v)] as a first consequence of (iv), $\abs{\nabla \tilde\phi}=1$ almost everywhere in $U_{x_{\Sigma}}$;
\item[(vi)] as a second consequence of (iv),
\begin{equation}\label{eq:laplatildeeps}
\Delta\tilde{\phi}\ge \eps'>0\, ,
\end{equation}
in the sense of distributions on $U_{x_{\Sigma}}$.
\end{itemize}
Property (vi) above is a consequence of the completely non trivial fact that the transform mapping $\phi$ into $\tilde{\phi}$ preserves, in a suitable sense, Laplacian lower bounds. We shall focus more in detail later on this fact.
\medskip

Let us see how to combine the ingredients above to reach a contradiction with the assumption that $\Sigma$ is a locally area minimizing surface.

Suppose that $\tilde{\phi}$ is also smooth in a neighbourhood of $x_{\Sigma}$ and let us cut the original surface $\Sigma$ along the level sets of $\tilde{\phi}$. By (ii), (iii) and (iv) above we obtain a family of compactly supported perturbations $\Sigma_t$, $t\in[0,\delta)$ of $\Sigma=\Sigma_0$ in this way. We claim that, for some $t\in[0,\eps)$, $\Sigma_t$ has area smaller than $\Sigma$. 

Let $\Omega_t$ be the region bounded between $\Sigma$ and $\Sigma_t$. The boundary $\partial\Omega_t$ is made of two components, one along $\Sigma$, denoted by $\Sigma_{old}$, and one along $\Sigma_t$, denoted by $\Sigma_{new}$. Then we can compute: 
\begin{align*}
0<\,  \int_{\Omega_t}\Delta\tilde{\phi}=\, &- \int_{\Sigma_{old}}\nabla\tilde{\phi}\cdot\nu_{\Sigma_{old}}\di \haus^{n-1}+\int_{\Sigma_{new}}\nabla\tilde{\phi}\cdot\nu_{\Sigma_{new}}\di \haus^{n-1}\\
=\, & -\int_{\Sigma_{old}}\nabla\tilde{\phi}\cdot\nu_{\Sigma_{old}}\di \haus^{n-1}-\haus^{n-1}(\Sigma_{new})\\
\le\,  & \haus^{n-1}(\Sigma_{old})-\haus^{n-1}(\Sigma_{new})\, .
\end{align*}
Above, the first inequality follows from (vi), the first identity follows from the Gauss-Green formula, the second one from the fact that $\Sigma_{new}$ is along the level hypersurface of $\tilde{\phi}$ therefore (taking into account also (v)) we have $-\nu_{\Sigma_{new}}=\nabla \tilde{\phi}$. The last inequality follows from (i), which guarantees in turn that 
\begin{equation*}
\abs{\nabla\tilde{\phi}\cdot\nu_{\Sigma_{old}}}\le 1\, .
\end{equation*}
Hence
\begin{equation*}
 \haus^{n-1}(\Sigma_{old})-\haus^{n-1}(\Sigma_{new})>0\, , 
 \end{equation*}
 contradicting the local minimality of $\Sigma$.
\medskip

Let us now comment on the main steps in the formal argument above.
\medskip
\begin{itemize}
\item We will deal with sets of finite perimeter: their boundaries provide a weak notion of codimension one hypersurface suitable for compactness and stability arguments. The Euclidean theory was developed in the 50's and later partially generalized to metric measure spaces in \cite{Am01,Am02}.
In the framework of $\RCD$ spaces they are quite well understood after \cite{AmbrosioBrueSemola19,BruePasqualettoSemola19,BruePasqualettoSemola21}.\\ 
This class is very natural to consider. Indeed, we recall that the classical regularity theory for area minimizing surfaces in codimension one was built on top of the regularity theory for minimal boundaries. 
\item In order to exploit the variational structure of the problem in the contradiction argument we rely on the \textit{viscous} perspective, while for the sake of applying the Gauss-Green theorem it is important to understand Laplacian bounds in the \textit{sense of distributions}. To this aim, we are going to develop a theory of Laplacian bounds in viscous sense on $\RCD(K,N)$ spaces and prove the equivalence with other weak notions of Laplacian bounds, including the distributional one. This part will be used in some of the geometric applications but it is also of independent analytical interest.
\item Conclusion (vi) above is a consequence of a completely non trivial statement about the preservation of Laplacian bounds via sup-convolution in the Euclidean setting. As we shall see, this statement holds, in a suitable sense, also for $\RCD$ spaces and it turns that it characterizes lower Ricci curvature bounds, at least on smooth Riemannian manifolds. 
\end{itemize}

\subsection{Weak notions of Laplacian bounds}\label{sec:LapBoundIntro}
Notions of superharmonicity for non smooth functions and, more in general, a weak theory of bounds for the Laplacian on smooth Riemannian manifolds have been fundamental in the Geometric Analysis of manifolds with lower curvature bounds. In \cite{Calabi58} a global version of the Laplacian comparison theorem was formulated in the sense of barriers; such a barrier formulation played a role also in the proof of the splitting theorem in \cite{CheegerGromoll71}. 
Then a viscous notion of Laplacian bounds was considered in \cite{Wu79} and its equivalence with other notions, such as the distributional one, was studied in \cite{GreeneWu79}. Since then, these different perspectives have played key roles in the theory. We refer for instance to \cite{Andrews15} for a survey of some recent applications of the viscous perspective.
\medskip

In more recent years, some of these weak notions of Laplacian have been necessary for the developments of an analysis on metric (measure) spaces.\\ 
In the first approaches \cite{KinnunenMartio02,Shanmugalingam02} the perspective was variational. This was made possible by the presence of a good notion of modulus of the gradient on metric measure spaces (see \cite{Cheeger99,HeinonenKoskela98}). More recently, on the one hand the point of view of gradient flows came into play in \cite{AmbrosioGigliSavare14a}, also in connection with the heat flow. On the other hand, in \cite{Gigli15} a distributional approach to the Laplacian on metric measure spaces was put forward.
\medskip

All of the theories above were dealing with quite general metric measure spaces. We aim to show that the further regularity of $\RCD(K,N)$ spaces allows to partially fill the gap with the classical Riemannian theory.\\ 
The first contribution in this regard is a theory of viscous bounds for the Laplacian.

\begin{definition}[Viscous bound for the Laplacian]\label{def:viscosityintro}
Let $(X,\dist,\meas)$ be an $\RCD(K,N)$ metric measure space and let $\Omega\subset X$ be an open and bounded domain. Let $f:\Omega\to\setR$ be locally Lipschitz and $\eta\in\Cb(\Omega)$. We say that $\Delta f\le \eta$ in the viscous sense in $\Omega$ if the following holds. For any open domain $\Omega'\Subset\Omega$ and for any test function $\phi:\Omega'\to\setR$ such that 
\begin{itemize}
\item[(i)] $\phi\in D(\Delta, \Omega')$ and $\Delta\phi$ is continuous on $\Omega'$;
\item[(ii)] for some $x\in \Omega'$ it holds 
$\phi(x)=f(x)$ and $\phi(y)\le f(y)$ for any $y\in\Omega'$, $y\neq x$;
\end{itemize}
it holds
\begin{equation*}
\Delta \phi(x)\le \eta(x)\, .
\end{equation*} 
\end{definition}

The starting point for the viscosity theory of PDEs is the observation that a smooth function at a minimum point has vanishing gradient and non-negative Hessian. By tracing the Hessian, it has also non-negative Laplacian (since also the gradient is vanishing, this principle holds true in the weighted Riemannian setting as well).\\ 
For evident reasons, this is a delicate point on metric measure spaces. The first issue is singling out a class of sufficiently smooth functions that is rich enough to make definitions non trivial. The second is that there is no pointwise notion of Hessian available in this setting. Nevertheless we are able to prove the equivalence between viscosity bounds on the Laplacian and distributional bounds.

\begin{theorem}\label{thm:viscoimpldistriintro}
Let $(X,\dist,\haus^N)$ be an $\RCD(K,N)$ metric measure space. Let $\Omega\subset X$ be an open and bounded domain, $f:\Omega\to\setR$ be a Lipschitz function and $\eta:\Omega\to\setR$ be continuous. Then $\Delta f\le \eta$ in the sense of distributions if and only if $\Delta f\le \eta$ in the viscous sense.
\end{theorem}  

The key difficulty discussed above will be circumvented relying on a powerful maximum principle obtained in \cite{ZhangZhu16}, reminiscent of the Omori-Yau and Jensen's maximum principles.\\
To prove that - at a minimum point of a sufficiently regular function - the Laplacian is non-negative, we will build a family of auxiliary functions playing the role of the distance squared in the Euclidean setting, i.e. sufficiently regular, with a strict minimum at a prescribed point and with non-negative Laplacian. This construction, of independent interest, is based on the study of the local Green function of the Laplacian on domains.
\medskip

As we already remarked, the connection between the heat flow and the distributional Laplacian is classical, see for instance \cite{AmbrosioGigliSavare14a,GrygoryanHu14,Gigli15}. Another contribution of the paper will be the proposal and the analysis of a new approach to Laplacian bounds, based on the \textit{pointwise} short time behaviour of the heat flow.\\
For a smooth function $f$ on a (compact and possibly weighted) Riemannian manifold,
\begin{equation}\label{eq:pointheatintro}
P_tf(x)=f(x)+t\Delta f(x)+o(t^2)\, ,  \quad\text{as $t\to 0$}\, .
\end{equation} 
Then we propose the following.
\begin{definition}\label{def:heatlaplboundsintro}
Let $(X,\dist,\meas)$ be an $\RCD(K,N)$ metric measure space and let $\Omega\subset X$ be an open and bounded domain. Let $f:\Omega\to\setR$ be a Lipschitz function and let $\eta\in\Cb(\Omega)$. We say that $\Delta f\le \eta$ on $\Omega$ in the \textit{heat flow sense} if the following holds. For any $\Omega'\Subset\Omega$ and any function $\tilde{f}:X\to\setR$ extending $f$ from $\Omega'$ to $X$ and with polynomial growth, we have
\begin{equation*}
\limsup_{t\downarrow 0}\frac{P_t\tilde{f}(x)-\tilde{f}(x)}{t}\le \eta(x)\, ,\quad \text{for any $x\in\Omega'$}\, .
\end{equation*}
\end{definition}
Building on the top of \autoref{thm:viscoimpldistriintro} we shall prove that also the notion in \autoref{def:heatlaplboundsintro} is an equivalent characterization of Laplacian bounds, see \autoref{prop:heatimplvisco} and \autoref{prop:distrimplheat}.\\ 
Besides its own theoretical interest, this perspective will be the key to understand the interplay between the Hopf-Lax semigroup and the preservation of Laplacian bounds under lower Ricci curvature bounds, as discussed below.

\subsection{Hopf-Lax semigroup and lower Ricci curvature bounds}\label{ss:HLIntro}
The Hopf-Lax semigroup is a fundamental tool in the viscosity theory of Partial Differential Equations, in Optimal Transport and in Geometric Analysis. In this paper we establish a new principle about the stability of Laplacian bounds through the Hopf-Lax semigroup under (possibly synthetic) lower Ricci curvature bounds.
\medskip

Let $1\le p<\infty$ and let $(X,\dist)$ be a metric space. Let us consider $f:X\to\setR\cup\{\pm\infty\}$ not identically $+\infty$ and let the evolution via the $p$-Hopf-Lax semigroup, for $0<t<\infty$ be defined by
\begin{equation}\label{eq:HLdefintro}
\mathcal{Q}^p_tf(x):=\inf_{y\in X}\left(f(y)+\frac{\dist(x,y)^p}{p\, t^{p-1}}\right)\, .
\end{equation}
Notice that when $p=1$ there is a simpler expression for the Hopf-Lax semigroup, actually independent of $t$, namely:
\begin{equation*}
f^c(x):=\mathcal{Q}^1_tf(x)=\mathcal{Q}^1 f(x)=\inf_{y\in X}\big(f(y)+\dist(x,y)\big)\, .
\end{equation*}

The role of the $2$-Hopf-Lax semigroup (commonly known also as \textit{inf-convolution}) as a non linear regularization tool was put forward in \cite{LasryLions86}. 
The connection of the $2$-Hopf Lax semigroup with the viscous theory was made clear later in \cite{CrandallIshiiLions92} where the \textit{magic property} of this non linear convolution (see Lemma A.5 therein) is that viscosity supersolutions are mapped into viscosity supersolutions by $\mathcal{Q}_t^2$. All these properties, in this generality, are usually proved relying on the Hilbert space structure of the Euclidean space.
\medskip

The $2$-Hopf-Lax semigroup was then used in \cite{Cabre98} in the analysis of elliptic operators in non-divergent form on Riemannian manifolds with non-negative sectional curvature, later extended to lower Ricci curvature bounds in \cite{Kim04,WangZhang13}. The Hopf-Lax semigroup also played a key role in the characterization of lower Ricci bounds for smooth Riemannian manifolds in terms of optimal transport  \cite{OttoVillani00, Cordetal01, SturmVonRenesse05} which paved the way to the synthetic theory of  Lott-Sturm-Villani  $\CD(K,N)$ spaces \cite{Sturm06a,Sturm06b,LottVillani}.
\medskip

A subsequent breakthrough came in \cite{Kuwada10} with a new connection between the Hopf-Lax semigroup (for general exponents $p$) and lower bounds on the Ricci curvature. On a smooth Riemannian manifold $(M,g)$ with Riemannian distance $\dist$ the following conditions are equivalent:
\begin{itemize}
\item[(i)] $\Ric\ge K$, for some $K\in\setR$;
\item[(ii)] let $1\le p<\infty$ be fixed. For any non-negative Lipschitz function with bounded support $f:M\to\setR$ it holds
\begin{equation}\label{eq:Kuwadaintro}
P_s\left(\mathcal{Q}^p_1 f\right)(x)-P_sf(y)\le \frac{e^{-pKs}}{p}\dist(x,y)^p\, ,
\end{equation}
for any $x,y\in X$ and for any $s\ge 0$,
where we denoted by $P_s$ the heat flow at time $s$.
\end{itemize}
The robustness of condition (ii) (notice that it involves only objects that do have a meaning in the setting of metric measure spaces) and of the proof of the equivalence, opened the way to several developments in the smooth and in the non-smooth theory of lower Ricci curvature bounds, see for instance \cite{AmbrosioGigliSavare14,AmbrosioGigliSavare15,Bakryetal15}. In particular, (ii) is a synthetic condition, valid also in the framework of $\RCD(K,\infty)$ metric measure spaces.
\medskip

A striking consequence of the Kuwada duality \eqref{eq:Kuwadaintro} which is explored in this paper is that the Hopf-Lax semigroup maps superharmonic functions into superharmonic functions on spaces with non-negative Ricci curvature, in synthetic sense, for any $1\le p<\infty$. More in general, it preserves (up to errors depending on the lower Ricci curvature bound) Laplacian upper bounds.\\
Indeed, suppose that $(M,g)$ is a compact manifold with non-negative Ricci curvature and that $f:M\to\setR$ is a smooth function. Let $x,y\in M$ be such that
\begin{equation}\label{eq:eq0}
\mathcal{Q}^p_1 f(x)-f(y)=\frac{1}{p}\dist(x,y)^p\, .
\end{equation}
Then, assuming for the sake of this presentation that $\mathcal{Q}^p_1 f$ is smooth at $x$, we can take the right derivatives at time $s=0$ in \eqref{eq:Kuwadaintro}, taking into account \eqref{eq:eq0} to obtain
\begin{equation*}
\Delta \mathcal{Q}^p_1 f(x)\le \Delta f(y)\, .
\end{equation*}

Focusing on the case $p=1$, the theory of Laplacian bounds for non-smooth functions allows to remove the (un-natural, even on smooth manifolds) regularity assumptions and prove the following.

\begin{theorem}\label{thm:HLlaplaintro}
Let $(X,\dist,\haus^N)$ be an $\RCD(K,N)$ metric measure space. Let $f:X\to\setR$ be a locally Lipschitz function. Let $\Omega,\Omega'\subset X$ be open domains and $\eta\in\setR$. Then the following holds. Assume that $f^c$ is finite and that, for any $x\in\Omega'$ the infimum defining $f^c(x)$ is attained at some $y\in \Omega$. Assume moreover that
\begin{equation}\label{eq:lapassumption}
\Delta f\le \eta \quad\text{on $\Omega$}\,.
\end{equation}
Then
\begin{equation*}
\Delta f^c\le \eta - \min_{x\in\Omega', y\in\Omega} K \dist(x,y) \quad\text{on $\Omega'$},
\end{equation*}
where the Laplacian bounds have to be intended in any of the equivalent senses discussed in \autoref{sec:LapBoundIntro} (see also \autoref{thm:mainequivlapla}).
\end{theorem}
Similar results can be obtained for general exponents $p\in[1,\infty)$, covering in particular the case $p=2$ that was classically considered in the viscosity theory, as we recalled above. 
\medskip

We are not aware of any reference for the above stability of Laplacian bounds with respect to the Hopf-Lax semigroup for general exponents $p\in[1,\infty)$, even in the setting of smooth Riemannian manifolds.
The property is stated in the unpublished \cite{Petrunin00} for Alexandrov spaces with lower sectional curvature bounds, where a sketch of the proof is also presented. The only other references we are aware of are \cite{ZhangZhu12}, dealing only with the case $p=2$ on Alexandrov spaces with lower Ricci curvature bounds and relying on the existence of a parallel transport between tangent cones along minimizing geodesics and on the second variation formula for the arc length from \cite{Petrunin98}, and the more recent \cite{ZhangZhongZhu19}, dealing with $1<p<\infty$ on smooth Riemannian manifolds. Also in this case, our proof is completely different and more robust, as it avoids completely the use of parallel transport along geodesics.\\
Let us also mention that the property in \autoref{thm:HLlaplaintro} is equivalent to a lower Ricci curvature bound, at least on smooth Riemannian manifolds (see  \autoref{thm:HLsmoothmanifolds}). The range of the applications of this PDE principle is expected to be broad. For instance, it plays a key role in the solution of the well known open question about Lipschitz continuity of harmonic maps from $\RCD(K,N)$ to $\mathrm{CAT}(0)$ spaces by the authors in \cite{MondinoSemola22} (see also the subsequent \cite{Gigli22}). 

Finally, we also mention that some of the results of the present work (namely: the equivalence of Laplacian bounds, \autoref{thm:mainequivlapla},  and the  Laplacian bounds on the distance function from locally perimeter minimizers, \autoref{thm:meancurvminimal1Intro}) have been subsequently extended \cite{GigliMondinoSemola23} to $\RCD(K,N)$ spaces endowed with a general reference measure $\mathfrak{m}$ (i.e. not necessarily the $N$-dimensional Hausdorff measure $\haus^N$).

\subsection*{Organization of the paper}
The paper is organised as follows:
\begin{itemize}
\item In \autoref{Sec:prelim}, we collect some background results about $\RCD(K,N)$ metric measure spaces that will be needed in the subsequent developments. Let us mention that this preliminary section already contains some original result about the pointwise short time behaviour of the heat flow and about local Green functions of the Laplacian. In particular, the properties of the local Green functions are  employed in the construction of a local \textit{Green distance} with good properties, which is of independent interest.
\item  In \autoref{sec:lap} we consider some new equivalences between different notions of Laplacian and bounds for the Laplacian on an $\RCD(K,N)$ metric measure space $(X,\dist,\haus^N)$, as outlined in \autoref{sec:LapBoundIntro}. 
\item \autoref{sec:HopfLax}  is dedicated to analyze the interplay between the Hopf-Lax semigroups (associated to exponents $1\le p<\infty$),  Ricci curvature lower bounds and Laplacian upper bounds, as sketched in \autoref{ss:HLIntro}.
\item \autoref{sec:meancurv} is devoted to the study of mean curvature bounds for boundaries of locally perimeter minimizing sets of finite perimeter, in the framework of $\RCD(K,N)$ metric measure spaces $(X,\dist,\haus^N)$.
Mean curvature bounds will be encoded into Laplacian bounds for distance functions, as outlined  in \autoref{subsec:overview} and \autoref{ss:OutlineMCIntro}.
\item Finally, \autoref{sec:regularity} is dedicated to the partial regularity theory for minimal boundaries on non collapsed $\RCD$ spaces, as sketched in \autoref{ss:RegIntro}.
\end{itemize}

\section*{Acknowledgements}
The authors are  supported by the European Research Council (ERC), under the European Union Horizon 2020 research and innovation programme, via the ERC Starting Grant  “CURVATURE”, grant agreement No. 802689.\\
The second author is grateful to Gioacchino Antonelli and Giovanni Comi for useful comments on a preliminary version of this note. The authors are grateful to the anonymous reviewers for their careful reading and comments.

\section{Preliminaries}\label{Sec:prelim}
In this preliminary section we collect some background results about $\RCD(K,N)$ metric measure spaces that will be needed in the subsequent developments of the paper. This section already contains some original result of independent interest, as detailed below.

In \autoref{sec:SlopeChEnergy} we fix some notation and quickly recall the definition and basic properties of the Cheeger energy.
In \autoref{subsec:generalRCD} we briefly introduce $\RCD(K,N)$ spaces and recall some of their fundamental properties, together with some useful terminology. In \autoref{subsec:noncollapsed} we focus on the regularity properties of those $\RCD(K,N)$ metric measure spaces where the reference measure $\meas$ is the $N$-dimensional Hausdorff measure $\haus^N$. We dedicate \autoref{subsec:setsoffiniteperimeter} to the background material about sets of finite perimeter.
In \autoref{subsec:heat} we focus on the Laplacian, the heat flow and the heat kernel. After recalling the basic notions and properties, we present some new results about the pointwise short time behaviour of the heat flow. Then in \autoref{subsec:Poisson} we recall some existence and regularity results about the Poisson equation and in \autoref{subsec:Green} we present a new analysis of the local Green function of the Laplacian in this framework. The properties of the local Green function are finally employed in the construction of a local \textit{Green distance} with good properties, which is of independent interest.

\subsection{Slope, Cheeger energy and weak upper gradient}\label{sec:SlopeChEnergy}
Throughout the paper, $(X,\dist,\meas)$ will be a metric measure space, i.e. $(X,\dist)$ is a complete and separable metric space endowed with a non-negative Borel measure which is finite on bounded sets.
\\Given $f:X\to \setR$, we denote with $\lip f $ the slope of $f$ defined as
\begin{equation*}
\lip f (x_{0}):=\limsup_{x\to x_{0}}  \frac{|f(x)-f(x_{0})|}{\dist(x, x_{0})} \; \text{ if $x_{0}$ is not isolated}\, 
\end{equation*}
and $\lip f(x_{0})=0$ otherwise.

We denote with $\Lip(X)$ (resp. $\Lipb (X),  \Lipbs(X)$) the space of Lipschitz functions on $(X, \dist)$ (resp. bounded Lipschitz functions, and Lipschitz functions with bounded support).
For $f\in \Lip(X)$, let $\LipConst(f)$ denote the Lipschitz constant of $f$.  Clearly, $\lip f \leq \LipConst (f)$ on all $X$.

The Cheeger energy (introduced in \cite{Cheeger99} and further studied in \cite{AmbrosioGigliSavare14a}) is defined as the $L^{2}$-lower semicontinuous envelope of the functional $f \mapsto \frac{1}{2} \int_{X} (\lip f)^2 \, \di \meas$, i.e.:
\begin{equation*}
\Ch(f):=\inf \left\{ \liminf_{n\to \infty}  \frac{1}{2} \int_{X} (\lip f_n)^2\, \di \meas \, :\,   f_n\in \Lip(X), \; f_{n}\to f \text{ in }L^{2}(X,\meas) \right \}\, .
\end{equation*}
If $\Ch(f)<\infty$ it was proved in \cite{Cheeger99,AmbrosioGigliSavare14a} that the set
$$
G(f):= \left\{g \in L^{2}(X,\meas) \, :\,   \exists \, f_{n} \in \Lip(X), \, f_n\to f, \, \lip f_n \rightharpoonup h\geq g  \text{ in } L^{2}(X,\meas) \right\}
$$
is closed and convex, therefore it
admits a unique element of minimal norm called  \textit{minimal weak upper gradient} and denoted by $|\nabla f|$.   The Cheeger energy can be then  represented by integration as
$$\Ch(f):=\frac{1}{2} \int_{X} |\nabla f|^{2} \di \meas\, . $$
It is not difficult to see that $\Ch$ is a $2$-homogeneous, lower semi-continuous, convex functional on $L^{2}(X,\meas)$, whose proper domain 
${\rm Dom}(\Ch):=\{f \in L^{2}(X,\meas)\,:\, \Ch(f)<\infty\}$ is a dense linear subspace of $L^{2}(X,\meas)$. It then admits an $L^{2}$-gradient flow which is a continuous semigroup of contractions $(P_{t})_{t\geq 0}$ in $L^{2}(X,\meas)$, whose continuous trajectories $t \mapsto P_{t} f$, for $f \in L^{2}(X,\meas)$,  are locally Lipschitz  curves from $(0,\infty)$ with values into  $L^{2}(X,\meas)$. 

Throughout the paper, we will assume that $\Ch: {\rm Dom}(\Ch)\to \setR$ satisfies the parallelogram identity (i.e. it is a quadratic form) or, equivalently, that $P_t:  L^{2}(X,\meas) \to  L^{2}(X,\meas)$ is a linear operator for every $t\geq 0$. This, in turn, is equivalent to require that ${\rm Dom}(\Ch)$ endowed with the norm $\|f\|_{H^{1,2}}^2:= \|f\|_{L^2}+ 2 \Ch(f)$ is a Hilbert space (in general it is only a Banach space) that will be denoted by $H^{1,2}(X, \dist,\meas)$, see \cite{AmbrosioGigliSavare14, Gigli15}.

\subsection{General properties of $\RCD(K,N)$ spaces}\label{subsec:generalRCD}

The main subject of our investigation will be the so-called $\RCD(K,N)$ metric measure spaces $(X,\dist,\meas)$, i.e. infinitesimally Hilbertian metric measure spaces with Ricci curvature bounded from below and dimension bounded from above, in synthetic sense.\\
The Riemannian Curvature Dimension condition $\RCD(K,\infty)$ was introduced in \cite{AmbrosioGigliSavare14} (see also the subsequent \cite{AmbrosioGigliMondinoRajala15}) coupling the Curvature Dimension condition $\CD(K,\infty)$, previously developed in \cite{Sturm06a,Sturm06b} and independently in \cite{LottVillani}, with the assumption that the heat semigroup $(P_{t})_{t\geq 0}$  is linear  in $L^{2}(X,\meas)$.
The  finite dimensional refinements subsequently led to the notions of $\RCD(K,N)$ and $\RCD^*(K, N)$ spaces, corresponding to $\CD(K, N)$ (resp. $\CD^*(K, N)$, see \cite{BacherSturm10}) coupled with  linear heat flow. The class $\RCD(K,N)$ was proposed in \cite{Gigli15}. The (a priori more general) $\RCD^*(K,N)$ condition was thoroughly analysed in \cite{ErbarKuwadaSturm15} and (subsequently and independently) in \cite{AmbrosioMondinoSavare15} (see also \cite{CavallettiMilman16} for the equivalence betweeen $\RCD^*$ and $\RCD$ in the case of finite reference measure).
\medskip

We avoid giving a detailed introduction to this notion, addressing the reader to the survey \cite{Ambrosio19} and references therein for the relevant background. Below we recall some of the main properties that will be relevant for our purposes.
\medskip

Note that, if $(X,\dist,\meas)$ is an $\RCD(K,N)$ m.m.s., then so is $(\supp\, \meas,\dist,\meas)$, hence in the following we will always tacitly assume $\supp\, \meas = X$.
\medskip

Any $\RCD(K,N)$ m.m.s.\ $(X,\dist,\meas)$ satisfies the Bishop-Gromov inequality:
\begin{equation}\label{eq:BishopGromovInequality}
\frac{\meas(B_R(x))}{v_{K,N}(R)}\le\frac{\meas(B_r(x))}{v_{K,N}(r)}
\quad \text{for any $0<r<R$ and $x\in X$}\, ,
\end{equation}
where $v_{K,N}(r)$ is the volume of the ball with radius $r$ in the model space with dimension $N$ and Ricci curvature $K$. 
In particular $(X,\dist,\meas)$ is locally uniformly doubling. Furthermore, it was proved in \cite{Rajala12} that it satisfies a local Poincar\'{e} inequality. Therefore $\RCD(K,N)$ spaces fit in the framework of PI spaces.
\medskip


We assume the reader to be familiar with the notion of (pointed measured) Gromov-Hausdorff convergence (pmGH-convergence for short), referring to \cite[Chapter 27]{Villani09} and \cite{GigliMondinoSavare15} for an overview on the subject. 

\begin{definition}\label{def: mpGH convergence}
	A sequence $\set{(X_i, \dist_i, \meas_i, x_i)}_{i\in \setN}$ of pointed m.m.s. is said to converge in the pmGH topology to $(Y,\varrho,\mu, y)$ if there exist a complete separable metric space $(Z, \dist_Z)$ and isometric embeddings 
	\begin{align*}
	 &\Psi_i : (\supp \meas_i, \dist_i)\to (Z,\dist_Z)
	 \qquad
	 \forall i\in \setN\, ,\\
	&\Psi: (\supp\mu, \varrho)\to (Z,\dist_Z)\, ,
	\end{align*}
	such that for every $\eps>0$ and $R>0$ there exists $i_0$ such that for every $i>i_0$
	\begin{equation*}
		\Psi(B^Y_R(y))\subset [\Psi_i(B^{X_i}_R(x_i))]_{\eps}\, ,
		\qquad
		\Psi_i(B^{X_i}_R(x_i))\subset [\Psi(B^Y_R(y))]_{\eps}\, ,
	\end{equation*}
	where $[A]_{\eps}:=\set{z\in Z\ : \dist_Z(z,A)<\eps}$ for every $A\subset Z$. Moreover $(\Psi_i)_{\#} \meas_i\weakto \Psi_{\#} \mu$, where the convergence is understood in duality with $\Cbs(Z)$.
\end{definition}

In the case of a sequence of uniformly locally doubling metric measure spaces $(X_i,\dist_i,\meas_i,x_i)$ (as in the case of $\RCD(K,N)$ spaces), the pointed measured Gromov-Hausdorff convergence to $(Y,\varrho,\mu,y)$ can be equivalently characterized asking for the existence of a proper metric space $(Z,\dist_Z)$ such that all the metric spaces $(X_i,\dist_i)$ are isometrically embedded into $(Z,\dist_Z)$, $x_i\to y$ and $\meas_i\weakto\mu$ in duality with $\Cbs(Z)$. 
Notice also that the pmGH convergence is metrizable, and therefore it makes sense to say that two pointed metric measure spaces are $\eps$-close in this sense. Analogous remarks hold for the Gromov-Hausdorff distance between metric spaces.
\medskip

	A fundamental property of $\RCD(K,N)$ spaces, that will be used several times in this paper, is the stability w.r.t.\ pmGH-convergence, meaning that a pmGH-limit of a sequence of (pointed) $\RCD(K_n,N_n)$ spaces for some $K_n\to K$ and $N_n\to N$ is an $\RCD(K,N)$ m.m.s..  
\medskip


Given a m.m.s.\ $(X,\dist,\meas)$, $x\in X$ and $r\in(0,1)$, we consider the rescaled and normalized pointed m.m.s.\ $(X,r^{-1}\dist,\meas_r^{x},x)$, where
\begin{equation*}
C(x,r):= \left( \int_{B_r(x)} \left(1-\frac{\dist(x,y)}{r}\right) \di \meas(y)\right)\quad  \meas_r^x=C(x,r)^{-1}\meas\, .
\end{equation*}
\begin{definition}[Tangent cone]
	We say that a pointed m.m.s.\ $(Y,\dist_Y,\eta,y)$ is tangent to $(X,\dist,\meas)$ at $x$ if there exists a sequence $r_i\downarrow 0$ such that $(X,r_i^{-1}\dist,\meas_{r_i}^x,x)\rightarrow(Y,\dist_Y,\eta,y)$ in the pmGH-topology. The collection of all the tangent spaces of $(X,\dist,\meas)$ at $x$ is denoted by $\Tan_x(X,\dist,\meas)$.
\end{definition}

A compactness argument, which is due to Gromov, together with the rescaling and stability properties of the $\RCD(K,N)$ condition, yields that $\Tan_x(X,\dist,\meas)$ is non-empty for every $x\in X$ and its elements are all $\RCD(0,N)$ pointed m.m.\ spaces.\\
Let us recall below the notion of $k$-regular point and $k$-regular set. 

\begin{definition}\label{def:regular point}
	Given any natural $1\le k\le N$, we say that $x\in X$ is a $k$-regular point if
	\begin{equation*}
		\Tan_x(X,\dist,\meas)=\left\lbrace (\setR^k,\dist_{eucl},c_k\Leb^k,0)  \right\rbrace\, .
	\end{equation*}
	We shall denote by $\mathcal{R}_k$ the set of $k$-regular points in $X$.
\end{definition}

Combing the results in \cite{MondinoNaber19} with \cite{KellMondino18,DePhlippisMarcheseRindler17,GigliPasqualetto16a} and \cite{BrueSemola18}, we have a good understanding of the rectifiable structure of $\RCD(K,N)$ metric measure spaces.

\begin{theorem}[Rectifiable structure]
	Let $(X,\dist,\meas)$ be an $\RCD(K,N)$ m.m.s.\ with $K\in\setR$ and $N\ge 1$. Then there exists a natural number $1\le n\le N$, called essential dimension of $X$, such that $\meas(X\setminus \mathcal{R}_n)=0$. Moreover $\mathcal{R}_n$ is $(\meas,n)$-rectifiable and $\meas$ is representable as $\theta\haus^n\res {\mathcal{R}_n}$ for some non-negative density $\theta\in L^1_{\rm loc}(X,\haus^n\res\mathcal{R}_n)$.
\end{theorem}

Recall that $X$ is said to be $(\meas,n)$-rectifiable if there exists a family $\left\lbrace A_i \right\rbrace_{i\in\setN}$ of Borel subsets of $X$ such that each $A_i$ is bi-Lipschitz to a Borel subset of $\setR^n$ and $\meas(X\setminus \cup_{i\in\setN}A_i)=0$.

\subsection{Non collapsed spaces}\label{subsec:noncollapsed}

We will mainly focus on the so called \textit{noncollapsed} $\RCD(K,N)$ metric measure spaces, i.e. those spaces for which the reference measure is the $N$-dimensional Hausdorff measure $\haus^N$.\\ 
As it happens for noncollapsed Ricci limits, whose regularity is much better than that of collapsed limits (see \cite{CheegerColding97,CheegerColding2000a,CheegerColding2000b}), noncollapsed $\RCD$ spaces are more regular than general $\RCD$ spaces. Their properties have been investigated throughout in \cite{Kitabeppu17,DePhilippisGigli18,KapovitchMondino21,AntonelliBrueSemola19,BrueNaberSemola20}. 
\medskip

Below we state a fundamental $\eps$-regularity result for non collapsed spaces. 
For smooth manifolds and their limits it was proved in \cite{Colding97,CheegerColding97}, building on a variant of the classical Reifenberg theorem valid for metric spaces (see also the earlier \cite{Anderson90}). We refer to \cite{DePhilippisGigli18,KapovitchMondino21} for the generalization to $\RCD$ spaces and the present form. 

\begin{theorem}[$\eps$-regularity]\label{thm:epsregcolding}
Let $1\le N<\infty$ be a fixed natural number. Then, for any $0<\eps<1/5$ there exists $\delta=\delta(\eps,N)>0$ such that for any $\RCD(-\delta(N-1),N)$ space $(X,\dist,\haus^N)$, if 
\begin{equation*}
\dist_{GH}(B_{2}(x),B_2(0^N))<\delta\, ,
\end{equation*}
then:
\begin{itemize}
\item[i)] $\abs{\haus^N(B_1(x))-\haus^N(B_1(0^N))}<\eps$;
\item[ii)] for any $y\in B_1(x)$ and for any $0<r<1/2$ it holds
\begin{equation*}
\dist_{GH}(B_r(y),B_r(0^N))<\eps r\, ;
\end{equation*}
\item[iii)] $B_1(x)$ is $C^{1-\eps}$-biH\"older homeomorphic to the Euclidean ball $B_1(0^N)$.
\end{itemize}
\end{theorem}

Another key regularity property of noncollapsed $\RCD$ spaces is that all their tangents are metric cones, see \cite{DePhilippisGigli18}. This is a consequence of the so-called \textit{volume cone implies metric cone} property, originally proved in \cite{CheegerColding96} for limits of smooth manifolds and later extended to $\RCD$ spaces in \cite{DePhilippisGigli16}.\\
Building on the top of this, one can introduce a natural stratification of the singular set of an $\RCD(K,N)$ metric measure space $(X,\dist,\haus^N)$, i.e. the set $\mathcal{S}:=X\setminus\mathcal{R}=X\setminus\mathcal{R}_N$, based on the maximal number of Euclidean factors in any tangent cone.

\begin{definition}\label{def:singularstrata}
Let $(X,\dist,\haus^N)$ be an $\RCD(K,N)$ metric measure space. Then for any $0\le k\le N$ we let
\begin{equation*}
\mathcal{S}_k:=\{x\in X\, : \quad\text{no tangent cone at $x$ splits a factor $\setR^{k+1}$}\}\, .
\end{equation*}
\end{definition}
A classical dimension reduction argument then allows to get the Hausdorff dimension bounds
\begin{equation}\label{eq:dimHbound}
\dim_H\mathcal{S}_k\le k\, ,
\end{equation} 
for any $0\le k\le N-1$.
\medskip

When combined with the $\eps$-regularity \autoref{thm:epsregcolding}, together with its counterpart for points in the top dimensional singular stratum obtained in \cite{BrueNaberSemola20} (see \autoref{thm:boundepsreg}), the Hausdorff dimension bound \eqref{eq:dimHbound} allows to understand the topological regularity of non collapsed $\RCD$ spaces away from sets of codimension two.

\begin{theorem}[Topological structure of non collapsed spaces]\label{thm:topRCDnc}
Let $(X,\dist,\haus^N)$ be an $\RCD(K,N)$ metric measure space for some $K\in\setR$ and $1\le N<\infty$. Then, for any $0<\alpha<1$ there exists a decomposition
\begin{equation*}
X=\partial X\cup O_{\alpha}\cup \mathcal{S}_{\alpha}\, ,
\end{equation*}
where $\partial X=\overline{\mathcal{S}^{N-1}\setminus\mathcal{S}^{N-2}}$ is the boundary of $(X,\dist,\haus^N)$, $O_{\alpha}$ is an open neighbourhood of the regular set $\mathcal{R}$ that is $C^{\alpha}$-biH\"older to a smooth $N$-dimensional Riemannian manifold and $\dim_H\mathcal{S}_{\alpha}\le N-2$.

Moreover, for any $0<\alpha<1$ there exists an open neighbourhood $V_{\alpha}$ of $\mathcal{S}^{N-1}\setminus \mathcal{S}^{N-2}$ inside $\partial X$ such that $V_{\alpha}$ is $C^{\alpha}$-biH\"older to a smooth $(N-1)$-dimensional Riemannian manifold.
\end{theorem}

Further estimates for singular sets on non collapsed $\RCD$ spaces will be recalled later in the note.

\subsection{Sets of finite perimeter}\label{subsec:setsoffiniteperimeter}

This subsection is aimed at introducing some classical and most recent results about sets of finite perimeter in the framework of $\RCD(K,N)$ metric measure spaces. 

\subsubsection{Introduction and basic properties}
We recall the definition of function of bounded variation in the present setting.

\begin{definition}[Function of bounded variation]\label{def:bvfunction}
We say that a function $f\in L^1(X,\meas)$ has bounded variation (and we write $f\in\BV(X,\dist,\meas)$) if there exist
locally Lipschitz functions $f_i$ converging to $f$ in $L^1(X,\meas)$ such that
\begin{equation*}
\limsup_{i\to\infty}\int_X \lip f_i\di \meas<\infty\, .
\end{equation*}  
By localizing this construction one can define 
\begin{equation*}
\abs{Df}(A):=\inf\left\lbrace \liminf_{i\to\infty}\int_A \lip f_i \di \meas: f_i\in\Lip_{\loc}(A),\quad f_i\to f \text{ in } L^1(A,\meas)\right\rbrace  
\end{equation*}
for any open set $A\subset X$. In \cite{AmbDiM14} (see also \cite{MirandaJr} for the case of locally compact spaces) it is proven that this set function 
is the restriction to open sets of a finite Borel measure that we call \textit{total variation of $f$} and still denote $\abs{Df}$.
\end{definition}

Dropping the global integrability condition on $f=\chi_E$, let us recall now the analogous definition of a set of finite perimeter 
in a metric measure space (see again \cite{Am02,MirandaJr,AmbDiM14}).

\begin{definition}[Perimeter and sets of finite perimeter]\label{def:setoffiniteperimeter}
Given a Borel set $E\subset X$ and an open set $A$, the perimeter $\Per(E,A)$ is defined in the following way:
\begin{equation*}
\Per(E,A):=\inf\left\lbrace \liminf_{n\to\infty}\int_A \lip u_n \di\meas: u_n\in\Lip_{\loc}(A),\quad u_n\to\chi_E\quad \text{in } L^1_{\loc}(A,\meas)\right\rbrace\, .
\end{equation*}
We say that $E$ has finite perimeter if $\Per(E,X)<\infty$. In that case it can be proved that the set function $A\mapsto\Per(E,A)$ is the restriction to open sets of a finite Borel measure $\Per(E,\cdot)$ defined by
\begin{equation*}
\Per(E,B):=\inf\left\lbrace \Per(E,A): B\subset A,\text{ } A \text{ open}\right\rbrace\, .
\end{equation*}
\end{definition}

Let us remark for the sake of clarity that $E\subset X$ with finite $\meas$-measure is a set of finite perimeter if and only if $\chi_E\in\BV(X,\dist,\meas)$ and that $\Per(E,\cdot)=\abs{D\chi_E}(\cdot)$. In the following we will say that $E\subset X$ is a set of locally finite perimeter if $\chi_E$ is a function of locally bounded variation, that is to say $\eta\chi_E\in\BV(X,\dist,\meas)$ for any $\eta\in \Lipbs(X,\dist)$. In the sequel we shall adopt both the notations $\abs{D \chi_E}$ and $\Per_E$ to denote the perimeter measure of a set with finite perimeter $E$.
\medskip

We will usually assume that a set of finite perimeter $E\subset X$ is normalized in the following sense (see \cite[Proposition 12.19]{Maggi12} for an analogous classical result in the Euclidean space and the proof of \cite[Theorem 4.2]{Kinnunenetal13} for the present setting): up to modification on an $\meas$-negligible set of $E$, it holds that $\meas(E\cap B_r(x))>0$ for any $x\in E$ and $r>0$ and $\meas(B_r(x)\setminus E)>0$ for any $x\in X\setminus E$ and $r>0$.\\
This implies in particular that, for any $x\in\partial E$ (where we denoted by $\partial E$ the topological boundary of $E$), it holds
\begin{equation}\label{eq:normsetsfin}
\meas(B_r(x)\cap E)>0\, \quad\text{and }\, \meas(B_r(x)\setminus E)>0\, ,\quad\text{for any $r>0$}\, .
\end{equation}

\begin{definition}
We  adopt the terminology \textit{measure theoretic interior} to indicate
\begin{equation*}
\mathrm{Int} (E):=\Big\{x\in X\, :\, \lim_{r\to 0}\frac{\meas(E\cap B_r(x))}{\meas(B_r(x))}=1\Big\}\, ,
\end{equation*} 
i.e. the set of point of density $1$ of $\chi_{E}$. Note that, by Lebesgue differentiation theorem, $\meas(E\Delta \mathrm{Int}(E))=0$.
\end{definition}

When considering the lower and upper approximate limits of the indicator function $\chi_E$ of $E$, i.e.
\begin{equation*}
\chi_E^{\vee}(x):=\inf\Big\{t\in\setR\, :\, \lim_{r\to 0}\frac{\meas(\{\chi_E<t\}\cap B_r(x))}{\meas(B_r(x))}=0 \Big\}\, 
\end{equation*}
and 
\begin{equation}\label{eq:pointwedgedef}
\chi_E^{\wedge}(x):=\sup\Big\{t\in\setR\, :\, \lim_{r\to 0}\frac{\meas(\{\chi_E>t\}\cap B_r(x))}{\meas(B_r(x))}=0 \Big\}\, , 
\end{equation}
it is easy to verify that 
\begin{equation*}
\chi_E^{\vee}(x)=1\, ,\quad\text{on $X\setminus \mathrm{Int}(E^c)$}\, \quad\text{and }\quad \chi_E^{\vee}(x)=0\, \quad\text{otherwise}\, ,
\end{equation*}
while 
\begin{equation}\label{eq:pointwedgeprop}
\chi_E^{\wedge}(x)=1\, ,\quad\text{on $ \mathrm{Int}(E)$}\, \quad\text{and }\quad \chi_E^{\wedge}(x)=0\, \quad\text{otherwise}\, .
\end{equation}

Following \cite{Am01,Am02} we recall the notion of essential boundary of a set of finite perimeter.

\begin{definition}[Essential boundary]
Let $(X,\dist,\meas)$ be an $\RCD(K,N)$ metric measure space and let $E\subset X$ be a set of locally finite perimeter. Then we introduce the \textit{essential boundary} $\partial ^*E$ as 
\begin{equation}\label{eq:defEssBoundary}
\partial^*E:= \Big\{x\in X\, :\, \lim_{r\to 0}\frac{\meas(B_r(x)\cap E)}{\meas(B_r(x))}\neq 0\, \quad\text{and }\quad\lim_{r\to 0}\frac{\meas(B_r(x)\setminus E)}{\meas(B_r(x))}\neq 0 \Big\}\, .
\end{equation}

\end{definition}

The following coarea formula for functions of bounded variation on metric measure spaces is taken from 
\cite[Proposition 4.2]{MirandaJr}, dealing with locally compact spaces and its proof works in the more general setting of metric measure spaces.

\begin{theorem}[Coarea formula]\label{thm:coarea}
Let $v\in\BV(X,\dist,\meas)$.
Then, $\{v>r\}$ has finite perimeter for $\Leb^1$-a.e. $r\in\setR$. Moreover, for any Borel function $f:X\to[0,+\infty]$, it holds 
\begin{equation}\label{eq:coarea}
\int_X f\di\abs{Dv}=\int_{-\infty}^{+\infty}\left(\int_X f\di\Per(\{v>r\},\cdot)\right)\di r\,  .
\end{equation}
\end{theorem}

Let us recall that if $(X,\dist, \meas)$ verifies doubling and Poincar\'e inequalities then a local, relative isoperimetric inequality holds, see for instance \cite[Theorem 3.3]{KorteLahti14}.  More precisely:  there exists constants  $\lambda>1, C>0, r_{0}>0$, depending only on the doubling and Poincar\'e constants, such that
\begin{equation}\label{eq:isop}
\min\{\meas(E\cap B_r(x)),\meas(B_r(x)\setminus E)\}\le C r \;\Per (E, B_{\lambda r}(x))\, ,
\end{equation}
for all  $x \in X$, $r\in (0,r_{0})$.


\subsubsection{Convergence and stability for sets of finite perimeter and functions of bounded variation}

Before introducing tangents for sets of finite perimeter over $\RCD$ spaces, let us recall some terminology about convergence and stability for $\BV$ functions along converging sequences of metric measure spaces. The discussion below is borrowed from \cite{AmbrosioBrueSemola19}, the main references being \cite{GigliMondinoSavare15,AmbrosioHonda} and \cite{AmbrosioHonda18}, to which we address the reader for details and relevant background.

 Let $(X_i, \dist_{i}, \meas_i, \bar x_{i})$ be a sequence of pointed metric measure spaces converging in pointed-measured-Gromov-Hausdorff sense  (or, more generally, in pointed measured Gromov sense) to $(Y, \varrho, \mu, y)$. 

\begin{definition}\label{def:BV convergence}
We say that a sequence $(f_i)\subset L^1(X_i,\meas_i)$ converges $L^1$-strongly to $f\in L^1(Y,\mu)$ if 
\begin{equation*}
   \sigma\circ f_i\meas_i\weakto \sigma\circ f\mu
   \qquad
   \text{and} 
   \qquad
   \int_{X_i}|f_i|\di \meas_i\to \int_Y |f|\di \mu\, ,
\end{equation*}
where $\sigma(z):=\sign(z)\sqrt{|z|}$ and the weak convergence is understood in duality with $\Cbs(Z)$.

We say that $f_i\in \BV(X_i,\meas_i)$ converge in energy in $\BV$ to $f\in \BV(Y,\mu)$ if $f_i$ converge $L^1$-strongly to $f$ and
\begin{equation*}
	\lim_{i\to\infty} |Df_i|(X_i)= |Df|(Y)\, .
\end{equation*}
\end{definition}

\begin{definition}\label{def:L1 convergence of sets}
	We say that a sequence of Borel sets $E_i\subset X_i$ such that $\meas_i(E_i)<\infty$ for any $i\in\setN$ converges in $L^1$-strong to a Borel set $F\subset Y$ with $\mu(F)<\infty$ if $\chi_{E_i}\meas_i\weakto \chi_F\mu$ in duality with $\Cbs(Z)$ and $\meas_i(E_i)\to \mu(F)$.
	
	We also say that a sequence of Borel sets $E_i\subset X_i$ converges in $L^1_{\loc}$ to a Borel set $F\subset Y$ if $E_i\cap B_R(x_i)\to F\cap B_R(y)$ in $L^1$-strong for any $R>0$.
\end{definition}

\subsubsection{De Giorgi's Theorem and integration by parts formulae}

Let us recall the definition of tangent to a set of finite perimeter from \cite{AmbrosioBrueSemola19}.

\begin{definition}[Tangents to a set of finite perimeter]\label{def:tan}
	Let $(X,\dist,\meas)$ be an $\RCD(K,N)$ m.m.s., fix $x\in X$ and let $E\subset X$ be a set of locally finite perimeter. 
	We denote by $\Tan_x(X,\dist,\meas,E)$ the collection of quintuples $(Y,\varrho,\mu,y, F)$ satisfying the following two properties:
	\begin{itemize}
		\item[(a)] $(Y,\varrho,\mu,y)\in\Tan_x(X,\dist,\meas)$ and $r_i\downarrow 0$ are such that the rescaled spaces $(X,r_i^{-1}\dist,\meas_x^{r_i},x)$ converge to $(Y,\varrho,\mu,y)$ in the pointed measured Gromov-Hausdorff topology;
		\item[(b)] $F$ is a set of locally finite perimeter in $Y$ with $\mu(F)>0$ and, if $r_i$ are as in (a), then the sequence $f_i=\chi_E$ converges in $L^1_\loc$ to $\chi_F$ according to \autoref{def:L1 convergence of sets}.
	\end{itemize}
\end{definition}

It is clear that the following locality property of tangents holds: if
\begin{equation}\label{eq:localitytangents}
\meas\bigl(A\cap (E\Delta F)\bigr)=0\, ,
\end{equation}
then 
\begin{equation}
\Tan_x(X,\dist,\meas,E)=\Tan_x(X,\dist,\meas,F)\qquad\forall x\in A\, ,
\end{equation}
whenever $E,\,F$ are sets of locally finite perimeter and $A\subset X$ is open.
\medskip

In \cite{BruePasqualettoSemola19,BruePasqualettoSemola21}, essential uniqueness of tangents and rectifiability of the reduced boundary were obtained for sets of finite perimeter on $\RCD(K,N)$ metric measure spaces.

\begin{theorem}[Uniqueness of tangents]\label{th:uniqueness}
	Let $(X,\dist,\meas)$ be an $\RCD(K,N)$ m.m.s. with essential dimension $1\le n\le N$ and let $E\subset X$ be a set of finite perimeter. Then, for $\abs{D\chi_E}$-a.e.\ $x\in X$ it holds 
	\begin{equation*}
	\Tan_x(X,\dist,\meas,E)=\left\lbrace (\setR^n,\dist_{eucl},c_n\Leb^n,0^n,\left\lbrace x_n>0\right\rbrace )\right\rbrace\, .
	\end{equation*}
\end{theorem}

We next introduce a notion of reduced boundary, in analogy with the Euclidean theory.

\begin{definition}
Let $(X,\dist,\meas)$ be an $\RCD(K,N)$ metric measure space with essential dimension equal to $n\in \setN$, and let $E\subset X$ be a set of locally finite perimeter. 
We set 
\begin{equation*}
\mathcal{F}E:=\left\lbrace x\in X\;:\;\Tan_x(X,\dist,\meas,E)=\left\lbrace (\setR^n,\dist_{eucl},c_n\Leb^n,0^n,\left\lbrace x_n>0\right\rbrace )\right\rbrace \right\rbrace \, .
\end{equation*}
\end{definition}

\begin{remark}
Let us point out, for the sake of clarity, that the reduced boundary in the above sense does not fully coincide with the reduced boundary in the classical Euclidean sense. Indeed the definition of reduced boundary point in the $\RCD$ framework does not prevent, when read in the Euclidean context, the possibility that different half-spaces arise as blow-ups when rescaling along different sequences of radii converging to $0$.
\end{remark}

\begin{theorem}[Rectifiability]\label{thm:rectifiability}
	Let $(X,\dist,\meas)$ be an $\RCD(K,N)$ m.m.s.  with essential dimension $1\le n\le N$ and let $E\subset X$ be a set of locally finite perimeter. Then 
	the reduced boundary $\mathcal{F}E$ is $\big(\abs{D\chi_E},(n-1)\big)$-rectifiable.
\end{theorem}

When specialized to the non-collapsed case,
where the essential dimension $n=N$ (cf. with the discussion before \autoref{def:singularstrata}), \autoref{thm:rectifiability} turns into:
\begin{corollary}\label{cor:Noncollapsed}
Let $(X,\dist,\haus^N)$ be a $\RCD(K,N)$ m.m.s.\ and $E\subset X$ a set of locally finite perimeter. Then $\mathcal{F}E$ is $\left(\abs{D\chi_E},N-1\right)$-rectifiable (equivalently, $\left(\mathcal{H}^{N-1},N-1\right)$-rectifiable). Furthermore
\begin{equation*}
\abs{D\chi_E}=\mathcal{H}^{N-1}\res\mathcal{F}E.
\end{equation*}
\end{corollary}


In \cite{BruePasqualettoSemola19} the following Gauss-Green integration by parts formula for sets of finite perimeter and Sobolev vector fields has been proved. We refer to \cite{Gigli18} for the notion of Sobolev vector fields in $H^{1,2}_C(TX)$ and to \cite{BruePasqualettoSemola19} for the notion of restriction of the tangent module over the boundary of a set of finite perimeter $L^2_E(TX)$.

\begin{theorem}[Theorem 2.4 in \cite{BruePasqualettoSemola19}]\label{thm:GGRCDsmooth}
Let $(X,\dist,\meas)$ be an $\RCD(K,N)$ metric measure space and let $E\subset X$ be a set with finite perimeter and finite measure. Then there exists a unique vector field $\nu_E\in L^2_E(TX)$ such that $\abs{\nu_E}=1$ holds $\Per$-a.e. and
\begin{equation*}
\int _E\div v\di\meas=-\int<\mathrm{tr}_E v,\nu_E>\di\Per_E\, , 
\end{equation*}
for any $v\in H^{1,2}_C(TX)\cap D(\div )$ such that $\abs{v}\in L^{\infty}(\meas)$.
\end{theorem}


For the sake of notation we shall denote 
\begin{equation}\label{eq:defGGMeas}
\mu_E:=\nu_E\cdot \Per_E,  \quad \text{the \textit{Gauss-Green measure}}.
\end{equation}
Notice that, by our choice of signs, $\nu_E$ corresponds to the \textit{inward-pointing} unit normal vector for a domain with smooth boundary in a smooth Riemannian manifold.
\medskip

Let us also recall a mild regularity result for sets of finite perimeter which follows again from \cite{BruePasqualettoSemola19} and has been proved in \cite[Proposition 4.2]{BruePasqualettoSemola21} (even for general $\RCD(K,N)$ metric measure spaces $(X,\dist,\meas)$). It can be considered as a counterpart tailored for this framework of the Euclidean Federer type characterization of sets of finite perimeter.

\begin{proposition}\label{prop:federertype}
Let $(X,\dist,\haus^N)$ be an $\RCD(K,N)$ metric measure space for some $N\ge 1$ and let $E\subset X$ be a set of locally finite perimeter. Then the following hold:
\begin{itemize}
\item[i)] for $\haus^{N-1}$-a.e. $x\in X$ it holds
\begin{equation*}
\lim_{r\downarrow 0}\frac{\haus^N(B_r(x)\cap E)}{\haus^N(B_r(x))}\in \Big\{0,\frac{1}{2},1 \Big\}\, .
\end{equation*}
Moreover, up to an $\haus^{N-1}$-negligible set it holds
\begin{equation*}
\mathcal{F}E=\Big\{x\in E\, :\, \lim_{r\downarrow 0}\frac{\haus^N(B_r(x)\cap E)}{\haus^N(B_r(x))}=\frac{1}{2} \Big\}\, .
\end{equation*}
\item[ii)] For $\haus^{N-1}$-a.e. $x\in X$ it holds
\begin{equation}\label{eq:densflow}
\lim_{t\downarrow 0}P_t\chi_E(x)\in \Big\{0,\frac{1}{2},1 \Big\}\, .
\end{equation}
Moreover, up to an $\haus^{N-1}$-negligible set it holds
\begin{equation*}
\mathcal{F}E=\Big\{x\in E\, :\, \lim_{t\downarrow 0}P_t\chi_E(x)=\frac{1}{2} \Big\}\, .
\end{equation*}
\end{itemize}
\end{proposition}

\begin{definition}\label{def:precrep}
Given a set of finite perimeter $E\subset X$ and any $0\le t\le 1$, we set
\begin{equation*}
E^{(t)}:=\Big\{x\in X: \lim_{r\downarrow 0}\frac{\haus^N(B_r(x)\cap E)}{\haus^N(B_r(x))}=t \Big\}\, .
\end{equation*}
\end{definition}
A consequence of \autoref{prop:federertype} above is that, up to an $\haus^{N-1}$-negligible set,
\begin{equation*}
X=E^{(1)}\cup E^{(1/2)}\cup E^{(0)}\, .
\end{equation*}

\begin{definition}
In the following we shall adopt the notation $M\sim N $ to indicate that two Borel sets coincide up to $\haus^{N-1}$ negligible sets, i.e. $\haus^{N-1}(M\Delta N)=0$.
\end{definition}

It follows from the discussion above that, for any Borel set $M\subset X$,
\begin{equation*}
M\sim (M\cap E^{(1)})\cup (M\cap E^{(0)})\cup (M\cap E^{(1/2)})\, .
\end{equation*}

In order to ease the notation, given a set of finite perimeter $E\subset X$ and $x\in X$ we shall denote by
\begin{equation*}
\theta(E,x):=\lim_{r\to 0}\frac{\haus^N(E\cap B_r(x))}{\haus^N(B_r(x))}\, ,
\end{equation*}
whenever the limit exists.\\
It follows again from the discussion above that $\theta(E,x)$ is well defined and belongs to $\{0,1/2,1\}$ for $\haus^{N-1}$-a.e. $x\in X$.

\begin{remark}
Analogous statements hold changing $\lim_{r\to 0}\haus^N(B_r(x)\cap E)/\haus^N(B_r(x))$ with $\lim_{t\to 0}P_t\chi_E$, see \cite[Remark 4.5]{BruePasqualettoSemola21}.
\end{remark}

\subsubsection{Gauss Green formulae for essentially bounded divergence measure vector fields}

In order to make rigorous the formal argument described in \autoref{subsec:overview}, we need to consider vector fields that are bounded and have measure valued divergence, but do not belong to $H^{1,2}_C(TX)$ in general. 

\begin{definition}\label{def:essbounddivmeas}
Let $(X,\dist,\meas)$ be an $\RCD(K,N)$ metric measure space. We say that a vector field $V\in L^{\infty}(TX)$ is an \textit{essentially bounded divergence measure vector field} if its distributional divergence is a finite Radon measure, that is if $\div V$ is a finite Radon measure such that, for any Lipschitz function with compact support $g:X\to\setR$, it holds
\begin{equation*}
\int_X g \di \div V=-\int_X \nabla g\cdot V\di\meas\, .
\end{equation*}
We shall denote the class of these vector fields by $\mathcal{DM}^{\infty}(X)$ and sometimes, to ease the notation, we will abbreviate $\int g\di\div V$ with $\int g\div V$.
\end{definition}

We recall a useful regularity result, whose proof can be found in the proof of \cite[Theorem 7.4]{BrueNaberSemola20}.

\begin{lemma}\label{cor:divac}
Let $(X,\dist,\haus^N)$ be an $\RCD(K,N)$ metric measure space. Let $V\in\mathcal{DM}^{\infty}(X)$. Then $\div V\ll \haus^{N-1}$.
\end{lemma}

Notice that the divergence measure of a vector field in this class might have singular parts with respect to the reference measure. In particular, it might charge the boundary of a set of finite perimeter and it becomes relevant to choose wether in the Gauss-Green formula we integrate the divergence of the vector field only over the interior of the set of finite perimeter or over its closure.\\
As a second issue, contrary to smooth vector fields (and to $H^{1,2}_C$-vector fields in the $\RCD$ framework) essentially bounded divergence measure vector fields do not have pointwise-a.e. defined representatives over boundaries of sets of finite perimeter.\\

It turns that, despite not being able to pointwise define the vector field over the reduced boundary of a set of finite perimeter, it is possible to define interior and exterior normal traces, possibly different, playing the role of the term $V\cdot\nu_E$ in the Gauss-Green formula.
\medskip

Given an essentially bounded divergence measure vector field $V\in \mathcal{DM}^{\infty}(X)$ and a set of finite perimeter $E\subset X$, it is proved in \cite[Section 6.5]{BuffaComiMiranda19} and \cite[Section 5]{BruePasqualettoSemola21} that there exist measures $D\chi_E(\chi_E V)$ and $D\chi_E(\chi_{E^c}V)$ such that
\begin{equation*}
\nabla P_t\chi_E\cdot (\chi_EV)\weakto D\chi_E(\chi_E V) \quad\text{and}\quad \nabla P_t\chi_E\cdot (\chi_{E^c}V)\weakto D\chi_E(\chi_{E^c} V)\, ,
\end{equation*}
as $t\to 0$.\\
Moreover, $D\chi_E(\chi_E V)$ and $D\chi_E(\chi_{E^c}X)$ are both absolutely continuous w.r.t. $\abs{D\chi_E}$. 
Therefore we are entitled to consider their densities, $\left(V\cdot\nu_E\right)_{\mathrm{int}}$ and $\left(V\cdot\nu_E\right)_{\mathrm{ext}}$, defined by
\begin{equation*}
2D\chi_E(\chi_E V)=\left(V\cdot\nu_E\right)_{\mathrm{int}}\abs{D\chi_E} \quad\text{and}\quad 2D\chi_E(\chi_{E^c} V)=\left(V\cdot\nu_E\right)_{\mathrm{ext}}\abs{D\chi_E}.
\end{equation*}

Below we report a Gauss-Green integration by parts for essentially bounded divergence measure vector fields and sets of finite perimeter on $\RCD(K,N)$ spaces. It is the outcome of \cite[Theorem 6.20]{BuffaComiMiranda19}, where the integration by parts formula has been obtained with non sharp bounds for the normal traces, and of \cite[Theorem 5.2]{BruePasqualettoSemola21}, where these bounds have been sharpened.

\begin{theorem}\label{thm:GaussGreenRCDBCM}
Let $(X,\dist,\meas)$ be an $\RCD(K,N)$ metric measure space. Let $E\subset X$ be a set of finite perimeter and let $V\in\mathcal{DM}^{\infty}(X)$. Then for any function $f\in \Lip_c(X)$ it holds
  \begin{align*}
 \int_{E^{(1)}}f \div V+\int_E\nabla f\cdot V\di\meas & =-\int_{\mathcal{F}E} f\left(V\cdot\nu_E\right)_{\mathrm{int}}\di\Per\,  ,\\
 \int_{E^{(1)}\cup\mathcal{F}E} f \div V+\int_E\nabla f\cdot V\di\meas & =-\int_{\mathcal{F}E} f\left(V\cdot\nu_E\right)_{\mathrm{ext}}\di\Per\, .
\end{align*}
Moreover 
\begin{align}
\norm{\left(V\cdot\nu_E\right)_{\mathrm{int}}}_{L^{\infty}(\mathcal{F}E,\Per)}& \le \norm{V}_{L^{\infty}(E,\meas)} \, , \label{eq:inftyboundtraceintsharp}\\
\norm{\left(V\cdot\nu_E\right)_{\mathrm{ext}}}_{L^{\infty}(\mathcal{F}E,\Per)}& \le \norm{V}_{L^{\infty}(X\setminus E,\meas)}\, . \label{eq:inftyboundtraceextsharp}
\end{align}
\end{theorem}


\subsubsection{Operations with sets of finite perimeter}

In order to build competitors for variational problems, we will rely on the following characterization theorem for the perimeter and the Gauss-Green measure of intersections, union and differences of sets of finite perimeter, that has been obtained in \cite[Theorem 4.11]{BruePasqualettoSemola21}. We refer to \cite[Theorem 16.3]{Maggi12} for the analogous statement for sets of finite perimeter on $\setR^n$. 

 Recall the definitions of essential boundary $\partial^* E$ given in \eqref{eq:defEssBoundary} and of Gauss-Green measure $\mu_E$ given in \eqref{eq:defGGMeas} for a set of finite perimeter $E\subset X$. We refer also to \cite[Definition 4.9]{BruePasqualettoSemola21} for the introduction of the set of coincidence $\{\nu_E=\nu_F \}$ of the unit normals to two sets of finite perimeter $E$ and $F$.

\begin{theorem}\label{thm:cutandpaste}
Let $(X,\dist,\haus^N)$ be an $\RCD(K,N)$ metric measure space and let $E,F\subset X$ be sets of finite perimeter. 
Let us set
\begin{equation*}
\{\nu_E=\nu_F\}:=\{x\in \partial^*E\cap\partial^*F\, :\, \nu_E=\nu_F\}
\end{equation*}
and 
\begin{equation*}
\{\nu_E=-\nu_F\}:=\{x\in \partial^*E\cap\partial^*F\, :\, \nu_E=-\nu_F\}\, .
\end{equation*}
Then $E\cap F$, $E\cup F$ and $E\setminus F$ are sets of finite perimeter; moreover the following hold:
\begin{align}
\mu_{E\cap F}&=\mu_E\res F^{(1)}+\mu_F\res E^{(1)}+\nu_E\haus^{N-1}\res\{\nu_E=\nu_F\}\, , \label{eq:GGEcapF} \\
\mu_{E\cup F}&=\mu_E\res F^{(0)}+\mu_F\res E^{(0)}+\nu_E\haus^{N-1}\res\{\nu_E=\nu_F\}\, , \ \label{eq:GGEcupF} \\
\mu_{E\setminus F}&=\mu_E\res F^{(0)}-\mu_F\res E^{(1)}+\nu_E\haus^{N-1}\res\{\nu_E=-\nu_F\}\, . \label{eq:GGEminF}
\end{align}
\end{theorem}

\begin{remark}
Let us clarify the meaning of \eqref{eq:GGEcapF}, \eqref{eq:GGEcupF} and \eqref{eq:GGEminF}. With this notation, \eqref{eq:GGEcapF} means that (and analogously for the others)  for any vector field $v\in H^{1,2}_C(TX)\cap D(\div )$ such that $\abs{v}\in L^{\infty}(\meas)$,
\begin{align*}
\int _{E\cap F}\div v\di\meas=&-\int_{F^{(1)}}<\mathrm{tr}_E v,\nu_E>\di\Per_E-\int_{E^{(1)}} <\mathrm{tr}_F v,\nu_F>\di\Per_F\\
&-\int_{E^{(1/2)}\cap F^{(1/2)}} <\mathrm{tr}_E v,\nu_E>\di\Per_E\, .
\end{align*}
\end{remark}

\begin{corollary}\label{cor:GGmeasureincluded}
Let $(X,\dist,\haus^N)$ be an $\RCD(K,N)$ metric measure space and let $E\subset F\subset X$ be sets of finite perimeter. Then $\nu_E=\nu_F$ on $\partial ^*E\cap\partial ^*F$ and 
\begin{equation*}
\mu_E=\mu_E\res F^{(1)}+\nu_F\haus^{N-1}\res\left(\partial^*E\cap\partial ^*F\right)\, .
\end{equation*}

\end{corollary}

We wish to understand to which extent the cut and paste operations for sets of finite perimeter are well behaved under the weaker regularity assumptions of \autoref{thm:GaussGreenRCDBCM}. This is the content of \cite[Proposition 5.4]{BruePasqualettoSemola21} that we report below. 

\begin{proposition}\label{prop:cutandpasteweaker}
Let $(X,\dist,\meas)$ be an $\RCD(K,N)$ metric measure space. Let $E,F\subset X$ be sets of (locally) finite perimeter and let $V\in\mathcal{DM}^{\infty}(X)$. Then the following hold:
\begin{align*}
\left(V\cdot \nu_{E\cap F}\right)_{int}&=\left(V\cdot \nu_{E}\right)_{\mathrm{int}}\, , \quad\text{$\Per$-a.e. on $F^{(1)}$}\, ,\\
\left(V\cdot \nu_{E\cap F}\right)_{int}&=\left(V\cdot \nu_{F}\right)_{\mathrm{int}}\, , \quad\text{$\Per$-a.e. on $E^{(1)}$}\, , \\
\left(V\cdot \nu_{E\cap F}\right)_{int}&=\left(V\cdot \nu_{E}\right)_{\mathrm{int}}\, , \quad\text{$\Per$-a.e. on $E^{(1/2)}\cap F^{(1/2)}$}\, .
\end{align*}
Analogous conclusions hold for the exterior normal traces and for the interior and exterior normal traces on $E\cup F$ and on $E\setminus F$.
\end{proposition}

Another technical result which is needed for the strategy we overviewed in \autoref{subsec:overview} is a rigorous version, within our framework, of the fact that the outward-pointing unit normal to a sub-level set of a distance function is the gradient of the distance function itself. We refer to \cite[Proposition 6.1]{BruePasqualettoSemola21} for its proof.

\begin{proposition}\label{prop:leveld}
Let $(X,\dist,\meas)$ be an $\RCD(K,N)$ metric measure space. Let $\Omega\Subset\Omega'\subset X$ be open domains and let $\phi:\Omega'\to\setR$ be a $1$-Lipschitz function such that 
\begin{itemize}
\item[i)] $\abs{\nabla \phi}=1$,\; $\meas$-a.e. on $\Omega'$;
\item[ii)] $\phi$ has measure valued Laplacian on $\Omega'$ with $\meas$-essentially bounded negative (or positive) part.
\end{itemize}
Then, for $\Leb^1$-a.e. $t$ such that $\{\phi=t\}\cap \Omega\neq\emptyset$, it holds that $\{\phi<t\}$ is a set of locally finite perimeter in $\Omega$; moreover, the following holds: 
\begin{equation*}
\left(\nabla\phi\cdot\nu_{\{\phi<t\}}\right)_{\mathrm{int}}=\left(\nabla\phi\cdot\nu_{\{\phi<t\}}\right)_{\mathrm{ext}}=-1\, \quad \Per_{\{\phi<t\}}\text{-a.e. on  $\Omega$} \, .
\end{equation*} 
\end{proposition}


\subsubsection{Some regularity results for quasi-minimizers}

Let us recall the definition of quasi-minimal set of finite perimeter in this framework.

\begin{definition}[Quasi-minimality]\label{def:quasimin}
Let $(X,\dist,\meas)$ be a metric measure space verifying the  doubling and Poincar\'e inequalities . Let $E\subset X$ be a Borel set with finite perimeter and $\Omega\subset X$ be an open set. Given any $\kappa \ge 1$ we say that $E$ is a $\kappa$-quasi-minimal set if for any $U\Subset \Omega$ and for all Borel sets $F,G\subset U$ it holds 
\begin{equation*}
\Per(E,U)\le \kappa \Per\left((E\cup F)\setminus G, U\right)\, .
\end{equation*}
\end{definition}

In the Euclidean setting, or on smooth Riemannian manifolds, quasi-minimality is a property shared by minimizers of many variational problems: the Plateau problem, the prescribed mean curvature problem, Cheeger sets and isoperimetric sets, among others. We refer to \cite[Chapter 21]{Maggi12} for a throughout discussion and references. This is indeed a general principle that holds also on $\RCD(K,N)$ metric measure spaces $(X,\dist,\haus^N)$: 
\begin{itemize}
\item perimeter minimizers are quasi minimizers as it directly follows from the definition; 
\item with minor modifications to the classical Euclidean proof it is possible to argue that solutions of the prescribed mean curvature problem are quasi minimizers under suitable assumptions; 
\item in \cite[Theorem 3.4]{AntonelliPasqualettoPozzetta21} it has been recently shown that isoperimetric sets are quasi minimizers.
\end{itemize}
\smallskip

A stronger notion involves a function in place of the constant $\kappa$, whose behaviour forces the set to be more and more almost minimizing inside smaller and smaller balls.
\begin{definition}\label{def:omegamin}
Let $(X,\dist,\meas)$ be a metric measure space verifying the  doubling and Poincar\'e inequalities. 
Given an increasing function $\omega:[0,\infty)\to[0,\infty]$ such that $\omega(0)=0$, we say that a set of finite perimeter $E\subset X$ is an $\omega$-minimizer if, for any $x\in X$ and $r>0$, for any $F\subset X$ such that $E\Delta F\Subset B_r(x)$, it holds
\begin{equation*}
\Per(E,B_r(x))\le (1+\omega(r))\Per(F,B_r(x))\, .
\end{equation*}

\end{definition}

\begin{remark}\label{rm:equivquasimin}
An equivalent reformulation of the quasi-minimality condition above is that $E$ is a $\kappa$-quasi-minimal set if for any $U\Subset \Omega$ and for all Borel sets $F\subset X$ such that $E\Delta F\Subset U$ it holds 
\begin{equation}\label{eq:qmin}
\Per(E,U)\le \kappa \Per\left(F, U\right)\, .
\end{equation}
Notice that $\kappa$-quasi-minimality for $\kappa=1$ corresponds to minimality, while it is a weaker notion for $\kappa>1$.
\end{remark}

\begin{remark}
We will sometimes work with the weaker assumption that \eqref{eq:qmin} holds for competitors $F$ such that $E\Delta F$ is supported in $B_r(x)$, where $r>0$ is fixed. This corresponds to a localized version of the quasi-minimality condition, which has the same consequences at the level of regularity.
\end{remark}

One of the main results in \cite{Kinnunenetal13} is the following theorem, asserting that a quasi-minimal set of finite perimeter, up to modification on a negligible set as in \eqref{eq:normsetsfin}, has measure theoretic boundary coinciding with the topological boundary. This is a generalization of the Euclidean result 
in \cite{DeGiorgi61}.

\begin{theorem}[Theorem 4.2 of \cite{Kinnunenetal13}]\label{thm:regqmin}
Let $E\subset X$ be a quasi-minimal set in $\Omega$. Then, up to modifying $E$ on a $\meas$-negligible set, there exists $\gamma_0>0$ such that, for any $x\in \partial E\cap \Omega$, we have
\begin{equation}\label{eq:measboundqmin}
\frac{\meas(E\cap B_r(x))}{\meas(B_r(x))}\ge \gamma_0\, ,\quad \frac{\meas(B_r(x)\setminus E)}{\meas(B_r(x))}\ge \gamma_0\, ,
\end{equation}
for any $r>0$ such that $B_{2r}(x)\subset \Omega$. The density constant $\gamma_0$ depends only on the quasi-minimality constant $\kappa$, the doubling constant and the Poincar\'e constant.
\end{theorem}
 
Given the measure bounds \eqref{eq:measboundqmin}, perimeter bounds follow from the isoperimetric inequality \eqref{eq:isop}.

\begin{corollary}[Lemma 5.1 of \cite{Kinnunenetal13}]\label{cor:perest}
Let $E\subset X$ be a quasi-minimal set in $\Omega$. Then there exist $r_0>0$ and $C>0$ such that for any $x\in\partial E\cap \Omega$ and $0<r<r_0$, it holds
\begin{equation}\label{eq:perboundqmin}
C^{-1}\frac{\meas(B_r(x))}{r}\le \Per(E,B_r(x))\le C\frac{\meas(B_r(x))}{r}\, ,
\end{equation}
whenever $B_{2r}(x)\subset \Omega$. The constants $C>0$ and $r_{0}>0$ depend only on the quasi-minimality constant $\kappa$, the doubling constant and the Poincar\'e constant.
\end{corollary}



The main outcome of \autoref{thm:regqmin}, together with \cite{Am01,Am02} and \cite{AmbrosioBrueSemola19}, is that, in the framework of noncollapsed $\RCD(K,N)$ metric measure spaces, the reduced boundary of a quasi-minimal set of finite perimeter is closed.

\begin{corollary}
Let $(X,\dist,\haus^N)$ be an $\RCD(K,N)$ metric measure space. Let $E\subset X$ be a set of finite perimeter and $\Omega\subset X$ be an open set such that $E$ is quasi-minimal in $\Omega$. Then, up to a modification of $E$ on an $\haus^N$-negligible set, it holds that:
\begin{itemize}
\item[(i)] the perimeter measure $\Per$ coincides with $\haus^{N-1}\res\partial E$ on $\Omega$ (up to a normalization constant);
\item[(ii)] $\partial E$ is $\haus^{N-1}$-rectifiable and $\haus^{N-1}\res\partial E$ is a locally Ahlfors regular measure.
\end{itemize} 
\end{corollary}

\begin{proof}
The identification of the reduced boundary with the topological boundary follows from \autoref{thm:regqmin}.\\
Rectifiability of the reduced boundary (and hence of the topological boundary) and identification of the perimeter measure with the $(N-1)$-Hausdorff measure are then consequences of \autoref{thm:rectifiability} and \autoref{cor:Noncollapsed}.
\end{proof}


A classical consequence of the local Ahlfors regularity of the perimeter for quasi-minimal sets is a measure estimate for the tubular neighbourhood of their boundaries. Given a subset $U\subset X$ and $r>0$, 
 we adopt the notation that $U^{r}:=\{ x\in X \,:\, \dist(x,U)<r \}$ denotes the $r$-enlargement of $U$.

\begin{lemma}\label{lemma:tubneighbounds}
There exist constants $C_{\kappa,K,N}>0$ and $0<r_0= r_0(\kappa,K,N)<1$ with the following property.
Let $(X,\dist,\meas)$ be an $\RCD(K,N)$ metric measure space and let $E\subset X$ be a set of finite perimeter. Assume that $E \cap B_{2}(x)$ is $\kappa$-quasi-minimal in $B_{2}(x)$. Then, for any open subset $\Omega\subset B_{1}(x)$ it holds
\begin{equation*}
\meas\left(\{x\in X\, :\, \dist(x,\partial E \cap \Omega )\le r\}\right)\le C_{\kappa,K,N} \, r\, \Per(E, B_{2}(x)) \, 
\end{equation*}
for every $r\in (0, r_{0})$.\\
In particular, if $E \cap \Omega$ is locally perimeter minimizing in $B_{2}(x)$, then the dependence on $\kappa$ in the constant $C_{\kappa,K,N}>0$ can be dropped.
\end{lemma}

\begin{proof}
By \autoref{cor:perest}, there exist $r_0=r_{0}(\kappa, K,N)>0$ and $C=C_{\kappa, K,N}>0$ such that, for any $x\in\partial E \cap \Omega$ and for any $r\in (0, r_0)$ it holds
\begin{equation}\label{eq:contrqmin}
C^{-1}\frac{\meas(B_r(x))}{r}\le \Per(E,B_r(x))\le C\frac{\meas(B_r(x))}{r}\, .
\end{equation}
We wish to estimate the volume of the tubular neighbourhood of $\partial E \cap \Omega$.\\ 
Let $r<r_0/5$ be fixed and let us consider, thanks to Vitali's covering lemma, a covering of $\{x\in X\, :\, \dist(x,\partial E \cap \Omega )\le r\}$ with balls  $B_{5r_i}(x_i)$ such that $x_i\in\partial E \cap \Omega$, $r_i<r<r_0/5$ and $\{B_{r_i}(x_i)\}$ is a disjoint family of subsets of $B_2(x)$.
Relying on \eqref{eq:contrqmin} and the disjointedness of the family $\{B_{r_i}(x_i)\}$ we can estimate
\begin{align*}
\meas\left(\{x\in X\, :\, \dist(x,\partial E \cap \Omega)\le r\}\right)\le&\,  \meas\left(\bigcup_{i}B_{5r_i}(x_i)\right)\le\,  \sum_{i}\meas\left(B_{5r_i}(x_i)\right)\\
\le &\, C_{K,N}\sum_{i}\meas(B_{r_i}(x_i))\\
\le &\, C_{\kappa,K,N}\sum_{i}\Per(E,B_{r_i}(x_i))r_i\\
\le &\,  C_{\kappa,K,N}r\Per(E, B_2(x))\, .
\end{align*}
\end{proof}

In the Euclidean setting a well known fact is that, when dealing with a family of sets of finite perimeter that are uniformly quasi-minimizing, the usual $L^1_{\loc}$ convergence up to subsequence guaranteed for uniformly bounded $\BV$ functions can be improved. We refer for instance to \cite[Section 21.5]{Maggi12} and references therein for the treatment of this topic on $\setR^n$.\\ 
This principle has already played a role in the proof of De Giorgi's theorem for sets of finite perimeter on $\RCD(K,N)$ spaces in \cite{AmbrosioBrueSemola19}. Below we present a slight enforcement of \cite[Proposition 3.9]{AmbrosioBrueSemola19}, allowing for more general quasi-minimality conditions and dealing with the Hausdorff convergence of the topological/measure theoretic boundaries.

\begin{theorem}\label{thm:closurecompactnesstheorem}
	Let $(X_i,\dist_i,\meas_i,x_i)$ be $\RCD(K,N)$ m.m. spaces converging in the pmGH topology to $(Y,\varrho,\mu,y)$ and let $(Z,\dist_Z)$ be realizing the convergence. For any $i\in\setN$, let $\omega_i:[0,\infty):\to [0,\infty)$ be a modulus of continuity and let $E_i\subset X_i$ be sets of finite perimeter satisfying the following 
	$\omega_i$-minimality condition: there exists $R_i>0$ such that
	\begin{equation*}
	\abs{D\chi_{E_{i}}}(B_{r}(z_i))\le (1+\omega_{i}(r)) \abs{D\chi_{E'}}(B_{r}(z_i))
		\end{equation*}
	for any $E'\subset X_i$ such that  $E_i\Delta E'\Subset B_{r}(z_i)\subset X_i$, for some $r<R_i$.

	Assume that, as $i\to\infty$, $E_i\to F$ in $L^1_{\loc}$ for some set $F\subset Y$ of locally finite perimeter, and $\omega_i\to \omega$ pointwise, where $\omega:[0,\infty)\to[0,\infty)$ is a modulus of continuity and $R_i\to \infty$. 
	Then:
	\begin{itemize}
		\item[(i)] $F$ is an entire $\omega$-minimizer of the perimeter (relative to $(Y,\varrho,\mu)$), namely
		\begin{equation}\label{eq:limitomegamin}
		\abs{D\chi_F}(B_r(z))\leq (1+\omega(r))\abs{D\chi_{F'}}(B_r(z)) 
		\end{equation}
		whenever $F\Delta F'\Subset B_r(z)\Subset Y$ and $r>0$;
		\item[(ii)] $|D\chi_{E_i}|\to|D\chi_F|$ in duality with $C_{\mathrm{bs}}(Z)$ as $i\to\infty$;
		\item[(iii)] $\partial E_i\to \partial F$ in the Kuratowski sense as $i\to\infty$.
	\end{itemize}
\end{theorem}

\begin{proof}
The statement is classical in the Euclidean setting, see for instance \cite{AmbrosioPaolini99}, and the adaptation to the present framework requires only minor adjustments. Therefore some details will be omitted. 
We will adapt  the arguments in the proof of \cite[Proposition 3.9]{AmbrosioBrueSemola19} to deal with the present setting.

The strategy is to consider a weak limit measure of the sequence of locally uniformly bounded perimeter measures $\abs{D\chi_{E_i}}$. Let us call it $\nu$. Then we show simultaneously that $\nu=\abs{D\chi_F}$ and that $F$ verifies the $\omega$-minimality condition \eqref{eq:limitomegamin}.

The inequality $\abs{D\chi_F}\le \nu$ follows from localizing the lower-semicontinuity of the perimeter \cite[Proposition 3.6]{AmbrosioBrueSemola19}, and does not require the $\omega$-minimality condition. It remains to check that $\nu\le \abs{D\chi_F}$.  Below we report part of the proof in \cite{AmbrosioBrueSemola19} and indicate where changes are needed.
%

	Let us fix $\bar y\in Y$ and let $F'\subset Y$ be a set of locally finite perimeter satisfying $F\Delta F'\Subset B_r(\bar{y})$.  
	Let $\bar{x}_i\in X_i$ converging to $\bar y$ in $Z$ and $R>0$ be such that the following properties hold true:
	\begin{equation}\label{eq:z20}
	\sup_{i\in \setN} \abs{D\chi_{B_R(x_i)}}(X_i)<\infty
	\qquad
	\text{and}
	\qquad
	B_r(\bar{x}_i)\Subset B_R(x_i)\qquad \forall i\in \setN\, .
	\end{equation}

	Using \cite[Proposition 3.8]{AmbrosioBrueSemola19}
	we can find a sequence of sets of finite 
	perimeter $E'_i\subset X_i$ converging to $F\cap B_R(y)$ in $\BV$ energy
	(notice that $F\cap B_R(y)$ is a set of finite perimeter thanks to \eqref{eq:z20}).

We claim that, for any set of finite perimeter $F'\subset Y$ such that $F\Delta F'\Subset B_r(\bar y)$, 
	\begin{equation}\label{eq:z18}
	\nu(B_s(\bar y))\le (1+\omega(r))\abs{D\chi_{F'}}(B_s(\bar y))\, ,
	\end{equation}
	for $\Leb^1$-a.e. $s\in (r',r)$, for some $0<r'<r$.\\
	Let us illustrate how to use \eqref{eq:z18} to conclude the proof. 
	
	If we apply \eqref{eq:z18} with $F'=F$ we get that 
\begin{equation*}
\nu(B_s(\bar y))\le (1+\omega(r))\abs{D\chi_F}(B_s(\bar y))\, ,\ \ 
\end{equation*}	
for $\Leb^1$-a.e. $s\in (r',r)$, for some $0<r'<r$.\\
Hence, letting $s\uparrow r$, we obtain
\begin{equation}\label{eq:slighimpr}
\nu(B_r(\bar y))\le (1+\omega(r))\abs{D\chi_F}(B_r(\bar y))\, .
\end{equation}
In particular $\nu\ll\abs{D\chi_F}$, which is an asymptotically doubling measure. Hence, noticing that by \eqref{eq:slighimpr} and the continuity at $0$ of $\omega$,
\begin{equation*}
\limsup_{r\downarrow 0}\frac{\nu(B_r(\bar{y}))}{\abs{D\chi_F}(B_{r}(\bar y))}\le \limsup_{r\downarrow 0}(1+\omega(r))=1\, ,
\end{equation*}
we can apply the differentiation theorem to infer that $\nu\le \abs{D\chi_F}$.
This proves (ii). 
\\Substituting back in \eqref{eq:z18}, we obtain that
	\begin{equation*}
	\abs{D\chi_F}(B_s(\bar y))\le (1+\omega(r))\abs{D\chi_{F'}}(B_s(\bar y))\, ,
	\end{equation*}
	for $\Leb^1$-a.e. $s\in (r',r)$, for some $0<r'<r$, 
and (i) follows by letting $s\uparrow r$.
\medskip

	Let us prove \eqref{eq:z18}. We first fix $0<r'<r$ such that $F\Delta F'\subset B_{r'}(y)$. Then
	we fix a parameter $s\in (r',r)$ with $\nu(\partial B_s(\bar y))=0$, $\abs{D\chi_{F'}}(\partial B_s(\bar y))=0$ and set
	\begin{equation*}
	\tilde{E}_i^s:=\left(E_i'\cap B_s(\overline{x}_i)\right)\cup \left(E_i\setminus B_s(\overline{x}_i)\right)\, .
	\end{equation*}
We also choose $s<s'<r$ such that $\nu(\partial B_{s'}(\bar y))=0$.

	From now on, up to the end of the proof, we are going to adopt the notation $\Per(G,A)$ to denote $\abs{D\chi_G}(A)$ whenever $G$ has finite perimeter and $A$ is a Borel set, to avoid multiple subscripts.
	
	Using the locality of the perimeter and the $\omega_i$-minimality of $E_i$ (notice that $R_i\ge r$ for $i$ big enough), we get
	\begin{align}
\nonumber	\Per(E_i, \overline{B}_s(\bar{x}_i)) =\,  & \Per(E_i, B_{s'}(\bar{x_i}))-\Per(E_i,B_{s'}(\bar{x_i})\setminus\overline{B}_s(\bar{x}_i))\\
\nonumber	\le\,  &  (1+\omega_i(r)) \Per(\tilde{E}^s_i, B_{s'}(\bar{x}_i))-\Per(E_i,B_{s'}(\bar{x}_i)\setminus\overline{B}_s(\bar{x}_i))\\
\nonumber 	=\, &(1+\omega_i(r)) \Per(\tilde{E}^s_i, B_s(\bar{x}_i))+(1+\omega_i(r))\Per(\tilde{E}^s_i, \partial B_s(\bar{x}_i))\\
\nonumber	&+(1+\omega_i(r)) \Per(\tilde{E}^s_i,B_{s'}(\bar{x}_i)\setminus\overline{B}_s(\bar{x}_i))\\
&-\Per(E_i,B_{s'}(\bar{x}_i)\setminus\overline{B}_s(\bar{x}_i))\\
\nonumber	=\, &(1+\omega_i(r)) \Per(E_i', B_s(\bar{x}_i))+(1+\omega_i(r))\Per(\tilde{E}^s_i, \partial B_s(\bar{x}_i))\\
	 &+\omega_i(r)\Per(E_i,B_{s'}(\bar{x}_i)\setminus\overline{B}_s(\bar{x}_i))\, .\label{al:deco}
	\end{align}
	Taking the limit as $i\to \infty$, arguing as in the last part of the proof of \cite[Proposition 3.9]{AmbrosioBrueSemola19} it is possible to prove that	
	\begin{equation}\label{eq:z19}
	\liminf_{i\to \infty}\Per(\tilde{E}_i^s, \partial B_s(\bar{x}_i))=0\, ,
	\qquad
	\text{for a.e.}\ s\in (r',r).
	\end{equation}
	
	Thanks to our choice of $s$, it holds that
	$\Per(E_i, \overline{B}_s(\bar{x}_i))\to \nu(B_s(\bar y))$ and moreover $(1+\omega_i(r)) \Per(E_i', B_s(\bar{x}_i))\to (1+\omega(r))\Per(F',B_s(\bar y))$, since $\chi_{E_i'}\to \chi_{F'\cap B_R(y)}$ in $\BV$ energy
	and therefore \cite[Corollary 3.7]{AmbrosioBrueSemola19}
	applies.  Combining these last observations with \eqref{eq:z19} and \eqref{al:deco} we obtain that
\begin{equation*}
\nu(B_s(\bar{y}))\le (1+\omega(r))\abs{D\chi_{F'}}(B_s(\bar y))+\omega(r)\nu(B_{s'}(\bar y)\setminus\overline{B}_s(\bar y))\, .
\end{equation*}	
Letting then $s\uparrow s'$ we infer
\begin{equation*}
\nu(B_{s'}(\bar y))\le (1+\omega(r))\abs{D\chi_{F'}}(B_{s'}(\bar y))\, ,
\end{equation*}
which is equivalent to \eqref{eq:z18} up to changing $s'$ into $s$.


\medskip
In order to prove (iii), it is enough to observe that all the $E_i$'s and the limit set of finite perimeter $F$ verify uniform upper and lower density estimates, thanks to $\omega_i$-minimality, convergence of $\omega_i$ to $\omega$, \autoref{thm:regqmin} and \autoref{cor:perest}.\\
By (ii) and the lower density estimate for $\abs{D\chi_F}$, any point in $\partial F$ can be approximated by points in $\partial E_i$. On the other hand, limit points of sequences $x_i\in\partial E_i$ do belong to $\partial F$ due to the uniform density estimates at $x_i$ and weak convergence of $\abs{D\chi_{E_i}}$ again. We refer to \cite[Section 7]{BrueNaberSemola20} for an analogous statement in the case of boundaries of noncollapsed $\RCD(K,N)$ spaces.
\end{proof}

\subsection{Laplacian, heat equation and heat kernel}\label{subsec:heat}

Unless otherwise stated from now on we assume that $(X,\dist,\meas)$ is an $\RCD(K,N)$ metric measure space for some $K\in\setR$ and $1\le N<\infty$.
\medskip

In the first part of this subsection we collect some basic notation and results about the Laplacian, the heat flow and the heat kernel, together with some terminology about first and second order differential calculus on $\RCD$ spaces. The basic references for this part are \cite{AmbrosioGigliSavare14,Gigli15,Gigli18}. The second part contains some new technical results about the pointwise short time behaviour of the heat flow.
	\begin{definition}\label{def:laplacian}
		The Laplacian $\Delta:D(\Delta)\to L^2(X,\meas)$ is a densely defined linear operator whose domain consists of all functions $f\in H^{1,2}(X,\dist,\meas)$ satisfying 
		\begin{equation*}
		\int hg\di\meas=-\int \nabla h\cdot\nabla f\di\meas \quad\text{for any $h\in H^{1,2}(X,\dist,\meas)$}
		\end{equation*}
		for some $g\in L^2(X,\meas)$. The unique $g$ with this property is denoted by $\Delta f$.
			\end{definition}
	As consequence of the infinitesimal hilbertianity, it is easily checked that $\Delta$ is an (unbounded) linear operator. More generally, we say that $f\in H^{1,2}_{\loc}(X,\dist,\meas)$ is in the domain of the measure valued Laplacian, and we write $f\in D(\boldsymbol{\Delta})$, if there exists a Radon measure $\mu$ on $X$ such that, for every $\psi\in\Lip_c(X)$, it holds
	\begin{equation*}
	\int\psi\di\mu=-\int\nabla f\cdot\nabla \psi\di\meas\, .
	\end{equation*} 
	In this case we write $\boldsymbol{\Delta}f:=\mu$. If moreover $\boldsymbol{\Delta}f\ll\meas$ with $L^{2}_{\loc}$ density we denote by $\Delta f$ the unique function in $L^{2}_{\loc}(X,\meas)$ such that $\boldsymbol{\Delta}f=\Delta f\, \meas$ and we write $f\in D_{\loc}(\Delta)$.
	
Notice that the definition makes sense even under the assumption that $f\in H^{1,p}_{\loc}(X,\dist,\meas)$ for some $1\le p<\infty$, and we will rely on this observation later.
\medskip

We shall also consider the Laplacian on open sets, imposing Dirichlet boundary conditions. Let us first introduce the local Sobolev space with Dirichlet boundary conditions.

\begin{definition}
Let $(X,\dist,\meas)$ be an $\RCD(K,N)$ metric measure space and let $\Omega\subset X$ be an open and bounded domain. Then we let $H^{1,2}_{0}(\Omega)$ be the $H^{1,2}(X,\dist,\meas)$ closure of $\Lip_c(\Omega,\dist)$.
\end{definition}

 We  also introduce the local Sobolev space (i.e. without imposing Dirichlet boundary conditions).

\begin{definition}
Let $(X,\dist,\meas)$ be an $\RCD(K,N)$ metric measure space and let $\Omega\subset X$ be an open and bounded domain. We say that a function $f\in L^2(\Omega,\meas)$ belongs to the local Sobolev space $H^{1,2}(\Omega,\dist,\meas)$ if 
\begin{itemize}
\item[i)] $f\phi\in H^{1,2}(X,\dist,\meas)$ for any $\phi\in\Lip_c(\Omega,\dist)$;
\item[ii)] $\abs{\nabla f}\in L^2(X,\meas)$.
\end{itemize}
Above we intend that $f\phi$ is set to be $0$ outside from $\Omega$. Notice that $\abs{\nabla f}$ is well defined on any $\Omega'\subset \Omega$ (and hence on $\Omega$) as $\abs{\nabla (f\phi)}$ for some $\phi\in \Lip_c(\Omega)$ such that $\phi\equiv 1$ on $\Omega'$.
\end{definition}

\begin{definition}
Let $f\in H^{1,2}(\Omega)$. We say that $f\in D(\Delta,\Omega)$ if there exists a function $h\in L^2(\Omega,\meas)$ such that
\begin{equation*}
\int_{\Omega}gh\di\meas=-\int_{\Omega}\nabla g\cdot\nabla f\di\meas\, ,\quad\text{for any $g\in H^{1,2}_0(\Omega,\dist,\meas)$}\, .
\end{equation*}
\end{definition}

We refer to \cite{Gigli18} for  the basic terminology and results about tangent and cotangent modules on metric measure spaces and for the interpretation of vector fields as elements of the tangent modules. The notations $L^2(TX)$, $L^2_{\loc}(TX)$ and $L^{\infty}(TX)$ will be adopted to indicate the spaces of $L^2$, $L^2_{\loc}$ and bounded vector fields, respectively. 

\begin{definition}
Let $V\in L^2(TX)$ be a vector field. We say that $V$ belongs to the domain of the divergence (and write $v\in D(\div)$) if there exists a function $f\in L^2(\meas)$ such that
\begin{equation*}
\int_X V\cdot\nabla g\di\meas=-\int_X fg\di\meas\, ,\quad\text{for any $g\in H^{1,2}(X)$}\, .
\end{equation*}
Under these assumptions, the function $f$ is uniquely determined and we shall denote $f=\div(V)$.
\end{definition}

We refer again the reader to \cite{Gigli18} for the introduction of more regular classes of vector fields, such as the class $H^{1,2}_C(TX)$ that will be relevant later in the paper.
\medskip
	
 The heat flow $P_t$, previously defined in \autoref{sec:SlopeChEnergy}  as  the $L^2(X,\meas)$-gradient flow of $\Ch$, can be equivalently  characterised by the following property:  for any $u\in L^2(X,\meas)$, the curve $t\mapsto P_tu\in L^2(X,\meas)$ is locally absolutely continuous in $(0,+\infty)$ and satisfies
	\begin{equation*}
	\frac{\di}{\di t}P_tu=\Delta P_tu \quad\text{for $\Leb^1$-a.e. $t\in(0,+\infty)$}\, .
	\end{equation*}  
	
	Under our assumptions the heat flow provides a linear, continuous and self-adjoint contraction semigroup in $L^2(X,\meas)$. Moreover $P_t$ extends to a linear, continuous and mass preserving operator, still denoted by $P_t$, in all the $L^p$ spaces for $1\le p<+\infty$.  
\medskip
	
	It has been proved in \cite{AmbrosioGigliSavare14,AmbrosioGigliMondinoRajala15} that, on $\RCD(K,\infty)$ metric measure spaces, the dual heat semigroup $\bar{P}_t:\mathcal{P}_2(X)\to\mathcal{P}_2(X)$ of $P_t$, defined by	
	\begin{equation*}
	\int_X f \di \bar{P}_t \mu := \int_X P_t f \di \mu\qquad\quad \forall \mu\in \mathcal{P}_2(X),\quad \forall f\in \Lipb(X)\, ,
	\end{equation*}
	is $K$-contractive (w.r.t. the $W_2$-distance) and, for $t>0$, maps probability measures into probability measures absolutely continuous w.r.t. $\meas$. Then, for any $t>0$, we can introduce the so called \textit{heat kernel} $p_t:X\times X\to[0,+\infty)$ by
	\begin{equation*}
	p_t(x,\cdot)\meas:=\bar{P}_t\delta_x\, .
	\end{equation*}  

A key property of the heat kernel follows, namely the so-called stochastic completeness: for any $x\in X$ and for any $t>0$ it holds
\begin{equation}\label{eq:stochcompl}
\int_X p_t(x,y)\di\meas(y)=1\, .
\end{equation}
\begin{remark}	
	From now on, for any $f\in L^{\infty}(X,\meas)$ we will denote by $P_tf$ the representative pointwise everywhere defined by
	\begin{equation*}
	P_tf(x)=\int_X f(y)p_t(x,y)\di\meas(y)\, .
	\end{equation*}
\end{remark}
Let us recall a few regularizing properties of the heat flow on $\RCD(K,N)$ spaces (which hold true more generally for any $\RCD(K,\infty)$ m.m.s.) referring again to \cite{AmbrosioGigliSavare14,AmbrosioGigliMondinoRajala15} for a more detailed discussion and the proofs of these results.   
	
	First we have the \textit{Bakry-\'Emery contraction} estimate:
	\begin{equation}\label{eq:BakryEmery}
	\abs{\nabla P_t f}^2\le e^{-2Kt}P_t\abs{\nabla f}^2\quad \text{$\meas$-a.e.,}
	\end{equation}
	for any $t>0$ and for any $f\in H^{1,2}(X,\dist,\meas)$.\\ 
	Later on it was proved in \cite{Savare14} that the Bakry-\'Emery contraction estimates extends to the full range of exponents $p\in [1,\infty)$, i.e.
	\begin{equation}\label{eq:selfimpr}
	\abs{\nabla P_tf}^p\le e^{-pKt}P_t\abs{\nabla f}^p\, ,\quad\text{$\meas$-a.e.}\, ,
	\end{equation}
	for any $t>0$, for any function $f\in H^{1,p}(X,\dist,\meas)$ if $p>1$ and for any function $f\in \BV(X,\dist,\meas)$ if $p=1$. 
\medskip
	
	Another non trivial regularity property is the so-called \textit{$L^{\infty}$-- $\Lip$ regularization} of the heat flow:  for any $f\in L^{\infty}(X,\meas)$, we have $P_tf\in\Lip(X)$ with
	\begin{equation}\label{eq:linftylipregularization}
	\sqrt{2I_{2K}(t)}\LipConst(P_tf)\le\norm{f}_{L^\infty}\, ,\quad\text{for any $t>0$}\, ,
	\end{equation}
	where $I_L(t):=\int_0^te^{Lr}\di r$.
\medskip
	
We also have the so-called \textit{Sobolev to Lipschitz property}: any $f\in H^{1,2}(X,\dist,\meas)$ with $\abs{\nabla f}\in L^{\infty}(X,\meas)$ admits a Lipschitz representative $\bar{f}$ such that $\LipConst \bar{f}\le \norm{\nabla f}_{\infty}$.
\begin{definition}	
	We introduce the space of ``test'' functions $\Test(X,\dist,\meas)$ by 
	\begin{align}\label{eq:test}
\nonumber	\Test(X,\dist,\meas):=&\{f\in D(\Delta)\cap L^{\infty}(X,\meas): \abs{\nabla f}\in L^{\infty}(X)\\
	&\quad\text{and}\quad\Delta f\in H^{1,2}(X,\dist,\meas) \}\, .
	\end{align}
	and the subspace $\Test^{\infty}(X,\dist,\meas)$ by 
	\begin{align}\label{eq:testinfty}
\nonumber	\Test^{\infty}(X,\dist,\meas):=&\{f\in D(\Delta)\cap \Lip_b(X)\\
	&\quad\text{and}\quad\Delta f\in L^{\infty}\cap H^{1,2}(X,\dist,\meas) \}\, .
	\end{align}	 
\end{definition}

\begin{remark}
	We remark that, for any $g\in L^{2}\cap L^{\infty}(X,\meas)$, it holds that $P_tg\in \Test(X,\dist,\meas)$ for any $t>0$, thanks to \eqref{eq:BakryEmery}, \eqref{eq:linftylipregularization}, the fact that $P_t$ maps $L^2(X,\meas)$ into $D(\Delta)$ and the commutation $\Delta P_t f= P_t\Delta f$, which holds true for any $f\in D(\Delta)$.
\end{remark}
On $\RCD(K,N)$ metric measure spaces it is possible to build \textit{regular} cut-off functions, see \cite[Lemma 3.1]{MondinoNaber19} (the Test regularity was not required in \cite{MondinoNaber19} but can be obtained with a similar construction, see also \cite[Lemma 6.7]{AmbrosioMondinoSavare14} and \cite{Gigli18}).

\begin{lemma}\label{lemma:cutoff}
Let $(X,\dist,\meas)$ be an $\RCD(K,N)$ metric measure space. Then, for any $R>0$ there exists a constant $C=C(K,N,R)>0$ such that, for any $x\in X$ and for any $0<r<R$, there exists a function $\phi_r:X\to[0,\infty)$ such that the following properties hold:
\begin{itemize}
\item[i)] $\phi_r\equiv 1$ on $B_r(x)$ and $\phi_r\equiv 0$ outside from $B_{2r}(x)$;
\item[ii)] $\phi_r$ is Lipschitz and belongs to $D(\Delta)$, moreover
\begin{equation*}
r^2\abs{\Delta \phi_r}+r\abs{\nabla \phi_r}\le C(K,N,R)\, .
\end{equation*}
\item[iii)] $\phi_r\in \Test(X,\dist,\meas)$.
\end{itemize}
\end{lemma}

\medskip

Since $\RCD(K,N)$ spaces are locally doubling and they satisfy a local Poincaré inequality (see \cite{VonRenesse08, Rajala12}), the general theory of Dirichlet forms guarantees that we can find a locally H\"older continuous heat kernel $p$ on $X\times X\times(0,+\infty)$, see \cite{Sturm96}.
\medskip
	
	Moreover in \cite{JangLiZhang} the following finer properties of the heat kernel over $\RCD(K,N)$ spaces, have been proved: there exist constants $C_1=C_1(K,N)>1$ and $c=c(K,N)\ge0$ such that
	\begin{align}\label{eq:kernelestimate}
\nonumber	\frac{1}{C_1\meas(B_{\sqrt{t}}(x))}\exp\left\lbrace -\frac{\dist^2(x,y)}{3t}-ct\right\rbrace&\le p_t(x,y)\\
	&\le \frac{C_1}{\meas(B_{\sqrt{t}}(x))}\exp\left\lbrace-\frac{\dist^2(x,y)}{5t}+ct \right\rbrace  
	\end{align}
	for any $x,y\in X$ and for any $t>0$. Moreover it holds
	\begin{equation}\label{eq:gradientestimatekernel}
	\abs{\nabla p_t(x,\cdot)}(y)\le \frac{C_1}{\sqrt{t}\meas(B_{\sqrt{t}}(x))}\exp\left\lbrace -\frac{\dist^2(x,y)}{5t}+ct\right\rbrace \quad\text{for $\meas$-a.e. $y\in X$},
	\end{equation}
	for any $t>0$ and for any $x\in X$.\\
	We remark that in \eqref{eq:kernelestimate} and \eqref{eq:gradientestimatekernel} above one can take $c=0$ whenever $(X,\dist,\meas)$ is an $\RCD(0,N)$ m.m.s..
\medskip

It is also possible to combine the upper bound for the heat kernel in \eqref{eq:kernelestimate} with the general theory of the heat kernels (see again \cite{Sturm96}) to infer that
\begin{equation*}
\abs{\frac{\di}{\di t}p_t(x,y)}=\abs{\Delta_xp_t(x,y)}\le \frac{C}{t\meas(B_{\sqrt{t}}(x))}\exp\left\lbrace -\frac{\dist^2(x,y)}{5t}+ct\right\rbrace\, ,
\end{equation*}
for all $t>0$ and $\meas\otimes\meas$-a.e. $(x,y)\in X\times X$.
\medskip

We will deal several times with the heat flow for initial data with polynomial growth, i.e. for those functions $f:X\to\setR$ such that for some $n\in\setN$, some constant $C>0$ and $x\in X$ it holds 
\begin{equation}\label{eq:polgrowth}
\abs{f(y)}\le C\dist(x,y)^n+ C\, , \quad\text{for any $y\in X$}\, .
\end{equation}
In this case the evolution via heat flow can be pointwise defined by
\begin{equation}\label{eq:pointdefheat}
P_tf(x):=\int_Xp_t(x,y)f(y)\di\meas(y)\, ,
\end{equation}
for any $x\in X$ and for any $t>0$.\\
Observe that the integral in \eqref{eq:pointdefheat} is absolutely convergent thanks to the upper heat kernel estimate in \eqref{eq:kernelestimate}, the Bishop-Gromov inequality \eqref{eq:BishopGromovInequality} and the polynomial growth assumption \eqref{eq:polgrowth}.  
\medskip

Whenever $f:X\to\setR$ has polynomial growth, it belongs to the domain of the Laplacian locally and has Laplacian with polynomial growth, it is possible to verify that $P_tf$ belongs to the domain of the Laplacian locally and 
\begin{equation}\label{eq:pointlaplac}
\Delta P_tf(x)=\int_X\Delta p_t(x,y)f(y)\di\meas(y)=\int_Xp_t(x,y)\Delta f(y)\di\meas(y)\, ,
\end{equation}
for any $x\in X$ and for any $t>0$. Then one can easily argue that
\begin{equation*}
\frac{\di}{\di t}P_tf(x)=\Delta P_tf(x)\, ,\quad\text{for a.e. $t>0$ and every $x\in X$}\, .
\end{equation*}

\medskip

Among the consequences of the Gaussian bounds there is the fact that the heat kernel is strictly positive. It follows that, whenever $f\in L^1_{\loc}(X,\meas)$ has polynomial growth and $f\ge 0$, then $P_tf$ is strictly positive at any point and any positive time unless $f\equiv 0$. Below we wish to show that, nevertheless, the action of the heat flow is still \textit{local}, to some extent.

\begin{lemma}\label{lemma:sufflocal}
Let $(X,\dist,\meas)$ be an $\RCD(K,N)$ metric measure space for some $K\in\setR$ and $1\le N<\infty$. Let $f\in L^1_{\loc}(X,\meas)$ be a function with polynomial growth and assume that there exist $x_{0}\in X$ and $r_{0}>0$ such that $f\equiv 0$ on $B_{r_{0}}(x_{0})$.
Then, for any $n\in\setN$, 
\begin{equation*}
P_tf(x_{0})=o(t^n)\, ,\quad\text{as $t\downarrow 0$}\, .
\end{equation*}
\end{lemma}

\begin{proof}
Observe that, since $p_t(x,\cdot)$ is a probability measure for any $x\in X$ and for any $t\ge 0$ (see \eqref{eq:stochcompl}), by Jensen's inequality it holds
\begin{equation*}
\abs{P_tf(x)}\le P_t\abs{f}(x)\, ,\quad\text{for any $t\ge 0$ and for any $x\in X$}\, .
\end{equation*}
Therefore we can assume without loss of generality that $f\ge 0$.
\medskip

Using the coarea formula and abbreviating by $\Per_r$ the perimeter measure of the ball $B_r(x)$, we can compute
\begin{equation}\label{eq:firstwr}
P_tf(x)=\int_Xp_t(x,y)f(y)\di\meas(y)=\int_0^{\infty}\int_{\partial B_r(x)}f(y)p_t(x,y)\di\Per_r(y)\di r\, .
\end{equation}

Using the upper bound for the heat kernel in \eqref{eq:kernelestimate} we estimate
\begin{align}\label{eq:secwr}
\nonumber \int_0^{\infty}&\int_{\partial B_r(x)}f(y)p_t(x,y)\di\Per_r(y)\di r\\
&\le \frac{Ce^{ct}}{\meas(B_{\sqrt{t}}(x))}\int_0^{\infty}e^{-\frac{r^2}{5t}}\int_{\partial B_{r}(x)}f(y)\di\Per_r(y)\di r\, .
\end{align}
Let us set now
\begin{equation*}
g(r):=\int_{\partial B_r(x_{0})}f(y)\di\Per_r(y)\, 
\end{equation*}
and 
\begin{equation*}
h(r):=\int_{B_r(x_{0})}f(y)\di\meas(y)\, .
\end{equation*}
By the coarea formula,
\begin{equation*}
h(r)=\int_0^rg(s)\di s\, ,\quad\text{for any $r>0$}\, ,
\end{equation*}
hence $r\mapsto h(r)$ is an absolutely continuous monotone map and 
\begin{equation}\label{eq:dervol}
h'(r)=g(r)\, ,\quad\text{for a.e. $r>0$}\, .
\end{equation}
Moreover, by the polynomial growth assumption and since $f\equiv 0$ on $B_{r_{0}}(x_{0})$, we know that, for any $n\ge n_0$ (where $n_0$ is the order in the polynomial growth assumption), there exists a constant $C=C(n)>0$ such that
\begin{equation}\label{eq:unipolgrowth}
\fint_{B_r(x_{0})}f(y)\di\meas(y)\le Cr^n\, ,\quad\text{for any $r>0$}\, .
\end{equation}
When read in terms of the function $h$, this can be rephrased by
\begin{equation*}
h(r)\le Cr^n\meas(B_r(x_{0}))\, ,\quad\text{for any $r>0$}.
\end{equation*}
With the above introduced notation, \eqref{eq:firstwr} and \eqref{eq:secwr} can be rephrased as
\begin{equation*}
P_tf(x_{0})\le \frac{Ce^{ct}}{\meas(B_{\sqrt{t}}(x_{0}))}\int_0^{\infty}e^{-\frac{r^2}{5t}}g(r)\di r\, .
\end{equation*}
Changing variables in the integral by setting $s:=r/\sqrt{5t}$ and integrating by parts, taking into account \eqref{eq:dervol} and the polynomial growth of $f$ and the Bishop-Gromov inequality \eqref{eq:BishopGromovInequality} to prove vanishing of the boundary terms, we obtain

\begin{align}
\nonumber P_tf(x_{0})\le\, & \frac{Ce^{ct}\sqrt{t}}{\meas(B_{\sqrt{t}}(x_{0}))}\int_0^{\infty}e^{-s^2}g(\sqrt{5t}s)\di s\\
\nonumber \le\,  &\frac{Ce^{ct}\sqrt{t}}{\meas(B_{\sqrt{t}}(x_{0}))}\frac{1}{\sqrt{5t}}\int_0^{\infty}se^{-s^2}h(\sqrt{5t}s)\di s\\
\nonumber \le\,  & \frac{C e^{ct}}{\meas(B_{\sqrt{t}}(x_{0}))}\int_0^{\infty}se^{-s^2}h(\sqrt{5t}s)\di s\\
=\, &C e^{ct}\int_0^{\infty}se^{-s^2}\frac{h(\sqrt{5t}s)}{\meas(B_{\sqrt{t}}(x_{0}))}\di s\, .\label{eq:boundheatconv}
\end{align}
Let us set, for any $0<t<1$ and for any $s>0$,  
\begin{equation*}
\phi_t(s):=\frac{h(\sqrt{5t}s)}{\meas(B_{\sqrt{t}}(x_{0}))}\, .
\end{equation*}
We wish to bound $\phi_s(t)$ in a sufficiently uniform way (w.r.t $t\in(0,1)$) in order to apply Fatou's Lemma and prove that $P_tf(x)=o(t^m)$, for any $m\in\setN$ as $t\downarrow 0$.\\
To this aim, fix $m\in \setN$ and let $n\in \setN,$  with $n>m$. 
We split $(0,\infty)$ into two intervals.
\\ If $s\in (0,1/\sqrt{5})$, then, for any $t\in(0,1)$, we can bound
\begin{equation}\label{eq:unidom1}
\phi_t(s)\le \frac{h(\sqrt{t})}{\meas(B_{\sqrt{t}}(x))}=\fint_{B_{\sqrt{t}}(x)}f(y)\di\meas(y)\le C t^{n/2}\, , 
\end{equation}
where we used \eqref{eq:unipolgrowth} for the last inequality. If  instead $s>1/\sqrt{5}$,  we can bound
\begin{align}
\nonumber \phi_t(s)=\, &\frac{h(\sqrt{5t}s)}{\meas(B_{\sqrt{t}}(x_{0}))}\le \frac{\meas(B_{\sqrt{5t}s}(x_{0}))}{\meas(B_{\sqrt{t}}(x_{0}))}\fint_{B_{\sqrt{5t}s}(x_{0})}f(y)\di\meas(y)\\ 
\nonumber \le\,  &  \frac{v_{K,N}(\sqrt{5t}s) }{v_{K,N} (\sqrt{t} )}   \fint_{B_{\sqrt{5t}s}(x_{0})}f(y)\di\meas(y)\\
\le\,  &C \frac{v_{K,N}(\sqrt{5t}s) }{v_{K,N} (\sqrt{t})} \left(\sqrt{5t}s\right)^{n}\, ,\label{eq:unidom2}
\end{align}
where we used the Bishop-Gromov inequality \eqref{eq:BishopGromovInequality} and the last bound follows  from \eqref{eq:unipolgrowth}.
\\From \eqref{eq:unidom2} we infer that for every  $t\in (0,1)$ it holds
\begin{equation}\label{eq:estConvDompsi}
0\leq t^{-m} \phi_{t}(s)\leq \psi_{K,N,n,m}(s) \quad \text{ with } \quad \int_0^{\infty} \psi_{K,N,n,m}(s) \,s\,   e^{-s^2} \di s<\infty.
\end{equation}
Moreover, since $f\equiv 0$ on $B_r(x)$, it holds
\begin{equation}\label{eq:pointconv1}
t^{-m} \phi_t(s) \to 0\, , \quad\text{as $t\downarrow 0$, for any $s>0$}\, .
\end{equation}
Now observe that \eqref{eq:boundheatconv} can be rewritten as
\begin{equation*}
t^{-m}P_tf(x_{0})\le Ce^{ct}\int_0^{\infty} t^{-m}\phi_t(s) \, se^{-s^2}\di s\, .
\end{equation*}
Thanks to the domination  \eqref{eq:estConvDompsi} and to the pointwise convergence \eqref{eq:pointconv1} we can apply the  Dominated Convergence Theorem and get
\begin{equation*}
\lim_{t\downarrow 0}t^{-m}P_tf(x_{0})=0\, .
\end{equation*}
Since $m\in \setN$ was arbitrary, the claim follows.
\end{proof}
The next lemma is an instance of the fact that the heat flow acts as an averaging operator on smaller and smaller scales as time goes to $0$, even though being non local.

\begin{lemma}\label{lemma:lebpoint}
Let $(X,\dist,\meas)$ be an $\RCD(K,N)$ metric measure space for some $K\in\setR$ and $1\le N<\infty$. Let us assume that $f\in L^1_{\loc}(X,\meas)$ has polynomial growth and let $x\in X$ be such that
\begin{equation}\label{eq:assLp}
\lim_{r\downarrow 0} \frac{1}{\meas(B_r(x))}\int_{B_r(x)}\abs{f(y)-f(x)}\di\meas=0\, . 
\end{equation}
Then 
\begin{equation}\label{eq:pointconvheat}
\lim_{t\downarrow 0}P_tf(x)=f(x)\, .
\end{equation}
\end{lemma}

\begin{proof}
We start by observing that, for any $t>0$, 
\begin{equation*}
P_tf(x)-f(x)=\int_Xp_t(x,y)(f(y)-f(x))\di\meas(y)\, ,
\end{equation*}
thanks to the stochastic completeness \eqref{eq:stochcompl}.\\
Therefore, in order to prove \eqref{eq:pointconvheat}, using Jensen's inequality it is sufficient to prove that
\begin{equation*}
\int_Xp_t(x,y)\abs{f(y)-f(x)}\di\meas(y)\to 0\, , \quad\text{as $t\downarrow 0$}\, .
\end{equation*} 
Thanks to \autoref{lemma:sufflocal}, we can assume without loss of generality that $f$ has compact support, up to multiplying with a compactly supported continuous cut-off function.\\
Under this assumption, \eqref{eq:assLp} can be rephrased by saying that
\begin{equation}\label{eq:proof2unif}
\fint_{B_r(x)}\abs{f(y)-f(x)}\di\meas(y)\le C<\infty\, ,\quad\text{for any $r>0$}
\end{equation}
and 
\begin{equation}\label{eq:proof2pointconv}
\fint_{B_r(x)}\abs{f(y)-f(x)}\di\meas(y)\to 0\, ,\quad\text{as $r\downarrow 0$}\, .
\end{equation}
Setting 
\begin{equation*}
h(r):=\int_{B_r(x)}\abs{f(y)-f(x)}\di\meas(y)
\end{equation*}
and arguing as in the proof of \autoref{lemma:sufflocal}, we can bound
\begin{equation*}
\int_Xp_t(x,y)\abs{f(y)-f(x)}\di\meas(y)\le C e^{ct}\int_0^{\infty}se^{-s^2}\frac{h(\sqrt{5t}s)}{\meas(B_{\sqrt{t}}(x))}\di s\, .
\end{equation*}
Relying on \eqref{eq:proof2unif} to get the uniform bounds and on \eqref{eq:proof2pointconv} to get the pointwise convergence to $0$ of the integrands as $t\downarrow 0$, we can argue as in \autoref{lemma:sufflocal} and prove that
\begin{equation*}
\int_0^{\infty}se^{-s^2}\frac{h(\sqrt{5t}s)}{\meas(B_{\sqrt{t}}(x))}\di s\to 0\, ,\quad\text{as $t\downarrow 0$}\, ,
\end{equation*}
hence \eqref{eq:pointconvheat} holds.
\end{proof}
\begin{remark}\label{rm:weakeninglebpoint}
In \autoref{lemma:lebpoint} above we can weaken the assumption by requiring only that
\begin{equation*}
\lim_{r\downarrow 0}\fint_{B_r(x)}f(y)\di\meas(y)=c\, .
\end{equation*}
In that case, the very same proof shows that 
\begin{equation*}
\lim_{t\to 0}P_tf(x)=c\, .
\end{equation*}
\end{remark}


\begin{lemma}\label{lemma:pointheat}
Let $(X,\dist,\meas)$ be an $\RCD(K,N)$ metric measure space for some $K\in\setR$ and $1\le N<\infty$. Let $f\in L^1_{\loc}(X,\meas)$ be a function with polynomial growth. Moreover, let us assume that: 
\begin{itemize}
\item[i)] There exists $B_r(x)\subset X$ such that $f\in D(\Delta,B_{2r}(x))$; 
\item[ii)] $\Delta f$ is $\meas$-essentially bounded on $B_r(x)$; 
\item[iii)] $x$ is a Lebesgue point for $\Delta f$, i.e.
\begin{equation*}
\lim_{r\to 0}\fint_{B_r(x)}\abs{\Delta f(y)-\Delta f(x)}\di\meas(y)=0\, .
\end{equation*} 
\end{itemize}
Then 
\begin{equation}\label{eq:heattolapla}
\lim_{t\downarrow 0}\frac{P_tf(x)-f(x)}{t}=\Delta f(x)\, .
\end{equation}
\end{lemma}

\begin{proof}
Thanks to \autoref{lemma:sufflocal}, up to multiplying $f$ with a cut-off function with good estimates from \autoref{lemma:cutoff}, we can assume that $f\in D(\Delta)$ and $\Delta f\in L^{\infty}(X,\meas)$. 
\\ 
Thanks to \eqref{eq:pointlaplac},  we can consider the pointwise defined versions of $P_tf$ and $P_t\Delta f$, and compute:
\begin{align}\label{eq:gaptocompute}
\nonumber \frac{P_tf(x)-f(x)}{t}&=\frac{1}{t}\int_0^t\frac{\di }{\di s}P_sf(x)\di s\\
&=\frac{1}{t}\int_0^t\Delta P_sf(x)\di s=\frac{1}{t}\int_0^tP_s\Delta f(x)\di s\, .
\end{align}
Observe that, in particular, $P_t\Delta f$ is continuous for any $t>0$ thanks to the $L^{\infty}$-$\Lip$ regularization property of the heat flow.   Thanks to \eqref{eq:gaptocompute}, in order to get \eqref{eq:heattolapla} it is sufficient to prove that
\begin{equation}\label{eq:conclrephrased}
P_t\Delta f(x)\to \Delta f(x)\, ,\quad\text{as $t\downarrow 0$}\, .
\end{equation}
In order to obtain \eqref{eq:conclrephrased}, it is now sufficient to apply \autoref{lemma:lebpoint} with $\Delta f$ in place of $f$.
\end{proof}

\begin{remark}
The technical lemmas above essentially provide a counterpart, tailored for the non smooth $\RCD(K,N)$ framework, of the classical fact that if one evolves a smooth initial datum $f$ through the heat flow on a Riemannian manifold, then $P_tf$ converges to $f$ smoothly as $t\to 0$. Moreover, local smoothness yields local smooth convergence.
\end{remark}

\subsection{The Poisson equation}\label{subsec:Poisson}

Let us collect here some existence and comparison results for the Poisson equation with Dirichlet boundary conditions on $\RCD(K,N)$ metric measure spaces. Some of them are valid in the much more general framework of metric measure spaces verifying doubling and Poincar\'e inequalities, but for the present formulation we rely on the $\RCD(K,N)$ structure.
\medskip

We will often rely on the following regularity result for the Poisson equation on $\RCD(K,N)$ spaces, which is in turn a corollary of \cite[Theorem 1.2]{Jiang12}.

\begin{theorem}\label{thm:lipregPoisson}
Let $(X,\dist,\meas)$ be an $\RCD(K,N)$ metric measure space for some $K\in\setR$ and $1\le N<\infty$. Let $\Omega\subset X$ be an open domain and let $f\in D(\Delta,\Omega)$ be such that $\Delta f$ is  continuous on $\Omega$.  

Then $f$ has a locally Lipschitz representative on $\Omega$.
\end{theorem}

From now on, when dealing with solutions of the Poisson problem $\Delta f=\eta$ for some continuous function $\eta$, we will always assume that $f$ is the continuous representative given by \autoref{thm:lipregPoisson} above.


\begin{theorem}\label{thm:existenceanccomparison}
Let $(X,\dist,\meas)$ be an $\RCD(K,N)$ metric measure space for some $K\in\setR$ and $1\le N<\infty$. Let $\Omega\subset X$ be an open and bounded domain. Then the following hold:
\begin{itemize}
\item[(i)] (Strong maximum principle) Assume that $\Delta f=0$ on $\Omega$ and that $f$ has a maximum point at $x_0\in\Omega$. Then $f$ is constant on the connected component of $\Omega$ containing $x_0$.
\item[(ii)] (Existence for the Dirichlet problem) Assume that $\meas(X\setminus \Omega)>0$ and that $g\in H^{1,2}(X)$ and let $\eta:\Omega\to\setR$ be continuous and bounded. Then there exists a unique solution $f$ of the Poisson problem with Dirichlet boundary conditions
\begin{equation*}
\Delta f=\eta\, \quad\text{on $\Omega$}\, ,\quad f-g\in H^{1,2}_0(\Omega)\, .
\end{equation*} 
\item[(iii)] (Comparison principle) Assume that, under the same assumptions above, $\boldsymbol{\Delta} g\le \eta$ on $\Omega$, then $g\ge f$ on $\Omega$. 
\end{itemize}

\begin{proof}
i) (resp. iii))  follows by combining \cite[Theorem 8.13]{BjornBjorn11} (resp. \cite[Theorem 9.39]{BjornBjorn11}) with the PDE characterization of sub-harmonic functions obtained in \cite{GigliMondino13}.
\\  ii) follows from the solvability of the Poisson equation with null boundary conditions proved in \cite[Corollary 1.2]{BiroliMosco95}, combined with the existence of harmonic functions with Dirichlet boundary conditions (see for instance \cite[Theorem 10.12]{BjornBjorn11}). Alternatively, one can argue as in the proof of \cite[Theorem 10.12]{BjornBjorn11} and minimize the functional $J_{\eta}(u):=\int_{\Omega} |\nabla u|^{2} \di\meas- \int_{\Omega} \eta \, u \, \di \meas$ instead of $J_{0}(u):=\int_{\Omega} |\nabla u|^{2} \di\meas$, among the functions $u\in H^{1,2}(\Omega)$ such that  $u-g\in H^{1,2}_0(\Omega)$. 
\end{proof}

\end{theorem}

\subsection{The Green function of a domain and applications}\label{subsec:Green}
Here we deal with some relevant estimates for the Green function of the Laplacian on a domain of an $\RCD(K,N)$ metric measure space $(X,\dist,\haus^N)$. We assume that $N\ge 3$, for the sake of this discussion. The arguments can be adapted to deal with the case $N=2$, as it is classical in geometric analysis when dealing with Green's functions. 
\medskip

A classical way (cf. for instance with \cite[Lemma 5.15]{JiangNaber21} and \cite{Grygorian09}) to construct a positive Green's function for the Laplacian with Dirichlet boundary condition (and estimate it) on a smooth domain of a Riemannian manifold is given by the following procedure. 

Let $p_t:X\times X\to[0,\infty)$ denote the global heat kernel of the Riemannian manifold. Fix a time parameter $T>0$ and consider 
\begin{equation*}
G^T(x,y):=\int_0^Tp_t(x,y)\di t\, .
\end{equation*}
This is formally a solution of $\Delta_xG^T(\cdot,y)=-\delta_y+p_T(\cdot,y)$. Indeed, we can compute
\begin{align*}
\Delta_xG^T(\cdot,y)=&\Delta_x\int_0^Tp_t(\cdot,y)\di t=\int_0^T\Delta_xp_t(\cdot,y)\di t\\
=&-\int_0^T\frac{\di}{\di t}p_t(\cdot,y)\di t=p_T(\cdot,y)-\delta_y\, .
\end{align*}
Then we solve the Dirichlet boundary value problem 
\begin{equation*}
\Delta f=p_T(\cdot,y)
\end{equation*}
with boundary condition
\begin{equation*}
f=G^T(\cdot,y)\, ,\;\;\;\text{on $\partial \Omega$}\, ,
\end{equation*}
and subtract the solution $f$ to $G^T(\cdot,y)$. In this way we obtain, for $y\in\Omega$ fixed, a solution for the problem
\begin{equation*}
\Delta_x G(\cdot,y)=-\delta_y\, ,\;\;\; G(\cdot,y)=0\,\;\;\;\text{on $\partial\Omega$}\, .
\end{equation*} 
Good properties such as regularity away from the pole and strict positivity can be proven by regularization and exploiting harmonicity outside from the pole, once suitable integrability is established.
\medskip

We wish to prove that the construction above can be carried over even in the non smooth framework. This will require some slight adjustments to the construction of global Green functions on $\RCD(K,N)$ metric measure spaces verifying suitable volume growth assumptions performed in \cite{BrueSemola18} following one of the classical Riemannian strategies.
\medskip

Notice that, as it is classical in the study of Green functions of the Laplacian, the case of dimension $2$ would require a separate treatment, that we omit here since it does not involve really different ideas.

\begin{proposition}\label{prop:Greenexistence}
Let $(X,\dist,\haus^N)$ be an $\RCD(K,N)$ metric measure space and let $\Omega\subset X$ be an open domain such that $\haus^N(X\setminus \Omega)>0$. Assume that $N\ge 3$. Then, for any $x\in\Omega$ there a exists a positive \textit{Green's function} of the Laplacian on $\Omega$ with pole at $x$, i.e. a function $G_x:\Omega\to (0,\infty]$ such that
\begin{equation*}
\boldsymbol{\Delta}G_x(\cdot)=-\delta _x\, , 
\end{equation*} 
i.e. $G_x$ is locally Lipschitz away from $x$, $\abs{\nabla G_x}\in L^1_{\loc}(\Omega)$ and
\begin{equation*}
\int_{\Omega}\nabla G_x\cdot \nabla \phi\di\meas=\phi(x)\, ,
\end{equation*}
for any function $\phi\in \Lip_c(\Omega)$.
In particular, $G_x$ is harmonic away from the pole $x$.
\end{proposition}

\begin{proof}
Let us fix $x\in \Omega\subset X$. For  $T>0$ sufficiently small,  we set
\begin{equation*}
G^T_x(y)=G^T(x,y):=\int_0^Tp_t(x,y)\di t\, 
\end{equation*}
and, for any $0<\eps<T$ we also set
\begin{equation*}
G^{T,\eps}_x(y):=\int_{\eps}^Tp_{t}(x,y)\di t\, .
\end{equation*}
Let us consider $G^T_x$ as a function of $y$. Then, relying on \eqref{eq:kernelestimate}, the smallness of $T>0$, and the local Ahlfors regularity of $(X,\dist,\haus^N)$, we can estimate 
\begin{align}
\nonumber G^T_x(y)=&\int_0^Tp_t(x,y)\di t\le \int _0^T\frac{C_1}{\meas(B_{\sqrt{t}}(x))}\exp\left\lbrace-\frac{\dist^2(x,y)}{5t}+ct \right\rbrace \di t\\
\nonumber \le & C\int_0^T\frac{e^{-\frac{\dist^2(x,y)}{5t}}}{\meas(B_{\sqrt{t}}(x))}\di t\le C\int_0^T\frac{e^{-\frac{\dist^2(x,y)}{5t}}}{t^{N/2}}\di t\\ 
\le &C\dist(x,y)^{2-N}\, .\label{eq:estaboveGT}
\end{align}
In an analogous way, relying on the lower Gaussian heat kernel bound \eqref{eq:kernelestimate}, we obtain
\begin{equation}\label{eq:estbelowGT}
G^T_x(y)\ge C'\dist(x,y)^{2-N}\, ,\quad\text{for any $y\in X$, $y\neq x$}\, ,
\end{equation}
for some constant $C'=C'_{x,T}>0$. 

Using the gradient bound for the heat kernel \eqref{eq:gradientestimatekernel} it is also possible to prove that $G^T_x$ is locally Lipschitz away from $x$ with the bound
\begin{equation}\label{eq:gradGT}
\abs{\nabla G^T_x(y)}\le C\dist(x,y)^{1-N}\, , \quad\text{for a.e. $y\in X$}\, .
\end{equation}
It follows in particular that $G^T_x\in L^1_{\loc}(X,\meas)$ and $\abs{\nabla G^T_x}\in L^1_{\loc}(X,\meas)$.
\medskip

Arguing as in the proof of \cite[Lemma 2.5]{BrueSemola18} it is then possible to prove that, for any function $\phi\in \Lip_c(X,\dist)$, it holds
\begin{equation*}
\int_X\nabla G^T_x(y)\cdot\nabla \phi(y)\di\meas(y)=\phi(x)-\int_Xp_T(x,y)\phi(y)\di\meas(y)\, , 
\end{equation*}
which is the distributional formulation of $\boldsymbol\Delta G_x^T=-\delta_x+p_T(x,\cdot)$.
\\Let us also notice (cf. again with \cite{BrueSemola18}) that $G^{T,\eps}_x$ is a regularized version of $G^T_x$. Indeed, it is possible to show that $G^{T,\eps}_x\in \Test_{\loc}(X,\dist,\haus^N)$ for any $0<\eps<T$ and 
\begin{equation*}
\Delta G^{T,\eps}_x(\cdot)=-p_{\eps}(x,\cdot)+p_T(x,\cdot)\, .
\end{equation*}
Now let us notice that $p_T(x,\cdot)\in \Test_{\loc}(X,\dist,\haus^N)$ as it follows from the regularization properties of the heat flow and the semigroup law.\\ 
Using \autoref{thm:existenceanccomparison} (ii), for any $\eps>0$ we can consider a solution $g^{\eps}$ of the Dirichlet problem
\begin{equation}\label{eq:poissonaux}
\Delta g^{\eps}=p_T(x,\cdot)\, \quad\text{on $\Omega$}\, , \quad g^{\eps}-G^{T,\eps}_x\in H^{1,2}_0(\Omega)\, .
\end{equation} 
Setting $G^{\eps}_x:=G^{T,\eps}_x-g^{\eps}$, it holds
\begin{equation*}
\Delta G^{\eps}_{x}=-p_{\eps}(x,\cdot)\, ,
\end{equation*}
and $G^{\eps}_x\in H^{1,2}_0(\Omega)$. Moreover, by the comparison principle \autoref{thm:existenceanccomparison} (iii), we get that $G^{\eps}_x\ge 0$ on $\Omega$. 
\medskip

Now we can fix $0<\eps_0<T$ and set $G_x:=G^T_{x}-g^{\eps_0}$. Observe that
\begin{equation*}
G_{x}:=G^{T}_{x}-g^{\eps_{0}}=G^{T,\eps_{0}}-g^{\eps_{0}}+\int_{0}^{\eps_{0}}p_{t}(x,\cdot) \di t > G^{\eps_{0}}_{x}\geq 0 \quad \text{on }\Omega\, .
\end{equation*}
Notice that \autoref{thm:lipregPoisson} applied to the Poisson problem \eqref{eq:poissonaux} yields that $g^{\eps}$ is a locally Lipschitz function. Hence $G_x$ is locally Lipschitz away from the pole $x$ and $\abs{\nabla G_x}\in L^{1}_{\loc}(\Omega)$. Moreover, by the very construction of $g^{\eps}$, it holds that
\begin{equation*}
\boldsymbol{\Delta}G_x=-\delta_x\, .
\end{equation*} 
\end{proof}

\begin{remark}
With an additional limiting argument (basically setting $\eps_0=0$ in the proof above) it is possible to obtain the Green function of the Laplacian on $\Omega$ with pole at $x$ and homogeneous Dirichlet boundary conditions.
\end{remark}

\begin{proposition}\label{prop:Gestimates}
Let $(X,\dist,\haus^N)$ be an $\RCD(K,N)$ metric measure space for some $N\ge 3$ and let $\Omega\subset X$ be an open domain such that $\haus^N(X\setminus \Omega)>0$. Let $x\in\Omega$ and consider the positive Green function of the Laplacian with Dirichlet boundary conditions on $\Omega$ and pole at $x$, constructed in \autoref{prop:Greenexistence}.\\ 
Then the following estimates hold: there exist constants $c_x,C_x>0$ such that
\begin{equation*}
\frac{c_x}{\dist^{N-2}(x,y)}\le G_x(y)\le \frac{C_x}{\dist^{N-2}(x,y)}\, ,
\end{equation*}
for every $y\in B_r(x)$ such that $y\neq x$ (where $r>0$ is such that $B_r(x)\subset \Omega$),
and
\begin{equation*}
\abs{\nabla G_x(y)}\le \frac{C_x}{\dist^{N-1}(x,y)}\, ,
\end{equation*}
for a.e. $y\in B_r(x)$.
\end{proposition}

\begin{proof}
The sought estimates follow from the estimates for the function $G^T_x$ and its gradient (see \eqref{eq:estaboveGT}, \eqref{eq:estbelowGT} and \eqref{eq:gradGT}) combined with the local uniform Lipschitz estimate for the solution of the Dirichlet problem $g^{\eps}$ considered in the proof of \autoref{prop:Greenexistence}, that follow in turn from \autoref{thm:lipregPoisson}.
\end{proof}

Our next step is to use the local Green function in order to build a replacement of the distance function with better regularity properties.\\ 
On the Euclidean space of dimension $N\ge 3$, the Green function of the Laplacian is a negative power of the distance function. On a general Riemannian manifold this is not the case of course, but still a suitable power of the Green function of the Laplacian is comparable to the distance function (under suitable curvature and volume growth assumptions). Moreover, the Green function solves an equation, which makes it sometimes more suitable for the applications. We refer to \cite{Colding12,JiangNaber21,BrueSemola18} for previous instances of this idea in Geometric Analysis.

\begin{proposition}[The Green distance]\label{prop:Gdistance}
Let $(X,\dist,\haus^N)$ be an $\RCD(K,N)$ space for some $N\ge 3$. Let $\Omega\subset X$ be an open and bounded domain with $\meas(X\setminus\Omega)>0$ and $x\in\Omega$. Let us suppose, up to scaling, that $B_1(x)\subset \Omega$ and let $G_x$ be the positive Green function of the Laplacian with pole at $x$ and Dirichlet boundary conditions, constructed  in \autoref{prop:Greenexistence}.\\
Then, setting
\begin{equation*}
b_x(y):=G^{-\frac{1}{N-2}}_x(y)\, ,
\end{equation*}
the following hold:
\begin{itemize}
\item[(i)] there exist constants $c_x,C_x>0$ such that 
\begin{equation}\label{eq:twosidedbx}
c_x\dist(x,y)\le b_x(y)\le C_x\dist(x,y)\, \quad\text{for any $y\in B_1(x)$}\, ;
\end{equation}
\item[(ii)] there exists $C_x>0$ such that
\begin{equation}\label{eq:gradbx}
\abs{\nabla b_x(y)}\le C_x\, \quad\text{for a.e. $y\in B_1(x)$}\, ;
\end{equation}
\item[(iii)] $b_x^2\in D(\Delta, B_1(x))$ and 
\begin{equation}\label{eq:laplab}
\Delta b_x^2=2N\abs{\nabla b_x}^2\, ;
\end{equation}
\end{itemize}
\end{proposition}

\begin{proof}

The estimates in items (i) and (ii) directly follows from the estimates for the Green function $G_x$ of \autoref{prop:Gestimates}.
\medskip

In order to prove \eqref{eq:laplab} we argue in two steps. First we prove that $b_x^2\in D(\Delta, B_1(x)\setminus\{x\})$ and that \eqref{eq:laplab} holds on $B_1(x)\setminus \{x\}$, then we verify that $b_x^2$ is globally in the domain of the Laplacian on $B_1(x)$ and that the pole gives no singular contribution.

Let us point out that $G_x$ is harmonic outside from the pole $x$. Given this remark, it can be easily verified via the chain rule for the gradient and the Leibniz formula for the Laplacian that $b_x^2\in D(\Delta, B_1(x)\setminus\{x\})$ and that \eqref{eq:laplab} holds on $B_1(x)\setminus \{x\}$.
\medskip

To conclude, we need to verify that $b_x^2$ belongs locally to the domain of the Laplacian. This conclusion will be achieved through a standard cutting-off and limiting procedure.

We wish to prove that 
\begin{equation}\label{eq:intbyparts}
\int_{B_1(x)}\nabla b_x^2\cdot\nabla \phi\, \di\haus^N=-2N\int_{B_1(x)}\phi\abs{\nabla b_x}^2\di\haus^N\, , 
\end{equation}
for any Lipschitz function $\phi$ with compact support in $B_1(x)$.\\
We already argued that $b_x^2\in D(\Delta, B_1(x)\setminus \{x\})$, hence \eqref{eq:intbyparts} holds true as soon as $\phi$ has compact support in $B_1(x)\setminus \{x\}$.  

Let us consider then radial Lipschitz cut-off functions $\eta_{\eps}$, for $0<\eps<1$ such that $\eta_{\eps}\equiv 1$ on $B_1(x)\setminus B_{2\eps}(x)$, $\eta_{\eps}\equiv 0$ on $B_{\eps}(x)$ and $\abs{\nabla \eta_{\eps}}\le C/\eps$. Then we can apply \eqref{eq:intbyparts} to $\phi_{\eps}:=\phi\eta_{\eps}$ for any $\eps>0$ and get
\begin{align}
\label{eq:f1}\int_{B_1(x)}&\eta_{\eps}\nabla b_x^2\cdot\nabla \phi\, \di\haus^N+\int_{B_1(x)}\phi\nabla \eta_{\eps}\cdot\nabla b_x^2\, \di\haus^N\\ 
\nonumber=&\int_{B_1(x)}\nabla b_x^2\cdot\nabla \phi_{\eps}\di\haus^N \\
=&-2N\int_{B_1(x)}\phi\eta_{\eps}\abs{\nabla b_x}^{2}\di\haus^N\, .  \nonumber
\end{align}
The last term above converges to 
\begin{equation*}
-2N\int_{B_1(x)}\phi\abs{\nabla b_x}^{2}\, \di\haus^N
\end{equation*}
as $\eps\to 0$ by the dominated convergence theorem. By the same reason, also the first term in the left hand side of \eqref{eq:f1} converges to 
\begin{equation*}
\int_{B_1(x)}\nabla b_x^2\cdot\nabla \phi\, \di\haus^N\, ,
\end{equation*}
as $\eps\to 0$. Hence to complete the proof of \eqref{eq:intbyparts}, it remains to prove that the second term in the left hand side of \eqref{eq:f1} converges to $0$ as $\eps\to 0$. To this aim, it is sufficient to observe that
\begin{align*}
\abs{\int_{B_1(x)}\phi\nabla \eta_{\eps}\cdot\nabla b_x^2\, \di\haus^N}&=\int_{B_{2\eps}(x)\setminus B_{\eps}(x)}\abs{\phi}\abs{\nabla \eta_{\eps}}\abs{\nabla b_x^2}\di\haus^N\\
&\le\frac{C \max_{B_1(x)}\abs{\phi}}{\eps}\haus^N(B_{2\eps}(x)\setminus B_{\eps}(x))\, ,
\end{align*}
which is easily seen to converge to $0$ as $\eps\to 0$.
\end{proof}

\begin{remark}\label{rm:goodfunction}
The main use of the Green function for the purposes of the present paper will be the possibility (guaranteed by the construction of the function $b_x$ above) of considering locally a sufficiently regular function $f:B_r(x)\to \setR$ with the following properties:
\begin{itemize}
\item[i)] it is non-negative;
\item[ii)] it vanishes only at $x$ and is strictly positive in a neighbourhood of $x$;
\item[iii)] also its gradient is vanishing at $x$, at least in a weak sense;
\item[iv)] its Laplacian is non-negative, in a weak sense.
\end{itemize}
This function plays the role of a power of the distance function in the development of a viscous theory of bounds for the Laplacian on $\RCD$ metric measure spaces.

In the Euclidean setting, by considering powers of the distance function it is possible to work with smooth functions whose derivatives are vanishing at any given order. In the synthetic framework this is of course too much to ask.
\end{remark}

\section{The Laplacian on $\RCD(K,N)$ spaces}\label{sec:lap}

We are going to consider some new equivalences between different notions of Laplacian and bounds for the Laplacian on an $\RCD(K,N)$ metric measure space $(X,\dist,\haus^N)$. We will be guided by the equivalences that hold in the Euclidean setting and on smooth Riemannian manifolds. In particular we shall address bounds on the Laplacian:
\begin{itemize}
\item in the sense of distributions; 
\item in the viscous sense; 
\item in the sense of sub/super minimizers of Dirichlet type energies; 
\item in the sense of comparison with solutions of the Dirichlet problem; 
\item in the sense of pointwise behaviour of the heat flow.
\end{itemize}
Some of the equivalences had already appeared in the literature, even under less restrictive assumptions on the metric measure spaces. The main contribution here will be in the direction of the viscous theory, in which case the only previous treatment we are aware of is \cite{ZhangZhu12}, dealing with Alexandrov spaces (and inspired in turn by the unpublished \cite{Petrunin00}), and of the pointwise behaviour of the heat flow, a notion that seems to be new also in the smooth setting.
\medskip

We are going to restrict the analysis to locally Lipschitz functions, in order to avoid technicalities and since this class will be sufficiently rich for the sake of the applications in later sections of the paper. We remark that likely more general functions could be considered.

\subsection{Notions of Laplacian bounds}


We start with distributional Laplacian bounds, borrowing the definition from \cite{Gigli15}.
\begin{definition}\label{def:distributions}
Let $(X,\dist,\meas)$ be an $\RCD(K,N)$ metric measure space and let $\Omega\subset X$ be an open domain. Let $f:\Omega\to\setR$ be a locally Lipschitz function and $\eta\in\Cb(\Omega)$. Then we say that $\boldsymbol{\Delta}f\le \eta$ in the sense of distributions if the following holds. For any non-negative function $\phi\in\Lip_c(\Omega)$,
\begin{equation*}
-\int_{\Omega}\nabla f\cdot\nabla \phi\di\meas\le \int_{\Omega}\phi\eta\di\meas\, .
\end{equation*}
\end{definition}

The following is a classical result, relying on the fact that a distribution with a sign is represented by a measure, in great generality. We refer to \cite{GigliMondino13,Gigli15} for a proof.

\begin{proposition}\label{prop:distrmeasure}
Let $(X,\dist,\meas)$ be an $\RCD(K,N)$ metric measure space and let $\Omega\subset X$ be an open domain. Let moreover $f:\Omega\to\setR$ be a locally Lipschitz function and $\eta\in\Cb(\Omega)$. Then $\boldsymbol{\Delta}f\le\eta$ in the sense of distributions if and only if there exists a locally finite measure $\mu$ on $\Omega$ such that
\begin{equation}\label{eq:laplmeas}
-\int_{\Omega}\nabla f\cdot\nabla \phi\di\meas=\int_{\Omega}\phi\di\mu\, ,
\end{equation}
for any $\phi\in\Lip_c(\Omega)$.  Moreover, under these assumption $\mu\le \eta\meas$, $\mu$ is uniquely determined by \eqref{eq:laplmeas} and we shall denote it by $\boldsymbol{\Delta}f$.
\end{proposition}

Given a function $\eta\in\Cb(\Omega)$, we introduce the energy
\begin{equation*}
E_{\eta}:\Lip(\Omega)\to\setR\, ,
\end{equation*}
by 
\begin{equation}\label{eq:energyeta}
E_{\eta}(v):=\frac{1}{2}\int_{\Omega}\abs{\nabla v}^2\di\meas+\int_{\Omega}v\eta\di\meas\, .
\end{equation}

\begin{definition}\label{def:subsupermin}
Let $(X,\dist,\meas)$ be an $\RCD(K,N)$ metric measure space and let $\Omega\subset X$ be an open domain. Let $f:\Omega\to\setR$ be a locally Lipschitz function and $\eta\in\Cb(\Omega)$. Let us consider the energy functional $E_{\eta}:\Lip(\Omega)\to\setR$ defined above. Then we say that $f$ is a superminimizer of $E_{\eta}$ on $\Omega$ if 
\begin{equation*}
E_{\eta}(f+\phi)\ge E_{\eta}(f)\, , \quad\text{for any non-negative function $\phi\in\Lip_c(\Omega)$}\, .
\end{equation*}

\end{definition}

The following result comparing superminimizers with functions having Laplacian bounded from above in the sense of distributions will be of some relevance for our purposes. A version of this statement tailored for more general ambient spaces (but restricted to the case of subharmonic/superharmonic functions) appears for instance in \cite[Theorem 4.1, Corollary 4.4]{GigliMondino13}. The extension to more general upper/lower bounds for the Laplacian requires just slight modifications to the original argument, that we omit for brevity.

\begin{proposition}\label{prop:compdistrsbmin}
Let $(X,\dist,\meas)$ be an $\RCD(K,N)$ metric measure space and let $\Omega\subset X$ be an open domain. Let $f:\Omega\to\setR$ be a locally Lipschitz function and $\eta\in\Cb(\Omega)$. Then $\boldsymbol{\Delta}f\le \eta$ in the sense of distributions if and only if $f$ is a superminimizer of the energy $E_{\eta}$ on $\Omega$ according to \autoref{def:subsupermin}.
\end{proposition}
Various definitions of sub/superharmonic functions on metric measure spaces in the sense of comparison with Dirichlet boundary value problems have appeared in the last twenty years. Here we choose a slight modification of \cite[Definition 14.8]{BjornBjorn11} tailored to the purpose of studying locally Lipschitz functions (and general Laplacian bounds).

\begin{definition}\label{def:classicalsupersolution}
Let $(X,\dist,\meas)$ be an $\RCD(K,N)$ metric measure space and let $\Omega\subset X$ be an open domain. Let $f:\Omega\to\setR$ be a locally Lipschitz function and $\eta\in\Cb(\Omega)$. We say that $f$ is a \textit{classical supersolution} of $\Delta f=\eta$ if the following holds: for any open domain $\Omega'\Subset\Omega$ and for any function $g\in C(\overline{\Omega'})$ such that $\Delta g=\eta$ in $\Omega'$ and $g\le f$ on $\partial\Omega'$ it holds $g\le f$ on $\Omega'$.
\end{definition}

\begin{remark}\label{rm:classimplclass}
If $f\in D(\Delta,\Omega)$ and $\Delta f=\eta$ on $\Omega$, then it is a classical supersolution of $\Delta f=\eta$ according to \autoref{def:classicalsupersolution} above.\\ 
Indeed, for any test function $g$ as in the definition above, $\Delta f=\Delta g=\eta$ on $\Omega'$ and $g$ is continuous on $\Omega'$ by assumption. Moreover $f$ is continuous on $\Omega'$, since it is locally Lipschitz on $\Omega$ by \autoref{thm:lipregPoisson}. Therefore, letting $h:=f-g$, $h$ is harmonic and continuous on $\Omega'$ and $h\ge 0$ on $\partial \Omega'$.\\ 
We claim that $h\ge 0$ on $\Omega'$. Suppose that this is not the case, then $h$ admits a strictly negative minimum in the interior of $\Omega'$. Therefore it is constant and strictly negative in the connected component of $\Omega'$ where this minimum is achieved by \autoref{thm:existenceanccomparison} (iii). This yields a contradiction since $h\ge 0$ on $\partial\Omega'$.
\end{remark}

\begin{remark}
By \autoref{rm:classimplclass} and thanks to the linearity of the Laplacian on $\RCD(K,N)$ spaces, the extension of the results in \cite{BjornBjorn11} from the case of sub/supersolutions of the equation $\Delta f =0$ to the general Poisson problem $\Delta f=\eta$ is harmless. Indeed we can always subtract off a solution of the Poisson problem and reduce to the harmonic case.
\end{remark}

\begin{remark}
In \cite[Chapter 11]{BjornBjorn11} it is proved that \autoref{def:classicalsupersolution} is equivalent to another (a priori stronger, since we test with more functions) definition of supersolution of the problem $\Delta f =\eta$.\\ 
The outcome is that $f$ (verifying the usual assumptions) is a classical supersolution of $\Delta f=\eta$ on $\Omega$ if and only if for any $\Omega'\Subset\Omega$ and for any $g\in\Lip(\partial\Omega')$ such that $g\le f$ on $\partial\Omega'$, it holds that $H_{\eta}g\le f$ on $\Omega'$. Here $H_{\eta}g$ is the solution of the Poisson problem with Dirichlet boundary conditions 
\begin{equation*}
\Delta H_{\eta} g=\eta\, ,\quad H_{\eta}g-\tilde g\in H^{1,2}_0(\Omega')\, ,
\end{equation*}
with $\tilde{g}$ any global extension of $g$.
\end{remark}

Let us quote a fundamental result connecting (classical) supersolutions of the equation $\Delta f=\eta$ with superminimizers. Under our assumptions, it is a direct corollary of \cite[Theorem 9.24]{BjornBjorn11} (see also \cite{KinnunenMartio02}), where equivalence of supersolutions with superminimizers of the energy is addressed, and \autoref{prop:compdistrsbmin}, that gives the equivalence between the superminimizing property and bounds for the Laplacian in the sense of distributions.

\begin{theorem}\label{thm:supersolsupermin}
Let $(X,\dist,\meas)$ be an $\RCD(K,N)$ metric measure space and let $\Omega\subset X$ be an open and bounded domain. Let $f:\Omega\to\setR$ be locally Lipschitz and bounded and let $\eta\in \Cb(\Omega)$. Then $f$ is a classical supersolution of $\Delta f=\eta$ in the sense of \autoref{def:classicalsupersolution} if and only if $\boldsymbol{\Delta} f\le\eta$ in the sense of \autoref{def:distributions}.
\end{theorem}  
Next we propose a definition of sub/supersolutions of the equation $\Delta f=\eta$ in the \textit{viscous} sense tailored to the setting of $\RCD(K,N)$ metric measure spaces. 
\medskip

The viscous theory for the Laplacian allows for several simplifications with respect to the general viscosity theory of PDEs in the Euclidean case. 

When considering general smooth Riemannian manifolds, there are intrinsic definitions of Laplacian bounds in the viscosity sense, see for instance \cite{Wu79} and the more recent \cite{Mantegazzaetal14}, that require essentially no modification with respect to the classical Euclidean notion.\\ 
In the non smooth framework, the development of a viscous theory of Laplacian bounds presents some further challenges, the first one being the necessity to single out the right class of \textit{smooth tests} to use as comparison functions.

\begin{definition}[Viscous bounds for the Laplacian]\label{def:viscosity}
Let $(X,\dist,\meas)$ be an $\RCD(K,N)$ metric measure space and let $\Omega\subset X$ be an open and bounded domain. Let $f:\Omega\to\setR$ be locally Lipschitz and $\eta\in\Cb(\Omega)$. We say that $\Delta f\le \eta$ in the \textit{viscous sense} in $\Omega$ if the following holds. For any $\Omega'\Subset\Omega$ and for any test function $\phi:\Omega'\to\setR$ such that 
\begin{itemize}
\item[(i)] $\phi\in D(\Delta, \Omega')$ and $\Delta\phi$ is continuous on $\Omega'$;
\item[(ii)] for some $x\in \Omega'$ it holds 
$\phi(x)=f(x)$ and $\phi(y)\le f(y)$ for any $y\in\Omega'$, $y\neq x$;
\end{itemize}
it holds
\begin{equation*}
\Delta \phi(x)\le \eta(x)\, .
\end{equation*} 
\end{definition}

\begin{remark}
In the classical definitions of viscosity supersolutions for PDEs on the Euclidean space or on Riemannian manifolds, test functions are required to be $C^2$. Therefore the notion considered above is a priori stronger than the classical one on smooth Riemannian manifolds, since it is well known that there are functions with continuous Laplacian that are not $C^2$. Nevertheless, it follows from the equivalence \autoref{thm:viscoimpldistri} that this notion is equivalent to the classical one.
\end{remark}
We introduce yet another definition of supersolution of the equation $\Delta f=\eta$ based on the pointwise behaviour of the heat flow.

\begin{definition}[Supersolution in the heat flow sense]\label{def:heatlaplbounds}
Let $(X,\dist,\meas)$ be an $\RCD(K,N)$ metric measure space and let $\Omega\subset X$ be an open and bounded domain. Let $f:\Omega\to\setR$ be a Lipschitz function and let $\eta\in\Cb(\Omega)$. We say that $\Delta f\le \eta$ on $\Omega$ in the \textit{heat flow sense} if the following holds. For any $\Omega'\Subset\Omega$ and any function $\tilde{f}:X\to\setR$ extending $f$ from $\Omega'$ to $X$ and with polynomial growth, we have
\begin{equation*}
\limsup_{t\downarrow 0}\frac{P_t\tilde{f}(x)-\tilde{f}(x)}{t}\le \eta(x)\, ,\quad \text{for any $x\in\Omega'$}\, .
\end{equation*}
\end{definition}

\begin{remark}
\autoref{def:heatlaplbounds} is independent of the choice of the global extension with polynomial growth of the function $f$ to $X$, therefore it is well-posed. This is a consequence of \autoref{lemma:sufflocal}, applied to the difference of two global extensions of $f$ with polynomial growth. 
\end{remark}

\begin{remark}
The role of the heat flow in the treatment of weak notions of Laplacian bounds on smooth Riemannian manifolds can be traced back at least to \cite{GreeneWu79}, where the original idea is attributed to Malliavin. Notions of Laplacian and Laplacian bounds related to the asymptotic behaviour of the heat flow appear also in \cite[Section 4]{Gigli15} and \cite{GrygoryanHu14}. The novelty of \autoref{def:heatlaplbounds} is  the absence of integrations against test functions and that the bound is required to hold  \textit{pointwise}.
\end{remark}

\begin{remark}
We can consider counterparts of all the notions in the case of lower bounds for the Laplacian of the type $\Delta f\ge \eta$. The only difference being that all the signs in the inequalities need to be switched.
\end{remark}
%
\begin{remark}
Since we chose to adopt the same notation $\Delta f\le \eta$ for most of the notions of Laplacian bounds that we have introduced, we shall usually clarify in which sense the bound has to be intended, whenever there is risk of confusion.
\end{remark}


\subsection{The main equivalence results}

The aim of this subsection is to establish the equivalence of the upper bounds for the Laplacian in the viscous sense and in the sense of distributions. This will allow also to prove equivalence with the less classical notion of Laplacian bounds through pointwise behaviour of the heat flow that we have introduced in \autoref{def:heatlaplbounds}.
\medskip

 We will mostly consider the case of an $\RCD(K,N)$ metric measure space $(X,\dist,\haus^N)$ and limit our analysis to functions that are locally Lipschitz continuous. We shall give the proofs under the additional assumption that $N\ge 3$. The case $N=1$ is elementary, due to the classification of non collapsed $\RCD(K,1)$ metric measure spaces, see \cite{KitabeppuLakzian16}. The case $N=2$ could be treated with arguments analogous to those considered here, with the slight modifications due to the different behaviour of the Green function. Notice also that the theory of non collapsed $\RCD(K,2)$ metric measure spaces is very well understood, thanks to \cite{LytchakStadler18}, where it is shown that they are Alexandrov spaces with curvature bounded from below.

\begin{remark}
Let us remark that the case of general $\RCD(K,N)$ metric measure spaces $(X,\dist,\meas)$ could be handled with similar arguments, after imposing some mild lower bounds on the measure growth of balls, necessary in order to have a good definition of local Green's functions.
\end{remark}

A fundamental tool in order to establish the equivalence between viscous and distributional bounds will be the following maximum principle, which follows from \cite{ZhangZhu16}. It is reminescent of the Omori-Yau maximum principle 
and of Jensen's maximum principle in the viscous theory of PDEs. 
\medskip

Below, given a measure $\boldsymbol{\mu}$ on an $\RCD(K,N)$ metric measure space $(X,\dist,\meas)$ we shall denote by $\boldsymbol{\mu}^{ac}$ its absolutely continuous part w.r.t. $\meas$ and by $\mu^{ac}$ its density, i.e. $\boldsymbol{\mu}^{ac}=\mu^{ac}\, \meas$.

\begin{theorem}\label{thm:maximumprincipleOYJ}
Let $(X,\dist,\meas)$ be an $\RCD(K,N)$ metric measure space. Let $\Omega\subset X$ be an open and bounded domain. Let $f\in \Lip(\Omega)$ be such that $\boldsymbol{\Delta}f$ is a signed Radon measure with non-negative singular part. Suppose that $f$ achieves one of its locally strict maxima in $\Omega$. Then there exists a sequence of points that are approximate continuity points of $\Delta^{ac}f$ and such that
\begin{equation*}
f(x_n)\ge \sup f_{\Omega} -\frac{1}{n}\, , \;\;\;\;\; \Delta^{ac}f(x_n)\le \frac{1}{n}\, .
\end{equation*}
In particular, if $\bar{x}\in\Omega$ is a strict maximum point of $f$ in $\Omega$, then there exists a sequence $(x_n)$ of approximate continuity points for $\Delta^{ac}f$ such that
\begin{equation*}
x_n\to \bar{x}\, ,\;\;\;\;\Delta^{ac}f(x_n)\to 0\, .
\end{equation*}
More strongly, for any $\eps>0$ it holds that
\begin{equation*}
\meas\left(\set{x\in B_{\eps}(\bar{x})\, : \;\; \Delta^{ac}f(x)\le \eps}\right)>0\, .
\end{equation*}

\end{theorem}

\begin{proof}
The proof follows from the more general statement of \cite[Theorem 1.3]{ZhangZhu16}. 
The conclusion that the points $x_n$ can be chosen to be converging to $\bar{x}$ follows from the fact that $\bar{x}$ is assumed to be the unique strict maximum in a neighbourhood of $\bar{x}$ in $\Omega$, i.e., there exists a neighbourhood $U_x\ni x$ such that $f(y)<f(x)$ for any $y\in U_x$ with $y\neq x$.
\end{proof}

\begin{remark}\label{rm:reversed}
A dual statement holds when dealing with functions whose distributional Laplacian is a signed Radon measure with non-positive singular part and local minima instead of local maxima.
\end{remark}

One of the steps towards a viscosity theory is the comparison between classical bounds for the Laplacian and bounds in the viscous sense for \textit{sufficiently smooth} functions. 


\begin{proposition}[Classical \textit{vs} viscous for functions with continuous Laplacian]\label{prop:smoothcomp}
Let $(X,\dist,\haus^N)$ be an $\RCD(K,N)$ metric measure space.
Let us consider a function $f\in D(\Delta, \Omega)$ and assume that $\Delta f$ has a continuous representative. Let $\eta:\Omega\to\setR$ be a continuous function. Then $\Delta f\le \eta$ pointwise if and only if $\Delta f\le \eta$ in the viscous sense on $\Omega$.
\end{proposition}

\begin{proof}
Let us suppose that $\Delta f\le \eta$ in the viscous sense. We wish to prove that $\Delta f\le \eta$ pointwise. To this aim, it is enough to observe that we can take $f$ as a test function in the definition of Laplacian bound in the viscous sense. This directly yields that $\Delta f\le \eta$ pointwise.
\medskip

Let us prove conversely that if $\Delta f\le \eta$ pointwise, then $\Delta f\le \eta$ in the viscous sense. To this aim,  fix $x\in\Omega$ and $\Omega'\Subset \Omega$. Let $\phi:\Omega'\to\setR$ be such that $\phi\le f$ on $\Omega'$, $\phi(x)=f(x)$ and $\phi$ has continuous Laplacian on $\Omega'$. We wish to prove that $\Delta\phi(x)\le \eta(x)$.  Set $\psi:= f-\phi$. 
\\Without loss of generality we can assume $\Omega'$ to be small enough in order for the Green type distance $b_x$ to be well defined with good properties on $\Omega'$, as in \autoref{prop:Gdistance}. Set $\tilde{\psi}:=\psi+b_x^4$. Then $\tilde{\psi}$ has a strict local minimum at $x$. Observe also that $\tilde{\psi}$ is locally Lipschitz. Hence, by \autoref{thm:maximumprincipleOYJ} (see also \autoref{rm:reversed}), we can find a sequence of points $(x_n)$ converging to $x$ and such that
\begin{equation}\label{eq:laplpsitilde}
\liminf_{n\to\infty} \Delta \tilde{\psi}(x_n)\ge  0\, .
\end{equation}
By the properties of the auxiliary function $b_x$, we infer that 
\begin{equation}\label{eq:laplpsi}
\liminf_{n\to\infty} \Delta\psi(x_n)\ge 0\, .
\end{equation} 
Indeed
\begin{equation}\label{eq:laplb4}
\Delta b_x^4=\Delta (b^2_x)^2=2\abs{\nabla b^2_x}^2+4Nb_x^2\abs{\nabla b_x}^2=4(N+2)b_x^2\abs{\nabla b_x}^2\, ,
\end{equation}
where we rely on the identity $\Delta b_x^2=2N\abs{\nabla b_x}^2$ obtained in \autoref{prop:Gdistance}. Then \eqref{eq:laplpsi} follows from \eqref{eq:laplpsitilde}, via \eqref{eq:laplb4} and relying on the two sided estimates \eqref{eq:twosidedbx} for $b_x$ and on the gradient estimate \eqref{eq:gradbx}.\\
Hence 
\begin{equation*}
\liminf _{n\to\infty}(\Delta f(x_n)-\Delta\phi (x_n))\ge 0\, .
\end{equation*}
Since $\Delta f\le \eta$ pointwise and $\eta$ is continuous, we infer 
\begin{equation*}
\limsup_{n\to\infty}\Delta\phi(x_n)\le\liminf_{n\to\infty}\eta(x_n)=\eta(x)\, .
\end{equation*} 
Hence $\Delta\phi(x)\le \eta(x)$ and we can conclude that $\Delta f\le \eta$ in the viscous sense, as claimed.

\end{proof}


\begin{remark}\label{rm:shiftingto0}
An easy consequence of the existence for solutions to the Dirichlet problem \autoref{thm:existenceanccomparison} and of the linearity of the Laplacian is now the following: given a continuous function $\eta$ and a function $u$ with continuous Laplacian, it holds that $\Delta u\le \eta$ in the viscous sense if and only if, denoting by $v_{\eta}$ a local solution of $\Delta v_{\eta}=\eta$, it holds that $\Delta (u-v_{\eta})\le 0$ in the viscous sense. 
\end{remark}



\begin{proposition}\label{prop:distrimplvisc}
Let $(X,\dist,\haus^N)$ be an $\RCD(K,N)$ metric measure space. Assume that $f:\Omega\to\setR$ is a locally Lipschitz function and that $\eta:\Omega\to\setR$ is a continuous function. If $\boldsymbol{\Delta} f\le\eta$ in the sense of distributions,  then $\Delta f\le \eta$ in the viscous sense.
\end{proposition}

\begin{proof}
If $\boldsymbol{\Delta} f\le \eta$ in the sense of distributions, then $\boldsymbol{\Delta}f$ is a signed Radon measure whose singular part is non-positive. Moreover, for any Lebesgue point $x\in\Omega$ of $\boldsymbol{\Delta}^{\mathrm{ac}}f$, it holds
\begin{equation*}
\Delta^{\mathrm{ac}}f(x)\le \eta (x)\, .
\end{equation*}
This is a direct consequence of the observation that $\boldsymbol{\Delta} ^{\mathrm{ac}}f\le \eta$ and of the very definition of Lebesgue point.

The proof now follows from the same argument used in the proof of \autoref{prop:smoothcomp} with the only adjustment that we have to consider Lebesgue points $(x_n)$ of the absolutely continuous part of the Laplacian in place of general points and $\Delta^{\mathrm{ac}}$ in place of the pointwise defined Laplacian $\Delta$.
\end{proof}



\begin{lemma}[Maximum principle for viscosity sub/super solutions]\label{lemma:maxprincvisco}
Let $(X,\dist,\meas)$ be an $\RCD(K,N)$ metric measure space for some $K\in\setR$ and $1\le N<\infty$. Let $\Omega\subset X$ be an open and bounded domain such that there exists $\Omega\Subset\tilde{\Omega}$ with $\meas(X\setminus\tilde{\Omega})>0$.
Let moreover $f:\Omega\to\setR$ be a Lipschitz function such that $\Delta f\le 0$ in the viscous sense. Then
\begin{equation*}
\min_{x\in\Omega} f(x)=\min_{x\in\partial \Omega} f(x)\, .
\end{equation*}
\end{lemma}

\begin{proof}
Let us suppose by contradiction that 
\begin{equation*}
\min_{x\in\Omega} f(x)<\min_{x\in\partial \Omega} f(x)\, .
\end{equation*}
Then the minimum in the left hand side is attained at an interior point $x_0\in\Omega$. In particular
\begin{equation}\label{eq:intmin}
\min_{x\in\partial \Omega}f(x)>f(x_0)\, .
\end{equation}
Consider a solution of the Poisson problem $\Delta v=1$ on $\Omega'$ such that $v\ge 0$ on $\Omega$ and 
\begin{equation*}
M:=\max_{\partial \Omega} v\ge \min_{\partial \Omega}v=:m>0\, .
\end{equation*} 
This function can be obtained with an additive perturbation from any solution of $\Delta f=1$ on $\Omega'$, by the local Lipschitz regularity \autoref{thm:lipregPoisson}.
\\We claim that, for $\eps>0$ sufficiently small, also
\begin{equation*}
f_{\eps}(x):=f(x)-\eps v(x)
\end{equation*}
attains a local minimum at an interior point in $\Omega$.\\
Let us suppose by contradiction that this is not the case. Then, for any $\eps>0$, the global minimum  of  $f_{\eps}$ on $\bar{\Omega}$ is attained  on $\partial \Omega$. In particular there exists $x_{\eps}\in \partial \Omega$ such that
\begin{equation*}
f(x_{\eps})-\eps M\le f(x_{\eps})-\eps v(x_{\eps})=f_{\eps}(x_\eps)\le f_{\eps}(x_{0})\le f(x_0)\, .
\end{equation*}
Hence
\begin{equation*}
\min_{x\in\partial\Omega}f(x)-f(x_0)\le f(x_{\eps})-f(x_0)\le M\eps\, ,\quad\text{for any $\eps>0$}\, ,
\end{equation*}
which yields a contradiction with \eqref{eq:intmin} a soon as $\eps$ is sufficiently small.
\medskip

Let now $\eps>0$ be small enough to get that $f_{\eps}=f-\eps v$ has a local minimum $c\in \setR$ at $\bar{x}\in\Omega$. Note that, by assumption, the function $g:=f-c$ has $\Delta g\leq 0$ in the viscous sense. Using $\eps v$ as a test function in the definition of the bound $\Delta g\leq 0$ in viscous sense,  we infer 
\begin{equation*}
\Delta (\eps v)(\bar{x})\le 0\, ,
\end{equation*}
a contradiction since $\Delta v=1$ on $\Omega$.
\end{proof}


\begin{theorem}\label{thm:viscoimpldistri}
Let $(X,\dist,\haus^N)$ be an $\RCD(K,N)$ metric measure space. Let $\Omega\subset X$ be an open and bounded domain, $f:\Omega\to\setR$ be a Lipschitz function and $\eta:\Omega\to\setR$ be continuous. Then $\boldsymbol{\Delta} f\le \eta$ in the sense of distributions if and only if $\Delta f\le \eta$ in the viscous sense.
\end{theorem}

\begin{proof}
We already proved in \autoref{prop:distrimplvisc} that distributional bounds on the Laplacian imply viscous bounds, so we are left to prove the converse implication.
\medskip

We claim that if $\Delta f\le \eta$ in the viscous sense, then $f$ is a classical supersolution to $\Delta f=\eta$ in the sense of \autoref{def:classicalsupersolution}. This is a consequence of \autoref{lemma:maxprincvisco}. Indeed, let us consider any open subdomain $\Omega'\Subset\Omega$ and any function $g\in C(\overline{\Omega'})$ such that $\Delta g=\eta$ on $\Omega'$ and $g\le f$ on $\partial\Omega'$.\\
Observe that $h:=f-g$ is continuous on $\overline{\Omega'}$ and verifies $\Delta h\le 0$ in the viscous sense on $\Omega'$, since $\Delta f\le \eta$ in the viscous sense and $\Delta g=\eta$. Therefore we can apply \autoref{rm:shiftingto0} and infer, by \autoref{lemma:maxprincvisco}, that
\begin{equation*}
\min_{x\in\Omega'} h(x)=\min_{x\in\partial \overline{\Omega'}} h(x)\ge 0\, .
\end{equation*}
It follows that $f\ge g$ on $\Omega'$, hence $f$ is a classical supersolution of $\Delta f=\eta$. 

The validity of the bound $\boldsymbol{\Delta} f\le \eta$ in the sense of distributions follows then from \autoref{thm:supersolsupermin}.
\end{proof}


The following is a counterpart, tailored to our purposes and under simplified assumptions, of the classical fact that the infimum of a family of viscosity supersolutions to a given equation is still a supersolution. Notice that the viscous approach fits particularly well with the stability issue for Laplacian bounds under infima. This property seems to be known to experts but we are not aware of any reference.

\begin{proposition}\label{prop:stabinf}
Let $(X,\dist,\haus^N)$ be an $\RCD(K,N)$ metric measure space. Let $\Omega\subset X$ be an open domain and let $f:\Omega\to\setR$ be continuous. Let $\mathcal{F}$ be a family of uniformly Lipschitz functions $u:\Omega\to\setR$ such that 
\begin{equation*}
\Delta u\le f\quad \text{in the viscous sense on $\Omega$}.
\end{equation*}
Let $v:\Omega\to\setR\cup\{-\infty\}$ be defined by
\begin{equation*}
v(x):=\inf\{u(x)\, :\, u\in\mathcal{F}\}\, .
\end{equation*}
Assume there exists a point $x_0\in\Omega$ such that $v(x_0)>-\infty$.  Then
\begin{equation*}
\Delta v\le f  \quad\text{in the viscous sense on $\Omega$}.
\end{equation*}
\end{proposition}

\begin{proof}
Let us preliminarily point out that, if $v(x_0)>-\infty$, then $v:\Omega\to\setR$ and, by the uniform Lipschitz assumption on the family $\mathcal{F}$, $v$ is Lipschitz on $\Omega$.
\medskip

We wish to verify that $\Delta v\le f$ in the viscous sense. To this aim, let $\Omega'\Subset\Omega$, $x\in\Omega'$ and $\phi:\Omega'\to\setR$ be such that $\phi\le v$ on $\Omega'$, $\phi(x)=v(x)$ and $\phi$ has continuous Laplacian on $\Omega'$.\\
Let us suppose by contradiction that $\Delta\phi(x)>f(x)$. Then there exist $\eps>0$ and a neighbourhood $U_x\ni x$ such that $\Delta\phi>f+\eps$ on $U_x$, by continuity of $\Delta\phi$ and $f$. 
\\
Let $b_x$ be the Green-type distance of \autoref{prop:Gdistance}, and recall the expression \eqref{eq:laplb4} of $\Delta b_{x}^{4}$. Using  the two sided estimates \eqref{eq:twosidedbx} for $b_x$ and the gradient estimate \eqref{eq:gradbx}, we can find $\eps'>0$ small enough such that, setting $\phi_{\eps'}:=\phi-\eps' b_x^4$,  it holds $\Delta\phi_{\eps'}>f+\eps''$ on $U_x$, for some $\eps''>0$.\\
Observe that $v-\phi_{\eps'}$ is non-negative and, thanks to the perturbation, it has a strict minimum at $x$. Let us consider now $u_h\in\mathcal{F}$ such that 
\begin{equation*}
v(x)=\lim_{h\to\infty}u_h(x)\, .
\end{equation*}
Let $\tilde{u}_h:=u_h-\phi_{\eps'}$. Let $y_h\in \overline{U_x}$ be a minimum point of $\tilde{u}_h$ on $\overline{U_x}$. Then it is easy to prove that $y_h\to x$ as $h\to\infty$, since $v-\phi_{\eps'}$ has its unique minimum on $\overline{U_x}$ at $x$. 

It is now sufficient to observe that $\Delta u_h\le f$ in the viscous sense and use that $\Delta \phi_{\eps'}>f+\eps''$ in the viscous and a.e. sense. Hence
\begin{equation}\label{eq:Deltauh<eps''}
\Delta \tilde{u}_h<-\eps''\, \quad\text{in the viscous sense on $U_x$}.
\end{equation}
From the proof of \autoref{thm:viscoimpldistri}, we infer that  $\tilde{u}_h$ is a classical supersolution of $\Delta w=0$, i.e. it is superharmonic in classical sense.
 Since $\tilde{u}_h$  is achieving its minimum at an interior point of $U_x$, by strong maximum principle for superharmonic functions (see for instance  \cite[Theorem 8.13]{BjornBjorn11}),  it is constant on $U_{x}$. But then  $\Delta \tilde{u}_h=0$ on $U_x$, contradicting \eqref{eq:Deltauh<eps''}.
\end{proof}

The last part of this subsection is dedicated to the relationship between \autoref{def:heatlaplbounds} and the other notions of Laplacian bounds that we have introduced and investigated so far.
\medskip

For a sufficiently smooth function $f$ on the Euclidean space or on a Riemannian manifold, the Laplacian $\Delta f(x)$ determines the first non trivial term in the asymptotic expansion of the average of $f$ on balls centred at $x$:
\begin{equation}\label{eq:asball}
\fint_{B_r(x)}f(y)\di\haus^n(y)=f(x)+C_n\Delta f(x)r^2+o(r^2)\, ,\quad\text{as $r\to 0$}\, ,
\end{equation} 
where $C_n$ is a constant depending only on the ambient dimension. A classical result is the fact that a continuous function $u:\Omega\to\setR$ on a Euclidean domain is harmonic (in the classical sense) if and only if 
\begin{equation*}
\lim_{r\to 0}\fint_{B_r(x)}(u(y)-u(x))\di\Leb^n(y)=0\, ,\quad\text{for any $x\in\Omega$}\, .
\end{equation*}

Although being a really powerful tool, at first sight, the asymptotic expansion above seems to require smoothness of the ambient space for its validity. Moreover, it is easy to check that it fails in general on smooth weighted Riemannian manifolds.

There have been recent attempts of understanding the connections between this approach through asymptotic mean values and the distributional notion of Laplacian on metric measure spaces. Let us mention in particular \cite[Section 4]{ZhangZhu12} where, relying on some ideas originally due to \cite{Petrunin00,Petrunin03}, it is shown that the asymptotic of the average on balls determines the Laplacian of a semiconcave function at sufficiently regular points on Alexandrov spaces. 
\medskip

Here we consider an alternative approach: instead of looking at the behaviour of averages on balls, we look at the pointwise behaviour of the heat flow. Basically, we consider weighted averages instead of averages, the weight being given by the heat kernel.\\ 
As we shall see, this turns to be a more intrinsic approach (the infinitesimal generator of the heat semigroup is the Laplacian) and allows for a counterpart of \eqref{eq:asball} better suited for the non smooth framework. 

\begin{proposition}\label{prop:heatimplvisco}
Let $(X,\dist,\meas)$ be an $\RCD(K,N)$ metric measure space.  Let $\Omega\subset X$ be an open subset,  $f:\Omega\to\setR$ be Lipschitz and  $\eta:\Omega\to\setR$ be continuous. Assume that for some global extension $\tilde{f}:X\to\setR$ of $f$ with polynomial growth,  it holds
\begin{equation}\label{eq:assheat}
\limsup_{t\downarrow 0}\frac{P_t\tilde{f}(x)-\tilde{f}(x)}{t}\le\eta (x)\,, \quad \text{ for any } x\in \Omega.
\end{equation}
Then $\Delta f\le \eta$ on $\Omega$ in the viscous sense.
\end{proposition}

\begin{proof}
We need to verify that, for any open subdomain $\Omega'\Subset\Omega$ and for any function $\phi:\Omega'\to\setR$ with continuous Laplacian on $\Omega'$ satisfying $\phi\le f$ on $\Omega'$ and $\phi(x)=f(x)$ for some $x\in \Omega'$, the following estimate holds:
\begin{equation*}
\Delta \phi(x)\le \eta(x)\, .
\end{equation*}

Let us first assume that $\phi$ extends to a global function $\tilde{\phi}:X\to\setR$ with polynomial growth and such that $\tilde{\phi}\le \tilde{f}$. Then
\begin{equation*}
\Delta\phi(x)=\Delta\tilde{\phi}(x)=\lim_{t\downarrow 0}\frac{P_t\tilde{\phi}(x)-\tilde{\phi}(x)}{t}\le \limsup_{t\downarrow 0}\frac{P_t  \tilde{f}(x)-  \tilde{f}(x)}{t}\le \eta(x)\, ,
\end{equation*}
where the first equality follows from the locality of the Laplacian, the second one from \autoref{lemma:pointheat}, the first inequality follows from the comparison principle for the heat flow and the last one from \eqref{eq:assheat}.

To complete the proof, we need to extend locally defined test functions for the Laplacian bound in viscous sense to globally defined functions, keeping the comparison.\\ 
To this aim, observe that we can extend any test function for the Laplacian bound in viscous sense to a global function $\hat{\phi}$ by multiplying it with a cut-off function with good estimates which is constantly $1$ on a small ball centred at $x$, see \autoref{lemma:cutoff}. In this way, we loose the comparison with $f$ but we obtain a globally defined function which coincides with $\phi$ in a neighbourhood of $x$. Then, setting $\tilde{\phi}:=\min\{\tilde{f}, \hat{\phi}\}$, we can easily verify that $\tilde{\phi}$ has polynomial growth, and $\tilde{\phi}\le \tilde{f}$ globally. Moreover, since $\tilde{\phi}\equiv \phi$ in a neighbourhood of $x$, still $\tilde{\phi}(x)=f(x)$ and $\tilde{\phi}$ has continuous Laplacian in a neighbourhood of $x$.
\end{proof}

\begin{proposition}\label{prop:distrimplheat}
Let $(X,\dist,\meas)$ be an $\RCD(K,N)$ metric measure space. Let $\Omega\subset X$ be an open domain and let $f:\Omega\to\setR$ be a locally Lipschitz function. Let $\eta\in\Cb(\Omega)$ and assume that 
\begin{equation*}
\boldsymbol{\Delta} f\le \eta\,, \quad\text{in the distributional sense on $\Omega$.}
\end{equation*}
 Then, for any $\Omega'\Subset \Omega$ and for any function $\tilde{f}:X\to\setR$ with polynomial growth and such that $\tilde{f}\equiv f$ on $\Omega'$, it holds
\begin{equation*}
\limsup_{t\downarrow 0}\frac{P_t\tilde{f}(x)-\tilde{f}(x)}{t}\le \eta(x)\, , \quad\text{for any $x\in\Omega'$}\, .
\end{equation*}

\end{proposition}

\begin{proof}
We divide the proof into three steps: first we deal with the case of superharmonic functions. Then we deal with the case of solutions of Poisson equations with continuous right hand sides. To conclude we combine the previous two steps to treat the general case.
\medskip

\textbf{Step 1.} Let us assume that $\eta\equiv 0$ on $\Omega$. Thanks to \autoref{lemma:cutoff} we can choose a good cut-off function $\phi:X\to\setR$ supported on $B_{2r}(x)$ and such that $\phi\equiv 1$ on $B_r(x)$. Computing the Laplacian of $f\phi$ by the standard calculus rules, we infer that 
\begin{equation}\label{eq:claim1eq2}
\boldsymbol{\Delta}(f\phi)=\phi \boldsymbol{\Delta}f+2\nabla \phi\cdot\nabla f+f\Delta\phi\, .
\end{equation}
By \autoref{lemma:sufflocal}, it is sufficient to prove that
\begin{equation*}
\limsup_{t\downarrow 0}\frac{P_t(f\phi)(x)-(f\phi)(x)}{t}\le 0\, .
\end{equation*}
Moreover, setting $\boldsymbol{\mu}:=\boldsymbol{\Delta}(f\phi)$, we have that $\boldsymbol{\mu}$ is the sum of a bounded function $\psi:=2\nabla \phi\cdot\nabla f+f\Delta\phi$ supported on $B_{2r}(x)\setminus B_r(x)$ and a non-positive measure $\boldsymbol{\nu}:=\phi \boldsymbol{\Delta}f$. We claim that 
\begin{equation}\label{eq:cla}
P_t(f\phi)(x)-(f\phi)(x)\le \int_0^tP_s\psi(x)\di s\, .
\end{equation}
In order to establish the claim we borrow the argument from the proof of \cite[Lemma 3.2]{Gigli22}. We set
\begin{equation}
\tilde{\Test}^{\infty}:=\{\eta\in L^1(X)\cap \Test^{\infty}(X)\, :\, \abs{\nabla \eta}, \Delta \eta\in L^1(X) \}\, ,
\end{equation} 
and let $\tilde{\Test}^{\infty}_+$ be the cone of nonnegative functions in $\tilde{\Test}^{\infty}$. We recall that for any nonnegative function $\eta\in L^1\cap L^{\infty}$ there exists a sequence $\eta_n\in \tilde{\Test}^{\infty}_+$ such that $\eta_n$ are uniformly bounded in $L^{\infty}$ and converge to $\eta$ in $L^1$. Hence, in order to prove \eqref{eq:cla} it is sufficient to show that
\begin{equation}
\int_X \eta\left(P_t(f\phi)-f\phi\right)\di\meas \le \int_X \eta\left(\int_0^t P_s\psi \di s\right)\di\meas\, ,
\end{equation}
for any $\eta \in \tilde{\Test}^{\infty}_+$. To this aim we can compute
\begin{align*}
\int_X \eta\left(P_t(f\phi)-f\phi\right)\di\meas=&\int_Xf\phi\left(P_t\eta-\eta\right)\di\meas\\
=&\int_X\int_0^tf\phi\Delta P_s\eta\di s\di\meas\\
\le &\int_X\int_0^t \psi P_s\eta\di s\di\meas\\
=&\int_X\eta \left(\int_0^tP_s\psi\di s\right)\di\meas\, .
\end{align*}
\smallskip

Since $\psi$ is bounded and supported on $B_{2r}(x)\setminus B_r(x)$, by \autoref{lemma:sufflocal} we infer:
\begin{equation*}
\limsup_{t\downarrow 0}\frac{P_t(f\phi)(x)-(f\phi)(x)}{t}\le \limsup_{t\downarrow 0}\frac{\int_0^tP_s\psi(x)\di s}{t}= 0\, ,
\end{equation*}
which proves \eqref{eq:claim1eq2}.
\medskip

\textbf{Step 2.} By \autoref{lemma:pointheat}, if $g:X\to\setR$ has polynomial growth and, for some $r>0$ and $x\in X$, $g$ belongs to the domain of the Laplacian on $B_r(x)$ and it has bounded and continuous Laplacian $\Delta g=\eta$ therein, then
\begin{equation*}
\lim_{t\downarrow 0}\frac{P_tg(x)-g(x)}{t}=\eta(x)\, ,\quad\text{for any $x\in B_r(x)$}\, .
\end{equation*}

\medskip

\textbf{Step 3.} Let us combine the outcomes of the previous two steps to prove the statement.
\\Let us consider a ball $B_{2r}(x)\Subset\Omega'$ and let $\phi:B_{2r}(x)\to\setR$  be a solution (see \autoref{thm:existenceanccomparison}) of
\begin{equation*}
\Delta \phi=\eta\, ,\quad\text{on $B_{2r}(x)$}\, .
\end{equation*} 
Observe that $f-\phi$ is Lipschitz on $B_r(x)$ and 
\begin{equation*}
\boldsymbol{\Delta}(f-\phi)\le 0\, ,\quad\text{on $B_{r}(x)$}\, . 
\end{equation*}
From Step 1, we infer that  for any extension $\tilde{f}_{\phi}:X\to\setR$ of $f-\phi$ with polynomial growth it holds
\begin{equation}\label{eq:condfphi}
\limsup_{t\downarrow 0}\frac{P_t\tilde{f}_{\phi}(x)-\tilde{f}_{\phi}(x)}{t}\le 0\, .
\end{equation}
Moreover, we can consider an extension $\tilde{\phi}:X\to\setR$ of $\phi$ and observe that, by Step 2,  
\begin{equation} \label{eq:condphi}
\lim_{t\downarrow 0}\frac{P_t\tilde{\phi}(x)-\tilde{\phi}(x)}{t}=\eta(x)\, .
\end{equation}
It is straightforward to check that, for any extension $\tilde{f}:X\to\setR$ of $f$, $\tilde{f}-\tilde{\phi}$ is an extension of $f-\phi$. Hence, applying \eqref{eq:condfphi} to $\tilde{f}_{\phi}:=\tilde{f}-\tilde{\phi}$ and then \eqref{eq:condphi}, we get
\begin{equation*}
\limsup_{t\downarrow 0}\frac{P_t\tilde{f}(x)-\tilde{f}(x)}{t}\le \eta(x)\, ,
\end{equation*}
as we claimed. 
\end{proof}


We collect the main equivalence results for Laplacian bounds in a single statement below. Many of the equivalences are proved without the restriction that $\meas=\haus^N$ and we expect all of them to hold in general. We do not pursue the most general statements as they will not be needed in the sequel of the paper.

\begin{theorem}[Equivalent notions of Laplacian bounds]\label{thm:mainequivlapla}
Let $(X,\dist,\haus^N)$ be an $\RCD(K,N)$ metric measure space. Let $\Omega\subset X$ be an open domain, $\eta\in\Cb(\Omega)$ and $f:\Omega\to\setR$ be a locally Lipschitz function. Then the following are equivalent:
\begin{itemize}
\item[(i)] $\boldsymbol{\Delta}f\le \eta$ in the sense of distributions on $\Omega$, as in \autoref{def:distributions};
\item[(ii)] $f$ is a superminimizer of the energy $E_{\eta}$, as in \autoref{def:subsupermin};
\item[(iii)] $f$ is a classical supersolution of $\Delta f=\eta$ in the sense of \autoref{def:classicalsupersolution};
\item[(iv)] $f$ satisfies $\Delta f\le \eta$ in the viscous sense as in \autoref{def:viscosity};
\item[(v)] $f$ is a supersolution of $\Delta f\le \eta$ in the heat flow sense as in \autoref{def:heatlaplbounds}.
\end{itemize}
\end{theorem}

While the equivalences between (i), (ii) and (iii) are well established within the theory of metric measure spaces that are doubling and verify a Poincar\'e inequality, our proofs of the equivalence between (iv), (v) and the previous ones heavily rely on the $\RCD(K,N)$ assumption. Indeed, the Omori-Yau-Jensen type maximum principle \autoref{thm:maximumprincipleOYJ}, the existence of a nice auxiliary function with the properties detailed in \autoref{rm:goodfunction} and the Gaussian heat kernel bounds played a fundamental role in all of the arguments above.
\medskip

%
%

\section{Ricci curvature bounds, Hopf-Lax semigroups and Laplacian bounds}\label{sec:HopfLax}

This section is dedicated to analyse the interplay between the Hopf-Lax semigroups (associated to exponents $1\le p<\infty$),  Ricci curvature lower bounds and Laplacian upper bounds.\\ 
Let us introduce some notation and terminology. 
\medskip

Let $1\le p<\infty$. We shall consider the evolution via the $p$-Hopf-Lax semigroup on a general metric space $(X,\dist)$. Let us consider $f:X\to\setR\cup\{\pm\infty\}$, not identically $+\infty$, and let the evolution via $p$-Hopf-Lax semigroup be defined by
\begin{equation}\label{eq:HLdef}
\mathcal{Q}^p_tf(x):=\inf_{y\in X}\left(f(y)+\frac{\dist(x,y)^p}{pt^{p-1}}\right)\, .
\end{equation}
Observe that, in the case $p=1$, the expression for the Hopf-Lax semigroup is actually independent of $t$:
\begin{equation*}
\mathcal{Q}^1_tf(x)=\mathcal{Q}^1 f(x)=\inf_{y\in X}\left(f(y)+\dist(x,y)\right)\, .
\end{equation*}
The key result of this section will be that the Hopf-Lax semigroup preserves upper bounds on the Laplacian on $\RCD$ spaces, when suitably interpreted, for any exponent $1\le p<\infty$. This observation appears to be new for general exponents $p$, even for smooth Riemannian manifolds. The only previous references we are aware of are \cite{ZhangZhu12}, dealing with the case $p=2$ on Alexandrov spaces with lower Ricci curvature bounds, (the result had been previously announced in the unpublished \cite{Petrunin00}, where a strategy on Alexandrov spaces was also indicated) and the more recent \cite{ZhangZhongZhu19}, where exponents $1<p<\infty$ on smooth Riemannian manifolds are considered. Even in this case, our proof seems more robust and it is based on a completely different idea, relying on the connection between the heat flow and lower Ricci curvature bounds instead of the second variation formula. 
\medskip

In the Euclidean setting, the inf-convolution preserves the property of being a supersolution of the Laplace equation, $\Delta u=0$. Classical proofs of this fact, that allow for extensions to more general PDEs, are based on the affine invariance of the Euclidean space.
 
In \autoref{subsec:smooth} we generalize this statement to Riemannian manifolds with lower Ricci curvature bounds. The proof introduces a different approach based on the characterization of the Laplacian of smooth functions through asymptotics of averages on balls. To avoid technicalities we will consider only smooth functions, though it is worth pointing out that the Hopf-Lax semigroup does not preserve smoothness, even in the Euclidean setting.

The extension to non smooth $\RCD(K,N)$ spaces, that we shall address in \autoref{subsec:nonsmooth}, requires two further ideas: a weak theory of Laplacian bounds in the non smooth context, that we have at our disposal after \autoref{sec:lap}, and a new intrinsic way to connect the Laplacian to the Hopf-Lax semigroup under the $\RCD$ condition. This connection will be achieved exploiting a powerful duality formula, originally due to Kuwada \cite{Kuwada10}, that we review in \autoref{subsec:kuwada}.

\subsection{Smooth Riemannian manifolds}\label{subsec:smooth}
For the sake of motivation, let us present a characterization of lower Ricci bounds for smooth Riemannian manifolds involving the interplay between the Hopf-Lax semigroup and Laplacian bounds.
\medskip

Let $(M^n,g)$ be a smooth Riemannian manifold and, given a sufficiently smooth function $f:M\to\setR$, let us set
\begin{equation*}
\sigma_rf(x):=\fint_{\partial B_r(x)}f(y)\di\haus^{n-1}(y)=\int f(y)\di\sigma_{x,r}(y)\, ,
\end{equation*}
where we denoted by $\haus^{n-1}$ the surface measure of $\partial B_r(x)$ and notice that, by its very definition, $\sigma_{x,r}:= \haus^{n-1}(\partial B_r(x))^{-1}  \,  \haus^{n-1} \res \partial B_r(x) $ is a probability measure.
\medskip

Let us recall (see for instance the proof of \cite[Theorem 1.5]{SturmVonRenesse05}) the following classical fact: for any $x\in U_x\subset M$ and any function $f\in C^3(U_x)$, it holds
\begin{equation}\label{eq:asymptoticlapla}
\sigma_rf(x)=f(x)+\frac{r^2}{2n}\Delta f(x)+o(r^2)\, ,\quad\text{as $r\downarrow 0$}\, .
\end{equation}
We will denote by $f^c$ the dual of a function $f$, with respect to the optimal transport duality induced by cost equal to distance, i.e.
\begin{equation*}
f^c(y):=\inf_{z\in M}\{f(z)+\dist(y,z)\}\, , \quad \text{for any $y\in M$}.
\end{equation*}

\begin{theorem}\label{thm:HLsmoothmanifolds}
Let $(M^n, g)$ be a smooth closed Riemannian manifold and let $K\in \setR$. The following conditions are equivalent:
\begin{itemize}
\item[(i)] $\Ric\ge K$ on $M$;
\item[(ii)] for any function $f:M\to\setR$ and for any $x,y\in M$ such that
\begin{equation*}
f^c(x)-f(y)=\dist(x,y)\, ,
\end{equation*} 
if $f$ is smooth in a neighbourhood of $y$ and $f^c$ is smooth in a neighbourhood of $x$, then
\begin{equation}\label{eq:HL1smooth}
\Delta f^c(x)\le \Delta f(y)-K\dist(x,y)\, ,
\end{equation}
\end{itemize}

\end{theorem}

\begin{proof}
Let us start proving the implication from (i) to (ii). 

By \cite[Theorem 1.5]{SturmVonRenesse05}, if $(M^n,g)$ is a smooth Riemannian manifold such that $\Ric\ge K$, then for any couple of points $x,y\in M$,
\begin{equation}\label{eq:boundW1}
W_1(\sigma_{x,r},\sigma_{y,r})\le \left(1-\frac{K}{2n}r^2+o(r^2)\right)\dist(x,y)\, ,\quad\text{as $r\downarrow 0$}\, ,
\end{equation}
where we denoted by $W_1$ the Wasserstein distance associated to the exponent $p=1$. 

Then we can apply the classical Kantorovich-Rubinstein duality to infer that
\begin{equation}\label{eq:asymsigma}
\sigma_rf^c(x)-\sigma_rf(y)\le  \left(1-\frac{K}{2n}r^2+o(r^2)\right)\dist(x,y)\, ,\quad\text{as $r\downarrow 0$}\, .
\end{equation}
Indeed
\begin{equation*}
\sigma_rf^c(x)=\int f^c(z)\di\sigma_{x,r}(z)\, ,
\end{equation*}
\begin{equation*}
\sigma_rf(y)=\int f(z)\di\sigma_{y,r}(z)\, 
\end{equation*}
and 
\begin{equation*}
f^c(x)-f(y)\le \dist(x,y)\, ,\quad\text{for any $x,y\in M$}\, .
\end{equation*}
Therefore
\begin{equation*}
\sigma_rf^c(x)-\sigma_rf(y)\le W_1(\sigma_{x,r},\sigma_{y,r})\, 
\end{equation*}
and we can apply \eqref{eq:boundW1} to get \eqref{eq:asymsigma}.
\\Taking into account \eqref{eq:asymptoticlapla}, the assumption $f^c(x)-f(y)=\dist(x,y)$ and the fact that $x$ and $y$ are smooth points for $f^c$ and $f$ respectively, starting from \eqref{eq:asymsigma} we can easily infer that
\begin{equation*}
\Delta f^c(x)\le \Delta f(y)- K\dist(x,y)\, ,
\end{equation*}
as we claimed.
\medskip

Let us prove the converse implication. As for the classical implications between different characterizations of lower Ricci bounds in \cite{SturmVonRenesse05}, we wish to apply \eqref{eq:HL1smooth} to suitably chosen functions $f$ in order to control from below the Ricci curvature at any point and in any direction. 

To this aim, let us choose $x\in M$ and a tangent vector $v\in T_xM$. Let us assume without loss of generality that $\abs{v}_x=1$. Then we can find, via a standard construction, a smooth hypersurface $\Sigma_{x,v}\subset B_r(x)$ for $r>0$ small enough, such that $x\in \Sigma_{x,v}$, the tangent hyperplane to $\Sigma_{x,v}$ is the orthogonal to $v$ in $T_xM$ and the second fundamental form of the hypersurface is vanishing at $x$.\\
It is a standard fact in Riemannian geometry that the signed distance function $\dist^{\pm}_{\Sigma}$ from $\Sigma_{x,v}$ is a smooth $1$-Lipschitz function in a neighbourhood of $x$. Moreover, for some $\eps>0$ sufficiently small, we can consider a unit speed geodesic $\gamma:(-\eps,\eps)\to M$ such that $\gamma(0)=x$, $\gamma'(0)=v$ and 
\begin{equation*}
\dist^{\pm}_{\Sigma_{x,v}}(\gamma(t))=t\, ,\quad\text{for any $t\in(-\eps,\eps)$}\, .
\end{equation*} 
The following is a well known identity in Riemannian geometry (observe that $\Hess \dist^{\pm}_{\Sigma}=0$ at $x$ due to the vanishing of the second fundamental form of $\Sigma_{x,v}$ at $x$):
\begin{equation}\label{eq:identityRic}
\left.\frac{\di}{\di t}\right|_{{t=0}}\Delta \dist^{\pm}_{\Sigma_{x,r}}(\gamma(t))=-\Ric_x(v,v)\, .
\end{equation}
Now, applying \eqref{eq:HL1smooth} to $f=f^c=\dist^{\pm}_{\Sigma_{x,r}}$ at the points $\gamma(0)$ and $\gamma(t)$, we obtain that
\begin{equation*}
\Delta  \dist^{\pm}_{\Sigma_{x,r}}(\gamma(t))\le \Delta  \dist^{\pm}_{\Sigma_{x,r}}(\gamma(0))-Kt\, ,\quad\text{for any $t\in(0,\eps)$.}\, 
\end{equation*}
Thus, we infer that
\begin{equation*}
\left.\frac{\di}{\di t}\right|_{{t=0}}\Delta \dist^{\pm}_{\Sigma_{x,r}}(\gamma(t))\le -K\, ,
\end{equation*}
which proves that $\Ric_x(v,v)\ge K$, thanks to \eqref{eq:identityRic}.
\end{proof}

\begin{remark}
For brevity, we discussed only the case $p=1$, however it is possible to consider variants of \autoref{thm:HLsmoothmanifolds} above dealing with the Hopf-Lax semigroups associated to any exponent $1\le p<\infty$. 
\end{remark}

\begin{remark}
Smoothness of the test function $f$ in condition (ii) in \autoref{thm:HLsmoothmanifolds} above is an assumption which can be relaxed, if we understand the Laplacian bounds in a more general sense. This will be the key to formulate a counterpart of this results on general $\RCD(K,N)$ metric measure spaces and it will be a key for the applications later in the paper.

Moreover, as the forthcoming discussion will clarify, also the compactness of the manifold is a completely unnecessary assumption.
\end{remark}

\subsection{Kuwada's lemma}\label{subsec:kuwada}

We recall here a fundamental result highlighting the interplay between lower Ricci curvature bounds, contractivity estimates for the heat-flow and the Hopf-Lax semigroup. The original formulation on smooth Riemannian manifolds is due to Kuwada \cite{Kuwada10}. Later on, due to its particular robustness, it has been extended to $\RCD(K,\infty)$ metric measure spaces in \cite[Lemma 3.4]{AmbrosioGigliSavare15} in the case of exponent $p=2$.

\begin{theorem}[Kuwada duality]\label{thm:KuwadaK=0}
Let $(X,\dist,\meas)$ be an $\RCD(K,\infty)$ metric measure space and $f\in\Lipb(X)$ be non-negative and with bounded support. 
Then, for any $t\ge 0$, $\mathcal{Q}_t^2f$ is Lipschitz, non-negative, with bounded support and it holds
\begin{equation}\label{eq:kuwadaK=0}
P_s\left(\mathcal{Q}^2_1 f\right)(x)-P_sf(y)\le \frac{e^{-2Ks}}{2}\dist(x,y)^2\,,
\end{equation}
for any $x,y\in X$ and for any $s\ge 0$.
\end{theorem}

Thanks to the self-improvement of the Bakry-\'Emery gradient contraction estimate for the heat flow obtained on $\RCD(K,\infty)$ spaces in \cite{Savare14} (see \eqref{eq:selfimpr}), \autoref{thm:KuwadaK=0} can then be generalized to arbitrary exponents $p$, along the original lines of \cite{Kuwada10}. Since it can be proved with the very same strategy of the case $p=2$ we omit the proof.

\begin{theorem}\label{thm:kuwadap}
Let $(X,\dist,\meas)$ be an $\RCD(K,\infty)$ metric measure space and $f\in\Lipb(X)$ be non-negative and with bounded support. Let $1\le p<\infty$.
Then, for any $t\ge 0$, $\mathcal{Q}_t^pf$ is Lipschitz, non-negative, with bounded support and it holds
\begin{equation}\label{eq:kuwadaK=0}
P_s\left(\mathcal{Q}^p_1 f\right)(x)-P_sf(y)\le \frac{e^{-pKs}}{p}\dist(x,y)^p\,,
\end{equation}
for any $x,y\in X$ and for any $s\ge 0$
\end{theorem}

For our purposes it will be relevant to apply Kuwada's duality under milder assumptions on the function $f$. This is possible under the $\RCD(K,N)$ condition for finite $N$, thanks to the Gaussian estimates for the heat kernel, that, as we already pointed out (see \eqref{eq:polgrowth} and the discussion following it), enlarge the class of functions to which the heat flow can be applied. We focus for simplicity on the case $p=1$, which is the relevant one for our purposes.

\begin{theorem}\label{thm:kuwadapolgrowth}
Let $(X,\dist,\meas)$ be an $\RCD(K,N)$ metric measure space for some $K\in\setR$ and $1\le N<\infty$. Let $f:X\to\setR$ be a locally Lipschitz function with polynomial growth. Let us assume that there exists $x_0\in X$ such that $\mathcal{Q}^1_1f(x_0)\in\setR$. Then
\begin{equation}\label{eq:dualityp1}
P_s\left(\mathcal{Q}^1_1 f\right)(x)-P_sf(y)\le e^{-Ks}\dist(x,y)\ ,
\end{equation}  
 for any $x,y\in X$ and for any $s\ge 0$.
\end{theorem}

\begin{proof}
Let us set $f^c:=\mathcal{Q}^1_1 f$, in order to ease the notation.\\
Observe that, if there exists $x_0\in X$ such that $f^c(x_0)\in \setR$, then $f^c$ is a $1$-Lipschitz function. Moreover, since for any function $f$ as above, it holds that $f^c\le f$, it is sufficient to prove \eqref{eq:dualityp1} for $1$-Lipschitz functions. Indeed, if the statement holds for $1$-Lipschitz functions, then
\begin{equation*}
\left(P_sf^c\right)(x)-\left(P_sf^c\right)(y)\le e^{-Ks}\dist(x,y)\, ,
\end{equation*} 
for any $x,y\in X$ and for any $s\ge 0$.\\
Hence, since $f^c\le f$ and therefore $P_sf^c\le P_sf$, we obtain
\begin{equation*}
\left(P_sf^c\right)(x)-P_sf(y)\le e^{-Ks}\dist(x,y)\;\;\;\text{ for any $x,y\in X$ and for any $s\ge 0$}\, ,
\end{equation*}
as we wished.
Now, given any $1$-Lipschitz function $f:X\to\setR$,  observe that $f^c=f$. Using \cite[Theorem 6.1 (iv)]{AmbrosioGigliSavare14}, we can estimate
\begin{equation*}
\LipConst (P_sf)\le e^{-Ks}P_s \big( \LipConst(f) \big) \le  e^{-Ks} \, .
\end{equation*}
Hence
\begin{equation*}
\abs{P_sf(x)-P_sf(y)}\le e^{-Ks}\dist(x,y)\, ,\quad\text{for any $x,y\in X$ and for any $s\ge 0$}\, .
\end{equation*}
\end{proof}

\begin{remark}
Note that, if $f:X \to \setR$ is 1-Lipschitz, then one can reinforce the estimate   \eqref{eq:dualityp1} by putting the modulus in the left hand side.
\end{remark}


\subsection{Hopf-Lax semigroup and Laplacian bounds: the non smooth framework}\label{subsec:nonsmooth}
Let us consider an $\RCD(K,N)$ metric measure space $(X,\dist,\meas)$. Recall the definition of the $p$-Hopf-Lax semigroup \eqref{eq:HLdef}.\\ 
In order to motivate the next developments, let us start with some formal computations, neglecting the regularity issues.
\medskip

To this aim let $x\in X$ and suppose that there exists $y\in X$ such that 
\begin{equation}\label{eq:HLattained}
\mathcal{Q}^p_1f(x)=f(y)+\frac{\dist(x,y)^p}{p}\, ,
\end{equation}
i.e., $y$ is a point where the infimum defining the $p$-Hopf-Lax semigroup for $t=1$ at $x$ is attained.\\
Observe that, for $x$ and $y$ as above, equality holds at time $s=0$ in \eqref{eq:kuwadaK=0}. Hence, by taking the right derivative,
\begin{equation}\label{eq:laplafromkuwada}
\limsup_{s\downarrow 0}\frac{P_s\left(\mathcal{Q}^p_1 f\right)(x)-\mathcal{Q}^p_1 f(x)}{s}\le \limsup_{s\downarrow 0}\frac{P_sf(y)-f(y)}{s} - K \dist(x,y)^{p}\, .
\end{equation}
If $f$ is \textit{regular} at $y$, the first term in the right hand side of \eqref{eq:laplafromkuwada} is the value $\Delta f(y)$. Hence \eqref{eq:laplafromkuwada} can be turned into
\begin{equation*}
\limsup_{s\downarrow 0}\frac{P_s\left(\mathcal{Q}^p_1 f\right)(x)-\mathcal{Q}^p_1 f(x)}{s}\le \Delta f(y) - K \dist(x,y)^{p}\, ,
\end{equation*}
where we recall that $x$ and $y$ are such that \eqref{eq:HLattained} holds. If also $\mathcal{Q}^p_1f$ happens to be \textit{regular} near to $x$, then
\begin{equation*}
\Delta \mathcal{Q}^p_1f(x)\le\Delta f(y) - K \dist(x,y)^{p}\, .
\end{equation*}
As we shall see, the viscous theory of Laplacian bounds allows to let the heuristic above become rigorous. 
\medskip

In order to ease the notation, we shall indicate
\begin{equation*}
\Delta ^hf(x):=\limsup_{t\downarrow 0}\frac{P_tf(x)-f(x)}{t}\, ,
\end{equation*}
whenever $f:X\to\setR$ is a locally Lipschitz function with polynomial growth.

\begin{proposition}\label{prop:HLprese hsense}
Let $(X,\dist,\meas)$ be an $\RCD(K,N)$ metric measure space for some $K\in\setR$ and $1\le N<\infty$. Let $f:X\to\setR$ be a locally Lipschitz function with polynomial growth. Let us assume that there exists $x_0\in X$ such that $f^c(x_0):=\mathcal{Q}^1_1f(x_0)\in\setR$. If $x,y\in X$ verify
\begin{equation*}
f^c(x)-f(y)=\dist(x,y)\, ,
\end{equation*}  
then
\begin{equation*}
\Delta^h f^c(x)\le \Delta^h f(y)-K\dist(x,y)\, .
\end{equation*}
\end{proposition}

\begin{proof}
The conclusion follows from \autoref{thm:kuwadapolgrowth}, relying on the very definition of $\Delta^h$ through the formal argument presented above.
Indeed, under the assumption of the statement, by \autoref{thm:kuwadapolgrowth} we have:
\begin{equation}\label{kuw}
P_sf^c(x)-P_sf(y)\le e^{-Ks}\dist(x,y)\,, \quad \text{ for any $x,y\in X$ and for any $s\ge 0$}\, .
\end{equation}
Moreover, by assumption, equality holds in \eqref{kuw} at time $s=0$. Hence, by taking the right derivative at both sides, we infer that
\begin{align*}
\Delta ^hf^c(x)&=\limsup_{s\downarrow 0}\frac{P_sf^c(x)-f^c(x)}{s}\\
&\le \limsup_{s\downarrow 0}\frac{P_sf(y)-f(y)}{s}+\lim_{s\downarrow 0}\frac{e^{-Ks}-1}{s}\dist(x,y)\\
&=\Delta ^hf(y)-K\dist(x,y)\, .
\end{align*}
\end{proof}

Thanks to the equivalences for Laplacian bounds over noncollapsed $\RCD(K,N)$ metric measure spaces (see \autoref{thm:mainequivlapla}), we obtain the following.

\begin{theorem}\label{thm:HLlapla}
Let $(X,\dist,\haus^N)$ be an $\RCD(K,N)$ metric measure space for some $K\in\setR$ and $1\le N<\infty$. Let $f:X\to\setR$ be a locally Lipschitz function with polynomial growth. Let $\Omega,\Omega'\subset X$ be open domains and $\eta\in\setR$. Then the following holds. Assume that $f^c$ is finite and that, for any $x\in\Omega'$ the infimum defining $f^c(x)$ is attained at some $y\in \Omega$. Assume moreover that
\begin{equation}\label{eq:lapassumption}
\Delta f\le \eta \quad\text{on $\Omega$}\,.
\end{equation}
Then
\begin{equation*}
\Delta f^c\le \eta -\min_{x\in\Omega', y\in\Omega} K \dist(x,y) \quad\text{on $\Omega'$}
\end{equation*}
 where the Laplacian bounds have to be intended in any of the equivalent senses of \autoref{thm:mainequivlapla}.
\end{theorem}

\begin{proof}
The statement follows from \autoref{prop:HLprese hsense} and \autoref{thm:mainequivlapla}. Indeed, by \eqref{eq:lapassumption} and \autoref{thm:mainequivlapla}, we have
\begin{equation*}
\Delta ^hf(y)\le \eta\, \quad\text{for any $y\in \Omega$}\, .
\end{equation*}
Hence, by \autoref{prop:HLprese hsense},
\begin{equation*}
\Delta ^hf^c(x)\le \eta-K\dist(x,y)\, ,
\end{equation*}
where $y\in\Omega$ is such that $f^c(x)-f(y)=\dist(x,y)$. 
\\The conclusion follows applying \autoref{thm:mainequivlapla} again to $f^c$ on $\Omega'$.
\end{proof}

Specializing to the case of non-negative Ricci curvature $K=0$, we get a cleaner statement.

\begin{corollary}\label{cor:HLdistr}
Let $(X,\dist,\haus^N)$ be an $\RCD(0,N)$ metric measure space. Let $f:X\to\setR$ be a locally Lipschitz function with polynomial growth. Let $\Omega,\Omega'\subset X$ be open domains and $\eta\in\setR$. Assume that $f^c$ is finite and that, for any $x\in\Omega'$ the infimum defining $f^c(x)$ is attained at some $y\in \Omega$. Assume moreover that
\begin{equation*}
\Delta f\le \eta\, \quad\text{on $\Omega$}\,.
\end{equation*}
Then
\begin{equation*}
\Delta f^c\le \eta\, \quad\text{on $\Omega'$}\, ,
\end{equation*}
where the Laplacian bound can be intended in any of the equivalent senses of \autoref{thm:mainequivlapla}.
\end{corollary}

\begin{remark}
For brevity, we discussed only the case $p=1$, however it is possible to obtain counterparts of all the results above dealing with the Hopf-Lax semigroup associated to an arbitrary exponent $1\le p<\infty$.
\end{remark}

\section{Mean curvature bounds for minimal boundaries}\label{sec:meancurv}

This section is dedicated to the study of mean curvature bounds for boundaries of locally perimeter minimizing sets of finite perimeter, in the framework of $\RCD(K,N)$ metric measure spaces $(X,\dist,\haus^N)$.

Mean curvature bounds will be encoded into Laplacian bounds for distance functions. As it is well known, this is equivalent to the classical information about the vanishing mean curvature condition in the smooth setting, see \autoref{thm:smoothminimalimpldistsubh}. At the same time, this perspective allows for a meaningful formulation and analysis in our non smooth framework: switching to global Laplacian bounds, avoids the necessity of considering second order objects (like the mean curvature, the Laplacian of the distance, the Hessian of a function) on a prescribed codimension one hypersurface.
This is key, indeed, in our non-smooth framework, as second order objects are usually well defined $\meas$-a.e. and thus it can be quite tricky to work with them on a codimesion one hypersurface.

As we shall see, this way of formulating mean curvature bounds is also fine enough to allow for several extensions of classical results in Riemannian geometry to the synthetic framework. Here we focus on the beginning of a regularity theory, see \autoref{sec:regularity}, and on some direct geometric applications, see for instance \autoref{thm:generalizedFrankel} for a generalized version of the Frankel property. The extension to different notions of minimal hypersurfaces and their geometric applications are left to future investigation.\\
We mention that the Laplacian bounds on the distance function, in addition to encoding the vanishing of the mean curvature (i.e. a ``first variation-type'' information), also encode ``second variation-type'' information. Moreover, such ``second variation-type'' information is encoded not only at an infinitesimal level, but at a \textit{finite level}; see for example \autoref{prop:tubecomparison} where the case of equidistant surfaces is treated.
\medskip

Our treatment is inspired by \cite{CaffarelliCordoba93}, where a new approach to mean curvature bounds for perimeter minimizing sets was proposed by Caffarelli and Cordoba. Their strategy partially avoids the first variation formula (that was a fundamental tool in the previous approach due to De Giorgi \cite{DeGiorgi61}) and is inspired by the viscosity theory in PDEs, instead. Later on, the possibility of relying on this approach on non smooth spaces was suggested by Petrunin in \cite{Petrunin03}, with a sketch of proof of the L\'evy-Gromov isoperimetric inequality on Alexandrov spaces along similar lines.
\medskip


\subsection{Minimal boundaries and the Laplacian of the distance function}\label{subsec:mean}

The subject of our study will be sets of finite perimeter that locally minimize the perimeter, according to the following.

\begin{definition}\label{def:locper}
Let $(X,\dist,\meas)$ be an $\RCD(K,N)$ metric measure space and let $\Omega\subset X$ be an open domain. Let $E\subset X$ be a set of locally finite perimeter. We say that $E$ is \textit{locally perimeter minimizing in $\Omega$} if for any $x\in \Omega$ there exists $r_x>0$ such that $E$ minimizes the perimeter among all the perturbations that are compactly supported in $B_{r_x}(x)$, i.e., for any Borel set $F$ such that $E\Delta F\subset B_{r_x}(x)$ it holds
\begin{equation*}
\Per(E,B_{r_x}(x))\le \Per(F,B_{r_x}(x))\, .
\end{equation*}
\end{definition}

Let us notice that the above is a very general condition. For instance, smooth minimal hypersurfaces in Riemannian manifolds are locally boundaries of locally perimeter minimizing sets according to \autoref{def:locper}, even though, in general, they do not minimize the perimeter among arbitrarily compactly supported variations (a simple example in this regard is the equator inside the sphere).

\begin{theorem}\label{thm:meancurvminimal1}
Let $(X,\dist,\haus^N)$ be an $\RCD(K,N)$ metric measure space. Let $E\subset X$ be a set of locally finite perimeter and assume that it is a local perimeter minimizer. Let $\dist_{\overline{E}}:X\setminus \overline{E}\to[0,\infty)$ be the distance function from $\overline{E}$. Then
\begin{equation}\label{eq:meanglob}
\Delta \dist_{\overline{E}}\le \ft_{K,N} \circ \dist_{\overline{E}}\, \quad\text{on $X\setminus\overline{E}$}\, ,
\end{equation}
where $ \ft_{K,N}$ is defined in \eqref{eq:deftKN}.
If $\Omega\subset X$ is an open domain and $E\subset X$ is locally perimeter minimizing in $\Omega$, then setting
\begin{equation}\label{eq:defK}
{\mathcal K}:=\{x\in X\, :\, \exists\,  y\in \Omega\cap \partial E\, :\, \dist_{\overline{E}}(x)=\dist(x,y)\}\, ,
\end{equation}
it holds
\begin{equation}\label{eq:meangen}
\Delta \dist_{\overline{E}}\le \ft_{K,N}\circ \dist_{\overline{E}}\,\, \quad\text{on any open subset $\Omega'\Subset \left(X\setminus\overline{E}\right)\cap {\mathcal K}$}\, .
\end{equation}
\end{theorem}

As observed in \autoref{rem:SharpUBLapIntro}, the upper bound \eqref{eq:meanglob} is sharp already in the class of smooth Riemannian manifolds with Ricci curvature bounded below by $K\in \setR$ and dimension equal to $N\in \setN, N\geq 2$.

\begin{remark}[How to interpret the Laplacian bounds]\label{rem:LaplacianBouundsDist}
The Laplacian bounds \eqref{eq:meanglob} and \eqref{eq:meangen} have to be intended in any of the equivalent ways stated in \autoref{thm:mainequivlapla}. However let us mention that, if suitably interpreted, the Laplacian bounds \eqref{eq:meangen} hold more generally on the whole (possibly non-open, but measurable) set $\left(X\setminus\overline{E}\right)\cap {\mathcal K}$. Indeed, from  the general representation theorem for the Laplacian of $\dist_{\overline{E}}$ obtained in  \cite[Corollary 4.16]{CavallettiMondino20}, we know  that $\Delta \dist_{\overline{E}}$ is a Radon functional, meaning that its positive and negative parts $\left(\Delta \dist_{\overline{E}}\right)^{\pm}$ are Radon measures. Thus it makes sense to consider the restrictions $\left(\Delta \dist_{\overline{E}}\right)^{\pm} \res \left(X\setminus\overline{E}\right)\cap {\mathcal K}$, and set
$$
\Delta \dist_{\overline{E}}\res \left(X\setminus\overline{E}\right)\cap {\mathcal K}:=  \left(\Delta \dist_{\overline{E}}\right)^{+} \res \left(X\setminus\overline{E}\right)\cap {\mathcal K} - \left(\Delta \dist_{\overline{E}}\right)^{-} \res \left(X\setminus\overline{E}\right)\cap {\mathcal K}\, .
$$
The same arguments used below to show \eqref{eq:meangen}, actually show the stronger claim that 
\begin{align}
\left(\Delta \dist_{\overline{E}}\right)^{+} &\res \left(X\setminus\overline{E}\right)\cap {\mathcal K}\leq  \ft_{K,N}^{+}\circ \dist_{\overline{E}} \; \meas  \res \left(X\setminus\overline{E}\right)\cap {\mathcal K} \\
-\left(\Delta \dist_{\overline{E}}\right)^{-} &\res \left(X\setminus\overline{E}\right)\cap {\mathcal K}\leq - \ft_{K,N}^{-}\circ \dist_{\overline{E}} \; \meas  \res \left(X\setminus\overline{E}\right)\cap {\mathcal K}.
\end{align}
In this sense, the bound  \eqref{eq:meangen} holds on the whole set $\left(X\setminus\overline{E}\right)\cap {\mathcal K}$.
\end{remark}

\begin{proof}[Proof of \autoref{thm:meancurvminimal1}]
The proof follows the outline in \autoref{subsec:overview}. We shall focus on the case $K=0$, assuming that $E$ is bounded and locally perimeter minimizing in $X$. Minor adjustments that are required to cover the more general situation will be mentioned at the end of the proof.
\medskip

Let us recall the general strategy. Set $f:=\dist_{\overline{E}}$, we rely on the equivalence between Laplacian bounds in distributional and viscous sense and prove by contradiction that $\Delta f\le 0$ in viscous sense. If this is not the case, we find a function with strictly positive Laplacian supporting $f$ from below. Then we apply the Hopf-Lax semigroup to obtain a $1$-Lipschitz function $\phi$ which has still positive Laplacian and touches the distance to the boundary of $E$ at a footpoint $x_{E}$ of a minimizing geodesic. Then,  cutting along the level sets of $\phi$, we build inner perturbations of $E$, compactly supported in a small ball centred at the footpoint $x_{E}$. The strictly positive Laplacian assumption on $\phi$ yields that these perturbations decrease the perimeter, a contradiction.
\medskip

\textbf{Step 1.} Mild regularity properties of $E$.

Since $E$ is locally a quasi-minimizer of the perimeter (see \autoref{def:quasimin}), \autoref{thm:regqmin} and \autoref{cor:perest} apply. We assume $E$ to be normalized according to \eqref{eq:normsetsfin}. Hence, the essential boundary of $E$ is closed and it coincides with the topological boundary $\partial E$. Moreover, $E$ verifies the lower and upper measure bounds and the lower and upper perimeter bounds \eqref{eq:perboundqmin} at any point of its topological boundary. We shall also assume that $E\subset X$ is an open subset.
\medskip

\textbf{Step 2.} Globalization of Laplacian upper bound.
\\We claim that if every $z\in \partial E$ admits a small neighbourhood $U$ such that $\boldsymbol{\Delta} f\res (U\setminus \overline{E})\leq 0$, then the upper bound globalises to $\boldsymbol{\Delta} f\res (X\setminus \overline{E})\leq 0$.
\\Such a claim follows from  the general representation theorem for the Laplacian of distance functions obtained in \cite{CavallettiMondino20} via the localization technique, we next outline the argument. From  \cite[Corollary 4.16]{CavallettiMondino20}, we know that 
\[\boldsymbol{\Delta} f \res X\setminus \overline{E}= (\boldsymbol{\Delta} f)^{reg}\res X\setminus \overline{E}+ (\boldsymbol{\Delta} f)^{sing}\res X\setminus \overline{E}\, ,\]
 where the singular part $(\boldsymbol{\Delta} f)^{sing}\perp \haus^{N}$ satisfies $(\boldsymbol{\Delta} f)^{sing} \res X\setminus \overline{E}\leq 0$ and the regular part $(\boldsymbol{\Delta} f)^{reg} \ll \haus^{N}$ admits the representation formula
\begin{equation}\label{eq:repDeltafreg}
(\boldsymbol{\Delta} f)^{reg} \res X\setminus \overline{E} = (\log h_{\alpha})' \haus^{N} \res X\setminus \overline{E}\,.
\end{equation}
In \eqref{eq:repDeltafreg}, $Q$ is a suitable set of indices,  $(h_{\alpha})_{\alpha\in Q}$  are suitable densities defined on geodesics $(X_{\alpha})_{\alpha \in Q}$, which are essentially partitioning $X\setminus \overline{E}$ (in the smooth setting, $(X_{\alpha})_{\alpha\in Q}$ correspond to the integral curves of $\nabla \dist_{E}$; note that here we are using the reverse parametrization of $X_{\alpha}$ with respect to \cite{CavallettiMondino20}, hence the reversed sign in the right hand side of \eqref{eq:repDeltafreg}), such that the following disintegration formula holds: 
 \begin{equation}\label{eq:disintegration}
 \haus^{N}\res X\setminus \overline{E}=\int_{Q} h_{\alpha} \haus^{1} \res X_{\alpha}\,  {\mathfrak{q}}(\di \alpha)\,.
 \end{equation}
 The non-negative measure  $\mathfrak{q}$ in \eqref{eq:disintegration}, defined on the set of indices $Q$, is obtained in a natural way from the essential partition $(X_{\alpha})_{\alpha \in Q}$ of $X\setminus \overline{E}$, roughly by projecting $\haus^{N}\res X\setminus \overline{E}$  on the set $Q$ of equivalence classes (we refer to \cite{CavallettiMondino20} for the details).
 \\The key point for the proof of Step 2 is that each $h_{\alpha}$ is a $\CD(0,N)$ density over the ray $X_{\alpha}$ (see \cite[Theorem 3.6]{CavallettiMondino20}), implying that $\log(h_{\alpha})$ is concave and thus $(\log h_{\alpha})'$ is non-increasing (recall that the geodesic $X_{\alpha}$ is parametrized in terms of $\dist_{E}\res X_{\alpha}$, i.e  in the direction ``from  $\overline{E}$  towards $X\setminus \overline{E}$''). 
\\From the discussion above,  the claim now easily follows. Indeed, if every $z\in \partial E$ admits a small neighbourhood $U$ such that $\boldsymbol{\Delta} f\res (U\setminus \overline{E})\leq 0$, then in particular  $(\log h_{\alpha})'\leq 0$ on $(X_{\alpha}\cap U)\setminus \overline{E}$ and the concavity of $\log(h_{\alpha})$ along $X_{\alpha}$ implies that  $(\log h_{\alpha})'\leq 0$ on $X_{\alpha}\setminus \overline{E}$.  Thus \eqref{eq:repDeltafreg} yields $(\boldsymbol{\Delta} f)^{reg} \res X\setminus \overline{E}\leq 0$. We conclude recalling that the singular part  $(\boldsymbol{\Delta} f)^{sing}\res X\setminus \overline{E}$ is non-positive.
\medskip

\textbf{Step 3.} Construction of the auxiliary function $\phi$ and properties. 
\\Suppose by contradiction that $\Delta f\le 0$ does not hold on $X\setminus \overline{E}$. Then, by Step 2, there exist arbitrarily small neighbourhoods $U$ centred at points of $\partial E$ such that  $\Delta f\le 0$ does not hold on $U\setminus \overline{E}$. Moreover, from the equivalence \autoref{thm:viscoimpldistri}, the bound is not verified in the viscous sense. It follows that there exist $x\in U\setminus\overline{E}$, a ball $B_r(x)\subset U\setminus\overline{E}$ and a lower supporting function $\psi:B_r(x)\to\setR$ with the following properties:
\begin{itemize}
\item[(i)] $\psi\in D(\Delta, B_r(x))$ and $\Delta\psi$ is continuous on $B_r(x)$;
\item[(ii)] $\psi(x)=f(x)$;
\item[(iii)] $\psi(y)\le f(y)$ for any $y\in B_r(x)$;
\item[(iv)] $0<\Delta\psi(x)<1/2$.
\end{itemize}

We wish to modify $\psi$ into a globally defined function $\overline{\psi}:X\to\setR$, while keeping all its good properties.\\ 
By the continuity of $\Delta\psi$ and (iv), there exists $\eps>0$ such that $\eps<\Delta\psi< 3/4$ on a neighbourhood of $x$. Then we can consider a local Green type distance $b_x$, see \autoref{prop:Gdistance} (possibly in a smaller neighbourhood of $x$) and subtract a small multiple of $b_x^4$ to $\psi$, to obtain a function 
$$\hat{\psi}:=\psi-\delta b_x^4.$$ For $\delta>0$ sufficiently small, possibly on a smaller ball $B_s(x)\subset B_r(x)$, it holds that:
\begin{itemize}
\item[(i')] $\hat{\psi}:B_s(x)\to\setR$ is Lipschitz and $\hat{\psi}\in D(\Delta,B_s(x))$;
\item[(ii')] $\hat{\psi}(x)=f(x)$;
\item[(iii')] $\hat{\psi}(y)< f(y)$ for any $y\in B_s(x)$, $y\neq x$ and there exist $s'<s$ and $\delta'>0$ such that $\hat{\psi}<f-\delta'$ on $B_s(x)\setminus B_{s'}(x)$;
\item[(iv')] $0<\eps'<\Delta\hat{\psi}\le 1$ on $B_s(x)$, for some $\eps'>0$.
\end{itemize}
Next, we extend $\hat{\psi}$ to a global function $\overline{\psi}:X\to\setR$ such that: 
\begin{itemize}
\item[(i'')] $\overline{\psi}:X\to\setR$ is Lipschitz and $\overline{\psi}\in D(\Delta,B_s(x))$;
\item[(ii'')] $\overline{\psi}(x)=f(x)$;
\item[(iii'')] $\overline{\psi}(y)< f(y)$ for any $y\neq x$ and there exist $s'>0$ and $\delta'>0$ such that $\overline{\psi}<f-\delta'$ on $X\setminus B_{s'}(x)$;
\item[(iv'')] $0<\eps'<\Delta\overline{\psi}\le 1$ on $B_s(x)$, for some $\eps'>0$.
\end{itemize}

\noindent
Now, let us define $\phi:X\to\setR$ by
\begin{equation}\label{eq:defphi}
\phi(z):=\sup_{y\in X}\{\overline{\psi}(y)-\dist(z,y)\}\, .
\end{equation}
Observe that the supremum in \eqref{eq:defphi} is always finite. Moreover,
\begin{equation}\label{eq:phileqf1Lip}
 \text{$\phi$ is $1$-Lipschitz and $\phi\le f$.}
 \end{equation}
In order to check these properties, observe that $\overline{\psi}\le f$. Hence, for any $z\in X$,
\begin{equation*}
\phi(z)=\sup_{y\in X}\{\overline{\psi}(y)-\dist(z,y)\}\le \sup_{y\in X}\{f(y)-\dist(z,y)\}= f(z)\, .
\end{equation*} 
Therefore $\phi$ is finite and, being the supremum of a family of $1$-Lipschitz functions (the functions $z\mapsto \overline{\psi}(y)-\dist(z,y)$, indexed by $y\in X$), it is $1$-Lipschitz.
\medskip

Let now $x_E\in\partial E$ be any footpoint of minimizing geodesic from $x$ to $\overline{E}$. In particular, $f(x_E)=0$ and $f(x)-f(x_E)=\dist(x,x_E)$. Let $\gamma:[0,\dist(x,x_E)]\to X$ be a unit speed minimizing geodesic between $\gamma(0)=x_E$ and $\gamma(\dist(x,x_E))=x$. Observe that 
\begin{equation}\label{eq:fgammat}
f(\gamma(t))=t\, \quad \text{for any $t\in[0,\dist(x,x_E)]$}\, . 
\end{equation}
Moreover, 
\begin{equation}\label{eq:phigammat}
\phi(\gamma(t))=f(\gamma(t)), \quad \text{for any $t\in [0,\dist(x,x_E)]$}
\end{equation}
 and, for any such $t$, the supremum defining $\phi(\gamma(t))$ in \eqref{eq:defphi} is attained only at $x$.\\
Indeed, by (iii'') above, $\overline{\psi}<f-\delta'$ on $X\setminus B_{s'}(x)$. Hence, for any $z\in X$ such that $\phi(z)>f-\delta'$, we can restrict the supremum defining $\phi(z)$ in \eqref{eq:defphi} to $\overline{B_{s'}(x)}$. Since $\overline{B_{s'}(x)}$ is compact, the supremum is attained. In details, if $\phi(z)>f(z)-\delta'$, then 
\begin{equation}\label{eq:phizlefz}
\phi(z)= \sup_{y\in \overline{B_{s'}(x)}}\{\overline{\psi}(y)-\dist(y,z)\}=\overline{\psi}(y_z)-\dist(y_z,z)\le f(y_z)-\dist(y_z,z)\le f(z)\, ,
\end{equation}
for some $y_z\in \overline{B_{s'}(x)}$. In particular, whenever $\phi(z)=f(z)$, all the inequalities above become equalities. Hence $\overline{\psi}(y_z)=f(y_z)$, that implies $y_z=x$ by (ii'') and (iii''), and $f(z)-f(x)=-\dist(x,z)$. Viceversa, if $f(z)-f(x)=-\dist(x,z)$ then $\phi(z)=f(z)$ and the supremum defining $\phi(z)$ is attained (only) at $x$.
\medskip

We claim that
\begin{equation}\label{eq:nablaphi1}
\abs{\nabla \phi}=1, \quad \haus^N\text{-a.e. on } \{\phi>f-\delta\}\setminus B_{s'}(x).
\end{equation} 
In order to verify this claim, we let $z\in \{\phi>f-\delta\}\setminus B_{s'}(x)$. By the argument above, the supremum defining $\phi(z)$ is a maximum and it is attained at some $x_z\in \overline{B_{s'}(x)}$. By assumption $x_z\neq z$. Let us consider now a minimizing geodesic $\gamma:[0,\dist(z,x_z)]\to X$ connecting $z$ with $x_z$ and with unit speed. We claim that
\begin{equation}\label{eq:slope1}
\phi(\gamma(t))=\phi(z)+t\, ,\quad\text{for any $t\in[0,\dist(z,x_z)]$}\, .
\end{equation}
The inequality $\phi(\gamma(t))\le \phi(z)+t$ follows from the fact that $\phi$ is $1$-Lipschitz. We only need to prove that $\phi(\gamma(t))\ge\phi(z)+t$. To this aim,  observe that
\begin{align*}
\phi(\gamma(t))=&\sup_{y\in X}\{\overline{\psi}(y)-\dist(y,\gamma(t))\}\\
\ge& \overline{\psi}(x_z)-\dist(\gamma(t),x_z)\\
=&\overline{\psi}(x_z)-\dist(z,x_z)+t\\
=&\phi(z)+t\, .
\end{align*}
From \eqref{eq:slope1} we infer that, for any $z\in \{\phi>f-\delta\}\setminus B_{s'}(x)$, the function $\phi$ has slope $1$ at $z$. The conclusion that $\abs{\nabla\phi}=1$-a.e. on $\{\phi>f-\delta\}\setminus B_{s'}(x)$ follows from the a.e. identification between slope and upper gradient obtained in \cite{Cheeger99}.
\medskip

Let us consider the Laplacian of $\phi$. By construction, $\overline{\psi}$ verifies the Laplacian bound (iv'') on $B_{s'}(x)$. In particular, $\Delta \overline{\psi}\ge \eps>0$ on $B_{s'}(x)$ in the sense of \autoref{def:heatlaplbounds}. Hence, since we already observed that for points $z\in \{\phi>f-\delta\}\setminus B_{s'}(x)$ the supremum defining $\phi(z)$ is a maximum attained in $\overline{B_{s'}(x)}$, we obtain by \autoref{cor:HLdistr} that
\begin{equation}\label{eq:Deltaphige}
\boldsymbol{\Delta}\phi\ge \eps\, \quad\text{on $\{\phi>f-\delta\} \setminus B_{s'}(x)$}\, ,
\end{equation} 
in the sense of distributions.
\medskip

\textbf{Step 4.} Construction of the inner variations of $E$.\\
Our next goal is to construct a suitable inner variation of $E$, compactly supported in a small ball centred at a point of $\partial E$. Such a perturbation is obtained by cutting along a level set of $\phi$, with value $-\delta< t<0$. In Step 5, we will reach a contradiction by showing that such an inner  perturbation has perimeter strictly less than $E$.
\medskip

Let us start by proving that for small values of $t\in(-\delta,0)$, we can cut $E$ along a level set of $\phi$ to obtain inner perturbations $E_t\subset E$, supported on suitable balls of arbitrary small radius.\\
Let us define
\begin{equation*}
E_t:=E\setminus \{\phi >t\}\, .
\end{equation*}
Observe that for $t=0$ it holds $\{\phi>0\}\cap E=\emptyset$, since from \eqref{eq:phileqf1Lip} we know that $\{\phi>0\}\subset \{f>0\}\subset X\setminus E$. When we decrease the value of $t$, the super-level set $ \{\phi >t\}$ starts cutting $E$.\\ 
Recall that $x_E\in \partial E$ is a footpoint of minimizing geodesic from $x$ to $\overline{E}$. We claim that for any $t<0$ sufficiently close to $0$, $E_t$ is a perturbation of $E$ supported in a small ball $B_{r}(x_{E})$, i.e. $\{\phi>t\}\cap E\subset B_r(x_E)$. To prove this claim, it is enough to observe that  from $f\equiv 0$ on $E$,  \eqref{eq:slope1},  and $B_{s'}(x)\subset X\setminus \bar{E}$,  we get
\begin{equation}\label{eq:phi>tinclusion}
\{\phi>t\}\cap E \subset \{\phi>f-\delta\}\setminus \overline{B_{s'}(x)}  \quad \text{for any } t\in (-\delta,0)\, .
\end{equation}
Moreover, for every $z\in \{\phi>t\}\cap E$,  the maximum defining $\phi(z)$ is attained inside $\overline{B_{s'}(x)}$, see \eqref{eq:phizlefz} and the nearby discussion.\\ 
Now we wish to bound the distance from $x_E$ to $\{\phi>t\}\cap E$. For any $z\in\{\phi>t\}\cap E$, there exists $x_z\in \overline{B_{s'}(x)}$ such that
\begin{equation*}
\phi(z)=\overline{\psi}(x_z)-\dist(x_z,z)\le f(x_z)-\dist(x_z,z)\le s'+\dist(x,\overline{E})-\dist(x_z,z)\, .
\end{equation*}
Hence 
\begin{equation*}
\dist(x_z,z)\le \dist(x,\overline{E})+s'-\phi(z)\le \dist(x,\overline{E})+s'-t\, .
\end{equation*}
In particular, we can bound the distance of $\{\phi>t\}\cap E$ from $x$, and hence from $x_E$, and obtain
\begin{equation}\label{eq:distEtxE}
\{\phi>t\}\cap E \subset B_{r}(x_{E}), \quad r:= 2 \dist(x, \overline{E})+s'+\delta\, .
\end{equation}
Recalling that $x$ can be chosen arbitrarily close to $\overline{E}$ (see beginning of Step 3), and that $s', \delta>0$ can be chosen arbitrarily small (see Step 3), we infer that $r:= 2 \dist(x, \overline{E})+s'+\delta$ can be chosen arbitrarily small. It follows that, for every $r>0$ arbitrarily small, one can perform the above construction in order to obtain $x_{E}\in \partial E$  and a family of  inner perturbations $(E_{t})_{t\in (-\delta, 0)}$ of $E$,  so that $E\setminus E_{t}\subset B_{r}(x_{E})$.

 Observe also that $E_t$ is a non trivial perturbation of $E$, i.e. $\haus^N(\{\phi>t\}\cap E)>0$. Indeed  from \eqref{eq:slope1} it is easily seen that  $\{\phi>t\}\cap E$ is non-empty and moreover it is open. Using \eqref{eq:slope1} it is also readily seen that the inclusion ``$\subset$'' in \eqref{eq:phi>tinclusion} can be improved to the compact inclusion ``$\Subset$''.
 \medskip

Thus, from the combination of \eqref{eq:fgammat}, \eqref{eq:phigammat}, \eqref{eq:nablaphi1}, \eqref{eq:Deltaphige} and \eqref{eq:phi>tinclusion},  $\phi$ verifies the assumptions of \autoref{prop:leveld} for some open subset $\Omega'' \subset X$ satisfying (note that $\Omega''$ plays the role of $\Omega$ in \autoref{prop:leveld}) 
\begin{equation}\label{eq:OmegaOmega'}
\{\phi>t\}\cap E  \Subset \Omega'' \Subset  \{\phi>f-\delta\} \setminus \overline{B_{s'}(x)}=:\Omega'\, .
\end{equation}
Hence, for $t\in (-\delta, 0)$, $E_t$ is a compactly supported inner perturbation of $E$ with finite perimeter and
\begin{equation}\label{eq:nablaphinu}
\left(\nabla\phi\cdot\nu_{\{\phi>t\}}\right)_{\mathrm{int}}=\left(\nabla\phi\cdot\nu_{\{\phi>t\}}\right)_{\mathrm{ext}}=1\, ,\quad \Per_{\{\phi>t\}}\text{-a.e. on $\Omega''$}\, . 
\end{equation}

\textbf{Step 5.} Estimate for the perimeter.\\
We aim to prove that there exists $t<0$, with $|t|$ small enough, such that 
\begin{equation}\label{eq:ContrEnonMin}
\Per(E, B_r(x_E))-\Per(E_t, B_r(x_E))>0,
\end{equation}
contradicting the local inner minimality of $E$.
Let 
\[F:=E\cap\{\phi>t\}\, = E\setminus E_{t}\, .\]
Neglecting the regularity issues, the boundary of $F$ has two components. The first one is along $\partial E$, with unit normal coinciding with the unit normal of $\partial E$. The second one is along the level set $\{\phi=t\}$, where the unit normal vector $\nu_{F}$ pointing inside of $F$ is $\nabla \phi$.  To make rigorous this description we rely on \autoref{thm:cutandpaste}, together with the remark that the boundaries of $\{\phi>t\}$ and $E$ have negligible intersections for a.e. $t\in (-\delta,0)$, for $\delta>0$ sufficiently small.
Let $\chi$ be a smooth cutoff function (see \autoref{lemma:cutoff}) with $\chi\equiv 1$ on a neighbourhood of $F$ and $\chi\equiv 0$ on $X\setminus \big(\{\phi>f-\delta\}\setminus B_{s'}(x) \big)$.  Notice that $\chi \nabla \phi\in\mathcal{DM}^{\infty}(X)$, by \eqref{eq:Deltaphige}. 
We can thus apply \autoref{thm:GaussGreenRCDBCM}, with test function $f\equiv 1$, vector field $V= \chi \nabla \phi$ and set of finite perimeter $F$, to obtain 
\begin{align*}
\int_{F^{(1)}}\boldsymbol{\Delta}\phi=& -\int_{\mathcal{F}F}\left(\nabla\phi\cdot\nu_F\right)_{\mathrm{int}}\di\Per\\
=&  - \int_{\mathcal{F}\{\phi>t\} \cap E^{(1)}}(\nabla \phi\cdot \nu_{\{\phi>t\}})_{\mathrm{int}}\di\Per \\ 
&-\int_{\mathcal{F}E\cap \{\phi>t\}^{(1)}}\left(\nabla\phi\cdot\nu_E\right)_{\mathrm{int}}\di\Per\\
=&-\Per\big(\mathcal{F}\{\phi>t\} \cap E^{(1)}\big)  -\int_{\mathcal{F}E\cap \{\phi>t\}^{(1)}}\big(\nabla\phi\cdot\nu_E\big)_{\mathrm{int}}\di\Per\\
\le& -\Per\big(\mathcal{F}\{\phi>t\} \cap E^{(1)}\big)+\Per\big(\mathcal{F}E\cap \{\phi>t\}^{(1)}\big)\, ,
\end{align*}
where the third equality follows from \autoref{prop:leveld} (see \eqref{eq:OmegaOmega'} and \eqref{eq:nablaphinu}), while the inequality follows from the sharp trace bound $\abs{\left(\nabla\phi\cdot\nu_E\right)_{\mathrm{int}}}\le 1$ in \eqref{eq:inftyboundtraceintsharp}.

Since $\boldsymbol{\Delta}\phi>\eps$ on a neighbourhood of $F$ by \eqref{eq:Deltaphige} and \eqref{eq:OmegaOmega'}, we get 
\begin{equation}\label{eq:estgap}
 -\Per\big(\mathcal{F}\{\phi>t\}\cap E^{(1)}\big)  +\Per\big(\mathcal{F}E\cap \{\phi>t\}^{(1)}\big)  >0\, .
\end{equation}
Combining \autoref{thm:cutandpaste} with \eqref{eq:distEtxE} and \eqref{eq:estgap}, we get the desired \eqref{eq:ContrEnonMin}:
\begin{align*}
\Per(E,B_r(x_E))-&\Per(E_t, B_r(x_E))=\\
&\Per\big(\mathcal{F}E\cap \{\phi>t\}^{(1)}\big)
-\Per\big(\mathcal{F}\{\phi>t\}\cap E^{(1)}\big)>0\, .
\end{align*}
\medskip

\textbf{Step 6.}  Adjustments to cover the case of a general lower Ricci curvature bound $K\in\setR$.\\ 
 In Step 2, the density $h_{\alpha}$ on $X_{\alpha}$ is a $\CD(K,N)$ density, yielding that $\log h_{\alpha}$ is semi-concave (thus locally Lipschitz and twice differentiable except at most at countably many points) and satisfies the differential inequality $ (\log h_{\alpha})''\leq -K  $ in the distributional sense and point-wise except countably many points. The singular part of $\Delta \dist_{\overline{E}}\res X\setminus \overline{E}$ is non-positive regardless of the value of $K\in \setR$. One can then argue along the lines of Step 2 to globalize the bound $\Delta \dist_{\overline E}\leq -K  \dist_{\overline E}$.
\\In Step 3, since in the contradiction argument we start from the assumption that \eqref{eq:meanglob} does not hold, arguing as before we can find an auxiliary function $\psi$ with properties (i) to (iii) and such that
\begin{equation*}
\Delta\psi(x)>-K\dist_{\overline{E}}(x)\, ,
\end{equation*}
that replaces the condition $\Delta\psi(x)>0$ that we found in the case $K=0$.
\\The construction of the functions $\hat{\psi}$ and $\overline{\psi}$ requires no modification, besides the natural ones for conditions (iv') and (iv''). Then, when building the function $\phi$ by duality as in \eqref{eq:defphi}, we only need to apply the general \autoref{thm:HLlapla} to infer that 
\begin{equation*}
\boldsymbol{\Delta}\phi\ge \eps\, \quad\text{on $\{\phi>f-\delta\}\setminus B_{s'}(x)$}\, ,
\end{equation*}
also in this case. Basically, whenever $K<0$, the argument by contradiction starts with a supporting function whose Laplacian is more positive than when $K=0$. This compensates the fact that the Hopf-Lax semigroup might decrease the lower Laplacian bound, though it does it only in a controlled way.

Notice that the bound  $\Delta \dist_{\overline E}\leq -K  \dist_{\overline E}$ is sharp in the $N=\infty$ case. The sharp dimensional bound can be obtained by the following self-improving argument. 
\\By the first part of Step 6 (see also Step 2), we know that $h_\alpha$ is a $\CD(K,N)$ density on the ray $X_\alpha$ for $\mathfrak{q}$-a.e. $\alpha\in Q$, i.e. it satisfies 
\begin{equation}\label{eq:DiffEqCDKN}
(\log h_{\alpha})''\leq - K - \frac{1}{N-1} \big( (\log h_{\alpha})' \big)^{2}
\end{equation}
in the sense of distributions and point-wise except countably many points. Moreover, from \eqref{eq:repDeltafreg} and the first part of Step 6, we know that 
\begin{equation}\label{eq:halpha'Kdist}
(\log h_{\alpha})'(\dist_{\overline{E}}) \leq - K \, \dist_{\overline{E}} \, \text{ on $X_{\alpha}$, for $\mathfrak{q}$-a.e. $\alpha\in Q$.}
\end{equation}
Observing that the function $\ft_{K,N}$ defined in \eqref{eq:deftKN} satisfies the following initial value problem
\begin{equation*}
\begin{cases}
\ft_{K,N}'(x)&= - K - \frac{1}{N-1} \big( \ft_{K,N}(x) \big)^{2} \\
\ft_{K,N}'(0)&= 0 
\end{cases}
\end{equation*}
on $I_{K,N}$, a standard argument via differential inequalities (using \eqref{eq:DiffEqCDKN} and \eqref{eq:halpha'Kdist}) implies that   
\[(\log h_{\alpha})' \circ \dist_{E}\leq \ft_{K,N}\circ \dist_{E}\, ,\quad \text{ for  $\mathfrak{q}$-a.e. $\alpha\in Q$.} \]
Recalling the representation formula  \eqref{eq:repDeltafreg} and that the singular part of $\Delta \dist_{\overline{E}}\res X\setminus \overline{E}$ is non-positive,  we infer that $\Delta   \dist_{\overline E} \leq \ft_{K,N}\circ \dist_{\overline E}$.
\medskip

\textbf{Step 7.}  Adjustments in case $E$ is locally perimeter minimizing in $\Omega$, i.e. proof of \eqref{eq:meangen}.\\ 
The key observation is the following: if the Laplacian bound \eqref{eq:meangen} holds in a neighbourhood of $\partial E \cap \Omega$, then it holds on  $\left(X\setminus\overline{E}\right)\cap {\mathcal K}$.  This can be proved along the lines of Step 2, since all the rays essentially partitioning $\left(X\setminus\overline{E}\right)\cap {\mathcal K}$ start from $\partial E \cap \Omega$: if we assume that the correct Laplacian bound holds in a neighbourhood of  $\partial E \cap \Omega$, then the bound holds globally on  $\left(X\setminus\overline{E}\right)\cap {\mathcal K}$ by one dimensional considerations along each ray and by the fact that the singular part of  $\Delta \dist_{\overline{E}}\res X\setminus \overline{E}$ is non-positive. One can then follow verbatim the previous argument by contradiction.
\end{proof}

For the sake of the applications it will be useful to understand the regularity of the distance function from $\partial E$ without the necessity of avoiding $\partial E$. Thanks to \autoref{thm:meancurvminimal1} we can prove that $\dist_{\partial E}$ has measure valued Laplacian and that its singular contribution along $\partial E$ is the surface measure of $\partial E$. 
\begin{proposition}\label{prop:fullrepr}
Let $(X,\dist,\haus^N)$ be an $\RCD(K,N)$ metric measure space and let $E\subset X$ be a set of locally finite perimeter. Assume that $E$ is locally perimeter minimizing inside an open domain $\Omega\subset X$, according to \autoref{def:locper} and let $\Omega'\Subset\Omega$. Then $\dist_{\overline{E}}:X\to[0,\infty)$ has locally measure valued Laplacian in a neighbourhood $U$ of $\partial E\cap \Omega'$. Moreover, the following representation formula holds:

\begin{equation}\label{eq:reprglob}
\boldsymbol{\Delta}\dist_{\overline{E}}=\haus^{N-1}\res\partial E+\boldsymbol{\Delta}\dist_{\overline{E}}\res(X\setminus \overline{E})\,, \quad \text{ on $U\supset \partial E\cap \Omega'$} \, .
\end{equation}
\end{proposition}

\begin{proof}
The proof relies on the following steps: first we will argue that $\dist_{\overline{E}}$ has locally measure valued Laplacian, relying on \autoref{thm:meancurvminimal1} and on the volume bound for the tubular neighbourhood of $\partial E$ in \autoref{lemma:tubneighbounds}. Then we observe that the Laplacian of $\dist_{\overline{E}}$ is absolutely continuous w.r.t. $\haus^{N-1}$. The sought representation formula follows by computing the density of $\boldsymbol{\Delta}\dist_{\overline{E}}\res\partial E$ w.r.t. $\haus^{N-1}\res\partial E$ via a blow-up argument. The strategy is inspired by the proofs of \cite[Lemma 7.5 and Theorem 7.4]{BrueNaberSemola20}, dealing with the Laplacian of the distance from the boundary on noncollapsed $\RCD$ spaces.
\medskip

\textbf{Step 1.} Our goal is to find a locally finite measure $\nu$ such that
\begin{equation*}
\int_X\nabla \phi\cdot\nabla\dist_{\overline{E}}\, \di\haus^N=-\int_X\phi\di\nu\, ,
\end{equation*}
for any Lipschitz function $\phi:X\to\setR$ with compact support.

Let us assume for simplicity that $\partial E$ is compact, the general case can be handled with an additional cut-off argument.
	
By the coarea formula \autoref{thm:coarea}, for almost every $r>0$, the superlevel set $\{\dist_{\overline{E}}>r\}$ has finite perimeter. Moreover, the volume bound for the tubular neighbourhood of the boundary 
		\begin{equation*}
		\haus^N(\{0\le \dist_{\overline{E}}<r\})\le Cr\, ,
		\end{equation*}
	that follows from \autoref{lemma:tubneighbounds},
		together with a further application of the coarea formula, yield the existence of a sequence $(r_i)$ with $r_i\downarrow 0$ as $i\to\infty$ such that 
		\begin{equation}\label{eq:boundper}
		\Per (\{\dist_{\overline{E}}>r_i\})\le C\;\;\;\text{for any $i\in\setN$}\, .
		\end{equation}

		Since $\dist_{\overline{E}}$ has measure valued Laplacian on $X\setminus\overline{E}=\{\dist_{\overline{E}}>0\}$, the bounded vector field $\nabla \dist_{\overline{E}}$ has measure valued divergence on the same domain. Therefore, applying \autoref{thm:GaussGreenRCDBCM} to the vector field $\phi\nabla\dist_{\overline{E}}$ on the domain $\{\dist_{\overline{E}}>r_i\}$ we infer that
		\begin{align}\label{eq:approxgg}
		\nonumber \int_{\{\dist_{\overline{E}}>r_i\}}&\nabla\phi\cdot\nabla\dist_{\overline{E}}\di\haus^N\\
		&=-\int_{\{\dist_{\overline{E}}>r_i\}}\phi\di\boldsymbol{\Delta}\dist_{\overline{E}}-\int_X\phi f_i\di\Per(\{\dist_{\overline{E}}>r_i\})\, ,
		\end{align} 
		for some Borel functions $f_i$ verifying 
		\begin{equation}\label{eq:boundf}
		\norm{f_i}_{L^{\infty}(\Per(\{\dist_{\overline{E}}>r_i\}))}\le 1\, .
		\end{equation}
		Thanks to \eqref{eq:boundper} and \eqref{eq:boundf}, up to extracting a subsequence, the measures $f_i\Per(\{\dist_{\overline{E}}>r_i\})$ weakly converge to a finite measure $\mu$ on $X$ in duality with continuous functions. Passing to the limit in \eqref{eq:approxgg} as $i\to\infty$, we get 
		\begin{equation}\label{eq:intlim}
		\int_X\nabla\phi\cdot\nabla\dist_{\overline{E}}\di\haus^N= -\lim_{r_i\to 0}\int_{\{\dist_{\overline{E}}>r_i\}}\phi\di\boldsymbol{\Delta}\dist_{\overline{E}} -\int_X\phi\di\mu\, ,
		\end{equation}
		as we claimed.

The next observation is that the first term at the right hand side in \eqref{eq:intlim} above is a linear function with sign (when $K=0$, otherwise there is a correction term), therefore it is represented by a measure.\\
Indeed, combining \eqref{eq:intlim} with \autoref{thm:meancurvminimal1}, we have 
\begin{equation*}
	\int_X \nabla\phi\cdot\nabla\dist_{\overline{E}}\di\haus^N \ge K\int_X\phi\, \dist_{\overline{E}}\, \di\haus^N -\int_X\phi\, \di\mu \, ,
	\end{equation*}
	for any $\phi \in \Lip_c(X)$ s.t. $\phi\ge 0$.\\
	In particular $\phi \mapsto \int \nabla\phi\cdot\nabla\dist_{\overline{E}}\di\haus^N + \int\phi\di\mu - K\int\phi\, \dist_{\overline{E}}\, \di\haus^N$ is a non-negative linear map. 
Hence there exists a non-negative locally finite measure $\eta$ such that
	\[
	\int_X \nabla\phi\cdot\nabla\dist_{\overline{E}}\di\haus^N + \int_X\phi\di\mu -K\int_X\phi\, \dist_{\overline{E}}\di\haus^N =  \int_X \phi \di \eta\, ,
	\]
for any $\phi \in \Lip_c(X)$.
		This implies that $\dist_{\overline{E}}$ has measure valued Laplacian on $X$.
\medskip

\textbf{Step 2.} Thanks to \autoref{cor:divac}, we have that $\boldsymbol{\Delta}\dist_{\overline{E}}\ll\haus^{N-1}$.

To check that $\boldsymbol{\Delta}\dist_{\overline{E}}\res\partial E=\haus^{N-1}\res\partial E$, by standard differentiation of measures (recall that in general the perimeter measure of any set of finite perimeter is asymptotically doubling, therefore the differentiation theorem applies), it suffices to prove that
\begin{equation}\label{eq:density}
\lim_{r\downarrow 0}\frac{\boldsymbol{\Delta}\dist_{\overline{E}}(B_r(x))}{\Per(E,B_r(x))}=1\, ,\quad\text{for $\Per$-a.e. $x\in\partial E$}\, .
\end{equation} 
The validity of \eqref{eq:density} can be proved thanks to
\autoref{thm:closurecompactnesstheorem}. Indeed, it is sufficient to prove that the density estimate holds at regular boundary points of $E$, i.e. those points where the blow-up is a Euclidean half-space $\mathbb{H}^N\subset \setR^N$.

Under this assumption, along the sequence $X_i:=(X,\dist/r_i,\haus^N/r_i^N,x,E)$ of scaled spaces converging to the blow-up, the sets $E\subset X$ converge in $L^1_{\loc}$ to $\mathbb{H}^N$. By \autoref{thm:closurecompactnesstheorem} the convergence can be stenghtned to Kuratowski convergence of $\partial E_i\subset X_i$ to $\partial \mathbb{H}^N$, which implies in turn the uniform convergence of $\dist_{\overline{E}}:X_i\to\setR$ to $\dist_{\mathbb{H}^N}$. Moreover, this is easily seen to imply the $H^{1,2}_{\loc}$ convergence of $\dist_{\overline{E}}:X_i\to\setR$ to $\dist_{\mathbb{H}^N}$. Then the distributional Laplacians of $\dist_{\overline{E}}$ weakly converge as measures to the distributional Laplacian of $\dist_{\mathbb{H}^N}$, and \eqref{eq:density} follows from the standard properties of weak convergence. 
\end{proof}

Up to now, we have studied the properties of the distance function from a locally perimeter minimizing set, \textit{outside of the set}. An inspection of the proof of \autoref{thm:meancurvminimal1} shows that we actually relied only on inner perturbations of the set $E$ to obtain properties of the Laplacian of the distance from $E$ outside of $E$.\\ 
As it is natural to expect, exploiting the full local minimality condition, we obtain sharper statements about the distance (and the signed distance) function from $\partial E$ on both sides of $E$, whenever $E$ is locally perimeter minimizing. Recall also that if $E\subset X$ is a set of finite perimeter, locally minimizing the perimeter functional, we can (and will) assume that $E$ is open (up to choosing the suitable a.e. representative).

\begin{theorem}\label{thm:meancurvminimal2}
Let $(X,\dist,\haus^N)$ be an $\RCD(K,N)$ metric measure space. Let $E\subset X$ be a set of locally finite perimeter and suppose that it is locally perimeter minimizing inside an open domain $\Omega\subset X$, according to \autoref{def:locper}. Let $\dist_{\partial E}:X\to\setR$ be the distance function from the boundary of $E$. Then $\dist_{\partial E}$ has locally measure valued Laplacian on $X$. Moreover, for any open subset $\Omega'\Subset \cK$ (where  $\cK$ was defined in \eqref{eq:defK}), it holds:
\begin{equation}\label{eq:DeltaddeE}
\boldsymbol{\Delta}\dist_{\partial E}  \le  \ft_{K,N} \circ \dist_{\partial E}\, \quad\text{on  $E\cap  \Omega'$}\, , \quad
\boldsymbol{\Delta}\dist_{\partial E}  \le   \ft_{K,N} \circ \dist_{\partial E}\, \quad\text{on $\big(X\setminus \overline{E}\big)\cap\Omega'$}\, ,
\end{equation}
where  $ \ft_{K,N}$ was defined in \eqref{eq:deftKN}. Moreover, 
\begin{equation}\label{eq:DeltadEdeE}
\boldsymbol{\Delta}\dist_{\partial E}\res  \big(\partial E\cap  \Omega'\big)=\haus^{N-1}\res \big(\partial E\cap  \Omega'\big)\, .
\end{equation}
Under the same assumptions, denoting by $\dist^s_{E}$ the signed distance function from $E$ (with the convention that it is positive outside of $E$ and negative inside), $\dist^s_E$ has measure valued Laplacian on $E\cap \Omega'$ and 
\begin{equation}\label{eq:laplsign}
\boldsymbol{\Delta}\dist^s_{ E}\ge  \ft_{K,N} \circ  \dist^s_{E}\, \quad\text{on $E\cap \Omega'$}\, , \quad \boldsymbol{\Delta}\dist^s_{E}\le  \ft_{K,N} \circ \dist^s_{ E}\, \quad\text{on $\left(X\setminus \overline{E}\right)\cap \Omega'$}\, ,
\end{equation}
and
\begin{equation}\label{eq:signalong}
\boldsymbol{\Delta}\dist^s_{E}\res \left(\partial E\cap \Omega'\right)=0\, .
\end{equation}

\end{theorem}

\begin{remark}\label{rem:InterLapBoundFinal}
With the same caveat about the interpretation of the Laplacian bounds when restricted to a measurable (possibly non-open) set  as in \autoref{rem:LaplacianBouundsDist}, the Laplacian bounds \eqref{eq:DeltaddeE}, \eqref{eq:DeltadEdeE}, \eqref{eq:laplsign} and \eqref{eq:signalong} actually hold more strongly by replacing $\Omega'$ with $\cK$.
\end{remark}

\begin{proof}
The first part of the statement follows from \autoref{thm:meancurvminimal1} and \autoref{prop:fullrepr}, applied to the distance from $\overline{E}$ and to the distance from $\overline{X\setminus E}$. Notice indeed that, under our assumptions on $E$, also $X\setminus E$ is locally perimeter minimizing inside $\Omega$.
\medskip

To deal with the signed distance function $\dist^s_E$, notice that it coincides with $\dist_{\partial E}$ on $\left(X\setminus \overline{E}\right)\cap K$ and with $-\dist_{\partial E}$ on $E\cap K$. Then, arguing as in the proof of \autoref{prop:fullrepr}, it is possible to prove that $\dist^s_E$ has measure valued Laplacian and \eqref{eq:laplsign} follows. 

To determine the restriction of $\boldsymbol{\Delta}\dist^s_E$ to $\partial E$, it is enough to adjust the argument in Step 2 of the proof of \autoref{prop:fullrepr}. The key remark is that, when blowing up, the distance function from the boundary converges to the distance function from the half-space, whose distributional Laplacian has a singular contribution given by the surface measure of the hyperplane. The signed distance function, instead, converges to the signed distance function from the half-space after blowing up, which is a coordinate function, hence in particular it is harmonic. This shows, through the density estimate via blow-up, that \eqref{eq:signalong} holds. 
\end{proof}

The range of applications of \autoref{thm:meancurvminimal1} and \autoref{thm:meancurvminimal2} is expected to be broad. For the sake of illustration, here we present an extension of a celebrated property of minimal surfaces in manifolds with positive Ricci curvature, the so-called Frankel's theorem.  As another application, in \autoref{sec:regularity} we will investigate some consequences of the mean curvature bounds at the level of regularity.
\medskip

It is a classical fact that two smooth minimal hypersurfaces in a manifold with (strictly) positive Ricci curvature must intersect each other. This is known as Frankel's theorem after \cite{Frankel66}, where similar results were obtained under the stronger assumption of  positive sectional curvature. In the present formulation the statement appears in \cite{PetersenWilhelm03}, whose proof we can now follow, given our understanding of mean curvature bounds for locally perimeter minimizing sets on $\RCD$ spaces, after \autoref{thm:meancurvminimal1} and \autoref{thm:meancurvminimal2}.

\begin{theorem}[Generalized Frankel's Theorem]\label{thm:generalizedFrankel}
Let $(X,\dist,\haus^N)$ be an $\RCD(N-1,N)$ metric measure space. Let $\Sigma_1,\Sigma_2\subset X$ be closed sets such that, for any $i=1,2$ and any $x\in\Sigma_i$, there exist a ball $B_r(x)$ and a set of finite perimeter $E\subset X$ such that $E$ is locally perimeter minimizing in $B_{2r}(x)$ and $\Sigma_i\cap B_r(x)=\partial E\cap B_r(x)$. Then 
\begin{equation*}
\Sigma_1\cap \Sigma_2\neq \emptyset\, .
\end{equation*}
\end{theorem}

\begin{proof}
Let $\dist_1$ and $\dist_2$ denote $\dist_{\Sigma_1}$ and $\dist_{\Sigma_2}$ respectively and let $\bar{\dist}:=\dist_1+\dist_2$.\\
Assume by contradiction that $\Sigma_1\cap \Sigma_2=\emptyset$. Then it is easily seen that $\bar{\dist}$ attains one of its minima at a point $x\in X\setminus (\Sigma_1\cup\Sigma_2)$. Indeed it is sufficient to consider a minimizing geodesic between $\Sigma_1$ and $\Sigma_2$ whose length is $\dist(\Sigma_1,\Sigma_2)>0$ and pick a point inside it. 

By \autoref{thm:meancurvminimal1}, 
\begin{equation*}
\boldsymbol{\Delta}\dist_1\le -(N-1)\dist_1\, ,\quad\text{and}\quad \boldsymbol{\Delta}\dist_2\le -(N-1)\dist_2\, ,\quad\text{on $X\setminus (\Sigma_1\cup\Sigma_2)$}\, .
\end{equation*} 
Hence
\begin{equation}\label{eq:supha}
 \boldsymbol{\Delta}\bar{\dist}\le -(N-1)\bar{\dist}\, ,\quad\text{on $X\setminus (\Sigma_1\cup\Sigma_2)$}\, .
\end{equation}
In particular, there is a neighbourhood $U$ of $x$ such that $\bar{\dist}$ is superharmonic on $U$ and attains a minimum at the interior point $x$. The strong maximum principle implies that $\bar{\dist}$ is constant in a neighbourhood of $x$, that contradicts the strict superharmonicity of $\bar{\dist}$ in \eqref{eq:supha}, since $\bar{\dist}(x)>0$.
\end{proof}

\begin{remark}
The assumptions of \autoref{thm:generalizedFrankel} cover in particular the classical case of smooth minimal hypersurfaces in closed manifolds with positive Ricci curvature. Indeed, as we already mentioned, smooth minimal hypersurfaces are, locally, perimeter minimizing boundaries.
\end{remark}

\section{Regularity theory}\label{sec:regularity}

This section is dedicated to the partial regularity theory for minimal boundaries on non collapsed $\RCD$ spaces. Our main result will be that they are topologically regular away from sets of ambient codimension three, and from the boundary of the space. Besides from a sharp Hausdorff dimension estimate (see \autoref{thm:sharpdimsing}), we will obtain also a Minkowski estimate for the quantitative singular set (see \autoref{thm:contestsing}). 
Following a classical pattern, these results will be achieved through two intermediate steps: 
\begin{itemize}
\item an $\eps$-regularity result, \autoref{thm:epsregularity} showing that under certain assumptions at a given location and scale a minimal boundary is topologically regular;
\item the analysis dedicated to guarantee that the assumptions of the $\eps$-regularity theorem are verified at many locations and scales along the minimal boundary. This is pursued as follows:
\begin{itemize}
\item in  \autoref{subsec:partreg}, via dimension reduction arguments, we prove  sharp Hausdorff dimension estimates  of the singular set (see \autoref{thm:sharpdimsing}). Here the arguments depart from the classical ones: in the Euclidean (resp. smooth) setting, minimal boundaries satisfy a very powerful monotonicity formula (resp. up to a lower order term) which implies that every tangent space to a minimal boundary is a cone. In the present non-smooth setting, it seems not possible to repeat the Euclidean/smooth computations and  it is not clear if such a (perturbed) monotonicity formula holds;
\item  in \autoref{subsec:mono} we prove sharp perimeter bounds for  the equidistant sets from locally minimal boundaries which will be used in \autoref{SubSec:QuantEst}  to obtain the quantitative regularity results (see \autoref{thm:contestsing}) through  a series of covering arguments that control the regularity of the space and the regularity of the minimal boundary together.  The interpretation of minimality via Laplacian bounds on the distance function obtained in \autoref{subsec:mean} will play a key role here.
\end{itemize}
 \end{itemize}

As some examples will show, the threshold dimension for the full regularity is lower in this framework than in the Euclidean case: our Hausdorff codimension three estimate for the singular set is sharp (see \autoref{rem:ExsharpCodim3}), moreover, already in ambient dimension $4$ there are examples of tangent cones with no Euclidean splittings (see \autoref{rm:Morgan}) and of topologically irregular minimal boundaries.

\subsection{An $\eps$-regularity theorem}

The aim of this subsection is to establish an $\eps$-regularity result for minimal boundaries. This will provide a (weak) counterpart of the classical statement for minimal boundaries in the Euclidean setting.
\medskip

Usually, the outcome of an $\eps$-regularity theorem is that if a certain \textit{solution} is close enough to a rigid model then it is \textit{regular}. The celebrated result for minimal boundaries in the Euclidean case from \cite{DeGiorgi61} says that a minimal boundary contained in a sufficiently small strip around a hyperplane is analytic.\\ 
Arguably, and as elementary examples show, this is too much to hope for in the present setting. Our $\eps$-regularity result will be more in the spirit of Reifenberg's original approach: we will show that a minimal boundary which is close enough to the boundary of a half-space (in the Gromov-Haudorff sense) is topologically regular.\\
This could be considered as the counterpart for minimal boundaries of the celebrated $\eps$-regularity result for manifolds with lower Ricci curvature bounds obtained in \cite{Colding97,CheegerColding97} and extended to $\RCD$ spaces in \cite{KapovitchMondino21}, see \autoref{thm:epsregcolding}.
\medskip

To avoid confusion let us clarify that in this subsection by local perimeter minimizer in an open domain we intend that the perimeter is minimized among all the competitors that are perturbations inside the domain. This is a much stronger requirement than the one considered in \autoref{def:locper} to obtain mean curvature bounds. For smooth hypersurfaces in smooth ambient spaces, \autoref{def:locper} would correspond to minimality (i.e. vanishing mean curvature), while here we will be concerned with locally area minimizers.\\
Moreover, this subsection will be independent of the theory of mean curvature bounds that we have developed so far. Mean curvature bounds will enter into play later on, when proving that the assumptions of the $\eps$-regularity theorem are in force at many locations and scales, see \autoref{subsec:mono} and \autoref{subsec:partreg}.   
\medskip

Let us introduce some useful terminology, adapting the notion of flatness from the Euclidean to the non smooth and non flat case. With respect to the Euclidean realm, in the non flat framework there are many more \textit{rigid} situations to be considered. This is also due to the following result, yielding existence of a large family of \textit{flat} minimal boundaries.

\begin{lemma}\label{lemma:splitminimal}
Let $(Y,\dist_Y,\haus^{N-1})$ be an $\RCD(0,N-1)$ metric measure space and let $X:=\setR\times Y$ be endowed with the canonical product metric measure structure. Let $E:=\{t<0\}$, where we denoted by $t$ the coordinate of the Euclidean factor $\setR$. Then $E$ is a perimeter minimizing set. 
\end{lemma}  

\begin{proof}
The vector field $\nabla t$ is easily checked to be a calibration for $E$, ($t$ is harmonic, hence $\nabla t$ has vanishing divergence). The conclusion follows from 
a classical calibration argument, exploiting \autoref{thm:GGRCDsmooth} and \autoref{thm:cutandpaste} as in the smooth setting.
\end{proof}

Recall that convergence in the $L^1$ strong sense of sets of finite perimeter along pmGH converging sequences of metric measure spaces is metrizable, see \cite[Appendix A]{AmbrosioBrueSemola19}. By the above, we are entitled to give the following.

\begin{definition}[$\eps$-flat points]\label{def:epsflat}
Let $\eps>0$. If $(X,\dist,\haus^N)$ is an $\RCD(-\eps,N)$ metric measure space and $E\subset X$ is a set of finite perimeter, perimeter minimizing in $B_2(x)\subset X$ such that:
\begin{itemize}
\item there exists an $\RCD(0,N-1)$ metric measure space $(Y,\dist_Y,\haus^{N-1},y)$ such that the ball $B_2(x)\subset X$ is $\eps$-GH close to the ball $B_2((0,y))\subset\setR\times Y$;
\item $E$ is $\eps$-close on $B_2(x)$ in the $L^1$ topology to $\{t<0\}\subset \setR\times Y$ and $\partial E\cap B_2(x)$ is $\eps$-GH close to $\{t=0\}\cap B_2(0,y)\subset\setR\times Y$;
\end{itemize}
then we shall say that $E$ is $\eps$-flat at $x$ in $B_2(x)$.

The notion of $\eps$-flat set at $x$ in $B_r(x)$ can be introduced analogously by scaling.
\end{definition}

\begin{definition}[$\eps$-regular points]\label{def:epsregular}
Let $\eps>0$. If $(X,\dist,\haus^N)$ is an $\RCD(-\eps,N)$ metric measure space and $E\subset X$ is a set of finite perimeter, perimeter minimizing in $B_2(x)\subset X$, such that:
\begin{itemize}
\item the ball $B_2(x)\subset X$ is $\eps$-GH close to the ball $B_2(0^N)\subset\setR^N$;
\item $E$ is $\eps$-close on $B_2(x)$ in the $L^1$ topology to $\{t<0\}\subset \setR^N$ and $\partial E\cap B_2(x)$ is $\eps$-GH close to $\{t=0\}\cap B_2(0^{N})\subset\setR^N$, where we denoted by $t$ one of the canonical coordinates on $\setR^N$;
\end{itemize}
then we shall say that $E$ is $\eps$-regular at $x$ in $B_2(x)$.

The notion of $\eps$-regular set at $x$ in $B_r(x)$ can be introduced analogously by scaling.
\end{definition}

\begin{remark}\label{rm:tangentgoodimpliesgood}
Let $E\subset X$ be perimeter minimizing inside an open domain $\Omega\subset X$. Let $x\in\partial E$ and assume that there exists an $\RCD(0,N-1)$ metric measure space $(Y,\dist_Y,\haus^{N-1},y)$ such that, denoting by $t$ the coordinate of the split factor $\setR$ in the product $\setR\times Y$ with canonical product metric measure structure,
\begin{equation*}
\big\{ \big(\{t<0\},(0,y),\setR\times Y\big) \big\}\in\Tan_x(E,X,\dist,\haus^N)\, .
\end{equation*}
Then, for any $\eps>0$ and any $r_0>0$, there exists $0<r<r_0$ such that $E$ is $\eps r$-flat in $B_r(x)$. This is a direct consequence of \autoref{thm:closurecompactnesstheorem}, together with the very definition of tangent to a set of finite perimeter.

Analogously, if
\begin{equation*}
\big\{ \big(\{t<0\},0^{N},\setR^N \big) \big\}\in\Tan_x(E,X,\dist,\haus^N)\, ,
\end{equation*}
then for any $\eps>0$ and for any $r_0>0$ there exists $0<r<r_0$ such that $E$ is $\eps r$-regular at $x$ on $B_r(x)$.
\end{remark}

Below, we shall fix the scale $r=1$. As we already argued, the statements are scale invariant, therefore this is not a loss of generality.
\medskip

The stability of perimeter minimizers allows to get a measure bound out from Gromov-Hausdorff closeness.

\begin{lemma}[Perimeter density estimate for perimeter minimizers]\label{lemma:areaconvergence}
For any $\delta>0$ there exists $\eps=\eps(\delta,N)>0$ such that the following holds. If $(X,\dist,\haus^N)$ is an $\RCD(-\eps,N)$ metric measure space, $E\subset X$ is perimeter minimizing in $B_4(x)$, $x\in\partial E$ and $E$ is $\eps$-regular at $x$ in $B_2(x)$, then 
\begin{equation}\label{eq:volconvper}
1-\delta\le \frac{\Per(E,B_1(x))}{\omega_{N-1}}\le 1+\delta\, ,
\end{equation}
where $\omega_{N-1}$ denotes the volume of the unit ball in $\setR^{N-1}$.
\end{lemma}

\begin{proof}
The statement can be proved by a contradiction argument. 

Consider a sequence of sets of finite perimeter $E_n\subset X_n$, where $x_n\in\partial E_n$, $(X_n,\dist_n,\haus^N)$ are $\RCD(-1/n,N)$ metric measure spaces, $E_n$ is $1/n$-regular in $B_2(x_n)$ and perimeter minimizing in $B_4(x_n)$. Then the following holds: the balls $B_2(x_n)\subset (X_n,\dist_n,\haus^N,x_n)$ are converging to $B_2(0^{N})\subset (\setR^N,\dist_{\mathrm{eucl}},\haus^N,0^{N})$ in the pmGH topology and the sets of finite perimeter $E_n$ are converging to $\mathbb{H}^N$ on $B_2(0^{N})$ in the $L^1_{\loc}$-topology, with boundaries $\partial E_n$ Hausdorff converging to the boundary $\partial \mathbb{H}^N$ on $B_2(x_n)$.\\ 
Then 
\begin{equation*}
\Per(E_n,B_1(x_n))\to \Per(\mathbb{H}^N,B_1(0^{N}))=\omega_{N-1}\, ,\quad\text{as $n\to\infty$}\, ,
\end{equation*}
thanks to the weak convergence of perimeter measures in \autoref{thm:closurecompactnesstheorem} and the observation that  $\Per(\mathbb{H}^N,\partial B_1(0^{N}))=0$.
\end{proof}

\begin{remark}\label{rm:densitygapEuclidean}
Let us recall that we can associate to any locally area minimizing cone $C\subset \setR^N$ with vertex at $0\subset\setR^N$ its density
\begin{equation*}
\Theta_{0,C}:=\frac{\Per(C,B_1(0))}{\omega_{N-1}}=\frac{\Per(C,B_r(0))}{\omega_{N-1}r^{N-1}}\, ,\quad\text{for any $0<r<\infty$}\, .
\end{equation*}
Then, among all the possible densities of minimal cones $C\subset \setR^N$, the halfspace attains the minimal one, and there is a strictly positive gap between the density of the half-space and the densities of all the other minimal cones.\\
This can be rephrased by saying that there exists $c_N>0$ such that, for any minimal cone $C\subset \setR^N$ with vertex at $0^{N}$ and different from the half-space,
\begin{equation}\label{eq:densitygapeucl}
\Theta_{0,C}> 1+c_N=\Theta_{0,\mathbb{H}^N}+c_N\, .
\end{equation}

The statement is classical, and it can be proved arguing by contradiction by relying on the regularity theory for perimeter minimizers. More in detail, the density at the vertex of a cone equals its density at infinity, which is independent of the chosen base point. Namely
\begin{equation}
\Theta_{0,C}=\lim_{r\to \infty}\frac{\Per(C,B_r(0))}{\omega_{N-1}r^{N-1}}=\lim_{r\to \infty}\frac{\Per(C,B_r(p))}{\omega_{N-1}r^{N-1}}\, ,
\end{equation}
for any $p\in \partial C$. By the regularity theory, we can choose $p$ to be a regular boundary point and apply the monotonicity formula to infer that
\begin{equation}
\Theta_{0,C}\ge \lim_{r\to 0}\frac{\Per(C,B_r(p))}{\omega_{N-1}r^{N-1}}=\Theta_{0,\mathbb{H}^N}\, .
\end{equation}
The argument above also shows that a cone with the same density of the half-space must be the half-space.\\
In order to prove \eqref{eq:densitygapeucl} we argue by contradiction. If there is a sequence of cones $C_n$, all different from the half-space, and with densities converging to the density of the half-space, by compactness and stability we can extract a subsequence converging to a perimeter minimizer. The density at infinity of this limit minimizer is easily seen to equal $\Theta_{0,\mathbb{H}^N}$. By the above considerations, the limit is the half-space. By the $\eps$-regularity theorem $C_n$ is smooth on $B_1(0)$ for any sufficiently large $n$. This is a contradiction to the assumption that $C_n$ is a cone different from the half-space.
\end{remark}

In the Euclidean theory minimal boundaries are smooth, if the ambient dimension is less or equal than $7$. Moreover, they are smooth in any dimension in a region where they are sufficiently flat. These statements are the outcome of the classification of minimal cones up to dimension $7$ and of the already mentioned $\eps$-regularity theorem in \cite{DeGiorgi61}.

Notice that \autoref{lemma:splitminimal} shows that there is no hope for such a statement in our setting: consider a (possibly singular) Alexandrov space of dimension two and its product with a line, then the Alexandrov space is a minimal boundary inside the product. Hence the best regularity we can achieve for minimal boundaries in ambient dimension three is the regularity of two dimensional Alexandrov spaces.\\
Nevertheless one might hope that sufficiently flat minimal boundaries in the sense of \autoref{def:epsflat} have flat tangents (i.e. $0$-flat). It turns that this is not the case, at least when the ambient dimension is greater than $4$, due to the following.

\begin{remark}\label{rm:Morgan}
Denote by $\mathbb{S}^3_r$ the three dimensional sphere of radius $r$ endowed with the canonical Riemannian metric, and by $\mathbb{H}^{3}_{r}$ the upper hemisphere. Let also $0$ denote the tip of the cone.
In \cite{Morgan02} it is shown that the cone $C(\mathbb{H}^{3}_{r})$ is perimeter minimizing  in $B_1(0)\subset C(\mathbb{S}^3_r)$, for $r<1$ sufficiently close to $1$. 
\end{remark} 

The effect of this remark is that in our framework there cannot be an improvement of flatness, as it happens in the classical case, at least for ambient dimension greater than $4$. The best we can hope for is that flatness is preserved along scales.\\

\begin{theorem}[$\eps$-regularity]\label{thm:epsregularity}
Let $N>1$ be fixed. For any $\eps>0$ there exists $\delta=\delta(\eps,N)>0$ such that the following holds. Let  $(X,\dist,\haus^N)$ be an $\RCD(-\delta,N)$ metric measure space,  $E\subset X$ be a set of locally finite perimeter, $x\in \partial E$ be such that $E$ is perimeter minimizing on $B_4(x)$ and $E$ is $\delta$-regular in $B_2(x)$; then for any $y\in \partial E\cap B_1(x)$ and for any $0<r<1$, $E$ is $\eps r$-regular in $B_r(y)$.

Moreover, for any $0<\alpha<1$, there exists $\delta=\delta(\alpha,N)>0$ such that if $X$ and $E$ are as above (in particular, $x\in \partial E$ and $E$ is $\delta$-regular at $x$ in $B_2(x)$), then $\partial E\cap B_1(x)$ is $C^{\alpha}$-homeomorphic to the ball $B_1(0^{N-1})\subset \setR^{N-1}$.
\end{theorem}

\begin{proof}
We argue by contradiction. Let us suppose that the conclusion is not true. Then we can find $\eps>0$, a sequence of $\RCD(-1/n,N)$ metric measure spaces $(X_n,\dist_n,\haus^N,x_n)$ and sets of locally minimal perimeter $E_n\subset X_n$ such that $x_n\in\partial E_n$, $E_n$ is $1/n$-regular at $x_n$ in $B_2(x_n)$ but there exist $y_n\in B_1(x_n)\cap\partial E_n$ and $r_n>0$ such that:
\begin{itemize}
\item[(i)] $E_n$ is $\eps r$-regular at $y_n$ in $B_r(y_n)$ for any $r_n<r<1$ and for any $n\in \setN$; 
\item[(ii)] $E_n$ is not $\eps r_n/2$-regular at $y_n$ in $B_{r_n/2}(y_n)$.
\end{itemize} 
It is easy to check that these assumptions force $r_n\to 0$. Moreover, we can assume $\eps>0$ small enough so that $\delta>0$ in \eqref{eq:volconvper} is smaller than the density gap $c_N$ of \eqref{eq:densitygapeucl}.

Now let us rescale along the sequence in order to let the critical scales $r_n$ become scale $1$. If we do so, letting $\tilde{X}_n:=(X_n,\dist_{n}/r_n,\haus^N,y_n)$ and looking at the sets $E_n$ in the rescaled metric measure spaces, by \autoref{thm:epsregcolding},  $\tilde{X}_n$ converge in the pmGH topology to $(\setR^N,\dist_{\mathrm{eucl}},\haus^N,0^{N})$. Moreover, thanks to  \autoref{thm:closurecompactnesstheorem} the sets $E_n$ converge in the $L^1_{\loc}$ topology to an entire minimizer of the perimeter $F\subset \setR^N$. \\
Taking into account (i) and \autoref{lemma:areaconvergence}, we can also infer that
\begin{equation}\label{eq:densitypinched}
1-\delta\le \frac{\Per(F,B_r(0^{N}))}{\omega_{N-1}r^{N-1}}\le 1+\delta\, ,\quad\text{for any $1<r<\infty$}\, .
\end{equation}
Since $F$ is an entire perimeter minimizer in $\setR^N$, the standard Euclidean monotonicity formula yields that
\begin{equation}\label{eq:monoF}
r\mapsto \frac{\Per(F,B_r(z))}{\omega_{N-1}r^{N-1}}
\end{equation}
is an increasing function, for any $z\in\partial F$. By \eqref{eq:densitypinched}, that guarantees compactness of the sequence of scalings $F_{0,r}$ of $F$ for $r>1$, we are allowed to consider a blow-down $G$ of $F$. A standard consequence of the monotonicty formula is that $G$ is an entire minimal cone in $\setR^N$. Moreover, by \eqref{eq:densitypinched} and our choice of $\delta>0$, we have that
\begin{equation*}
1-\delta\le \Theta_{0,G}\le 1+\delta\le 1+c_N\, .
\end{equation*}
Hence, by the Euclidean density gap \autoref{rm:densitygapEuclidean} and monotonicity, we infer that $\Theta_G=1$. Therefore
\begin{equation}\label{eq:densinfty}
\lim_{r\to\infty}\frac{\Per(F,B_r(0))}{\omega_{N-1}r^{N-1}}=\Theta_{0,G}=1\, .
\end{equation}
Observe that the density at infinity of the entire minimal surface $F$ is independent of the base point $z\in\partial F$, as one can easily verify. Moreover, by De-Giorgi's theorem, there exists $z_0\in F\cap B_1(0)$ such that
\begin{equation}\label{eq:dens0}
\lim_{r\to 0}\frac{\Per(F,B_r(z_0))}{\omega_{N-1}r^{N-1}}=1\, .
\end{equation} 
Relying again on the monotonicity formula, by \eqref{eq:densinfty} and \eqref{eq:dens0} we infer that
\begin{equation*}
\frac{\Per(F,B_r(z_0))}{\omega_{N-1}r^{N-1}}=1\, ,\quad\text{for any $0<r<\infty$}\, .
\end{equation*}
Then with a standard argument we obtain that $F$ is a half-space $\mathbb{H}^N$ passing through $0$. 

By condition (ii) above, the sets $E_n$, when considered in the scaled metric measure spaces $\tilde{X}_n$, are not $\eps/2$-regular at $x_n$ in $B_{1/2}(x_n)$. This clearly gives a contradiction, since their limit is a half-space, as we just argued; in particular, they are $\eps/2$-regular at $x_n$ in $B_{1/2}(x_n)$ as soon as $n$ is large enough.
\medskip

The second part of the statement follows from the previous one via Reifenberg's theorem for metric spaces, see for instance \cite[Appendix 1]{CheegerColding97}.
\end{proof}

\begin{corollary}\label{cor:smooth}
Let $N>1$ be fixed. Then there exists $\delta=\delta(N)>0$ such that the following holds. If $(M^N,g)$ is a smooth $N$-dimensional Riemannian manifold and $E\subset M$ is a set of locally finite perimeter such that, for some $x\in M$ and $r>0$,
\begin{itemize}
\item[(i)] $\Ric_M\ge -\delta r^{-2}$ on $B_{4r}(x)$;
\item[(ii)] $E$ is perimeter minimizing in $B_{4r}(x)$;
\item[(iii)] $E$ is $\delta$-regular at $x$ on $B_{2r}(x)$. 
\end{itemize}
Then $\partial E\cap B_r(x)$ is smooth.
\end{corollary}

\begin{proof}
We only need to verify that all tangent cones at all points $x\in\partial E\cap B_r(x)$ are Euclidean half-spaces. Then the classical regularity in Geometric Measure Theory provides smoothness.\\
To this aim, observe that, by \autoref{thm:epsregularity}, all the tangent cones at any $x\in\partial E\cap B_r(x)$ are entire perimeter minimizers in $\setR^n$ close to the Euclidean half-space at all scales. Then an argument analogous to the one exploited in the proof of \autoref{thm:epsregularity}, relying on the Euclidean density gap (see \autoref{rm:densitygapEuclidean}), shows that the tangent cones are half-spaces.  
\end{proof}

\begin{remark}
In \autoref{cor:smooth} there is no assumption on the injectivity radius of the Riemannian manifold, nor on the full curvature tensor, which are the classical assumptions for the $\eps$-regularity theorems for minimal surfaces on Riemannian manifolds, see for instance \cite{CheegerNaber13b,NaberValtorta20}.
\end{remark}

\begin{remark}
\autoref{cor:smooth} should be compared with some previous results obtained in \cite{Gromov96} and \cite[Section 4]{Gromov14b}. Therein, uniform Reifenberg flatness was proved for minimal bubbles w.r.t. families of smooth Riemannian metrics $g_{\eps}$ uniformly converging to a background metric $g$ on a fixed manifold $M$. In this regard \autoref{cor:smooth} is much stronger, since it deals with a weaker notion of convergence of metrics. Moreover, \autoref{thm:epsregularity} shows that ambient regularity is not a key assumption for Reifenberg flatness, provided there is a synthetic lower Ricci bound on the background.  
\end{remark}

\subsection{Sharp perimeter bounds for  the equidistant sets from minimal boundaries }\label{subsec:mono}
In this subsection we consider again local perimeter minimizers in the sense of \autoref{def:locper}. Our goal is to prove some  sharp perimeter bounds for the equidistant sets from minimal boundaries which will turn to be very useful to establish the quantitative regularity results in \autoref{SubSec:QuantEst}. The interpretation of minimality via Laplacian bounds on the distance function obtained in \autoref{subsec:mean} will play a key role here.
\medskip

The following useful lemma is essentially taken from \cite{BrueNaberSemola20}, see in particular the proof of Theorem 7.4 therein. We omit the proof that can be obtained relying on \autoref{prop:fullrepr}, with arguments similar to those appearing in the proofs of previous results in this note.

\begin{lemma}\label{lemma:intbyparts}
Let $(X,\dist,\haus^N)$ be an $\RCD(K,N)$ metric measure space and let $E\subset X$ be a set of locally finite perimeter which locally minimizes the perimeter in an open domain $\Omega\subset X$ according to \autoref{def:locper}. Then, for any Lipschitz function $\phi:X\to\setR$ with compact support in $\Omega$, it holds:
\begin{equation}\label{eq:intbyp}
\int\phi\di\Per(\{\dist_{\bar{E}}>r\})=\int_{\{0\le \dist_{\bar{E}}<r\}}\di \div(\phi\nabla\dist_{\bar{E}})\, ,\quad\text{for a.e. $r>0$}\, .
\end{equation}
\end{lemma}

\begin{remark}
The local perimeter minimizing assumption above is used only to infer regularity properties of the distance function, namely the fact that it has measure valued Laplacian whose singular part on the boundary of the set is the surfaces measure, rather than to obtain specific mean curvature bounds. Indeed, the conclusion of \autoref{lemma:intbyparts} holds for the boundary of any smooth set on a smooth Riemannian manifold.
\end{remark}

In order to ease the notation, let us denote by $E^t$ the open $t$-enlargement of $E$, i.e.
\begin{equation}\label{eq:defEt}
E^t:=\{x\in X\, :\, \dist(x,\bar{E})<t\}\,.
\end{equation}

We will need to compare the perimeter measure of the set $E$ and the measures obtained by normalizing the restriction of the ambient volume measure to a tubular neighbourhood of the set. Again, for all smooth hypersurfaces in the smooth Riemannian setting, the perimeter and such a Minkowski-type measure coincide, even though they do not for general sets. The next result states that the perimeter minimality condition is robust enough to guarantee such an extra regularity also in the $\RCD$ setting. 

\begin{proposition}\label{prop:regularitymeasure}
Let $(X,\dist,\haus^N)$ be an $\RCD(K,N)$ metric measure space and let $E\subset X$ be a set of locally finite perimeter which locally minimizes the perimeter in an open domain $\Omega\subset X$ according to \autoref{def:locper}. For any $0<\eps<1$, let
\begin{equation*}
\mu_{\eps}^+:=\frac{1}{\eps}\haus^N\res\{0\le \dist_{{\bar{E}}}<\eps\}\quad \text{and} \quad \mu_{\eps}^-:=\frac{1}{\eps}\haus^N\res\{0\le \dist_{E^c}<\eps\}\,.
\end{equation*} 
Then both $\mu_{\eps}^+$ and $\mu_{\eps}^-$ weakly converge to $\Per_E$ on $\Omega$ as $\eps\to 0$.
\end{proposition}

\begin{proof}
Let us prove the weak convergence to the perimeter of $\mu_{\eps}^{+}$. The weak convergence of $\mu_{\eps}^{-}$ can be proved with an analogous argument, replacing $\dist_{\bar{E}}$ with $\dist_{E^c}$.
\medskip

The family of measures $\mu_{\eps}^+$ has locally uniformly bounded mass, as it follows from \autoref{lemma:tubneighbounds}. We claim that for any weak limit $\mu$ of the sequence of measures $\mu_{\eps_i}^+$, where $\eps_i\downarrow 0$ as $i\to\infty$, it holds $\mu=\Per_E$.

Let us start from the inequality $\mu\ge \Per_E$.\\ 
Letting 
$\phi_{\eps}^+:X\to\setR$ be defined by $\phi_{\eps}^+(x)=1$ on $\bar{E}$, $\phi_{\eps}^+=0$ on $X\setminus E^{\eps}$ and 
\begin{equation*}
\phi_{\eps}^+=\frac{1}{\eps}(\eps-\dist(x,\bar{E}))\, ,\quad\text{on $E^{\eps}\setminus E$}\, ,
\end{equation*}
it holds
\begin{equation*}
\mu_{\eps}^+=\abs{\nabla \phi_{\eps}^+}\haus^N\, .
\end{equation*}
Moreover, it is easy to check that $\phi_{\eps}^+$ converge locally in $L^1$ to $\chi_E$. Hence, by the lower semicontinuity of the total variation (in localized form), it is easy to infer that, for any open set $A\subset \Omega$ such that $\mu(\partial A)=0$,
\begin{equation*}
\Per(E,A)\le \liminf_{i\to \infty}\mu_{\eps_{i}}^+(A)=\mu(A)\, .
\end{equation*}
To prove the converse inequality, let us focus for simplicity on the case $K=0$, the general case introduces only an additional error term of lower order. Let us consider any non-negative Lipschitz function $\phi:X\to [0,\infty)$ with compact support in $\Omega$. We claim that
\begin{equation*}
\int\phi\di\mu\le \int\phi\di\Per\, ,
\end{equation*}
which will imply the inequality $\mu\le \Per_E$.\\ 
To prove this claim, we rely on \autoref{lemma:intbyparts}. Indeed, for a.e. $r>0$ sufficiently small, it holds that
\begin{equation*}
\int\phi\di\Per(\{\dist_{\bar{E}}>r\})=\int_{\{\dist_{\bar{E}}<r\}}\di \div(\phi\nabla\dist_{\bar{E}})\, .
\end{equation*}
Hence, for a.e. $r>0$, using the Leibniz rule for the divergence, \autoref{thm:meancurvminimal1} and \autoref{prop:fullrepr}, we get
\begin{align*}
\int\phi\di\Per(\{\dist_{\bar{E}}>r\})=&\int_{E^r}\nabla\phi\cdot\nabla\dist_{\bar{E}}\di\haus^N+\int_{E^r}\phi\boldsymbol{\Delta}\dist_{\bar{E}}\\
\le &\int_{E^r}\nabla\phi\cdot\nabla\dist_{\bar{E}} \di\haus^N+\int\phi\di\Per_E\, .
\end{align*}
Therefore, for any $s>0$ sufficiently small, by the coarea formula we get
\begin{align*}
\int_{E^s}\phi\di\haus^N=&\int_0^s\int\phi\di\Per(\{\dist_{\bar{E}}>r\})\\
\le &\int_0^s\left(\int_{E^r}\nabla\phi\cdot\nabla\dist_{\bar{E}} \di\haus^N+\int\phi\di\Per_E\right)\\
\le& s \LipConst (\phi)\haus^N(E^s\cap\mathrm{spt}\phi)+s\int\phi\di\Per_E\, .
\end{align*}
Hence
\begin{align*}
\int\phi\di\mu=&\lim_{i\to \infty} \frac{1}{s_{i}}\int_{E^{s_{i}}}\phi\di\haus^N\\
\le& \limsup_{s\to 0}\frac{1}{s}\left(s \LipConst (\phi) \haus^N(E^s\cap\mathrm{spt}\phi)+s\int\phi\di\Per_E\right)\\
= &\int\phi\di\Per_E\, ,
\end{align*}
where we used \autoref{lemma:tubneighbounds} in the last inequality. This concludes the proof of the inequality $\mu\le \Per_E$ and hence the proof.
\end{proof}

Let us introduce the notation $\Sigma$ for the boundary $\partial E$ of a set of finite perimeter $E$ which is locally perimeter minimizing in $\Omega\subset X$ and let us denote, for any $h>0$,
\begin{equation*}
\Sigma^h:=\{x\in \Omega\, :\, \dist(\bar{E},x)=h\}\, .
\end{equation*}

The next result is a kind of monotonicity formula for equidistant sets from minimal boundaries. Its proof is inspired by \cite[Lemma 2]{CaffarelliCordoba93}, which deals with the Euclidean case.
The Laplacian bound for the distance from a locally minimizing set of finite perimeter under lower Ricci curvature bounds (obtained in \autoref{thm:meancurvminimal1}) allows to extend it to the present framework.

\begin{proposition}\label{prop:tubecomparison}
Let $(X,\dist,\haus^N)$ be an $\RCD(K,N)$ metric measure space and let $E\subset X$ be a set of locally finite perimeter which locally minimizes the perimeter in an open domain $\Omega\subset X$ according to \autoref{def:locper}. Let $h>0$ be fixed. Let $\Gamma\subset\Sigma^h$ be any compact set and denote
\begin{align}
 \Gamma_{\Sigma}&:=\{y\in\Sigma\cap \Omega\ :\, \dist(x,y)=h\quad\text{for some $x\in\Gamma$}\}\, ,\\
G&:=\{ x\in \Omega\,:\,\dist_{\Gamma_{\Sigma}}(x)+\dist_{\Gamma}(x)=h\}\, .\label{eq:G}
\end{align}
If $G\Subset \mathcal{K}$, where $\mathcal{K}$ has been defined in \eqref{eq:defK}, then
\begin{equation}\label{eq:PerEhPerE}
\Per(E^h,\Gamma)\le \Per(E,\Gamma_{\Sigma}) + \int_G   \ft_{K,N}(\dist_{E}) \di\haus^N\ \, ,
\end{equation}
where  $\ft_{K,N}$ was defined in \eqref{eq:deftKN}, 
and
\begin{equation}\label{eq:PerEhGr}
\Per(E^h,\Gamma)\le \begin{cases}
\Per(E,\Gamma_{\Sigma})\cos \left(\sqrt{\frac{K}{N-1}}h \right)^{N-1}\,& \quad \text{if } K>0\\
\quad \Per(E,\Gamma_{\Sigma})\, & \quad \text{if } K=0 \\
\Per(E,\Gamma_{\Sigma})\cosh\left(\sqrt{\frac{-K}{N-1}}h\right)^{N-1}\,& \quad \text{if } K<0 \, .
\end{cases} 
\end{equation}
\end{proposition}

\begin{remark}
Note that $G$ is made by the union of minimizing geodesics connecting $\Gamma_{\Sigma}$ with $\Sigma$ along which $\dist_{E}$ is attained.
\end{remark}

\begin{remark}
The bounds obtained in \autoref{prop:tubecomparison} are sharp. Indeed it is easily seen that equality is achieved in the model spaces:
\begin{itemize}
\item for $K>0$, let $(X,\dist,\haus^N)$ be the $N$-dimensional round sphere of constant sectional curvature $K/(N-1)$ and $E$ be a half-sphere. It is a standard fact that $E$ is locally perimeter minimizing inside a sufficiently small open domain $\Omega$. It is immediate to see that $\Sigma^h$ is (part of the boundary of) a spherical cap and one can check that equality is attained in \eqref{eq:PerEhGr} by  direct computations;
\item for $K=0$, let $(X,\dist,\haus^N)$ be the $N$-dimensional Euclidean space and $E$ be a half-space. It is a standard fact that $E$ is locally perimeter minimizing inside any open domain $\Omega$. It is immediate to see that $\Sigma^h$ is (part of the boundary of) an equidistant  half space and that equality is attained in \eqref{eq:PerEhGr};
\item for $K<0$, let $(X,\dist,\haus^N)$ be the $N$-dimensional hyperbolic space of constant sectional curvature $K/(N-1)$ and $E$ be a horo-ball. It is a standard fact that $E$ is locally perimeter minimizing inside  any open domain $\Omega$. Also in this case, one can check that equality is attained in \eqref{eq:PerEhGr} by direct computations.
\end{itemize}
\end{remark}

\begin{proof}
Notice that $G$ is the set spanned by those rays connecting $\Gamma_{\Sigma}$ to $\Gamma$. We would like to apply the Gauss-Green integration by parts formula to the vector field $\nabla \dist_{E}$ on $G$. Indeed, at an heuristic level, the boundary of $G$ is made of three parts, $\Gamma_{\Sigma}$, $\Gamma$ and some \textit{lateral} faces whose unit normal we expect to be orthogonal to $\nabla\dist_{E}$. Then the conclusion would follow from the fact that $\div\nabla\dist_E\le \ft_{K,N} \circ \dist_{E}$, by \autoref{thm:meancurvminimal1}.

In order to make the argument rigorous, we are going to approximate the characteristic function of the set $G$ (which in general may not be regular enough), by suitable cut-off functions. \\ 
Let us introduce the shortened notation $\bar{\dist}$ for the distance from $\bar{E}$. Moreover, let us denote by $\dist_{\Gamma}$ the distance function from the compact set $\Gamma$ in the statement. Then, for any $\eps\in(0,\eps_0)$ let us set
\begin{equation*}
\phi_{\eps}:=\frac{1}{\eps}\left(h+\eps-(\bar{\dist}+\dist_{\Gamma})\right)_{+}\, ,
\end{equation*} 
where we denoted by $(\cdot)_+$ the positive part. 
For any  $\delta\in (0,h)$, we introduce the monotone function $g_{\delta}$ satisfying:
\begin{equation*}
g_{\delta}(0)=g'_{\delta}(0)=0\, ,\quad g''_{\delta}=\frac{1}{\delta}\left(\chi_{[0,\delta]}-\chi_{[h-\delta,h]}\right)\, .
\end{equation*}
Observe that, in particular, $g'_{\delta}(h)=0$.  Recalling that $\abs{\nabla \bar{\dist}}\le 1$ a.e. and that $g_{\delta}'(\bar{\dist})\boldsymbol{\Delta}\bar{\dist}\le g_{\delta}'(\bar{\dist})\,  \ft_{K,N}(\bar{\dist})$  by \autoref{thm:meancurvminimal1}, using chain rule we obtain:
\begin{equation}\label{eq:bounddistrdiv}
\boldsymbol{\Delta}g_{\delta}(\bar{\dist})\le g_{\delta}''(\bar{\dist}) + g_{\delta}'(\bar{\dist})\, \ft_{K,N}(\bar{\dist})\, .
\end{equation}
Now let $F\subset \mathcal{K}$ be an open neighbourhood of $G$ inside $\mathcal{K}$. Relying on \eqref{eq:bounddistrdiv} and applying the Gauss Green integration by parts formula (see \autoref{thm:GaussGreenRCDBCM}), taking into account that there are no boundary terms since either $\phi_{\eps}=0$ or $g'_{\delta}(\bar{\dist})=0$ on the boundary of the domain for $\eps>0$ sufficiently small, we can compute:
\begin{align*}
\int_{F}g''_{\delta}(\bar{\dist})\phi_{\eps}\di\haus^N&\ge \int_F\phi_{\eps}\di \boldsymbol{\Delta}g_{\delta}(\bar{\dist}) -  \int_F \phi_{\eps}\, g_{\delta}'(\bar{\dist})\,  \ft_{K,N}(\bar{\dist}) \di\haus^N \\
& =-\int_F\nabla (g_{\delta}(\bar{\dist}))\cdot\nabla\phi_{\eps}\di\haus^N   -  \int_F  \phi_{\eps}\, g_{\delta}'(\bar{\dist})\,  \ft_{K,N}(\bar{\dist}) \di\haus^N\, .
\end{align*}
Let us observe that
\begin{align*}
\nabla (g_{\delta}(\bar{\dist}))\cdot\nabla\phi_{\eps}=&g'_{\delta}(\bar{\dist})\nabla\bar{\dist}\cdot (-\nabla\bar{\dist}-\nabla\dist_{\Gamma}) \eps^{-1}\\
=&g'_{\delta}(\bar{\dist})(-1-\nabla\bar{\dist}\cdot \nabla\dist_{\Gamma} )  \eps^{-1}\le 0\, ,\quad\text{$\haus^N$-a.e. on $F$}\, .
\end{align*}
Hence, for any $\eps,\delta>0$ sufficiently small, it holds
\begin{equation*}
\int_{F}g''_{\delta}(\bar{\dist})\phi_{\eps}\di\haus^N\ge  -  \int_F  \phi_{\eps}\, g_{\delta}'(\bar{\dist})\,  \ft_{K,N}(\bar{\dist}) \di\haus^N\, .
\end{equation*}
By the very definition of $g_{\delta}$, this implies that

\begin{align}\label{eq:ap}
\nonumber \frac{1}{\delta}\int_{F}\chi_{[0,\delta]}(\bar{\dist})\phi_{\eps}\di\haus^N\ge & \frac{1}{\delta}\int_{F}\chi_{[h-\delta,h]}(\bar{\dist})\phi_{\eps}\di\haus^N\\ 
& -  \int_F  \phi_{\eps}\, g_{\delta}'(\bar{\dist})\,  \ft_{K,N}(\bar{\dist}) \di\haus^N\, .
\end{align}
Relying on \autoref{prop:regularitymeasure}, which guarantees the weak convergence of the measures $\delta^{-1}\chi_{[0,\delta]}(\bar{\dist})\haus^N$ to $\Per_E$ as $\delta\to 0$, we can pass to the limit in the left hand side of \eqref{eq:ap}. Moreover, by semicontinuity of the total variation, for any weak limit $\nu$ of the sequence $\delta^{-1}\chi_{[h-\delta,h]}(\bar{\dist})\haus^N$ (which is easily seen to be pre-compact in the weak topology) as $\delta\to 0$, it holds $\nu\ge \Per(E_h)$. It is also easily seen that $0\leq g_{\delta}'(\bar{\dist})\uparrow 1$ $\haus^N$-a.e. on $F$, as $\delta \downarrow 0$. Hence
\begin{equation}\label{eq:PerEPerEhPf}
\int_{F}\phi_{\eps}\di\Per_E\ge \int_{F}\phi_{\eps}\di\Per_{E_h}   -  \int_F  \phi_{\eps}\,  \ft_{K,N}(\bar{\dist}) \di\haus^N\, .
\end{equation}
Next, we pass to the limit as $\eps\to 0$. 
Observe that 
\begin{equation*}
\lim_{\eps\to 0}\phi_{\eps}(x)=
\begin{cases}
1\,   \quad \text{if $x\in G$ } \\
0\,  \quad \text{otherwise}\, .
\end{cases}
\end{equation*}
Therefore, passing to the limit in \eqref{eq:PerEPerEhPf} as $\eps\to 0$,  we obtain that
\begin{equation*}
\Per(E^h,\Gamma)\le \Per(E,\Gamma_{\Sigma}) + \int_G   \ft_{K,N}(\bar{\dist}) \di\haus^N\ \, ,
\end{equation*}
as desired.

The bounds in \eqref{eq:PerEhGr} follow from \eqref{eq:PerEhPerE} thanks to the coarea formula and the integral form of Gr\"onwall's Lemma if $K<0$. In the case $K=0$ they follow directly from \eqref{eq:PerEhPerE} since $\ft_{0,N}=0$.\\ 
Let us deal with the remaining case $K>0$.\\
We introduce a function
\begin{equation*}
f_{K,N}:\left(0,\frac{\pi}{2}\sqrt{\frac{N-1}{K}}\right)\to \setR\, , \quad\, 
f_{K,N}(r):=\int_0^r\cos  \left(\sqrt{\frac{K}{N-1}}\, h \right)^{-(N-1)}\di h\, .
\end{equation*}
Notice that
\begin{equation}
f'_{K,N}(r)=\cos\left(\sqrt{\frac{K}{N-1}}\, r \right)^{-(N-1)}\, , 
\end{equation}
for any $r>0$ and in particular $f'_{K,N}(0)=1$.
Moreover, the chain rule for the Laplacian and a direct computation show that we can rephrase the bound in \autoref{thm:meancurvminimal1} as
\begin{equation}\label{eq:lapdistcomp}
\Delta f_{K,N}\circ \dist_{E}\le 0\, .
\end{equation} 
Then \eqref{eq:PerEhGr} in the case $K>0$ follows formally by applying the Gauss-Green integration by parts formula to the vector field $\nabla f_{K,N}\circ\dist_E$ on the set $G$ introduced in \eqref{eq:G}. Indeed, the contribution coming from the integration in the interior has a sign thanks to \eqref{eq:lapdistcomp}, one of the two boundary terms is $f'_{K,N}(0)\Per(E,\Gamma_{\Sigma})=\Per(E,\Gamma_{\Sigma})$ and the other one can be estimated by 
\begin{equation}
f'_{K,N}(h)\Per(E^h,\Gamma)=\Per(E^h,\Gamma)\cos\left(\sqrt{\frac{K}{N-1}}\, h \right)^{-(N-1)}\, .
\end{equation}
Therefore we obtain
\begin{equation}\label{eq:bdinverse}
\Per(E^h,\Gamma)\le \Per(E,\Gamma_{\Sigma})\cos\left(\sqrt{\frac{K}{N-1}}\, h \right)^{(N-1)}\, ,
\end{equation}
as we claimed. The rigorous justification of \eqref{eq:bdinverse} can be obtained with an approximation argument completely analogous to the one introduced in the first part of the proof, approximating the characteristic function of $G$ with suitable cut-off functions;  we omit the details for the sake of brevity.
\end{proof}

A very useful result proved in Simons' seminal paper on minimal varieties \cite{Simons} states that there are no two sided stable smooth minimal hypersurfaces on closed manifolds with positive Ricci curvature. Thanks to the perimeter monotonicity in \autoref{prop:tubecomparison} we can partially generalize this fact to the present framework.

\begin{corollary} [Simons' theorem in $\RCD$ spaces]\label{cor:SimonsThmRCD}
Let $(X,\dist,\haus^N)$ be an $\RCD(K,N)$ metric measure space, for some $K>0$. Then, for any $r>0$, there is no non trivial set of finite perimeter $E\subset X$ that minimizes the perimeter among all the perturbations $F\subset X$ such that 
\begin{equation*}
E\Delta F\subset B_r(\partial E)\, . 
\end{equation*}
\end{corollary}

\begin{proof}
Let us argue by contradiction. If a set of finite perimeter as in the statement exists, then it is locally perimeter minimizing according to \autoref{def:locper}. Hence it verifies the assumptions of \autoref{prop:tubecomparison}. Therefore, for any $h>0$,
\begin{equation*}
\Per(E^h)\le \Per(E) + \int_{E^h\setminus E}   \ft_{K,N}(\dist_{E}) \di\haus^N <\Per(E) \, .
\end{equation*}
To conclude it is sufficient to observe that $E^h\Delta E \subset B_r(\partial E)$ for any $h>0$ sufficiently small and we reach a contradiction.
\end{proof}

\subsection{Partial regularity of minimal boundaries away from sets of codimension three}\label{subsec:partreg}
Our goal in this subsection is to prove that minimal boundaries have regular blow-ups (and therefore are topologically regular) away from sets of ambient codimension three (assuming for simplicity that the ambient space $(X,\dist,\haus^N)$ is an $\RCD$ space without boundary). 
\medskip

 
\begin{definition}[Regular and singular sets on minimal boundaries]
Let $(X,\dist, \haus^{N})$ be an $\RCD(K,N)$ metric measure space and $E\subset X$ be locally perimeter minimizing inside a ball $B_{2}(x)\subset X$. Suppose also that $\partial X\cap B_2(x)=\emptyset$. The regular part $\mathcal{R}^E$ and the singular part $\mathcal{S}^E$ of $\partial E$ are defined as
\begin{equation*}
\mathcal{R}^E:=\{x\in\partial E\, :\, (\setR^N,\dist_{\mathrm{eucl}},\haus^N,0^{N},\{x_N<0\})\in \Tan_x(X,\dist,\haus^N,E)\} \, , 
\end{equation*}
\begin{equation*}
\mathcal{S}^E:=\partial E\setminus \mathcal{R}^E\, . 
\end{equation*}
\end{definition}

\begin{remark}
If $E\subset X$ is locally perimeter minimizing and $x\in\partial X$, then for any 
\begin{equation}
(Y,\dist_Y,\haus^N,F,y)\in \Tan_x(X,\dist,\haus^N,E)
\end{equation}
it holds that $F$ is an entire local perimeter minimizer in $Y$, as it follows from the stability \autoref{thm:closurecompactnesstheorem}.
\end{remark}

As a first regularity result, we establish topological regularity of the regular set. This is indeed a direct consequence of the $\varepsilon$-regularity  \autoref{thm:epsregularity}.

\begin{theorem}[Topological regularity of the regular set]\label{thm:RegSetTop}
Let $(X,\dist,\haus^N)$ be an $\RCD(K,N)$ metric measure space, assume that $B_{2}(x)\cap \partial X=\emptyset$ and let $E\subset X$ be a set of locally finite perimeter that is locally perimeter minimizing in $B_{2}(x)$.
Then, for every $\alpha\in (0,1)$ there exists a relatively open set $O_{\alpha}\subset \partial E \cap B_{1}(x)$ with $\mathcal{R}^E\subset O_{\alpha}$ such that $O_{\alpha}$ is $\alpha$-biH\"older homeomorphic to an open, smooth $(N-1)$-dimensional manifold.
\end{theorem}

\begin{remark}
The $C^{0,\alpha}$ regularity of the manifold $O_{\alpha}$ containing the regular set matches the (currently known) regularity of the regular part $\mathcal{R}(X)$ of the ambient space $X$ (after Cheeger-Colding's metric Reifenberg Theorem  \cite[Appendix 1]{CheegerColding97}  and \cite{KapovitchMondino21}). Higher regularity of  $\mathcal{R}^E$ (e.g. contained in a Lipschitz manifold), would require first improving the structure theory of the ambient space.\end{remark}

The classical regularity result for perimeter minimizers in the Euclidean (or smooth Riemannian) setting is that they are smooth away from sets of ambient codimension $8$.

A key intermediate step is the fact that the blow-ups are flat Euclidean half-spaces away from sets of ambient codimension $8$, see \cite{Federer70,Giusti84}.\\ 
The examples that we have already discussed in this note show that this statement is false in the non smooth framework. Singular blow-ups already appear in ambient dimension $3$, see the discussion after \autoref{rm:densitygapEuclidean}.

As a first regularity result, below we prove that, if we restrict to regular ambient points (or we consider $\RCD(K,N)$ metric measure spaces $(X,\dist,\haus^N)$ such that the singular set is empty) then the picture matches with the classical one and we can prove that the codimension of the singular set of a perimeter minimizer is at least  $8$.

\begin{theorem}\label{thm:sharpdimreg}
Let $(X,\dist,\haus^N)$ be an $\RCD(K,N)$ metric measure space, let $\Omega\subset X$ be an open domain and let $E\subset X$ be a set of locally finite perimeter that is locally perimeter minimizing in $\Omega$. Then 
\begin{equation}\label{eq:dimbdreg}
\dim_H\left(\mathcal{S}^E\cap \mathcal{R}\cap \Omega\right)\le N-8\, .
\end{equation}
In particular: 
\begin{itemize}
\item[i)] if $\Omega\subset \mathcal{R}$, then
\begin{equation}
\dim_H\left(\mathcal{S}^E\cap \Omega\right)\le N-8\, ;
\end{equation}
\item[ii)] if $\Omega\subset \mathcal{R}$ and $N\le 7$, then $\mathcal{S}^E\cap\Omega=\emptyset$.
\end{itemize}
\end{theorem}

\begin{proof}
With the tools that we have developed so far, the proof reduces to a variant of the classical dimension reduction technique to bound the dimension of singular sets. See \cite{Federer69,Federer70} and \cite[Chapter 11]{Giusti84} for the case of perimeter minimizers in the Euclidean setting, and \cite{DePhilippisGigli18} for the dimension bounds for the singular strata on $\RCD(K,N)$ spaces $(X,\dist,\haus^N)$.
\medskip

We argue by contradiction. If \eqref{eq:dimbdreg} is not satisfied, then we construct a (local) perimeter minimizer inside $\setR^N$ whose singular set has codimension less than $8$. This will lead to a contradiction, since the singular set of a Euclidean (local) perimeter minimizer has codimension at least $8$, by the classical regularity theory.
\medskip

Let us suppose that \eqref{eq:dimbdreg} is not verified. Then there exists $\eta>0$ such that
\begin{equation}
\mathcal{H}^{N-8+\eta}\left(\mathcal{S}^E\cap\mathcal{R}\cap \Omega\right)>0\, .
\end{equation}

By \cite[Lemma 3.6]{DePhilippisGigli18} (see also \cite[Lemma 11.3]{Giusti84} and \cite[Theorem 2.10.17]{Federer69}), there exists $x\in \mathcal{S}^E\cap \mathcal{R}\cap \Omega$ such that 
\begin{equation}\label{eq:densest}
\limsup_{r\to 0}\frac{\haus^{N-8+\eta}_{\infty}\left(B_r(x)\cap \left(\mathcal{S}^E\cap \mathcal{R}\cap \Omega\right)\right)}{r^{N-8+\eta}}\ge 2^{N-8+\eta}\omega_{N-8+\eta}\, ,
\end{equation}
where we denoted by $\haus^{N-8+\eta}_{\infty}$ the pre-Hausdorff measure of dimension $N-8+\eta$.
\medskip

Now we claim that for any sequence $r_i\downarrow 0$ there exists a subsequence, that we do not relabel, such that $E\subset (X,\dist/r_i,\haus^N/r_{i}^N,x)$ converge in the $L^1_{\loc}$ sense as sets of (locally) finite perimeter to an entire (local) perimeter minimizer $E_{\infty}\subset \left(\setR^N,\dist_{eucl},\haus^N,0^N\right)$. Here the compactness of the sequence follows from \cite[Corollary 3.4]{AmbrosioBrueSemola19} together with the uniform perimeter bounds for perimeter minimizers inside the ball, while the conclusion that $E$ is an entire (local) perimeter minimizer follows from \autoref{thm:closurecompactnesstheorem}.

By scaling, denoting by $E_i=E\subset (X,\dist/r_i,\haus^N/r_{i}^N,x)$ the set of finite perimeter $E$ considered inside the rescaled metric measure space, we can find a sequence $r_i\downarrow 0$ such that $E_i$ converge in $L^1_{\loc}$ to an entire (locally) perimeter minimizer $E_{\infty}$ as we discussed above and moreover
\begin{equation}\label{eq:limhinfty}
\lim_{i\to\infty}\haus^{N-8+\eta}_{\infty}\left(\mathcal{S}^{E_i}\cap B_1^i(x)\cap \mathcal{R}(X_i)\right) \ge 2^{N-8+\eta}\omega_{N-8+\eta}\, >\, 0.
\end{equation}
 
We claim that \eqref{eq:limhinfty} forces
\begin{equation}\label{eq:boundpre}
\haus^{N-8+\eta}_{\infty}\left(\mathcal{S}^{E_{\infty}}\cap B_1(0^N)\right) \, >\, 0\, .
\end{equation}
In order to check \eqref{eq:boundpre} it is sufficient to prove that any limit point $x_{\infty}$ of a sequence $(x_i)$ such that $x_i\in \mathcal{S}^{E_i}\cap \mathcal{R}(X_i)\cap B_1^i(x)$ belongs to $\mathcal{S}^{E_{\infty}}\cap B_1(0^N)$. Once this statement has been established, \eqref{eq:boundpre} will follow from \eqref{eq:limhinfty} and the upper semicontinuity of the pre-Hausdorff measure $\haus^{N-8+\eta}_{\infty}$ under GH convergence, see \cite[Equation (3.36)]{DePhilippisGigli18}.
\medskip

Let us pass to the verification of the claim. Let us consider any $x_{\infty}$ as above. If we suppose by contradiction that it is a regular point of $E_{\infty}$, then it follows from the $\eps$-regularity \autoref{thm:epsregularity} that for any $\eps>0$ there exists $i\in\setN$ such that, for any $j\ge i$, $E_j\cap B_1^j(x)$ is $\eps$-regular inside $B_1^j(x)$. In particular, if $\eps>0$ is sufficiently small, then by the Euclidean density gap \autoref{rm:densitygapEuclidean}, all the blow-ups of $E_j$ inside $B_1^j(x)\cap \mathcal{R}$ are flat half-spaces (see also the proof of \autoref{cor:smooth}). This leads to a contradiction since we are assuming that $x_{\infty}$ is a limit of singular points $x_k\in \mathcal{S}^{E_j}\cap B_1^j(x)\cap\mathcal{R}(X_j)$.
\medskip

Given \eqref{eq:boundpre} we obtain a contradiction, since $E_{\infty}$ is an entire   Euclidean (local) perimeter minimizer and the classical dimension estimates for the singular sets of perimeter minimizers give
\begin{equation}
\dim_H\left(\mathcal{S}^{E_\infty}\cap B_1(0^N)\right)\le N-8\, .
\end{equation}

\end{proof}

\begin{remark}
One of the key steps in the proof above is the fact that limits of singular points of perimeter minimizers where the blow-up of the ambient is Euclidean are singular points. In the smooth setting the second assumption is always verified, but in the non smooth setting this is a non trivial requirement and the example presented in \autoref{rm:Morgan} shows that it is necessary in order for the statement to hold. In particular, without further assumptions it is not true that limits of singular boundary points of a perimeter minimizer are singular boundary points of the limit.
\end{remark}

We aim at obtaining a sharp dimension bound for the singular set of local perimeter minimizers $\dim_H(\mathcal{S}^E)\le N-3$ within our framework. To this aim, it will be necessary to consider also the intersection of the minimal boundary with the ambient singular set $\partial E\cap\mathcal{S}(X)$.

In order to obtain the sharp dimension bound for the singular set in this setting, there is a key additional difficulty with respect to the classical case. Indeed it is not clear whether a monotonicity formula holds in this generality, therefore we do not know if any blow-up of a local perimeter minimizer is a cone. In order to circumvent this difficulty, following the classical pattern of the dimension reduction, we will need first to iterate blow-ups to reduce to the situation where the ambient is a cone of the form $\setR^{N-2}\times C(\mathbb{S}^1_r)$, where $\mathbb{S}^1_r$ is a circle, and then to perform the dimension reduction again in this simplified setting (where a monotonicity formula holds). 
\medskip

We will rely on some classical tools of Geometric Measure Theory. The first one is a monotonicity formula for perimeter minimizers inside cones, whose proof can be obtained as in the classical case, see \cite[Theorem 9.3]{Morgan00}, \cite[Theorem 5.4.3]{Federer69} and \cite{Morgan02}. 

\begin{theorem}\label{thm:monocone}
Let $(M,g)$ be a smooth Riemannian manifold of dimension $k\ge 1$ and with $\Ric\ge k-1$. Let $(X,\dist,\haus^N):=C(M)\times \setR^{N-k-1}$ be the product of the metric measure cone $C(M)$ of tip $\{o\}$, with an $(N-k-1)$-dimensional Euclidean factor.  Let $p\in \{o\}\times \setR^{N-k-1}$ and let $E\subset X$ be perimeter minimizing in $B_2(p)\subset X$. Then the ratio 
\begin{equation}
(0,1)\ni r\mapsto \frac{\Per_E(B_r(p))}{r^{N-1}} \quad \text{is increasing.}
\end{equation} 
Moreover, if the perimeter ratio is constant in $(0,2)$ then $E$ is a cone with vertex $p$ inside $B_1(p)$.
\end{theorem}

The second tool is an elementary non existence result for entire local perimeter minimizing cones passing through the tip inside non flat two dimensional cones, whose proof is well known, see \cite{Morgan02} and references therein.

\begin{proposition}\label{thm:2dimrigidity}
Let $N\ge 2$ be a given natural number. Let $0< r\le 1$ and let $\mathcal{C}$ be the metric measure cone $C(\mathbb{S}^1_r)\times\setR^{N-2}$ with canonical structure. Let $F:=C(I)\times \setR^{N-2}\subset \mathcal{C}$, where $I\subset \mathbb{S}^1_r$ is a set of finite perimeter. Then $F$ is a local perimeter minimizer if and only if $r=1$ (i.e. $\mathcal{C}=\setR^N$) and $I$ is the half-circle $[0,\pi]\subset \mathbb{S}^1_1$, up to isometry (i.e. $F\subset \setR^N$ is a half-space).
\end{proposition}

Notice that to pass from the case $N=2$ treated in \cite{Morgan02} to the case of $N\ge 3$ it is sufficient to rely on a slight modification of \cite[Lemma 28.13]{Maggi12} to drop the dimension.
\medskip


\begin{proposition}\label{prop:regincones}
Let $N\ge 2$ be a given natural number. Let $0< r\le 1$ and let $\mathcal{C}$ be the metric measure cone $\setR^{N-2}\times C(\mathbb{S}^1_r)$ with canonical structure and set of tips $\setR^{N-2}\times\{o\}$. Let $G\subset \mathcal{C}$ be any entire local perimeter minimizer. Then
\begin{equation}\label{eq:hausdimestcon}
\dim_H\left(\mathcal{S}^G\right)\le N-3\, . 
\end{equation}
\end{proposition}

\begin{proof}
We argue via dimension reduction, reducing to the situation where \autoref{thm:2dimrigidity} can be applied.\\ Let us suppose without loss of generality that $r<1$ (i.e. the cone is singular). If $r=1$ the statement follows from the classical Euclidean regularity theory.
\medskip

\textbf{Step 1.} We claim that any  blow-up of $G\subset \mathcal{C}$ at a point $x\in\mathcal{C}$ is either a minimal cone inside $\mathcal{C}$ if $x\in \setR^{N-2}\times \{o\}$ is an ambient singular point, or a minimal cone in $\setR^N$ if $x\in\mathcal{C}\setminus\left( \setR^{N-2}\times \{o\}\right)$ is an ambient regular point.\\
In order to check this statement it is sufficient to observe that the monotonicity formula
\begin{equation}
r\mapsto \frac{\Per_G(B_r(x))}{r^{N-1}}\, \quad\text{is non decreasing on $0<r<r_x$}
\end{equation}
holds for any $x\in\partial G$. Indeed, if $x\in \setR^{N-2}\times \{o\}$ is a vertex, this follows from \autoref{thm:monocone} (and we can take $r_x=\infty$ actually). If $x$ is a regular point of $\mathcal{C}$, the monotonicity formula follows from the fact that $\mathcal{C}$ is isometric to a (flat) Euclidean ball in a neighbourhood of $x$.\\
The fact that blow-ups are always cones follows then from a classical argument, thanks to the uniform perimeter density bound \eqref{cor:perest} and the rigidity in the monotonicity formula on $\mathcal{C}$ and $\setR^N$.
\medskip

\textbf{Step 2.} Let us assume by contradiction that \eqref{eq:hausdimestcon} fails. Then, arguing as in the proof of \autoref{thm:sharpdimreg} (see in particular \eqref{eq:densest}) we can find $\eta>0$ such that $\haus^{N-3+\eta}(\partial G\cap\mathcal{S}(\mathcal{C}))>0$ (notice that $\haus^{N-3+\eta}(\mathcal{S}^G\cap\mathcal{R}(\mathcal{C}))=0$, by \autoref{thm:sharpdimreg}). Therefore, there exist $x\in \partial G\cap\mathcal{S}(\mathcal{C})=\partial G\cap \setR^{N-2}\times \{o\}$ and a sequence $r_i\downarrow 0$ such that 
\begin{equation}\label{eq:lowdbsPre}
\limsup_{i\to \infty}\frac{\haus^{N-3+\eta}_{\infty}\left(\partial G\cap\mathcal{S}(\mathcal{C})\cap B_{r_i}(x)\right)}{r_i^{N-3+\eta}}>0\, .
\end{equation}  
By Step 1, we can find a subsequence of $(r_i)$, that we do not relabel, such that \eqref{eq:lowdbsPre} holds and the blow-up of $G$ along the sequence $r_i$ is an entire local perimeter minimizing cone $G^1$ inside a metric cone $\mathcal{C}$ (which is the blow-up of $\mathcal{C}$ at any point $x\in\mathcal{S}(\mathcal{C})$) with tip $o$. Moreover, it is easily seen that any limit of points $x_i\in \partial G\cap \mathcal{S}(\mathcal{C})\cap B_{r_i}(x)$ along this converging sequence belongs to $\partial G^1\cap \mathcal{S}(\mathcal{C})\cap B_1(o)$. Hence, the upper semicontinuity of the pre-Hausdorff measure implies 
\begin{equation}\label{eq:mb}
\haus^{N-3+\eta}_{\infty}\left(\partial G^1\cap \mathcal{S}(\mathcal{C})\cap B_1(o)\right)\, >\, 0\, ,
\end{equation}
that yields in turn 
\begin{equation}\label{eq:hausdimbound}
\haus^{N-3+\eta}\left(\partial G^1\cap \mathcal{S}(\mathcal{C})\cap B_1(o)\right)\, >\, 0\, .
\end{equation}
Let us write $G^1=\setR^k\times C(B^1)$, where $C(B^1)\subset \setR^{N-2-k}\times C(\mathbb{S}^1_r)$ is an entire local perimeter minimizing cone.\\ 
We claim that, after iterating a finite number of times the construction above, it is possible to take $k=N-2$. Indeed, if we suppose that $k\le N-3$, then we obtain $\haus^{N-3+\eta}\left(\left(\partial G^1\setminus \setR^k\times \{o\}\right)\cap\mathcal{S}(\mathcal{C})\right)>0$, by \eqref{eq:hausdimbound}. 

In particular, there exist $z\in \left(\partial G^1\cap\mathcal{S}(\mathcal{C})\right)\setminus \left(\setR^k\times \{o\}\right)$ and a sequence $r_j\downarrow 0$ such that 
\begin{equation}\label{eq:lowdbsPre2}
\limsup_{j\to \infty}\frac{\haus^{N-3+\eta}_{\infty}\left(\partial G^1\cap\mathcal{S}(\mathcal{C})\cap B_{r_j}(z)\right)}{r_j^{N-3+\eta}}>0\, .
\end{equation}  
Up to extraction of a subsequence, that we do not relabel, we find that a blow-up of $G^1$ at $z$ along the sequence $r_j\downarrow 0$ is an entire local perimeter minimizing cone of the form $G^2=\setR^{k+1}\times C(B^2)$ such that
\begin{equation}
\dim_H\left(\partial G^2\cap \mathcal{S}(\mathcal{C})\right)\, >\, N-3\, .
\end{equation}
This is due to Step 1 and to the fact that $G^1$ splits off a factor $\setR^k$, it is a cone, and we chose a point  $z\notin \setR^k\times\{o\}$ as base point for the blow-up. The additional splitting of $G^2$ can be justified with the very same arguments of the Euclidean case, we refer for instance to \cite[Theorem 28.11, Lemma 28.12, Lemma 28.13]{Maggi12} whose statements and proofs work \textit{mutatis mutandis} also in our setting.
\medskip

\textbf{Step 3.} The outcome of the previous two steps is that if \eqref{eq:hausdimestcon} fails, then there exists an entire local perimeter minimizing cone of the form $G=\setR^{N-2}\times C(B)\subset\mathcal{C}$. This is in contradiction with \autoref{thm:2dimrigidity}.
\end{proof}

\begin{remark}\label{rem:ExsharpCodim3}
The Hausdorff dimension estimate \eqref{eq:hausdimestcon} above is sharp. This is easily verified by considering as entire local perimeter minimizer the set $G:=\{x>0\}\subset \setR\times C(\mathbb{S}^1_r)$, where $0<r<1$ and $x\in\setR$ denotes the coordinate of the $\setR$ factor. Then $\partial G=\{0\}\times C(\mathbb{S}^1_r)$ which has one singular point $p=(0,o)$. Therefore $\haus^0(\mathcal{S}^G)=1$.
\end{remark}


\begin{theorem}\label{thm:sharpdimsing}
Let $(X,\dist,\haus^N)$ be an $\RCD(K,N)$ metric measure space, let $\Omega\subset X$ be an open domain such that $\Omega\cap \partial X=\emptyset$ and let $E\subset X$ be a set of locally finite perimeter that is locally perimeter minimizing in $\Omega$. Then 
\begin{equation}\label{eq:dimbdsing}
\dim_H\left(\mathcal{S}^E \cap \Omega\right)\le N-3\, .
\end{equation}
\end{theorem}

\begin{proof}
The strategy of the proof is a refinement of the one  of \autoref{thm:sharpdimreg}.\\
To simplify the notation, we assume throughout the proof that $E$ is an entire local perimeter minimizer. Moreover, we assume that $N\ge 3$ and $K\geq -(N-1)$. The case $K<-(N-1)$ can be reduced to $K=-(N-1)$ by scaling of the distance. The case $N=1$ is elementary and the case $N=2$ can be treated with a simpler variant of the argument presented below.
\medskip

\textbf{Step 1.} Reduction to perimeter minimizers inside cones.\\ 
We aim to show via blow-up that if \eqref{eq:dimbdsing} fails for some local perimeter minimizer on some $\RCD(K,N)$ metric measure space $(X,\dist,\haus^N)$, then it fails also for an entire perimeter minimizer inside a metric measure cone.\\
Notice that $\mathcal{S}^E\cap \mathcal{S}(X)=\partial E\cap \mathcal{S}(X)$ by the very definition of $\mathcal{S}^E$. Hence, if \eqref{eq:dimbdsing} fails, then $\dim_H(\partial E\cap\mathcal{S}(X))>N-3$, by \autoref{thm:sharpdimreg}. In particular, there exists $\eps>0$ such that $\dim_H(\partial E\cap\mathcal{S}_{\eps}(X))>N-3$, where $\mathcal{S}_{\eps}(X)$ is the quantitative $\eps$-singular set of $(X,\dist,\haus^N)$ 
defined by
\begin{equation*}
\mathcal{S}_{\eps}(X):= \{x\in X\, :\, \dist_{GH}(B_r(x),B_r(0^{N}))\ge \eps r\, , \text{for all $r\in (0,1)$}\}\, .
\end{equation*}
Indeed, by a well known argument (involving Bishop-Gromov volume monotonicity, volume convergence and volume rigidity; see for instance the proof of \cite[Theorem 3.1]{KapovitchMondino21} after \cite{CheegerColding97}), it is easy to check that $\mathcal{S}(X)= \bigcup_{\eps>0} \mathcal{S}_{\eps}(X)$.\\
Arguing as in the proof of \autoref{thm:sharpdimreg} (see in particular \eqref{eq:densest}) we can find $\eta>0$ such that $\haus^{N-3+\eta}(\partial E\cap\mathcal{S}_{\eps}(X))>0$. Therefore, there exist $x\in \partial E\cap\mathcal{S}_{\eps}(X)$ and a sequence $r_i\downarrow 0$ such that 
\begin{equation}\label{eq:lowdbs}
\limsup_{i\to \infty}\frac{\haus^{N-3+\eta}_{\infty}\left(\partial E\cap\mathcal{S}_{\eps}(X)\cap B_{r_i}(x)\right)}{r_i^{N-3+\eta}}>0\, .
\end{equation}  
Applying \autoref{thm:closurecompactnesstheorem}, we can find a subsequence of $(r_i)$, that we do not relabel, such that \eqref{eq:lowdbs} holds and the blow-up of $E$ along the sequence $r_i$ is an entire local perimeter minimizer $F$ inside a metric cone $C(Z)$ with tip $p$, for some $\RCD(N-2,N-1)$ metric measure space $(Z,\dist_Z,\haus^{N-1})$. Moreover, it is easily seen that any limit of points $x_i\in \partial E\cap \mathcal{S}_{\eps}(X)\cap B_{r_i}(x)$ along this converging sequence belongs to $\partial F\cap \mathcal{S}_{\eps}(C(Z))\cap B_1(p)$. Hence, the upper semicontinuity of the pre-Hausdorff measure implies 
\begin{equation}\label{eq:mb}
\haus^{N-3+\eta}_{\infty}\left(\partial F\cap \mathcal{S}_{\eps}(C(Z))\cap B_1(p)\right)\, >\, 0\, ,
\end{equation}
which yields $\dim_H(\mathcal{S}^F)>N-3$, as we claimed.
\medskip

\textbf{Step 2.} Dimension reduction.\\
In Step 1, we found an entire local perimeter minimizer $F\subset C(Z)$, where $C(Z)$ is a metric measure cone. Let us consider the maximal Euclidean factor $\setR^k$ split off by $C(Z)$ and write $C(Z)=\setR^k\times C(W)$ for some $0\le k\le N-2$ and some $\RCD(N-k-2,N-k-1)$ metric measure space $(W,\dist_W,\haus^{N-k-1})$. 
Arguing inductively, we wish to prove that it is possible to assume that $k=N-2$ iterating the construction of Step 1.\\ 
Indeed, let us suppose that $k\le N-3$. Then by \eqref{eq:mb} there exists a set of singular points of $F$ with positive $\haus^{N-3+\eta}_{\infty}$ pre-Hausdorff measure not contained $\setR^k\times\{p\}$. Iterating the construction of Step 1 with a base point $y\notin \left(\setR^{k}\times\{p\}\right)$ such that 
\begin{equation}
\limsup_{i\to \infty}\frac{\haus^{N-3+\eta}_{\infty}\left(\partial F\cap\mathcal{S}_{\eps}(C(Z))\cap B_{r_i}(y)\right)}{r_i^{N-3+\eta}}>0\, , 
\end{equation}
we obtain that, up to extraction of a subsequence that we do not relabel, the blow-up of $F$ at $y$ is an entire local perimeter minimizer $G\subset \setR^{k+1}\times C(V)$, where $(V,\dist_V,\haus^{N-k-1})$ is an $\RCD(N-k-2,N-k-1)$ metric measure space. Indeed, the blow-up $G$ is an entire local perimeter minimizer by the usual stability \autoref{thm:closurecompactnesstheorem}. Moreover, the fact that the ambient space splits an additional Euclidean factor follows by the choice of base point $y\notin \left(\setR^{k}\times\{p\}\right)$ (and the fact that $C(Z)$ is a cone), via the splitting theorem \cite{Gigli13}. 
\medskip

\textbf{Step 3.} Conclusion.\\
The outcome of the previous two steps is that, if \eqref{eq:dimbdsing} fails for a local perimeter minimizer $E\subset X$, where $(X,\dist,\haus^N)$ is an $\RCD(K,N)$ m.m.s. (without boundary), then it fails for an entire perimeter minimizer $F\subset \setR^{N-2}\times C(\mathbb{S}^1_r)$, where $0<r\le 1$. However, this would contradict \autoref{prop:regincones} and the proof is complete.

\end{proof}

In the next statement,  obtained combining \autoref{thm:RegSetTop},  \autoref{thm:sharpdimreg}  and \autoref{thm:sharpdimsing}, we summarize the main regularity results of the present section.

\begin{theorem}\label{thm:SummaryOa}
Let $(X,\dist,\haus^N)$ be an $\RCD(K,N)$ metric measure space. Let $E\subset X$ be a set of locally finite perimeter. Assume that $E$ is perimeter minimizing in $B_2(x)\subset X$ and $B_2(x)\cap \partial X=\emptyset$. Then for any $\alpha\in (0,1)$ there exists a relatively open set $O_{\alpha}\subset \partial E\cap B_1(x)$ such that:
\begin{itemize}
\item $O_{\alpha}$ is $\alpha$-bi-H\"older homeomorphic to a smooth open $(N-1)$-dimensional manifold;
\item $\mathcal{R}^{E}\subset O_{\alpha}$ and 
\begin{align}
\dim_H\big( (\partial E \setminus  \mathcal{R}^{E} ) \cap \mathcal{R}(X) \big)& \le N-8\, , \label{eq:dimHR}  \\
\dim_H\big( (\partial E \setminus  \mathcal{R}^{E}) \cap \mathcal{S}(X) \big)& \le N-3\,. \label{eq:dimHS}
\end{align}
\end{itemize}
\end{theorem}

\begin{remark}[Sharpness of \autoref{thm:SummaryOa} and a conjecture]\label{rem:sharpThmSum}
Both the Hausdorff codimension bounds \eqref{eq:dimHR} and \eqref{eq:dimHS} are sharp:
\begin{itemize}
\item  \eqref{eq:dimHR} is sharp already in $\setR^{N}$, by the classical example of Simons' cone $C_{S}\subset \setR^{8}$;
\item the sharpness of \eqref{eq:dimHS} was discussed in \autoref{rem:ExsharpCodim3}.
\end{itemize}

Since $\mathcal{R}^{E}\subset O_{\alpha}$, the bounds \eqref{eq:dimHR}-\eqref{eq:dimHS} of course imply
\begin{align}
\dim_H\big( (\partial E \setminus O_{\alpha}) \cap \mathcal{R}(X) \big)& \le N-8\, , \label{eq:dimHROa}  \\
\dim_H\big( (\partial E \setminus O_{\alpha}) \cap \mathcal{S}(X) \big)& \le N-3\,. \label{eq:dimHSOa}
\end{align}
Note that \eqref{eq:dimHROa} is sharp already in  $\setR^{N}$, by the example of the Simons' cone $C_{S}\subset \setR^{8}$: indeed,  for any $\alpha\in (0,1)$, it holds that $O_{\alpha}=C_{S}\setminus\{0^{8}\}$, so that $\dim_H\big( C_{S} \setminus O_{\alpha} \big)=0$.
\\ Instead, we conjecture  that the optimal dimension bound for the topologically regular part of $\partial E$  contained in the ambient singular set  is 
\[
\dim_H\big( (\partial E \setminus O_{\alpha}) \cap \mathcal{S}(X) \big) \le N-4\,. 
\]
Note that  ambient Hausdorff co-dimension 4 would be sharp, from the example given by $E:=C(\mathbb{RP}^{2})\times [0,\infty)\subset C(\mathbb{RP}^{2}) \times \setR=:X$.
\end{remark}

\subsection{Quantitative estimates for singular sets of minimal boundaries}\label{SubSec:QuantEst}

Our goal is to obtain Minkowski content estimates for the singular sets of boundaries of locally perimeter minimizing sets in our context, in analogy with the Euclidean theory \cite{CheegerNaber13b,NaberValtorta20} and with the Minkowski estimates for the quantitative singular sets of non collapsed Ricci limit spaces \cite{CheegerNaber13,CheegerJiangNaber21} and $\RCD$ spaces \cite{AntonelliBrueSemola19}.
\medskip

The  strategy that we adopt has been partly inspired by \cite{CaffarelliCordoba93}, which proposed an alternative approach to the regularity theory of locally perimeter minimizing boundaries in the Euclidean framework. A key additional difficulty in our setting, besides the fact that the spaces are not smooth, is that they are curved (and we aim to an \textit{effective} regularity theory, i.e. without the dependence on flatness parameters such as the injectivity radius). Therefore we will need to control at the same time the regularity of the space (with constants only depending on the  Ricci curvature and volume lower bounds) and the regularity of the minimal boundary inside it.
\medskip


\begin{definition}\label{def:quantreg}
Let $(X,\dist,\haus^N)$ be an $\RCD(K,N)$ space.  Let $\eta>0$ and $r\in (0,1)$ be fixed. The quantitative regular set $\mathcal{R}_{\eta,r}\subset X$   is defined by
\begin{equation*}
\mathcal{R}_{\eta,r}:=\{x \in X\, :\, \dist_{GH}(B_s(x),B_s(0^{N}))\le \eta s\, \quad\text{for any $0<s<r$}\}\, ,
\end{equation*}
where we indicated by $B_r(0^N)\subset\setR^N$ the Euclidean ball of radius $r$.
\end{definition}

\begin{definition}\label{def:quantsing}
Let $(X,\dist,\haus^N)$ be an $\RCD(K,N)$ space.  Let $\eta>0$ and $r\in (0,1)$ be fixed.  For any $0\le k\le N$, we shall denote 
\begin{align*}
\mathcal{S}^k_{\eta,r}:=\{x\in X&\, :\, \dist_{GH}(B_s(x),B_s(0^{k+1},z^*))\ge \eta s\, ,\\
&\quad\text{for any $\setR^{k+1}\times C(Z)$ and all $r<s<1$}\}\, ,
\end{align*}
where $B_{s}(0^{k+1},z^*)$ denotes the ball centred at the tip of a cone $\setR^{k+1}\times C(Z)$.
\end{definition}

In an analogous way we can deal with boundary points of local perimeter minimizers.

\begin{definition}[Quantitative singular sets for minimizing boundaries]\label{def:QuantSS}
Let $(X,\dist,\haus^N)$ be an $\RCD(K,N)$ metric measure space and let $E\subset X$ be a set of locally finite perimeter. Let us suppose that $E$ is locally perimeter minimizing inside a ball $B_{2}(x)\subset X$ and that $\partial X\cap B_2(x)=\emptyset$.\\ 
For any $\delta>0$, let $\mathcal{S}^E_{\delta}\subset \partial E$ be the quantitative singular set defined by
\begin{align*}
\mathcal{S}^E_{\delta}:=\{x\in\partial E\,& :\, 
\text{there exists no $r\in(0,1)$}\\ 
&\text{for which
$E\cap B_r(x)$ is $\delta$-regular at $x$}\}\, .
\end{align*}
Moreover, for any $r>0$ we shall denote
\begin{align*}
\mathcal{S}^E_{\delta,r} :=\{x\in\partial E\,& :\, \text{there exists no $s\in (r,1)$}\\ 
&\text{for which $E\cap B_s(x)$ is $\delta$-regular at $x$}\}\, 
\end{align*}
and 
\begin{equation*}
\overline{\mathcal{S}}^E_{\delta,r} :=\{x\in\partial E\, :\, \text{$E\cap B_r(x)$ is not $\delta$-regular at $x$}\}\, . 
\end{equation*}
\end{definition}

\begin{remark}
A direct consequence of the definitions is that
\begin{equation*}
\mathcal{S}^E_{\delta}=\cap _{r>0}\mathcal{S}^E_{\delta,r}=\cap_{i\in\setN}\mathcal{S}^E_{\delta,r_i}\,  \quad\text{and }\,\quad \mathcal{S}^E_{\delta,r}=\cap_{s>r}\overline{\mathcal{S}}^E_{\delta,s}\, ,
\end{equation*}
for any $\delta>0$ and for any sequence $r_i\downarrow 0$.
\end{remark}

\begin{definition}[Quantitative regular sets for minimal boundaries]
Let $(X,\dist, \haus^{N})$ and $E\subset X$ be as in \autoref{def:QuantSS}.
Given $\eta>0$ and $r>0$ we shall denote by
\begin{align*}
\mathcal{R}^E_{\eta,r}&:=\{x\in\partial E\, :\, E\cap B_s(x) \quad\text{is}\quad \text{ $\eta$-regular for any $s\in (0,r)$} \}\,  , \\
\mathcal{R}^E_{\eta}&:=\bigcup_{r>0}\mathcal{R}^E_{\eta,r}=\{x\in\partial E\, :\, \exists\,  r>0\, \text{ s.t. } E\cap B_r(x)\quad\text{is $\eta$-regular}\}\, ,
\end{align*}
the quantitative regular sets of the minimal boundary $\partial E$.
\end{definition}

\begin{remark}
Let us notice that 
\begin{equation}\label{eq:SESEdelta}
\mathcal{R}^E=\bigcap_{\eta>0}\mathcal{R}^E_{\eta}.
\end{equation}
This is a consequence of the very definitions and of the $\eps$-regularity \autoref{thm:epsregularity}. Also, $\mathcal{R}^E_{\eta}$ is open as soon as $\eta<\eta(N)$. Moreover,
\begin{equation*}
\mathcal{R}^E_{\eta}=\partial E\setminus\mathcal{S}^E_{\eta}\, ,\quad\text{for any $\eta>0$}\, .
\end{equation*}
\end{remark}

\begin{remark}
An inspection of the proof of \autoref{thm:epsregularity} shows that, if $\eta<\eta(N)$ and
\begin{equation*}
x\in \mathcal{R}\cap \mathcal{R}^E_{\eta}\, ,
\end{equation*}
then $x\in\mathcal{R}^E$ (cf. also with \autoref{cor:smooth}).
\end{remark}


\begin{theorem}\label{thm:contestsing}
For every $K\in \setR$ and $N\in [1,\infty)$ there exists $\delta_{K,N}>0$ with the following property.
Let $(X,\dist,\haus^N)$ be an $\RCD(K,N)$ metric measure space, $x\in X$ such that $\partial X\cap B_2(x)=\emptyset$, and  $E\subset X$  be a set of locally finite perimeter such that $E$ is perimeter minimizing in $B_2(x)$.
Then for any $0<\delta\in (0,\delta_{K,N})$ and for any $\gamma\in(0,1)$
there exist $C=C(K, N, \delta, \Per\big(E, B_{2}(x)),\gamma\big)>0$  and $r_{0}= r_0\big(K, N, \Per(E, B_{2}(x))\big)>0$ so that the following Minkowski content-type estimate on the quantitative singular set $\mathcal{S}^E_{\delta} \subset \partial E$ holds: 
\begin{equation}
\haus^N\big(T_r(\mathcal{S}^E_{\delta})\cap B_1(x)\big) \le C \,  r^{2-\gamma}\, \quad\text{for any $r\in (0,r_0)$ }, \label{eq:contbound}
\end{equation}
where $T_r$ denotes the tubular neighbourhood of radius $r>0$.\\

When $(X,\dist,\haus^N)$ is a non collapsed Ricci limit space or a finite dimensional Alexandrov space with curvature bounded below, the bounds \eqref{eq:contbound} can be strengthened to 
\begin{equation}
\haus^N(T_r(\mathcal{S}^E_{\delta})\cap B_1(x)) \le C \,  r^2 \, ,\quad\text{for any $r\in (0,r_0)$}\, .
\end{equation}

\end{theorem}

\begin{remark}
There is no direct implication between 
\eqref{eq:contbound} and the Hausdorff dimension estimate in \autoref{thm:sharpdimsing}. Indeed, while it is easily seen that \eqref{eq:contbound} is much stronger than the Hausdorff dimension estimate $\dim_H(\mathcal{S}^E)\le N-2$, it does not imply the sharp estimate $\dim_H(\mathcal{S}^E)\le N-3$. On the other hand, the Minkowski type estimate \eqref{eq:contbound} is not implied by any Hausdorff dimension estimate.
As an elementary example just to fix the ideas, note for instance that $\setQ^{N}\subset \setR^{N}$ has Hausdorff dimension $0$, but no Minkowski content-type estimate holds since any tubular neighbourhood of $\setQ^{N}$ is the whole space $\setR^{N}$.
\end{remark}

\begin{remark}
While the proof of the Hausdorff dimension bound $\dim_H(\mathcal{S}^E)\le N-3$ for local perimeter minimizers is independent of the mean curvature bounds proved in \autoref{sec:meancurv}, these play a key role in the proof of \autoref{thm:contestsing}.
\end{remark}

\autoref{thm:contestsing} will be proved at the end of the section. Below, we first establish a series of auxiliary results.
\medskip

Thanks to \autoref{lemma:tubneighbounds}, there exist constants $C_{K,N}, \bar{\delta}_{K,N}>0$ such that, if $(X,\dist,\haus^N)$ is an $\RCD(K,N)$ metric measure space and $E\subset X$ is a set of finite perimeter minimizing the perimeter in $B_2(x)\subset X$, then
\begin{equation}\label{eq:HNEdeltaBrho}
\haus^N\big((E^{\delta} \setminus \bar{E})\cap B_1(x)\big)\le C_{K,N} \, \Per(E, B_{2}(x))\, \delta \, ,\quad \text{for any } \delta\in (0, \bar{\delta}_{K,N})\, ,
\end{equation}
where we keep the notation $E^{\delta}$ for the $\delta$-enlargement of the set $E$, see \eqref{eq:defEt}.

\begin{corollary}\label{cor:cuttingneighminimal}
There exist constants $C_{K,N}, \bar{\delta}_{K,N}>0$ with the following property. Let $(X,\dist,\haus^N)$ be an $\RCD(K,N)$ metric measure space and $E\subset X$ is a set of finite perimeter minimizing the perimeter in $B_2(x)\subset X$. Then,  for any $\delta\in(0, \bar{\delta}_{K,N})$, there exists $\rho\in (1/2,1)$ such that
\begin{equation*}
\Per(B_{\rho}(x),E^{\delta}\setminus \bar{E})\le C_{K,N} \, \Per(E, B_2(x))\, \delta \, .
\end{equation*}
\end{corollary}

\begin{proof}
The conclusion follows from the estimate \eqref{eq:HNEdeltaBrho}.
Indeed, by the coarea formula \autoref{thm:coarea} applied to the distance function from $x$ we can bound
\begin{align*}
\int_{1/2}^1\Per \big(B_{s}(x),E^{\delta} \setminus \bar{E}\big)\di s &\le\haus^N\big((E^{\delta} \setminus \bar{E})\cap (B_1(x)\setminus B_{1/2}(x))\big)\\ 
&\le \haus^N\big((E^{\delta} \setminus \bar{E})\cap B_1(x)\big)\, \\
&   \leq  C_{K,N} \, \Per(E, B_{2}(x))\, \delta \, .
\end{align*}
\end{proof}

The so-called interior/exterior touching ball condition is a regularity property for domains $\Omega\subset X$. The interior one amounts to ask that at any given point $x\in\partial\Omega$, there exists a point $y\in\Omega$ and $r>0$ such that $\dist(x,y)=r$ and $B_r(y)\subset \Omega$.\\
When it holds uniformly, on a smooth Riemannian manifold, it yields a control on the second fundamental form of the boundary of the domain.
\medskip

Our next goal is to prove that minimal boundaries verify a weak interior/exterior touching ball condition in our setting. 

\begin{definition}[Set of touching points]\label{def:touch}
Let $(X,\dist,\meas)$ be an $\RCD(K,N)$ metric measure space and let $E\subset X$ be a local perimeter minimizer.
For any $\delta\in(0,\delta_0)$, we let $\mathcal{C}_{\delta}\subset\partial E$ be the set of interior and exterior touching points of balls of radius $\delta$, i.e.
\begin{align*}
\mathcal{C}_{\delta}:=\{x\in \partial E\, & :\, \text{there exist $B_{\delta}(x_1)\subset E$} \\
&\text{and $B_{\delta}(x_2)\subset E^c$ such that $x\in\partial B_{\delta}(x_1)\cap\partial B_{\delta}(x_2)$}\}\, .
\end{align*}
\end{definition}

\begin{proposition}\label{prop:touchingpoints}
There exist constants $\delta_{K,N}, \,C_{K,N}>0$ such that, for any $\RCD(K,N)$ metric measure space $(X,\dist,\haus^N)$, for any $x\in X$ and for any set of finite perimeter $E\subset X$ such that $E$ is perimeter minimizing in $B_2(x)\subset X$ the following holds: 
\begin{equation}\label{eq:touch}
\Per(E, B_{1/2}(x)\setminus \mathcal{C}_{\delta})\le  C_{K,N} \, \Per(E, B_2(x))\, \delta \, , \quad \text{for any } \delta\in (0,\delta_{K,N})\, .
\end{equation}

\end{proposition}

\begin{proof}
It is sufficient to estimate the size of the set $\mathcal{C}^e_{\delta}$ of touching points of exterior \textit{tangent} balls as in \eqref{eq:touch}. A similar argument will give the estimate for the size of the set of touching points of interior balls $\mathcal{C}^i_{\delta}$. Then the estimate for $\mathcal{C}_{\delta}$ will follow, since $\mathcal{C}_{\delta}=\mathcal{C}^i_{\delta}\cap \mathcal{C}^{e}_{\delta}$.
\medskip

Let us fix $\delta\in (0, \bar{\delta}_{K,N})$ and choose $\rho\in (1/2,1)$ given by \autoref{cor:cuttingneighminimal} above. We can also assume that $\Per(E^{\delta},\partial B_{\rho}(x))=0$ up to slightly perturb $\rho$. Observe that
$ E\cup (E^{\delta}\cap B_{\rho}(x)) $ is a compactly supported perturbation of $E$ in $B_2(x)$. Hence, by perimeter minimality, it holds:
\begin{equation*}
\Per(E,B_{2}(x))\le \Per(E\cup(E^{\delta}\cap B_{\rho}(x)),B_2(x))\, .
\end{equation*}
Therefore
\begin{align}\label{eq:toching}
\nonumber \Per(E,B_{\rho}(x))\le & \Per(E^{\delta},B_{\rho}(x))+\Per(B_{\rho}(x),E^{\delta} \setminus \bar{E})\\
\le &\Per(E^{\delta},B_{\rho}(x))+C_{K,N}  \, \Per(E, B_2(x))\, \delta\, .
\end{align}
Letting $\Gamma_{\delta}:=\partial E^{\delta}\cap \overline{B}_{\rho}(x)$ and $\Gamma_{\delta,\Sigma}$ be the set of touching points of minimizing geodesics from $\Gamma_{\delta}$ to $\Sigma$, we can estimate by \autoref{prop:tubecomparison}
\begin{equation}\label{eq:touching2}
\Per(E^{\delta},B_{\rho}(x))\le  \Per(E,\Gamma_{\delta,\Sigma}) + C_{K,N} \Per(E, B_{2}(x)) \, \delta \, .
\end{equation}
Notice that all the points in $\Gamma_{\delta,\Sigma}$ are touching points of exterior balls of radius $\delta$ on $\partial E$. Hence $\Gamma_{\delta,\Sigma}\subset \mathcal{C}_{\delta}^e$. Taking into account \eqref{eq:toching} and \eqref{eq:touching2}, then we can estimate
\begin{equation}\label{eq:estext}
\Per(E,B_{1/2}(x)\setminus \mathcal{C}^e_{\delta})\le C_{K,N}  \, \Per(E, B_2(x))\, \delta\, .
\end{equation}
Combining \eqref{eq:estext} with the analogous estimate valid for the set of touching points of interior balls, we get \eqref{eq:touch}.
\end{proof}

\begin{remark}
It is worth pointing out the following nontrivial consequence of \autoref{prop:touchingpoints}: if $E$ is locally perimeter minimizing, then $\Per_E$-a.e. point $x\in\partial E$ is an intermediate point of a minimizing geodesic along which the signed distance function from $E$ is realized (that would correspond to a perpendicular geodesic on a smooth Riemannian manifold). 
\end{remark}

Given a set of finite perimeter $E\subset \setR^n$, locally perimeter minimizing in an open domain, the existence of an interior and of an exterior touching balls at a given point $x\in\partial E$ are enough to guarantee the regularity of the boundary near to the touching point.\\
One way to verify this conclusion is to argue that the presence of both an interior and an exterior touching ball forces the tangent cone at the point to be flat and this is enough to guarantee regularity in a neighbourhood, as we already pointed out. 
\medskip

There is also a more quantitative approach, whose starting point is given by the following observation: there exists $C=C_n>0$ such that if $x\in\partial B_{C_n\lambda/\delta}(x_1)\cap\partial B_{C_n\lambda/\delta}(x_2)$, where $\delta,\lambda>0$ and $B_{C_n\lambda/\delta}(x_1)$ and $B_{C_n\lambda/\delta}(x_2)$ are an interior and an exterior touching ball respectively, then 
\begin{equation}\label{eq:eucl}
E\cap B_{\lambda}(y)\, \quad\text{is $\delta$-flat at $y$, for every $y$ such that $\abs{x-y}<\lambda$}\, .
\end{equation}

As \autoref{lemma:splitminimal} clearly illustrates, the existence of an interior and an exterior touching ball at a boundary point of a perimeter minimizing set is not enough to guarantee that the tangent is flat, nor that the boundary is regular in a neighbourhood of the point. 
\medskip

Even on a smooth Riemannian manifold, in order to guarantee $\delta$-regularity, the existence of interior/exterior touching balls needs to be combined with closeness (at the given scale) of the ball to the Euclidean ball, as shown in the next lemma.

\begin{lemma}\label{lemma:flatpropa}
There exists a constant $C=C_N>0$ such that the following holds. Let $(X,\dist,\haus^N)$ be an $\RCD(-(N-1),N)$ metric measure space and let $E\subset X$ be a set of finite perimeter that locally minimizes the perimeter in $B_{2}(x)\subset X$. Let $\lambda\in (0,1/2)$, $\delta>0$  and assume that: 
\begin{itemize}
\item[(i)] $x\in \partial E$ is a touching point of an interior and an exterior ball of radius $C_N\lambda/\delta$; 
\item[(ii)] $B_{\lambda}(x)$ is $\delta\lambda$-GH close to $B_{\lambda}(0)\subset\setR^N$.
\end{itemize}
Then, for any $y\in\partial E\cap B_{\lambda/2}(x)$, $E\cap B_{\lambda}(y)$ is $2\delta$-regular in $B_{\lambda}(y)$.
\end{lemma}

\begin{proof}
Condition (ii) guarantees scale invariant $\delta$-closeness, in GH sense, of $B_{\lambda}(x)$ to $B_{\lambda}(0)\subset \setR^N$ and of $B_{\lambda}(y)$ to $B_{\lambda}(0)\subset\setR^N$ for any $y\in\partial E\cap B_{\lambda/2}(x)$. The proof is then reduced to the Euclidean setting, where the existence of interior/exterior touching balls with radii $C_N\lambda/\delta$ guarantees $\delta$-flatness, as we remarked in \eqref{eq:eucl}.
\end{proof}

By \eqref{eq:eucl}, we can bound in an effective way the perimeter of the set where there are no interior/exterior touching balls of a given size. In order to guarantee that regularity of the ambient balls is in force at many locations and scales along $\partial E$, we will rely on the quantitative bounds for the singular strata of noncollapsed $\RCD$ spaces, obtained in \cite{AntonelliBrueSemola19} following the strategy of the previous \cite{CheegerNaber13}. 
\medskip

We will be focusing on codimension two singularities. With this aim, let us state an $\eps$-regularity result that follows from \cite{BrueNaberSemola20}.

\begin{theorem}[Boundary $\eps$-regularity]\label{thm:boundepsreg}
Let $K\in\setR$ and $1\leq N<\infty$ be fixed. Then there exists $\eps=\eps(K,N)>0$ such that the following holds. If $(X,\dist,\haus^N)$ is an $\RCD(K,N)$  space, $x\in X$ and $s\in (0,1)$ are such that 
\begin{equation*}
\dist_{GH}(B_{s}(x),B_{s}(0^{N-1},z^*))<\eps s\, ,
\end{equation*}
for some $B_{s}(0^{N-1},z^*)\subset \setR^{N-1}\times C(Z)$ and $\partial X\cap B_{s}(x)=\emptyset$, then 
\begin{equation*}
\dist_{GH}(B_s(x),B_s(0^{N}))<2\eps s\, ,
\end{equation*}
where $B_s(0^{N})\subset \setR^N$ is the Euclidean ball of dimension $N$.
\end{theorem}

\begin{proof}
There are only two possibilities for the cone $C(Z)$. Either $C(Z)=\setR^+$ or $C(Z)=\setR$, with the canonical metric measure structure, in both cases. The possibility that $C(Z)=\setR^+$ can be excluded thanks to \cite[Theorem 1.6]{BrueNaberSemola20}. Hence $C(Z)=\setR$ the ball is $2\eps$-regular, as we claimed.

\end{proof}

Thanks to \autoref{thm:boundepsreg}, we can easily check that if $(X,\dist,\haus^N)$ is an $\RCD(-(N-1),N)$ space and $B_{s}(x)\cap \partial X=\emptyset$ for some ball $B_s(x)\subset X$, then
\begin{equation}\label{eq:reph}
 B_s(x)\setminus \mathcal{S}^{N-2}_{\eta,r}=\{y\in B_s(x)\, :\, \dist_{GH}(B_t(y),B_t(0^{N}))<\eta t\, ,\text{ for any $t\in (0,r)$}\}\, ,
\end{equation}
for any $\eta\in(0,\eta(N))$.
\medskip

Let us recall the volume estimate for the quantitative singular stratum obtained in \cite{AntonelliBrueSemola19} (see \cite[Theorem 2.4]{AntonelliBrueSemola19} and the discussion below it) after \cite{CheegerNaber13}.

\begin{theorem}\label{thm:volesttubsing}
Let $K\in\setR$, $2\le N<\infty$, $1\le k\le N$, $v,\eta,\gamma>0$ be fixed. Then there exists a constant $c=c(K,N,k,\eta,v)>0$ such that the following holds. If $(X,\dist,\haus^N)$ is an $\RCD(K,N)$ space and 
\begin{equation*}
\frac{\haus^N(B_1(x))}{v_{K,N}(1)}\ge v\, ,
\end{equation*}
then for any $r\in (0,1/2)$ it holds
\begin{equation}\label{eq:volboundRCD}
\haus^N(B_{1/2}(x)\cap \mathcal{S}^{k}_{\eta,r})\le c(K,N,k,\eta,v,\gamma)r^{N-k-\gamma}\, .
\end{equation}
\end{theorem}

\begin{remark}\label{rm:volestRicci}
In \cite[Theorem 1.7]{CheegerJiangNaber21} it has been shown that for non collapsed Ricci limit spaces it is possible to replace $(N-k-\gamma)$ with $(N-k)$ at the exponent in \eqref{eq:volboundRCD} to obtain a much stronger estimate 
 \begin{equation}\label{eq:volboundRiccilimit}
\haus^N(B_{1/2}(x)\cap \mathcal{S}^{k}_{\eta,r})\le c(N,K,k,\eta,v)r^{N-k}\,, \quad\text{for any $r\in (0,1/2)$ }.
\end{equation}

The very same estimate \eqref{eq:volboundRiccilimit} was established in \cite[Corollary 1.5]{LiNaber20} for $N$-dimensional Alexandrov spaces with curvature bounded below by $K$.
\end{remark}

Relying on \autoref{thm:volesttubsing}, let us estimate the size of the intersection of the quantitative singular stratum $\mathcal{S}^{k}_{\eta,r}$ with the boundary of a locally perimeter minimizing set of finite perimeter.

\begin{proposition}\label{prop:quantstratalongmin}
Let $K\in\setR$, $2\le N<\infty$, $1\le k\le N$, $v,\eta, \gamma>0$ be fixed. Then there exists a constant $c=c(K,N,k,\eta,v)>0$ such that the following holds. If $(X,\dist,\haus^N)$ is an $\RCD(K,N)$ space such that
\begin{equation*}
\frac{\haus^N(B_1(x))}{v_{K,N}(1)}\ge v\, 
\end{equation*}
and $E\subset X$ is a set of finite perimeter which is perimeter minimizing in $B_{2}(x)$, then there exists $r_0=r_{0}(K,N)>0$ independent of $k,\,\eta$ and $\gamma$ such that
\begin{align}
\Per(E, B_{1/2}(x)\cap \mathcal{S}^{k}_{\eta,r}) &\le c(K,N,k,\eta,v) \, r^{N-k-1-\gamma}\, , \quad\text{for any $r\in (0,r_0)$}\, , \label{eq:singalongmin}\\
\Per(E, B_{1/2}(x)\setminus\mathcal{R}_{\eta, r}) &\le c(K,N,\eta,v)\, r^{1-\gamma}\, , \quad \qquad  \quad\text{ for any $r\in (0,r_0)$}\, . \label{eq:singalongmin2}
\end{align}
\end{proposition}

\begin{proof}
Let us consider a covering of $\partial E\cap B_{1/2}(x)\cap \mathcal{S}^{k}_{\eta,r}$ with balls $B_{r\eta}(x_i)$ such that the balls $B_{r\eta/5}(x_i)$ are disjoint, via a Vitali covering argument.\\ 
As shown for instance in \cite[equation (2.5)]{AntonelliBrueSemola19}, unwinding the definitions, one can check that
\begin{equation}\label{eq:incltub}
T_{\eta r}(\mathcal{S}^k_{2\eta,r})\subset\mathcal{S}^{k}_{\eta,r}\, ,
\end{equation}
where $T_{\eta r}$ denotes the tubular neighbourhood of radius $r\eta$.
Thus, we can estimate:
\begin{equation}\label{eq:1est}
\Per(E, B_{1/2}(x)\cap \mathcal{S}^{k}_{\eta,r})\le \sum_i\Per(E,B_{\eta r}(x_i))\le C \sum_i \frac{\haus^N(B_{\eta r}(x_i))}{\eta r}\, , 
\end{equation}
where the constant $C$ is given by \autoref{cor:perest}. \\
Relying on \eqref{eq:incltub}, \autoref{thm:volesttubsing} and the Vitali covering condition, we obtain
\begin{equation*}
\sum_i \frac{\haus^N(B_{\eta r}(x_i))}{\eta r}\le \frac{C}{\eta r}\haus^N(T_{\eta r}(\mathcal{S}^k_{2\eta,r}\cap B_{1/2}(x)))\le \frac{c}{\eta}r^{N-k-1-\gamma}\, ,
\end{equation*}
which gives \eqref{eq:singalongmin} when combined with \eqref{eq:1est}.
\\The estimate \eqref{eq:singalongmin2} follows from \eqref{eq:singalongmin}, thanks to \eqref{eq:reph}.
\end{proof}

\begin{remark}\label{rm:better L}
Relying on the observation in \autoref{rm:volestRicci}, in the case of non collapsed Ricci limit spaces and finite dimensional Alexandrov spaces with curvature bounded below, it is possible to strengthen \eqref{eq:singalongmin} and \eqref{eq:singalongmin2} to 
\begin{equation}\label{eq:singalongminRL}
\Per(E, B_{1/2}(x)\cap \mathcal{S}^{k}_{\eta,r}) \le c(K,N,k,\eta,v) \, r^{N-k-1}\, , 
\end{equation}
for any $r\in (0,r_0)$ and 
\begin{equation}\label{eq:singalongmin2RL}
\Per(E, B_{1/2}(x)\setminus\mathcal{R}_{\eta, r}) \le c(K,N,\eta,v)\, r\, , 
\end{equation}
for any $r\in (0,r_0)$.
\end{remark}

\begin{proof}[Proof of \autoref{thm:contestsing}]
We claim that for any $\delta\in (0,\delta_{K,N})$ there exists a constant 
$$C=C(K, N, \delta, \gamma, \Per(E, B_{2}(x))>0$$ such that for any $r\in(0,r_0)$ and any Vitali covering of $\overline{\mathcal{S}}^E_{\delta,r}\cap B_1(x)$ with balls $B_r(x_i)$ such that $x_{i}\in \overline{\mathcal{S}}^E_{\delta,r}$ and $B_{r/5}(x_i)$ are pairwise disjoint for $i=1,\dots, N(r)$, it holds
\begin{equation}\label{eq:estNr}
N(r)\le C \, r^{2-N-\gamma}\, .
\end{equation}
Indeed, for any ball $B_r(x_i)$ as above, it holds
\begin{equation*}
B_r(x_i)\cap \mathcal{R}_{\delta/2,2r}\cap \mathcal{C}_{C_{K,N}\frac{r}{\delta}}=\emptyset\, ,
\end{equation*}
where $\mathcal{C}_{C_{K,N}\frac{r}{\delta}}$ is the set of contact points of touching balls as in \autoref{def:touch}. This is a consequence of \autoref{lemma:flatpropa}: if by contradiction $y$ belongs to the intersection above, then $E$ is $\delta$-regular on $B_r(z)$ for any $z\in B_r(y)$. Hence $E$ is $\delta$-regular on $B_r(x_i)$, a contradiction.
\\ Since the balls $B_{r/5}(x_i)$ are disjoint, we can bound 
\begin{align*}
\sum_{i\le N(r)} \Per \big(E, B_{r/5}(x_i)\big)  \le &  \Per \bigg(E,   \bigcup_{i\le N(r)} B_{r}(x_i)\bigg)  \\
\le &\Per\big(E,B_1(x)\setminus(\mathcal{R}_{\delta/2,2r}\cap {\mathcal C}_{C_{K,N}\frac{r}{\delta}}) \big)\le  C \, r^{1-\gamma}\,  ,
\end{align*}
for some $C=C(K, N, \delta, \gamma, \Per(E, B_{2}(x))>0$, where the last inequality follows from \autoref{prop:touchingpoints} and \eqref{eq:singalongmin2}. 
By the Ahlfors regularity of the perimeter measure \autoref{cor:perest}, we easily get \eqref{eq:estNr}.

Notice that, for non collapsed Ricci limit spaces and finite dimensional Alexandrov spaces,  the estimate \eqref{eq:estNr} can be strengthened into 
\begin{equation}
N(r)\le C \, r^{2-N}\, .
\end{equation}
This is a consequence of the better Minkowski bounds obtained in \cite{CheegerJiangNaber21, LiNaber20} in such a setting, arguing as we did above, using \autoref{rm:volestRicci} and \autoref{rm:better L}.
\medskip

To conclude the proof, in all the cases of $\RCD(K,N)$ spaces,  non collapsed Ricci limit spaces and finite dimensional Alexandrov spaces, it is sufficient to rely on the Ahlfors regularity bound for $\haus^N$ and to recall that 
\begin{equation}
\mathcal{S}^E_{\delta,r}=\bigcap _{s>r}\overline{\mathcal{S}}^E_{\delta,s}\,
\end{equation}
and 
\begin{equation}
\mathcal{S}^E_{\delta}=\bigcap _{r>0}\mathcal{S}^E_{\delta,r}\, .
\end{equation}

\end{proof}

\begin{remark}\label{rem:RegNewSmooth}
If $(X,\dist,\haus^N)$ is a smooth Riemannian manifold equipped with its volume measure, then \eqref{eq:contbound} can be strengthened into 
\begin{equation}
\haus^N\big(T_r(\mathcal{S}^E_{\delta})\cap B_1(x)\big) \le C \,  r^{8}\, \quad\text{for any $r\in (0,r_0)$ }\, ,
\end{equation}
if we allow the constants $C$ and $r_0$ to depend on the norm of the full Riemann curvature tensor on $B_2(x)$ and on a lower bound on the injectivity radius on $B_2(x)$, as proved in \cite[Theorem 1.6]{NaberValtorta20}.\\
Since the constants in \eqref{eq:contbound} only depend on the dimension, the lower Ricci curvature bound and on the perimeter of $E$ on $B_1(x)$, our estimates are not encompassed by those in \cite{NaberValtorta20} even in the case of smooth manifolds.
\end{remark}

\appendix
\section{Laplacian bounds Vs mean curvature bounds: \\a  comparison with the classical literature } \label{subsec:othernotions}

The aim of this subsection is to put \autoref{thm:meancurvminimal1} and \autoref{thm:meancurvminimal2} into perspective. In particular, we wish to clarify why Laplacian bounds on the distance function can be understood as mean curvature bounds.
For this reason, we are going to present some mostly well known results about the distance function from minimal hypersurfaces on smooth Riemannian manifolds, focusing for simplicity on the non-negative Ricci curvature case.
\medskip

As we already remarked, the fact that the distance from a smooth minimal hypersurface is subharmonic in a manifold with non-negative Ricci curvature is classical. To the best of our knowledge, the first reference where this result is explicitly stated, even though without proof, is \cite{Wu79}. Therein, the Laplacian bound was understood in the viscosity sense. In subsequent contributions, such as \cite{PetersenWilhelm03} and \cite{ChoeFraser18}, superharmonicity of the distance was understood in the sense of barriers, following the seminal \cite{Calabi58,CheegerGromoll71}. 

\begin{theorem}\label{thm:smoothminimalimpldistsubh}
Let $(M^n,g)$ be a smooth Riemannian manifold  with non-negative Ricci curvature and let $\Sigma\subset M$ be a smooth hypersurface. Then $\boldsymbol{\Delta} \dist_{\Sigma}\le 0$ on $M\setminus \Sigma$ if and only if $\Sigma$ is minimal, in the sense that it has vanishing mean curvature.
\end{theorem}

\begin{proof}
We only give an indication of the argument, a complete proof of the implication from minimality to subharmonicity of the distance can be found for instance in \cite{ChoeFraser18}.
\medskip

Notice that the Laplacian of the distance from a smooth hypersurface coincides with its mean curvature along the hypersurface, thanks to a classical computation in Riemannian Geometry. One possible strategy to check subharmonicity of the distance is to observe that the singular part of the Laplacian has negative sign, in great generality. Then we can consider minimizing geodesics along which the distance to the hypersurface is realized. Along these rays, the vanishing mean curvature condition at the starting point propagates to nonnegativity of the Laplacian of the distance, thanks to the non-negative Ricci curvature condition.
\medskip

The converse implication, from subharmonicity of the distance to minimality, relies on the same principle, combined with the fact that $\dist_{\Sigma}$ is smooth on any side of $\Sigma$ locally in a neighbourhood of any point. In order to check that the mean curvature $H_{\Sigma}$ vanishes at a given $p\in\Sigma$, let us consider the minimizing geodesic $\gamma:(-\eps,\eps)\to M$ such that $\gamma(0)=p$ and $\gamma'(0)$ is perpendicular to $T_p\Sigma$. Then observe that, combining the superharmonicity of $\dist_{\Sigma}$ with the already mentioned connection between mean curvature and Laplacian of the distance,
\begin{equation*}
0\le -\lim_{t\uparrow 0}\Delta\dist_{\Sigma}(\gamma(t))=H_{\Sigma}(p)=\lim_{t\downarrow 0}\Delta\dist_{\Sigma}(\gamma(t))\le 0\, ,
\end{equation*}
hence $H_{\Sigma}(p)=0$.
\end{proof}


On smooth Riemannian manifolds with non-negative Ricci curvature, the distance from a minimal hypersurface is subharmonic even for certain minimal hypersurfaces that are not globally smooth. This is a key point for the sake of the applications, since minimal hypersurfaces that are built through variational arguments might be non smooth in ambient dimension greater than $8$.
\medskip

Notice that \autoref{thm:meancurvminimal1} already gives a substantial contribution in this direction. Indeed, we can cover at least all the minimal hypersurfaces that are locally boundaries of sets of locally minimal perimeter.\footnote{In particular it provides a different proof of the first implication in \autoref{thm:smoothminimalimpldistsubh} since, as we already mentioned, all smooth minimal hypersurfaces are locally boundaries of perimeter minimizing sets.} 
\medskip

Actually, the principle ``minimality implies subharmonicity of the distance function'' extends even to minimal hypersurfaces that are not necessarily locally boundaries.
\medskip

Let us introduce some terminology, following \cite{Zhou17} for this presentation. 

\begin{definition}\label{def:minvar}
Given a smooth Riemannian manifold $(M^n,g)$, a singular hypersurface with singular set of codimension no less than $k$ ($k<n-1$, $k\in\setN$) is a closed set $\overline{\Sigma}\subset M$ such that $\haus^{n-1}(\overline{\Sigma})<\infty$, where the \textit{regular} part $\mathcal{R}(\Sigma)$ is defined by
\begin{align*}
\mathcal{R}(\Sigma):=\{&x\in\overline{\Sigma}\,  :\, \overline{\Sigma}\\ 
&\quad\text{is a smooth embedded hypersurface in a neighbourhood of $x$}\}
\end{align*}
and $\mathcal{S}(\Sigma):=\overline{\Sigma}\setminus \mathcal{R}(\Sigma)$ is the \textit{singular} part which we assume to satisfy $\dim_H(\mathcal{S}(\Sigma))\le n+1-k$.

Given such a singular hypersurface, it represents an integral varifold, that we denote as $[\Sigma]$. We will say that $\Sigma$ is minimal if $[\Sigma]$ is a stationary varifold and the tangent cones of $[\Sigma]$ have all multiplicity one. 
\end{definition}

\begin{remark}
We recall that the minimality condition above is equivalent to the requirement that the mean curvature vanishes on $\mathcal{R}(\Sigma)$ and the density of $[\Sigma]$ is finite everywhere. Moreover, as shown in \cite[Lemma 6.3]{Zhou17}, minimal hypersurfaces produced through min-max are minimal according to \autoref{def:minvar} above.
\end{remark}
 
The next statement originates from an argument due to Gromov in his proof of the isoperimetric inequality \cite{Gromov07}.

\begin{theorem}\label{thm:minimalimpliesdistsubh}
Let $(M^n,g)$ be a smooth Riemannian manifold with non-negative Ricci curvature. Let $\overline{\Sigma}\subset X$ be minimal in the sense of \autoref{def:minvar}. Then $\dist_{\overline{\Sigma}}$ is subharmonic on $M\setminus\overline{\Sigma}$.
\end{theorem}

\begin{proof}
The proof is divided in two steps. The first is about controlling the mean curvature at footpoints of minimizing geodesics on the hypersurface. The second deals with the propagation of the mean curvature bound to obtain a Laplacian bound, as in previous arguments in this note.
\medskip

\textbf{Step 1.} As proved for instance in \cite[Lemma 2.1]{Zhou17} along the original argument due to Gromov, the following holds. For any $p\in M\setminus\overline{\Sigma}$, let $\gamma:[0,\dist(p,\overline{\Sigma})]\to M$ be a minimizing geodesic connecting $p$ to $\overline{\Sigma}$, and let $\gamma(0)=q$ be the footpoint of the geodesic on $\overline{\Sigma}$, then $q\in\mathcal{R}(\Sigma)$. 

Indeed, the geodesic sphere of radius $\dist(p,q)/2$ centred at $\gamma(\dist(p,q)/2)$ is a smooth hypersurface near to $q$ and $\overline{\Sigma}$ lies on one side of it. Since all tangent cones have multiplicity one, the tangent cone to $[\Sigma]$ at $q$ is unique and it is a hyperplane. Hence, by Allard's regularity theorem \cite{Allard72}, $\overline{\Sigma}$ is regular at $q$. Therefore, the mean curvature of $\Sigma$ is vanishing in a neighbourhood of $q$.

\medskip
\textbf{Step 2.} Let us propagate the information that the mean curvature is vanishing in the classical sense near to footpoints of minimizing geodesics to prove that $\dist_{\overline{\Sigma}}$ is subharmonic on $M\setminus\overline{\Sigma}$.

We can rely for instance on the localization technique to argue that it is sufficient to control the regular part of the Laplacian of $\dist_{\overline{\Sigma}}$ (see for instance  \cite[Theorem 1.3, Corollary 4.16]{CavallettiMondino20} and Step 2 in the proof of \autoref{thm:meancurvminimal1}). Then, to control the regular part, it is enough to observe that $\dist_{\overline{\Sigma}}$ is smooth near to initial points of rays in the localization (thanks to the smoothness of $\Sigma$ obtained in Step 1). Moreover the Laplacian of the distance is vanishing there, therefore it remains non-negative along the rays by the non-negative Ricci curvature assumption.
\end{proof}

\begin{remark}
The proof of \autoref{thm:minimalimpliesdistsubh} above works in particular for hypersurfaces that are locally boundaries of locally perimeter minimizing sets, once we appeal to the classical Euclidean regularity theory for local perimeter minimizers. In particular it provides a different proof of \autoref{thm:meancurvminimal1} for smooth Riemannian manifolds. However, the use of deep regularity theorems in Geometric Measure Theory, makes the extension of this strategy to non smooth ambient spaces unlikely, as already pointed out in \cite{Petrunin03}. The interest towards proofs of mean curvature bounds and regularity results for area minimizing surfaces not heavily relying on GMT tools was pointed out also in \cite{Gromov14a,Gromov14b}.
\end{remark}

\begin{remark}
As remarked in \cite{Savin07}, if $E\subset \setR^n$ is an open set and $\Delta \dist_{\partial E}\le 0$ locally in a neighbourhood of $\partial E$ and away from $\partial E$, then $\partial E$ satisfies the minimal surfaces equation in the viscosity sense. Indeed the signed distance from a smooth boundary is smooth in a neighbourhood of any point along the boundary, where its Laplacian corresponds to the mean curvature, as we pointed out in the proof of \autoref{thm:smoothminimalimpldistsubh} above. See also \cite{White16} for some arguments in the same spirit in the Riemannian framework.
\end{remark}

\end{document}